\documentclass[11pt]{memoir}

\setsecnumdepth{subsubsection}
\settocdepth{subsection}

\usepackage{hott}

\title{\textbf{\HUGE Classifying Types}}

\author{{\large Egbert Rijke}}

\date{{\large July 2018}}

\begin{document}

\frontmatter

\newgeometry{letterpaper,hmargin=1in,vmargin=1in} 
\pagenumbering{Roman}
\begin{titlingpage}
\maketitle
\end{titlingpage}
\restoregeometry
\pagenumbering{roman}

\tableofcontents*

\chapter{Introduction}

The study of homotopy theoretic phenomena in the language of type theory \cite{hottbook} is 
sometimes loosely called `synthetic homotopy theory' \cite{Brunerie16}. 
Homotopy theory in type theory \cite{Awodey12} is only one of the
many aspects of homotopy type theory, which also includes the study of the
set theoretic semantics (models of homotopy type theory and univalence in a
meta-theory of sets or categories \cite{Awodey14,AwodeyWarren,BezemCoquandHuber,KapulkinLeFanuLumsdaine,Shulman15,Voevodsky15}), type theoretic semantics (internal models of homotopy type
theory), and computational semantics \cite{AngiuliHarperWilson}, as well as the study of various questions
in the internal language of homotopy type theory which are not necessarily 
motivated by homotopy theory, or questions related to the development of
formalized libraries of mathematics based on homotopy type theory.
This thesis concerns the development of synthetic homotopy theory.

Homotopy type theory is based on Martin-L\"of's theory of dependent types \cite{MartinLof1984}, which
was developed during the 1970's and 1980's.
The novel additions of homotopy type theory are Voevodsky's univalence axiom \cite{Voevodsky06,Voevodsky10},
and higher inductive types \cite{Lumsdaine11Blog,Shulman11Blog,hottbook}. The univalence axiom characterizes the identity
type on the universe, and establishes the universe as an object classifier \cite{RijkeSpitters}.
Higher inductive types are a generalization of inductive types, in which both
point constructors (generators) and path constructors (relations) may be specified.
A simple class of higher inductive types, which includes most known higher inductive
types, are the homotopy pushouts. When the universe is assumed to be closed
under homotopy pushouts, it is also closed under homotopy coequalizers \cite{hottbook}, 
sequential colimits \cite{hottbook}, and propositional truncation \cite{vanDoorn2016}.
For instance, one way of obtaining the $n$-spheres \cite{Lumsdaine12Blog} using homotopy pushouts is by setting $\sphere{-1}\defeq\emptyt$ and
by inductively defining the $(n+1)$-sphere $\sphere{n+1}$ to be the pushout
of the span $\unit \leftarrow\sphere{n}\rightarrow\unit$. 
Then we can attach $(n+1)$-cells to a type $X$.
Let $f:(\sm{a:A}B(a))\to X$ be a family of attaching maps
for some $\sphere{n}$-bundle $B:A\to\mathrm{BAut}(\sphere{n})$, where $\mathrm{BAut}(\sphere{n})$ is the type $\sm{X:\UU}\brck{\eqv{\sphere{n}}{X}}$ of types merely equivalent to the $n$-sphere. We
attach the $(n+1)$-disks indexed by $A$ to $X$ by taking the homotopy pushout
\begin{equation*}
\begin{tikzcd}
\sm{a:A}B(a) \arrow[r,"f"] \arrow[d,swap,"\pi_1"]  & X \arrow[d,"\inr"] \\
A \arrow[r,swap,"\inl"] & P. \arrow[ul,phantom,very near start,"\ulcorner"]
\end{tikzcd}
\end{equation*}
In this thesis we study dependent type theory with univalent universes that are closed under homotopy pushouts, further developing the program set out in \cite{hottbook} on synthetic homotopy theory. We will assume in this dissertation that every type family is classified by a univalent universe, and that all universes are closed under homotopy pushouts, as well as under the usual type constructors including identity types, $\Pi$- and $\Sigma$-types, and a natural numbers object. The model in cubical sets by Coquand et al. \cite{BezemCoquandHuber} is a constructive model for this setup, although the fact that it is closed under homotopy pushouts is currently unpublished. Furthermore, we will exclusively work with objects that can be described either as (dependent) types, or as their terms. In other words, the objects of our study are all `formalizable' in the sense that they can be encoded in computer implementations of our setup of dependent type theory. Existing implementations of dependent type theory, supporting univalent universes and homotopy pushouts, with libraries developing synthetic homotopy theory, include the proof assistants Agda \cite{agda}, Coq \cite{Coq,hcoq}, and Lean \cite{vDvRB2017HoTTLean}, and parts of the material in this theses are formalized in each of them.

\section{Overview per chapter}
In \cref{chap:univalent} we will first establish notation, although we will mostly follow the notation from \cite{hottbook}\footnote{A difference of practice between \cite{hottbook} and this thesis is that \cite{hottbook} doesn't use the label `Proposition' for any results at all, whereas we use propositions for statements that might be of independent interest, but are not our main theorem. In this thesis, the label `Theorem' is reserved for the main results that were originally established in my thesis research, and is therefore used much more sparingly.}. Furthermore, we recall some of the most basic facts of homotopy type theory, and then we will show that the univalence axiom establishes any universe as an object classifier. 

In \cref{chap:descent} we will first recall the most basic properties of homotopy pushouts, and then we will proceed to prove the descent property for pushouts, using the univalence axiom. Note that the coherence problem of type theory plays a role in this chapter: we would much rather have shown that the descent property holds for any homotopy colimit, but this task requires the definition of an internal $\infty$-category. Some solace will be offered in \cref{chap:equifibrant}, where we will prove a descent property for any modality.

In \cref{chap:rcoeq} we will study reflexive graphs and, most importantly, reflexive coequalizers in type theory. We take interest in reflexive graphs, because the topos of reflexive graphs (over sets) is cohesive over the topos of sets. The left-most adjunction is the reflexive coequalizer which is left adjoint to the discrete functor. We will construct the reflexive coequalizer as a pushout, so it exists under our assumptions, and we know in some cases how to compute them into previously known operations. Here we run again into the limitation of homotopy type theory, because we cannot establish the type of reflexive graphs as the type of objects of an $\infty$-category, and neither can we establish the reflexive coequalizer as an $\infty$-functor. However, to state and prove that the reflexive coequalizer satisfies the universal property of the left adjoint of the discrete functor, we only need composition of reflexive graph morphisms, and the action on morphisms of the operation equipping a type with the structure of a discrete graph. The universal property is indeed sufficient for many purposes. We then introduce the notion of fibrations for reflexive graphs, and show that the fibrations are right orthogonal to the same maps as the discrete reflexive graphs: morphisms between representables. It follows that a morphism is a fibration if and only if it is cartesian. Then we proceed to prove the descent property of reflexive coequalizers. We also note that diagrams over reflexive graphs are just left fibrations of reflexive graphs, so we also obtain a descent theorem for diagrams over graphs. Of particular interest in the present work is the descent property for sequential colimits. The contents of this chapter are joint work with Bas Spitters.

In \cref{chap:image} we will show that any function factors as a surjective function followed by an embedding, even though we only assume that the universe is closed under pushouts and the basic type constructors. Of course, it is of essential importance here that the universe contains a natural numbers object. Our construction of the image of a map proceeds by iteratively taking the fiberwise join of a map with itself, so we call it the join construction. It follows from our construction that the image of a map from an essentially small type into a locally small type (notions that are explained in \cref{chap:univalent}), is again essentially small. The join construction can be used to construct the quotient of a type by a $\prop$-valued equivalence relation (i.e.~an equivalence relation in the usual sense), and it can be used to construct the Rezk completion of a pre-category. The construction of set-quotients includes the construction of the set truncation, and the construction of the Rezk completion includes the construction of the $1$-truncation, and of Eilenberg-Mac Lane spaces of the form $K(G,1)$ \cite{FinsterLicata}. Following \cite{FinsterLicata}, the Eilenberg Mac-Lane spaces $K(G,n)$ for $G$ abelian and $n\geq 2$ can be constructed once we have constructed the $k$-truncation for any $k\geq -2$. We note that Eilenberg-Mac Lane spaces are important classifying spaces in higher group theory, and so are the connected components of the universe (which we show to be essentially small). 

In \cref{chap:reflective} we consider general reflective subuniverses. Examples that we have at our disposal at this point are the $(-2)$-, $(-1)$-, $0$-, and $1$-truncations. Most examples of reflective subuniverses are obtained by localization at a family of maps. However, since the only assumed homotopy colimits are pushouts, we will construct localizations only in \cref{chap:compact}, and only of families between compact types, because we need more theory in order to establish the necessary basic results. Thus, in \cref{chap:reflective} we focus on general reflective subuniverses. We show that for any reflective subuniverse $L$, the subuniverse $L'$ of $L$-separated types (i.e.~types whose identity types are $L$-types) is again reflective. It follows at once that the subuniverse of $k$-truncated types is reflective, for any $k\geq -2$. Furthermore, we will study several classes of maps related to a reflective subuniverse. First of all, we study the $L$-equivalences, i.e.~the maps that become an equivalence by the functorial action of $L$, and second of all we study the $L$-connected maps, i.e.~ the maps with fibers that become trivial after applying $L$. Clearly, any $L$-connected map is also an $L$-equivalence, but the converse is one of the many characterizations of $L$ being lex given in \cite{RijkeShulmanSpitters}. 
Furthermore, we study modalities. One of the characterizations of modalities is as reflective subuniverses that are $\Sigma$-closed, but we provide three more equivalent definitions of modalities. One particularly important alternative definition is that of a stable orthogonal factorization system, i.e.~a pair $(\mathcal{L},\mathcal{R})$ of two classes of maps such that every map factors as an $\mathcal{L}$-map followed by a $\mathcal{R}$-map; the class $\mathcal{L}$ is left orthogonal to the class $\mathcal{R}$; and the pullback of an $\mathcal{L}$-map is again a $\mathcal{L}$-map. For any modality $\modal$, the stable orthogonal factorization system associated to it consists of the $\modal$-connected maps as the $\mathcal{L}$-maps, and the $\modal$-modal maps as the $\mathcal{R}$-maps. A final topic for this chapter is the notion of accessibility for reflective subuniverses and accessible modalities. 

The contents from this chapter are selected from \cite{RijkeShulmanSpitters} and \cite{ChristensenOpieRijkeScoccola}. I began to study reflective subuniverses with Mike Shulman and Bas Spitters, and continued to study them with my MRC teammates Morgan Opie and Luis Scoccola, under the lead of Dan Christensen.

In \cref{chap:equifibrant} we recall from \cite{WellenPhD} the notion of $\modal$-\'etale map for an arbitrary modality $\modal$. We prove a modal version of the descent theorem, which asserts that a maps into $\modal X$ from a $\modal$-modal type are equivalently described as \'etale maps into $X$. The $\modal$-\'etale maps form the right class of a second orthogonal factorization system associated to any modality: the \emph{reflective} factorization system. The left class of this factorization system is the class of $\modal$-equivalences. Using this factorization system we obtain that the universal cover of a type $X$ at a point $x_0$ is the left-right-factorization of the map $\unit\to X$. 
We then proceed to study the \'etale maps for the modality of discrete reflexive graphs. Note that the type of all reflexive graphs isn't exactly a universe, so it is not possible to directly apply our previous observations about $\modal$-\'etale maps. Nevertheless, most arguments are practically the same. Thus, we treat our section on $\modal$-\'etale maps as a blue-print for our section of $\Delta$-\'etale maps, and do most arguments a second time. Such is the current state of homotopy type theory. What we get out is a generalized flattening lemma, which states that for any morphism $f:\mathcal{B}\to\mathcal{A}$ factors uniquely as an $\Delta$-equivalence followed by a fibration of graphs. We use this to show that the loop space of the suspension of a pointed type $X$ is the free H-space $G$ with a base-point preserving map $X\to_\ast G$. The generalized flattening lemma also applies to diagrams over reflexive graphs, and in particular to sequential colimits. Moreover, the equifibrant replacement can be constructed by a telescope construction. We show in this chapter that sequential colimits commute with $\Sigma$- and identity types, and therefore also with pullbacks. In particular, the sequential colimit operation sends sequences of fiber sequences to fiber sequences. Moreover, we show that sequential colimits commute with $k$-truncation for all $k\geq -2$, from which it follows that sequential colimits commute with $\pi_k$ for any $k\geq 0$. We expect to be able to use these results also in showing that the spectrification of a pre-spectrum is indeed a spectrum, but we haven't done that yet.

The idea of a modal version of the descent theorem first arose in unpublished work on reflexive graphs with Bas Spitters, in the spring of 2016. However, I only learned about $\modal$-\'etale maps much later from Felix Wellen, and many of the results presented in \cref{sec:modal_descent} came out of a discussion I had with Felix Wellen and Mike Shulman, who also brought my attention to the reflective factorization system of a modality, of which the classical case is due to \cite{chk:reflocfact}. 
The material in \cref{sec:equifibrant_replacement} on the equifibrant replacement operation on reflexive graphs is joint work with Bas Spitters. 
The material in \cref{sec:seqcolim_eqf} on sequential colimits is joint work with Floris van Doorn and Kristina Sojakova \cite{DoornRijkeSojakova}, and all the results concerning sequential colimits are formalized in the proof assistant Lean. 

In \cref{chap:compact} we introduce the notion of (sequentially) compact types, in order to provide an application for the results in \cref{chap:equifibrant}. The most important basic result about compact types is that they are closed under pushouts, and in proving this fact we use that sequential colimits commute with pullbacks.
Our main purpose here, and the final main result of this dissertation, is to show that for any family $f$ of maps between compact types, the subuniverse of $f$-local types is reflective, thus providing a fairly large class of reflective subuniverses including types localized away from a prime \cite{ChristensenOpieRijkeScoccola}. It should be noted that all subuniverses of $f$-local types are reflective if enough higher inductive types are assumed. Furthermore, our result about $\omega$-compact types should in principle hold for $\kappa$-compact types for cardinals larger than $\omega$. However, we restrict to the case of $\omega$-compact types since we do not have a good theory of such larger cardinals available in homotopy type theory, while the natural numbers object is right there (by assumption).

\section{Acknowledgments\footnote{I gratefully acknowledge the support of the Air Force Office of Scientific Research through MURI grant FA9550-15-1-0053.}}
First and foremost, it is my pleasure to thank my advisor, professor Steve Awodey. I could not have done this PhD without his advice, support, and inspiration. I cherish the many pleasant discussions we had: over the blackboard, between us, with our many visitors, while visiting other places, or over a beer somewhere in Pittsburgh. I am very grateful for the guidance you offered in selecting a worthwhile research topic, and I feel honored to have had the privilege to work with you.

Then I would like to express my gratitude to the members of my committee: Jeremy Avigad for transmitting his enthusiasm for interactive theorem proving, and for his wisdom and unparalleled kindness; Ulrik Buchholz for his unabated energy to make large formalization projects possible, introducing me to many concepts of algebraic topology along the way; and Michael Shulman for shining a light on many beautiful subjects related to higher category theory, that would otherwise have remained obscure for me.

I would like to thank my other collaborators, the people I had projects with:
Simon Boulier, Dan Christensen, Floris van Doorn, Jonas Frey, Morgan Opie, Luis Scoccola, Kristina Sojakova, Bas Spitters, Nicolas Tabareau, Felix Wellen, and Alexandra Yarosh. I thank you for all your inspiration, insights, energy, and creativity.

I would also like to thank all the people who have generously invited me for a visit. I am deeply indebted to Joachim Kock, who has hosted me for half a year in 2013-2014 at the \textit{Departament de Matemàtiques} of the Universitat Autònoma de Barcelona and helped me finding my PhD position at CMU; I am grateful to Vladimir Voevodsky for inviting me to the Institute for Advanced Study in March 2015; Nicolas Tabareau and his students Kevin Quirin and Simon Boulier for hosting me at INRIA Nantes during the summers of 2015 and 2016; Andrej Bauer for hosting me at the \textit{Fakulteta za Matematiko in Fiziko} of the University of Ljubljana in January 2016; Lars Birkedal and Bas Spitters for hosting me at the Department of Computer Science of Aarhus University in February 2016; Marie-Françoise Roy for inviting me to speak in the Effective Geometry and Algebra seminar at the \textit{Institut de Recherche Mathématiques de Rennes} in June 2016; Dan Christensen and his student Luis Scoccola for hosting me at the University of Western Ontario in London, Ontario in October 2017; Guillaume Brunerie for hosting me at the Institute for Advanced Study in November 2017; Charles Rezk and his student Nima Rasekh for inviting me to speak in the Topology Seminar at the University of Illinois at Urbana-Champaign in December 2017; Tom Hales for inviting me to speak at the Algebra, Combinatorics, and Geometry seminar at the University of Pittsburgh in April 2018; and Pieter Hofstra for inviting me to speak for the Canadian Mathematical Society in Fredericton, New Brunswick, in June 2018. Tak, dankjewel, thank you, merci, danke sch\"on, spasibo, \v dankujem, hvala, gracias.

I would like to thank Karin Arnds for lending me her cabin along the Allegheny River, where I wrote and rewrote significant parts of this dissertation.

Finally, I thank my dear parents Mieke and Reinier for their unconditional support, and my siblings Jeroen, Sophie, and Fleur, whom I didn't get to see as much as I would have liked because I chose to study overseas. I miss you and I dedicate this thesis to you.\\[2em]

\noindent \hfill Pittsburgh, \today

\mainmatter

\chapter{The object classifier}\label{chap:univalent}

In this chapter we establish notation, and we highlight the basic results concerning fiberwise transformations and fiberwise equivalences, which we will use for the descent theorems \cref{thm:descent,thm:rcoeq_cartesian}. Of particular importance are the following theorems:
\begin{enumerate}
\item The Fundamental Theorem of Identity Types (\cref{thm:id_fundamental}), which establishes that a type family $B$ over $A$ with $b:B(a)$ for a given $a:A$ is fiberwise equivalent to the identity type $a=x$ if and only if its total space is contractible. This result appears in \cite{hottbook} as Theorem 5.8.2, which contains other equivalent conditions as well.
\item \cref{thm:pb_fibequiv}, in which we establish that for any fiberwise map
\begin{equation*}
g:\prd{x:A}P(x)\to Q(f(x)),
\end{equation*}
the commuting square
\begin{equation*}
\begin{tikzcd}[column sep=large]
\sm{x:A}P(x) \arrow[r,"{\total[f]{g}}"] \arrow[d,swap,"\proj 1"] & \sm{y:B}Q(y) \arrow[d,"\proj 1"] \\
A \arrow[r,swap,"f"] & B,
\end{tikzcd}
\end{equation*}
where $\total[f]{g}$ is defined as $\lam{(x,p)}(f(x),g(x,y))$, is a pullback square if and only if $g$ is a fiberwise equivalence\index{fiberwise equivalence}. As a consequence, we obtain that a commuting square
\begin{equation*}
\begin{tikzcd}
A \arrow[d,swap,"f"] \arrow[r] & B \arrow[d,"g"] \\
X \arrow[r,swap,"h"] & Y
\end{tikzcd}
\end{equation*}
is a pullback square if and only if the induced fiberwise transformation
\begin{equation*}
\prd{x:X} \fib{f}{x}\to\fib{g}{h(x)}
\end{equation*}
is a fiberwise equivalence. Our main reference \cite{hottbook} does not present many results of homotopy pullbacks, although the material we present about homotopy pullbacks is surely well-known. The connection between pullbacks and fiberwise equivalences has an important role in the descent theorem\index{descent} in \cref{chap:descent}, which is why we devote a section to this result.
\item \cref{thm:classifier}, in which we establish that the universe is an object classifier. This result appears in \cite{RijkeSpitters} as Theorem 2.31, and in \cite{hottbook} as Theorem 4.8.4
\end{enumerate}

\section{Notation and preliminary results}

We work in Martin-L\"of dependent type theory with $\Pi$-types, $\Sigma$-types and cartesian products, coproducts $A+B$ equipped with $\inl:A\to A+B$ and $\inr:B\to A+B$ for any two types $A$ and $B$, an empty type $\emptyt$, a unit type $\unit$ equipped with $\ttt:\unit$, a type $\bool$ of booleans equipped with $\btrue,\bfalse:\bool$, a type $\N$ of natural numbers equipped with $0:\N$ and $\suc:\N\to\N$, and identity types. 

\begin{rmk}
As usual, we write $\idfunc[A]:A\to A$ for the \define{identity function} $\lam{x}x$ on $A$, and we write $g\circ f:A\to C$ for the \define{composite function} $\lam{x}g(f(x))$ of $f:A\to B$ and $g:B\to C$. For any two types $A$ and $B$, and any $b:B$, we write
\begin{equation*}
\const_b : A\to B
\end{equation*}
for the \define{constant function} $\lam{x}b$. Sometimes we also write $\lam{\nameless}b$ for the constant function.

In the case of $\Sigma$-types, the empty type $\emptyt$, and the unit type, we use the following notation to define functions by pattern-matching:
\begin{align*}
\lam{(x,y)}f(x,y) & : \prd{t:\sm{x:A}B(x)} P(t) \\
\lam{\ttt}y & : \prd{t:\unit}P(t).
\end{align*}
For instance, the first and second projection maps 
\begin{align*}
\proj 1 & : (\sm{x:A}B(x))\to A \\
\proj 2 & : \prd{p:\sm{x:A}B(x)}B(\proj 1)
\end{align*}
are defined as $\proj 1\defeq \lam{(x,y)}x$ and $\proj 2\defeq\lam{(x,y)}y$.
We use similar notation for definitions by iterated pattern-matching. For instance, given a dependent function $f:\prd{x:A}{y:B(x)}{z:C(x,y)} P((x,y),z)$ we obtain the function
\begin{equation*}
\lam{((x,y),z)} f(x,y,z): \prd{t:\sm{s:\sm{x:A}B(x)}C(s)}P(t).
\end{equation*}
\end{rmk}

Given a type $A$ in context $\Gamma$, the \define{identity type} of $A$ at $a:A$ is the inductive type family 
\begin{equation*}
\Gamma,x:A\vdash a =_A x~\mathrm{type}
\end{equation*}
with constructor
\begin{equation*}
\Gamma \vdash \refl{a} : a=_A a.
\end{equation*}
The induction principle for the identity type of $A$ at $a$ asserts that for any type family
\begin{equation*}
\Gamma,x:A,\alpha: a=_A x\vdash P(x,\alpha)~\mathrm{type}
\end{equation*}
there is a term
\begin{equation*}
\ind{a=} : P(a,\refl{a})\to \prd{x:A}{\alpha:a=_A x}P(x,\alpha)
\end{equation*}
in context $\Gamma$, satisfying the computation rule
\begin{equation*}
\ind{a=}(p,a,\refl{a})\jdeq p.
\end{equation*}

A term of type $a=_A x$ is also called an \define{identification} of $a$ with $x$, or a \define{path} from $a$ to $x$.
The induction principle for identity types is sometimes called \define{identification elimination} or \define{path induction}. Occasionally, we also write $\idtypevar{A}$ for the identity type on $A$. 

Moreover, we assume that there is a universe $\UU$ with a universal family $\mathrm{El}$ over $\UU$, that is closed under the type forming operations. For example, there is a map
\begin{equation*}
\check{\idtypevar{}}:\prd{A:\UU}\mathrm{El}(A)\to\mathrm{El}(A)\to\UU
\end{equation*}
satisfying
\begin{equation*}
\mathrm{El}(\check{\idtypevar{}}(A,x,y))\jdeq (x=_{\mathrm{El}(A)} y),
\end{equation*}
establishing that the universe is closed under identity types.

Given a type $A$ the \define{concatenation} operation
\begin{equation*}
\concat : \prd{x,y,z:A} (\id{x}{y})\to(\id{y}{z})\to (\id{x}{z})
\end{equation*}
is defined by $\concat(\refl{x},q)\defeq q$. We will usually write $\ct{p}{q}$ for $\concat(p,q)$. 
The concatenation operation satisfies the unit laws
\begin{align*}
\leftunit(p) & : \ct{\refl{x}}{p}=p \\
\rightunit(p) & : \ct{p}{\refl{y}}=p.
\end{align*}

The \define{inverse operation} 
\begin{equation*}
\invfunc:\prd{x,y:A} (x=y)\to (y=x)
\end{equation*}
is defined by $\invfunc(\refl{x})\defeq\refl{x}$. We will usually write $p^{-1}$ for $\invfunc(p)$.
The inverse operation satisfies the inverse laws
\begin{align*}
\leftinv(p) & : \ct{p^{-1}}{p} = \refl{y} \\
\rightinv(p) & : \ct{p}{p^{-1}} = \refl{x}.
\end{align*}

The \define{associativity operation}, which assigns to each $p:x=y$, $q:y=z$, and $r:z=w$ the \define{associator}
\begin{equation*}
\assoc(p,q,r) : \ct{(\ct{p}{q})}{r}=\ct{p}{(\ct{q}{r})}
\end{equation*}
is defined by $\assoc(\refl{x},q,r)\defeq \refl{\ct{q}{r}}$.

Given a map $f:A\to B$, the \define{action on paths} of $f$ is an operation
\begin{equation*}
\apfunc{f} : \prd{x,y:A} (\id{x}{y})\to(\id{f(x)}{f(y)})
\end{equation*}
defined by $\ap{f}{\refl{x}}\defeq\refl{f(x)}$. 
Moreover, there are operations
\begin{align*}
\apid_A & : \prd{x,y:A}{p:\id{x}{y}} \id{p}{\ap{\idfunc[A]}{p}} \\
\apcomp(f,g) & : \prd{x,y:A}{p:\id{x}{y}} \id{\ap{g}{\ap{f}{p}}}{\ap{g\circ f}{p}}
\end{align*}
defined by $\apid_A(\refl{x})\defeq \refl{\refl{x}}$ and $\apcomp(f,g,\refl{x})\jdeq \refl{\refl{g(f(x))}}$, respectively.
It can be shown easily that the action on paths of a map preserves the groupoid operations, and that the groupoid laws are also preserved.

\begin{defn}
Let $A$ be a type, and let $B$ be a type family over $A$. The \define{transport} operation
\begin{equation*}
\tr_B:\prd{x,y:A} (\id{x}{y})\to (B(x)\to B(y))
\end{equation*}
is defined by $\tr_B(\refl{x}) \defeq \idfunc[B(x)]$. 
\end{defn}

\begin{defn}\label{defn:apd}
Given a dependent function $f:\prd{a:A}B(a)$ and a path $p:\id{x}{y}$ in $A$, the \define{dependent action on paths}
\begin{equation*}
\apdfunc{f} : \prd{x,y:A}{p:x=y}\id{\tr_B(p,f(x))}{f(y)}
\end{equation*}
is defined by $\apd{f}{\refl{x}}\defeq \refl{f(x)}$.
\end{defn}

\begin{defn}
Let $f,g:\prd{x:A}P(x)$ be two dependent functions. The type $f\htpy g$ of \define{homotopies}\index{homotopy|textbf} from $f$ to $g$ is defined as
\begin{equation*}
f\htpy g \defeq \prd{x:A} f(x)=g(x).
\end{equation*}
\end{defn}

Commutativity of diagrams is stated using homotopies. For instance, a triangle
\begin{equation*}
\begin{tikzcd}[column sep=tiny]
A \arrow[dr,swap,"f"] \arrow[rr,"h"] & & B \arrow[dl,"g"] \\
& X
\end{tikzcd}
\end{equation*}
is said to commute if it comes equipped with a homotopy $H:f\htpy g\circ h$, and a square
\begin{equation*}
\begin{tikzcd}
A \arrow[d,"i"'] \arrow[r,"g"] & X \arrow[d,"f"] \\
B \arrow[r,swap,"h"] & Y
\end{tikzcd}
\end{equation*}
is said to commute if it comes equipped with a homotopy $H:h\circ i\htpy f\circ g$. 

The reflexivity, inverse, and concatenation operations on homotopies are defined pointwise.
We will write $H^{-1}$ for $\lam{x}H(x)^{-1}$, and $\ct{H}{K}$ for $\lam{x}\ct{H(x)}{K(x)}$.
These operations satisfy the groupoid laws (phrased appropriately as homotopies). Apart from the groupoid operations and their laws, we will occasionally need \emph{whiskering} operations and the naturality of homotopies.

\begin{defn}\label{defn:htpy_whisering}
We define the following \define{whiskering}\index{homotopy!whiskering operations|textbf}\index{whiskering operations!of homotopies|textbf} operations on homotopies:
\begin{enumerate}
\item Suppose $H:f\htpy g$ for two functions $f,g:A\to B$, and let $h:B\to C$. We define
\begin{equation*}
h\cdot H\defeq \lam{x}\ap{h}{H(x)}:h\circ f\htpy h\circ g.
\end{equation*}
\item Suppose $f:A\to B$ and $H:g\htpy h$ for two functions $g,h:B\to C$. We define
\begin{equation*}
H\cdot f\defeq\lam{x}H(f(x)):h\circ f\htpy g\circ f.
\end{equation*}
\end{enumerate}
\end{defn}

We will frequently make use of commuting cubes. The commutativity of a cube is stated using the whiskering operations on homotopies.

\begin{defn}\label{defn:cube}
A \define{commuting cube}\index{commuting cube|textbf}
\begin{equation*}
\begin{tikzcd}
& A_{111} \arrow[dl] \arrow[dr] \arrow[d] \\
A_{110} \arrow[d] & A_{101} \arrow[dl] \arrow[dr] & A_{011} \arrow[dl,crossing over] \arrow[d] \\
A_{100} \arrow[dr] & A_{010} \arrow[d] \arrow[from=ul,crossing over] & A_{001} \arrow[dl] \\
& A_{000},
\end{tikzcd}
\end{equation*}
consists of 
\begin{enumerate}
\item types
\begin{equation*}
A_{111},A_{110},A_{101},A_{011},A_{100},A_{010},A_{001},A_{000},
\end{equation*}
\item \begin{samepage}%
maps
\begin{align*}
f_{11\check{1}} & : A_{111}\to A_{110} & f_{\check{1}01} & : A_{101}\to A_{001} \\
f_{1\check{1}1} & : A_{111}\to A_{101} & f_{01\check{1}} & : A_{011}\to A_{010} \\
f_{\check{1}11} & : A_{111}\to A_{011} & f_{0\check{1}1} & : A_{011}\to A_{001} \\
f_{1\check{1}0} & : A_{110}\to A_{100} & f_{\check{1}00} & : A_{100}\to A_{000} \\
f_{\check{1}10} & : A_{110}\to A_{010} & f_{0\check{1}0} & : A_{010}\to A_{000} \\
f_{10\check{1}} & : A_{101}\to A_{100} & f_{00\check{1}} & : A_{001}\to A_{000},
\end{align*}
\end{samepage}%
\item homotopies
\begin{align*}
H_{1\check{1}\check{1}} & : f_{1\check{1}0}\circ f_{11\check{1}} \htpy f_{10\check{1}}\circ f_{1\check{1}1} & H_{0\check{1}\check{1}} & : f_{0\check{1}0}\circ f_{01\check{1}} \htpy f_{00\check{1}}\circ f_{0\check{1}1} \\
H_{\check{1}1\check{1}} & : f_{\check{1}10}\circ f_{11\check{1}} \htpy f_{01\check{1}}\circ f_{\check{1}11} & H_{\check{1}0\check{1}} & : f_{\check{1}00}\circ f_{10\check{1}} \htpy f_{00\check{1}}\circ f_{\check{1}01} \\
H_{\check{1}\check{1}1} & : f_{\check{1}01}\circ f_{1\check{1}1} \htpy f_{0\check{1}1}\circ f_{\check{1}11} & H_{\check{1}\check{1}0} & : f_{\check{1}00}\circ f_{1\check{1}0} \htpy f_{0\check{1}0}\circ f_{\check{1}10},
\end{align*}
\item and a homotopy 
\begin{align*}
C & : \ct{(f_{\check{1}00}\cdot H_{1\check{1}\check{1}})}{(\ct{(H_{\check{1}0\check{1}}\cdot f_{1\check{1}1})}{(f_{00\check{1}}\cdot H_{\check{1}\check{1}1})})} \\
& \qquad \htpy \ct{(H_{\check{1}\check{1}0}\cdot f_{11\check{1}})}{(\ct{(f_{0\check{1}0}\cdot H_{\check{1}1\check{1}})}{(H_{0\check{1}\check{1}}\cdot f_{\check{1}11})})}
\end{align*}
filling the cube.
\end{enumerate}
\end{defn}

\begin{defn}
Let $f:A\to B$ be a function. We say that $f$ has a \define{section}\index{section!of a map|textbf} if there is a term of type\index{sec(f)@{$\sections(f)$}|textbf}
\begin{equation*}
\sections(f) \defeq \sm{g:B\to A} f\circ g\htpy \idfunc[B].
\end{equation*}
Dually, we say that $f$ has a \define{retraction}\index{retraction} if there is a term of type\index{retr(f)@{$\retractions(f)$}|textbf}
\begin{equation*}
\retractions(f) \defeq \sm{h:B\to A} h\circ f\htpy \idfunc[A].
\end{equation*}
If $f$ has a retraction, we also say that $A$ is a \define{retract}\index{retract!of a type} of $B$.
\end{defn}

\begin{defn}
We say that a function $f:A\to B$ is an \define{equivalence}\index{equivalence|textbf}\index{bi-invertible map|see {equivalence}} if it has both a section and a retraction, i.e.~if it comes equipped with a term of type\index{is_equiv@{$\isequiv$}|textbf}
\begin{equation*}
\isequiv(f)\defeq\sections(f)\times\retractions(f).
\end{equation*}
We will write $\eqv{A}{B}$\index{equiv@{$\eqv{A}{B}$}|textbf} for the type $\sm{f:A\to B}\isequiv(f)$.
\end{defn}

Clearly, if $f$ is \define{invertible}\index{invertible map} in the sense that it comes equipped with a function $g:B\to A$ such that $f\circ g\htpy\idfunc[B]$ and $g\circ f\htpy\idfunc[A]$, then $f$ is an equivalence. We write\index{has_inverse@{$\hasinverse$}|textbf}
\begin{equation*}
\hasinverse(f)\defeq\sm{g:B\to A} (f\circ g\htpy \idfunc[B])\times (g\circ f\htpy\idfunc[A]).
\end{equation*}
The section of an equivalence is also a retraction (and vice versa), so we define the \define{inverse} of an equivalence to be its section. It follows immediately that the inverse of any equivalence is again an equivalence.\index{equivalence!invertibility of} The identity function $\idfunc[A]$ on a type $A$ is an equivalence since it is its own section and its own retraction.

It is straightforward to show that for any two functions $f,g:A\to B$, we have
\begin{equation*}
(f\htpy g)\to (\isequiv(f)\leftrightarrow\isequiv(g)).
\end{equation*}
Given a commuting triangle
\begin{equation*}
\begin{tikzcd}[column sep=tiny]
A \arrow[rr,"h"] \arrow[dr,swap,"f"] & & B \arrow[dl,"g"] \\
& X.
\end{tikzcd}
\end{equation*}
with $H:f\htpy g\circ h$, we have:
\begin{enumerate}
\item If the map $h$ has a section, then $f$ has a section if and only if $g$ has a section.
\item If the map $g$ has a retraction, then $f$ has a retraction if and only if $h$ has a retraction.
\item (The \define{3-for-2 property} for equivalences.) If any two of the functions
\begin{equation*}
f,\qquad g,\qquad h
\end{equation*}
are equivalences, then so is the third.
\end{enumerate}

In the following theorem we characterize the identity type of a $\Sigma$-type as a $\Sigma$-type of identity types.

\begin{prp}[Theorem 2.7.2 of \cite{hottbook}]\label{thm:eq_sigma}
Let $B$ be a type family over $A$, let $s:\sm{x:A}B(x)$, and consider the dependent function\index{pair_eq@{$\paireq$}|textbf}
\begin{equation*}
\paireq_s:\prd{t:\sm{x:A}B(x)} (s=t)\to \sm{\alpha:\proj 1(s)=\proj 1(t)} \tr_B(\alpha,\proj 2(s))=\proj 2(t)
\end{equation*}
defined by $\paireq_s(\refl{s}) \defeq (\refl{\proj 1(s)},\refl{\proj 2(s)})$. Then $\paireq_{s,t}$ is an equivalence for every $t:\sm{x:A}B(x)$.\index{Sigma type@{$\Sigma$-type}!identity types of|textit}\index{identity type!of a Sigma-type@{of a $\Sigma$-type}|textit}
\end{prp}

We include the proof mainly to introduce some more notation.

\begin{proof}
The maps in the converse direction\index{eq_pair@{$\eqpair$}}
\begin{equation*}
\eqpair_{s,t} : \Big(\sm{p:\proj 1(s)=\proj 1(t)}\id{\tr_B(p,\proj 2(s))}{\proj 2(t)}\Big)\to(\id{s}{t})
\end{equation*}
is defined by
\begin{equation*}
\eqpair_{(x,y),(x',y')}(\refl{x},\refl{y})\defeq \refl{(x,y)}.
\end{equation*}
The proofs that the function $\eqpair_{s,t}$ is indeed an inverse of $\paireq_{s,t}$ are also by induction.
\end{proof}

\begin{defn}
We say that a type $A$ is \define{contractible}\index{contractible!type|textbf} if there is a term of type
\begin{equation*}
\iscontr(A) \defeq \sm{c:A}\prd{x:A}c=x.
\end{equation*}
Given a term $(c,C):\iscontr(A)$, we call $c:A$ the \define{center of contraction}\index{center of contraction|textbf} of $A$, and we call $C:\prd{x:A}a=x$ the \define{contraction}\index{contraction} of $A$.
\end{defn}

Suppose $A$ is a contractible type with center of contraction $c$ and contraction $C$. Then the type of $C$ is (judgmentally) equal to the type
\begin{equation*}
\const_c\htpy\idfunc[A].
\end{equation*}
In other words, the contraction $C$ is a \emph{homotopy} from the constant function to the identity function.

\begin{defn}
Consider a type $A$ with a base point $a:A$. We say that $A$ satisfies \define{singleton induction}\index{singleton induction|textbf} if for every type family $B$ over $A$, the map
\begin{equation*}
\evpt:\Big(\prd{x:A}B(x)\Big)\to B(a)
\end{equation*}
given by $f\mapsto f(a)$ has a section.
\end{defn}

\begin{prp}\label{thm:contractible}
A type $A$ is contractible if and only if it satisfies singleton induction.
\end{prp}

\begin{eg}
By definition the unit type\index{unit type!contractibility} $\unit$ satisfies singleton induction, so it is contractible.
\end{eg}

\begin{rmk}
For any family $P:\Big(\sm{x:A}B(x)\Big)\to\UU$ there is a map
\begin{equation*}
\evpair : \Big(\prd{t:\sm{x:A}B(x)}P(t)\Big)\to \prd{x:A}{y:B(x)}P(x,y)
\end{equation*}
that evaluates $f:\prd{t:\sm{x:A}B(x)}P(t)$ at pairs $(x,y)$. In other words, $\evpair$ is defined by $\lam{f}{x}{y}f(x,y)$. By the induction principle for $\Sigma$-types, this map has a section. It is easy to show that $\evpair$ is in fact an equivalence. 

Similarly, there is a map
\begin{equation*}
\evrefl : \Big(\prd{x:A}{p:a=x}B(x,p)\Big)\to B(a,\refl{a})
\end{equation*}
given by $\lam{f}f(a,\refl{a})$, for any type family $B:\prd{x:A} (a=x)\to\UU$. By path induction, this map has a section, and again it is easy to show that this map is in fact an equivalence. 
\end{rmk}

\begin{prp}[Lemma 3.11.8 in \cite{hottbook}]\label{thm:total_path}
For any $x:A$, the type
\begin{equation*}
\sm{y:A}x=y
\end{equation*}
is contractible.\index{identity type!contractibility of total space|textit}
\end{prp}

\begin{proof}
We have the term $(x,\refl{x}):\sm{y:A}x=y$, and both maps in the composite
\begin{equation*}
\begin{tikzcd}[column sep=large]
\prd{t:\sm{y:A}x=y}B(t) \arrow[r,"\evpair"] & \prd{y:A}{p:x=y}B((y,p)) \arrow[r,"\evrefl"] & B((x,\refl{x}))
\end{tikzcd}
\end{equation*}
have sections, so the composite has a section. The composite is $\evpt$, so we see that the asserted type satisfies singleton induction.
\end{proof}

\begin{defn}
Let $f:A\to B$ be a function, and let $b:B$. The \define{fiber}\index{fiber|textbf}\index{homotopy fiber|see {fiber}} of $f$ at $b$ is defined to be the type
\begin{equation*}
\fib{f}{b}\defeq\sm{a:A}f(a)=b.
\end{equation*}
\end{defn}

\begin{eg}[Lemma 4.8.1 of \cite{hottbook}]\label{eg:fib_proj}
Consider a type family $B$ over $A$. Then the map
\begin{equation*}
B(a)\to \fib{\proj 1}{a}
\end{equation*}
given by $b\mapsto ((a,b),\refl{a})$ is an equivalence. In other words, the fibers of the projection function $\proj 1 : \big(\sm{x:A}B(x)\big)\to A$ are just the fibers of the family $B$.
\end{eg}

\begin{defn}
We say that a function $f:A\to B$ is \define{contractible}\index{contractible!map|textbf} if there is a term of type
\begin{equation*}
\iscontr(f)\defeq\prd{b:B}\iscontr(\fib{f}{b}).
\end{equation*}
\end{defn}

We cite Chapter 4 of \cite{hottbook} for the following result, although it is well-known that it can be proven directly and without the use of function extensionality.

\begin{prp}[Chapter 4 in \cite{hottbook}]\label{thm:contr_equiv}
A function is an equivalence if and only if it is contractible.\index{contractible!map!is an equivalence|textit}
\end{prp}

\section{The Fundamental Theorem of Identity Types}
Consider a family
\begin{equation*}
f : \prd{x:A}B(x)\to C(x)
\end{equation*}
of maps. Such $f$ is also called a \define{fiberwise map} or \define{fiberwise transformation}.

\begin{defn}[Definition 4.7.5 of \cite{hottbook}]
We define the map
\begin{equation*}
\total{f}:\sm{x:A}B(x)\to\sm{x:A}C(x).
\end{equation*}
by $\lam{(x,y)}(x,f(x,y))$.
\end{defn}

\begin{lem}[Theorem 4.7.6 of \cite{hottbook}]\label{lem:fib_total}
For any fiberwise transformation $f:\prd{x:A}B(x)\to C(x)$, and any $a:A$ and $c:C(a)$, there is an equivalence
\begin{equation*}
\eqv{\fib{f(a)}{c}}{\fib{\total{f}}{\pairr{a,c}}}.
\end{equation*}
\end{lem}

\begin{eg}
There are equivalences
\begin{equation*}
\eqv{\fib{(\apfunc{f})_{x,y}}{q}}{\fib{\delta_f}{(x,y,q)}}
\end{equation*}
for any $q:f(x)=f(y)$, because the triangle
\begin{equation*}
\begin{tikzcd}[column sep=-1em]
A \arrow[rr,"{\lam{x}(x,x,\refl{x})}"] \arrow[dr,swap,"\delta_f"] & & \sm{x,y:A}x=y \arrow[dl,"\total{\total{\apfunc{f}}}"] \\
\phantom{\sm{x,y:A}x=y} & \sm{x,y:A}f(x)=f(y)
\end{tikzcd}
\end{equation*}
commutes, and the top map is an equivalence.
\end{eg}

\begin{prp}[Theorem 4.7.7 of \cite{hottbook}]\label{thm:fib_equiv}
Let $f:\prd{x:A}B(x)\to C(x)$ be a fiberwise transformation. The following are logically equivalent:
\begin{enumerate}
\item For each $x:A$, the map $f_x:B(x)\to C(x)$ is an equivalence. In this case we say that $f$ is a \define{fiberwise equivalence}.
\item The map $\total{f}:\sm{x:A}B(x)\to\sm{x:A}C(x)$ is an equivalence.
\end{enumerate}
\end{prp}

The following theorem is the key to many results about identity types, which we will use instead of the \emph{encode-decode method}\index{encode-decode method} of \cite{LicataShulman}. We refer to it as the \define{Fundamental Theorem of Identity Types}.

\begin{thm}[Theorem 5.8.2 of \cite{hottbook}]\label{thm:id_fundamental}
Let $A$ be a type with $a:A$, and let $B$ be a type family over $A$ with $b:B(a)$.
Then  the following are logically equivalent:
\begin{enumerate}
\item The canonical family of maps
\begin{equation*}
\ind{a{=}}(b):\prd{x:A} (a=x)\to B(x)
\end{equation*}
is a fiberwise equivalence.
\item The total space
\begin{equation*}
\sm{x:A}B(x)
\end{equation*}
is contractible.
\end{enumerate}
\end{thm}

\begin{proof}
By \autoref{thm:fib_equiv} it follows that the fiberwise transformation $\ind{a{=}}(b)$ is a fiberwise equivalence if and only if it induces an equivalence
\begin{equation*}
\eqv{\Big(\sm{x:A}a=x\Big)}{\Big(\sm{x:A}B(x)\Big)}
\end{equation*}
on total spaces. We have that $\sm{x:A}a=x$ is contractible. Now it follows by the 3-for-2 property of equivalences, applied in the case
\begin{equation*}
\begin{tikzcd}
\sm{x:A}a=x \arrow[rr,"\total{\ind{a{=}}(b)}"] \arrow[dr,swap,"\eqvsym"] & & \sm{x:A}B(x) \arrow[dl] \\
& \unit & \phantom{\sm{x:A}a=x}
\end{tikzcd}
\end{equation*}
that $\total{\ind{a{=}}(b)}$ is an equivalence if and only if $\sm{x:A}B(x)$ is contractible.
\end{proof}

Observe that in the proof of \cref{thm:id_fundamental} we haven't used the actual definition of the fiberwise transformation. Indeed, for any fiberwise transformation
\begin{equation*}
f:\prd{x:A}(a=x)\to B(x)
\end{equation*}
we have that $f$ is a fiberwise equivalence if and only if the total space of $B$ is contractible.

Since retracts of contractible types are again contractible, it follows that the only retract of the identity type is the identity type itself:

\begin{cor}\label{cor:id_fundamental_retr}
Let $a:A$, and let $B$ be a type family over $A$. If each $B(x)$ is a retract of $\id{a}{x}$, then $B(x)$ is equivalent to $\id{a}{x}$ for every $x:A$.
\end{cor}

As a first application of the fundamental theorem we give a quick new proof that equivalences are embeddings. The proof of the corresponding theorem in \cite{hottbook} is more involved.

\begin{defn}
An \define{embedding}\index{embedding|textbf} is a map $f:A\to B$ satisfying the property that
\begin{equation*}
\apfunc{f}:(\id{x}{y})\to(\id{f(x)}{f(y)})
\end{equation*}
is an equivalence for every $x,y:A$. We write $\isemb(f)$ for the type of witnesses that $f$ is an embedding.
\end{defn}

\begin{prp}[Theorem 2.11.1 in \cite{hottbook}]
\label{cor:emb_equiv} 
Any equivalence is an embedding.\index{embedding!equivalences are embeddings|textit}\index{equivalence!is an embedding|textit}
\end{prp}

\begin{proof}
Let $e:\eqv{A}{B}$ be an equivalence, and let $x:A$. By \autoref{thm:id_fundamental} it follows that
\begin{equation*}
\apfunc{e} : (\id{x}{y})\to (\id{e(x)}{e(y)})
\end{equation*}
is an equivalence for every $y:A$ if and only if the total space
\begin{equation*}
\sm{y:A}e(x)=e(y)
\end{equation*}
is contractible for every $y:A$. Now observe that $\sm{y:A}e(x)=e(y)$ is equivalent to the fiber $\fib{e}{e(x)}$, which is contractible by \cref{thm:contr_equiv}.
\end{proof}

\begin{defn}
A type $A$ is said to be a \define{proposition} if there is a term of type
\begin{equation*}
\isprop(A)\defeq\prd{x,y:A}\iscontr(x=y).
\end{equation*}
Furthermore, we write $\prop\defeq\sm{X:\UU}\isprop(X)$ for the type of all small propositions.
\end{defn}

We will often use either of the following characterizations of propositions.

\begin{lem}[Lemma 3.11.10 and Exercise 3.5 of \cite{hottbook}]\label{lem:prop_char}
For any type $A$ the following are equivalent:
\begin{enumerate}
\item $A$ is a proposition.
\item $A$ is \define{proof irrelevant} in the sense that $\prd{x,y:A}x=y$.
\item $A\to\iscontr(A)$. 
\end{enumerate}
\end{lem}

\begin{eg}\label{eg:prop_contr}
Any contractible type is a proposition. The empty type is a proposition by a direct application of the induction principle of the empty type. Furthermore, any retract of a proposition is again a proposition. In particular, propositions are closed under equivalences.
\end{eg}

\begin{lem}\label{lem:id_fib}
Consider a function $f:A\to B$, and let $(a,p),(a',p'):\fib{f}{b}$ for some $b:B$. Then the canonical map
\begin{equation*}
((a,p)=(a',p'))\to \fib{\apfunc{f}}{\ct{p}{p'^{-1}}}
\end{equation*}
is an equivalence.
\end{lem}

\begin{proof}
By \cref{thm:fib_equiv} it suffices to show that the type
\begin{equation*}
\sm{y:A}{q:f(y)=b}{p:\proj 1(s)=y} \ap{f}{p}=\ct{\proj 2(s)}{q^{-1}}
\end{equation*}
is contractible, which is immediate by two applications of \cref{thm:total_path}.
\end{proof}

\begin{prp}[Lemma 7.6.2 of \cite{hottbook}]\label{thm:prop_emb}
A map is an embedding if and only if its fibers are propositions.
\end{prp}

\begin{proof}
If $f$ is an embedding, then the fibers of $\apfunc{f}$ are contractible by \cref{thm:contr_equiv}. Thus it follows by \cref{lem:id_fib} that the fibers of $f$ are propositions.

Conversely, if the fibers of $f$ are propositions, then we have by \cref{lem:id_fib} an equivalence
\begin{equation*}
\eqv{((x,p)=(y,\refl{f(y)}))}{\fib{\apfunc{f}}{p}}
\end{equation*}
for any $p:f(x)=f(y)$, which shows that the fibers of $\apfunc{f}$ are contractible. Thus $f$ is an embedding by \cref{thm:contr_equiv}.
\end{proof}

\begin{defn}
A type family $B$ over $A$ is said to be a \define{subtype} of $A$ if for each $x:A$ the type $B(x)$ is a proposition.
\end{defn}

\begin{cor}\label{thm:subtype}
A type family $B$ over $A$ is a subtype if and only if the projection map
\begin{equation*}
\proj 1 : \Big(\sm{x:A}B(x)\Big)\to A
\end{equation*}
is an embedding.
\end{cor}

\begin{proof}
Immediate by \cref{eg:fib_proj,thm:prop_emb}.
\end{proof}

\section{Function extensionality}
\begin{prp}[Theorem 4.9.5 of \cite{hottbook}]\label{thm:funext_wkfunext}
The following are equivalent:
\begin{enumerate}
\item The \define{function extensionality principle}\index{function extensionality}: For every type family $B$ over $A$, and any two dependent functions $f,g:\prd{x:A}B(x)$, the canonical map\index{htpy_eq@{$\htpyeq$}|textbf}
\begin{equation*}
\htpyeq(f,g) : (\id{f}{g})\to (f\htpy g)
\end{equation*}
by path induction (sending $\refl{f}$ to $\lam{x}\refl{f(x)}$) is an equivalence. We will write $\eqhtpy$\index{eq_htpy@{$\eqhtpy$}} for its inverse.
\item The \define{weak function extensionality principle}\index{weak function extensionality} holds: For every type family $B$ over $A$ one has\index{contractible!weak function extensionality}
\begin{equation*}
\Big(\prd{x:A}\iscontr(B(x))\Big)\to\iscontr\Big(\prd{x:A}B(x)\Big).
\end{equation*}
\end{enumerate}
\end{prp}

From now on we will assume that function extensionality holds.

\begin{cor}[Theorem 7.1.9 of \cite{hottbook}]\label{thm:prop_pi}
For any type family $B$ over $A$ one has
\begin{equation*}
\Big(\prd{x:A}\isprop(B(x))\Big)\to \isprop\Big(\prd{x:A}B(x)\Big).
\end{equation*}
In particular, if $B$ is a proposition, then $A\to B$ is a proposition for any type $A$.
\end{cor}

We show in this section that a map $f:A\to B$ is an equivalence if and only if for any type family $P$ over $B$, the precomposition map
\begin{equation*}
\blank\circ f: \Big(\prd{y:B}P(y)\Big)\to\Big(\prd{x:A}P(f(x))\Big)
\end{equation*}
is an equivalence. 
In the proof we use the notion of \emph{path-split} maps, which was introduced in \cite{RijkeShulmanSpitters}.

\begin{defn}
We say that a map $f:A\to B$ is \define{path-split}\index{path-split|textbf} if $f$ has a section, and for each $x,y:A$ the map
\begin{equation*}
\apfunc{f}(x,y):(x=y)\to (f(x)=f(y))
\end{equation*}
also has a section. We write $\pathsplit(f)$\index{path_split(f)@{$\pathsplit(f)$}|textbf} for the type
\begin{equation*}
\sections(f)\times\prd{x,y:A}\sections(\apfunc{f}(x,y)).
\end{equation*}
\end{defn}

We will also use the notion of \emph{half-adjoint equivalences}, which were introduced in \cite{hottbook}.

\begin{defn}[Definition 4.2.1 of \cite{hottbook}]
We say that a map $f:A\to B$ is a \define{half-adjoint equivalence}\index{half-adjoint equivalence|textbf}, in the sense that there are
\begin{align*}
g & : B \to A\\
G & : f\circ g \htpy \idfunc[B] \\
H & : g\circ f \htpy \idfunc[A] \\
K & : G\cdot f \htpy f\cdot H.
\end{align*}
We write $\halfadj(f)$\index{half_adj(f)@{$\halfadj(f)$}|textbf} for the type of such quadruples $(g,G,H,K)$.
\end{defn}

Furthermore, we will need `type theoretic choice'. 

\begin{prp}[Theorem 2.15.7 of \cite{hottbook}]\label{thm:choice}
Let $C(x,y)$ be a type in context $\Gamma,x:A,y:B(x)$. Then the map
\begin{equation*}
\varphi:\Big(\prd{x:A}\sm{y:B(x)}C(x,y)\Big)\to \Big(\sm{f:\prd{x:A}B(x)}\prd{x:A}C(x,f(x))\Big)
\end{equation*}
given by $\lam{h}(\lam{x}\proj 1(h(x)),\lam{x}\proj 2(h(x)))$ is an equivalence.
\end{prp}

\begin{cor}
For type $A$ and any type family $C$ over $B$, the map
\begin{equation*}
\Big(\sm{f:A\to B} \prd{x:A}C(f(x))\Big)\to\Big(A\to\sm{y:B}C(x)\Big)
\end{equation*}
given by $\lam{(f,g)}{x}(f(x),g(x))$ is an equivalence.
\end{cor}

\begin{prp}\label{prp:equiv_precomp}
For any map $f:A\to B$, the following are equivalent:
\begin{enumerate}
\item $f$ is an equivalence.
\item $f$ is path-split.
\item $f$ is a half-adjoint equivalence.
\item For any type family $P$ over $B$ the map
\begin{equation*}
\Big(\prd{y:B}P(y)\Big)\to\Big(\prd{x:A}P(f(x))\Big)
\end{equation*}
given by $s\mapsto s\circ f$ is an equivalence.
\item For any type $X$ the map
\begin{equation*}
(B\to X)\to (A\to X)
\end{equation*}
given by $g\mapsto g\circ f$ is an equivalence. 
\end{enumerate}
\end{prp}

\begin{proof}
To see that (i) implies (ii) we note that any equivalence has a section, and its action on paths is an equivalence by \cref{cor:emb_equiv} so again it has a section.

To show that (ii) implies (iii), assume that $f$ is path-split. Thus we have $(g,G):\sections(f)$, and the assumption that $\apfunc{f}:(x=y)\to (f(x)=f(y))$ has a section for every $x,y:A$ gives us a term of type
\begin{equation*}
\prd{x:A}\fib{\apfunc{f}}{G(f(x))}.
\end{equation*}
By \cref{thm:choice} this type is equivalent to
\begin{equation*}
\sm{H:\prd{x:A}g(f(x))=x}\prd{x:A}G(f(x))=\ap{f}{H(x)},
\end{equation*}
so we obtain $H:g\circ f\htpy \idfunc[A]$ and $K:G\cdot f\htpy f\cdot H$, showing that $f$ is a half-adjoint equivalence.

To show that (iii) implies (iv), suppose that $f$ comes equipped with $(g,G,H,K)$ witnessing that $f$ is a half-adjoint equivalence. Then we define the inverse of $\blank\circ f$ to be the map
\begin{equation*}
\varphi:\Big(\prd{x:A}P(f(x))\Big)\to\Big(\prd{y:B}P(y)\Big)
\end{equation*}
given by $s\mapsto \lam{y}\tr_P(G(y),sg(y))$. 

To see that $\varphi$ is a section of $\blank\circ f$, let $s:\prd{x:A}P(f(x))$. By function extensionality it suffices to construct a homotopy $\varphi(s)\circ f\htpy s$. In other words, we have to show that
\begin{equation*}
\tr_P(G(f(x)),s(g(f(x)))=s(x)
\end{equation*}
for any $x:A$. Now we use the additional homotopy $K$ from our assumption that $f$ is a half-adjoint equivalence. Since we have $K(x):G(f(x))=\ap{f}{H(x)}$ it suffices to show that
\begin{equation*}
\tr_P(\ap{f}{H(x)},sgf(x))=s(x).
\end{equation*}
A simple path-induction argument yields that
\begin{equation*}
\tr_P(\ap{f}{p})\htpy \tr_{P\circ f}(p)
\end{equation*}
for any path $p:x=y$ in $A$, so it suffices to construct an identification
\begin{equation*}
\tr_{P\circ f}(H(x),sgf(x))=s(x).
\end{equation*}
We have such an identification by $\apd{H(x)}{s}$.

To see that $\varphi$ is a retraction of $\blank\circ f$, let $s:\prd{y:B}P(y)$. By function extensionality it suffices to construct a homotopy $\varphi(s\circ f)\htpy s$. In other words, we have to show that
\begin{equation*}
\tr_P(G(y),sfg(y))=s(y)
\end{equation*}
for any $y:B$. We have such an identification by $\apd{G(y)}{s}$. This completes the proof that (iii) implies (iv).

Note that (v) is an immediate consequence of (iv), since we can just choose $P$ to be the constant family $X$.

It remains to show that (v) implies (i). Suppose that
\begin{equation*}
\blank\circ f:(B\to X)\to (A\to X)
\end{equation*}
is an equivalence for every type $X$. Then its fibers are contractible by \cref{thm:contr_equiv}. In particular, choosing $X\jdeq A$ we see that the fiber
\begin{equation*}
\fib{\blank\circ f}{\idfunc[A]}\jdeq \sm{h:B\to A}h\circ f=\idfunc[A]
\end{equation*}
is contractible. Thus we obtain a function $h:B\to A$ and a homotopy $H:h\circ f\htpy\idfunc[A]$ showing that $h$ is a retraction of $f$. We will show that $h$ is also a section of $f$. To see this, we use that the fiber
\begin{equation*}
\fib{\blank\circ f}{f}\jdeq \sm{i:B\to B} i\circ f=f
\end{equation*}
is contractible (choosing $X\jdeq B$). 
Of course we have $(\idfunc[B],\refl{f})$ in this fiber. However we claim that there also is an identification $p:(f\circ h)\circ f=f$, showing that $(f\circ h,p)$ is in this fiber, because
\begin{align*}
(f\circ h)\circ f & \jdeq f\circ (h\circ f) \\
& = f\circ \idfunc[A] \\
& \jdeq f
\end{align*}
Now we conclude by the contractibility of the fiber that there is an identification $(\idfunc[B],\refl{f})=(f\circ h,p)$. In particular we obtain that $\idfunc[B]=f\circ h$, showing that $h$ is a section of $f$.
\end{proof}

\section{Homotopy pullbacks}
\subsection{Cartesian squares}

Recall that a square
\begin{equation*}
\begin{tikzcd}
C \arrow[r,"q"] \arrow[d,swap,"p"] & B \arrow[d,"g"] \\
A \arrow[r,swap,"g"] & X
\end{tikzcd}
\end{equation*}
is said to \define{commute}\index{commuting square|textbf} if there is a homotopy $H:f\circ p\htpy g\circ q$. 

\begin{defn}\label{defn:cospan}
A \define{cospan}\index{cospan|textbf} consists of three types $A$, $X$, and $B$, and maps $f:A\to X$ and $g:B\to X$. Given a type $C$, a \define{cone}\index{cone!on a cospan|textbf} on the cospan $A \stackrel{f}{\rightarrow} X \stackrel{g}{\leftarrow} B$ with \define{vertex} $C$\index{vertex!of a cone|textbf} consists of maps $p:C\to A$, $q:C\to B$ and a homotopy $H:f\circ p\htpy g\circ q$ witnessing that the square
\begin{equation*}
\begin{tikzcd}
C \arrow[r,"q"] \arrow[d,swap,"p"] & B \arrow[d,"g"] \\
A \arrow[r,swap,"f"] & X
\end{tikzcd}
\end{equation*}
commutes. We write\index{cone(C)@{$\cone(\blank)$}|textbf}
\begin{equation*}
\cone(C)\defeq \sm{p:C\to A}{q:C\to B}f\circ p\htpy g\circ q
\end{equation*}
for the type of cones with vertex $C$.
\end{defn}

Given a cone with vertex $C$ on a span $A\stackrel{f}{\rightarrow} X \stackrel{g}{\leftarrow} B$ and a map $h:C'\to C$, we construct a new cone with vertex $C'$ in the following definition.

\begin{defn}
For any cone $(p,q,H)$ with vertex $C$ and any type $C'$, we define a map\index{cone map@{$\conemap$}|textbf}
\begin{equation*}
\conemap(p,q,H):(C'\to C)\to\cone(C')
\end{equation*}
by $h\mapsto (p\circ h,q\circ h,H\cdot h)$. 
\end{defn}

\begin{defn}
We say that a commuting square
\begin{equation*}
\begin{tikzcd}
C \arrow[r,"q"] \arrow[d,swap,"p"] & B \arrow[d,"g"] \\
A \arrow[r,swap,"f"] & X
\end{tikzcd}
\end{equation*}
with $H:f\circ p\htpy g\circ q$ is a \define{pullback square}\index{pullback square|textbf}, or that it is \define{cartesian}\index{cartesian square|textbf}, if it satisfies the \define{universal property} of pullbacks\index{universal property!of pullbacks}, which asserts that the map
\begin{equation*}
\conemap(p,q,H):(C'\to C)\to\cone(C')
\end{equation*}
is an equivalence for every type $C'$. 
\end{defn}

We often indicate the universal property with a diagram as follows:
\begin{equation*}
\begin{tikzcd}
C' \arrow[drr,bend left=15,"{q'}"] \arrow[dr,densely dotted,"h"] \arrow[ddr,bend right=15,swap,"{p'}"] \\
& C \arrow[r,"q"] \arrow[d,swap,"p"] & B \arrow[d,"g"] \\
& A \arrow[r,swap,"f"] & X
\end{tikzcd}
\end{equation*}
since the universal property states that for every cone $(p',q',H')$ with vertex $C'$, the type of pairs $(h,\alpha)$ consisting of $h:C'\to C$ equipped with $\alpha:\conemap((p,q,H),h)=(p',q',H')$ is contractible by \cref{thm:contr_equiv}.

\begin{prp}\label{thm:pullback_up}
Consider a commuting square
\begin{equation*}
\begin{tikzcd}
C \arrow[r,"q"] \arrow[d,swap,"p"] & B \arrow[d,"g"] \\
A \arrow[r,swap,"f"] & X
\end{tikzcd}
\end{equation*}
with $H:f\circ p\htpy g\circ q$
Then the following are equivalent:\index{universal property!of pullbacks (characterization)|textit}
\begin{enumerate}
\item The square is a pullback square.
\item For every type $C'$ and every cone $(p',q',H')$ with vertex $C'$, the type of quadruples $(h,K,L,M)$ consisting of
\begin{align*}
h & : C'\to C \\
K & : p\circ h \htpy p' \\
L & : q\circ h \htpy q' \\
M & : \ct{(H\cdot h)}{(g\cdot L)} \htpy \ct{(f\cdot K)}{H'}
\end{align*}
is contractible.
\end{enumerate}
\end{prp}

\begin{rmk}
The homotopy $M$ in \cref{thm:pullback_up} witnesses that the square
\begin{equation*}
\begin{tikzcd}
f\circ p\circ h \arrow[r,"f\cdot K"] \arrow[d,swap,"H\cdot h"] & f\circ p' \arrow[d,"{H'}"] \\
g\circ q\circ h \arrow[r,swap,"g\cdot L"] & g\circ q'
\end{tikzcd}
\end{equation*}
of homotopies commutes.
\end{rmk}

\subsection{The unique existence of pullbacks}

\begin{defn}
Let $f:A\to X$ and $B\to X$ be maps. Then we define
\begin{align*}
A\times_X B & \defeq \sm{x:A}{y:B}f(x)=g(y) \\
\pi_1 & \defeq \proj 1 & & : A\times_X B\to A \\
\pi_2 & \defeq \proj 1\circ\proj 2 & & : A\times_X B\to B\\
\pi_3 & \defeq \proj 2\circ\proj 2 & & : f\circ \pi_1 \htpy g\circ\pi_2.
\end{align*}
The type $A\times_X B$ is called the \define{canonical pullback}\index{canonical pullback|textbf} of $f$ and $g$.
\end{defn}

Note that $A\times_X B$ depends on $f$ and $g$, although this dependency is not visible in the notation.

\begin{prp}[Exercise 2.11 of \cite{hottbook}]
Given maps $f:A\to X$ and $g:B\to X$, the commuting square\index{canonical pullback|textit}
\begin{equation*}
\begin{tikzcd}
A\times_X B \arrow[r,"\pi_2"] \arrow[d,swap,"\pi_1"] & B \arrow[d,"g"] \\
A \arrow[r,swap,"f"] & X,
\end{tikzcd}
\end{equation*}
is a pullback square.
\end{prp}

In the following lemma we establish the uniqueness of pullbacks up to equivalence via a \emph{3-for-2 property} for pullbacks.

\begin{lem}\label{lem:pb_3for2}\index{pullback!3-for-2 property|textit}\index{3-for-2 property!of pullbacks|textit}%
Consider the squares
\begin{equation*}
\begin{tikzcd}
C \arrow[r,"q"] \arrow[d,swap,"p"] & B \arrow[d,"g"] & {C'} \arrow[r,"{q'}"] \arrow[d,swap,"{p'}"] & B \arrow[d,"g"] \\
A \arrow[r,swap,"f"] & X & A \arrow[r,swap,"f"] & X
\end{tikzcd}
\end{equation*}
with homotopies $H:f\circ p \htpy g\circ q$ and $H':f\circ p'\htpy g\circ q'$.
Furthermore, suppose we have a map $h:C'\to C$ equipped with
\begin{align*}
K & : p\circ h \htpy p' \\
L & : q\circ h \htpy q' \\
M & : \ct{(H\cdot h)}{(g\cdot L)} \htpy \ct{(f\cdot K)}{H'}.
\end{align*}
If any two of the following three properties hold, so does the third:
\begin{samepage}%
\begin{enumerate}
\item $C$ is a pullback.
\item $C'$ is a pullback.
\item $h$ is an equivalence.
\end{enumerate}%
\end{samepage}%
\end{lem}

\begin{proof}
The type of triples $(K,L,M)$ is equivalent to the type of identifications
\begin{equation*}
\conemap((p,q,H),h)=(p',q',H').
\end{equation*}
Let $D$ be a type, and let $k:D\to C'$ be a map. We observe that
\begin{align*}
\conemap((p,q,H),(h\circ k)) & \jdeq (p\circ (h\circ k),q\circ (h\circ k),H\circ (h\circ k)) \\
& \jdeq ((p\circ h)\circ k,(q\circ h)\circ k, (H\circ h)\circ k) \\
& \jdeq \conemap(\conemap((p,q,H),h),k) \\
& = \conemap((p',q',H'),k).
\end{align*}
Thus we see that the triangle 
\begin{equation*}
\begin{tikzcd}[column sep=-1em]
(D\to C') \arrow[rr,"{h\circ \blank}"] \arrow[dr,swap,"{\conemap(p',q',H')}"] & & (D\to C) \arrow[dl,"{\conemap(p,q,H)}"] \\
& \cone(D) & \phantom{(D\to C')}
\end{tikzcd}
\end{equation*}
commutes. Therefore it follows from the 3-for-2 property of equivalences that if any two of the following properties hold, then so does the third:
\begin{enumerate}
\item The map $\conemap(p,q,H):(D\to C)\to \cone(D)$ is an equivalence,
\item The map $\conemap(p',q',H'):(D\to C')\to \cone(D)$ is an equivalence,
\item The map $h\circ\blank : (D\to C')\to (D\to C)$ is an equivalence.
\end{enumerate}
Thus the 3-for-2 property for pullbacks follows from the fact that $h$ is an equivalence if and only if $h\circ\blank : (D\to C')\to (D\to C)$ is an equivalence for any type $D$.
\end{proof}

\begin{defn}
Given a commuting square
\begin{equation*}
\begin{tikzcd}
C \arrow[r,"q"] \arrow[d,swap,"p"] & B \arrow[d,"g"] \\
A \arrow[r,swap,"f"] & X
\end{tikzcd}
\end{equation*}
with $H:f\circ p \htpy g \circ q$, we define the \define{gap map}\index{gap map|textbf}\index{pullback!gap map|textbf}
\begin{equation*}
\gap(p,q,H):C \to A\times_X B
\end{equation*}
by $\lam{z}(p(z),q(z),H(z))$. Furthermore, we will write\index{is_pullback@{$\ispullback$}|textbf}
\begin{equation*}
\ispullback(f,g,H)\defeq \isequiv(\gap(p,q,H)).
\end{equation*}
\end{defn}

\begin{prp}\label{thm:is_pullback}
Consider a commuting square
\begin{equation*}
\begin{tikzcd}
C \arrow[r,"q"] \arrow[d,swap,"p"] & B \arrow[d,"g"] \\
A \arrow[r,swap,"f"] & X
\end{tikzcd}
\end{equation*}
with $H:f\circ p \htpy g \circ q$. The following are equivalent:
\begin{enumerate}
\item The square is a pullback square
\item There is a term of type
\begin{equation*}
\ispullback(p,q,H)\defeq \isequiv(\gap(p,q,H)).
\end{equation*}
\end{enumerate}
\end{prp}

\begin{proof}
Note that there are homotopies
\begin{align*}
K & : \pi_1\circ \gap(p,q,H) \htpy p \\
L & : \pi_2\circ \gap(p,q,H) \htpy q \\
M & : \ct{(\pi_3\cdot \gap(p,q,H))}{(g\cdot L)} \htpy \ct{(f\cdot K)}{H}.
\end{align*}
given by 
\begin{align*}
K & \defeq \lam{z}\refl{p(z)} \\
L & \defeq \lam{z}\refl{q(z)} \\
M & \defeq \lam{z}\ct{\rightunit(H(z))}{\leftunit(H(z))^{-1}}.
\end{align*}
Therefore the claim follows by \cref{lem:pb_3for2}.
\end{proof}

\subsection{Fiberwise equivalences}

\begin{prp}\label{thm:pb_fibequiv}
Let $f:A\to B$, and let $g:\prd{a:A}P(a)\to Q(f(a))$ be a fiberwise transformation\index{fiberwise transformation|textit}. The following are equivalent:
\begin{enumerate}
\item The commuting square
\begin{equation*}
\begin{tikzcd}[column sep=large]
\sm{a:A}P(a) \arrow[r,"{\total[f]{g}}"] \arrow[d,swap,"\proj 1"] & \sm{b:B}Q(b) \arrow[d,"\proj 1"] \\
A \arrow[r,swap,"f"] & B
\end{tikzcd}
\end{equation*}
is a pullback square.
\item $g$ is a fiberwise equivalence.\index{fiberwise equivalence|textit}
\end{enumerate}
\end{prp}

\begin{proof}
The gap map factors as follows
\begin{equation*}
\begin{tikzcd}[column sep=-2em]
\sm{x:A}P(x) \arrow[dr,swap,"\total{g}"] \arrow[rr,"\gap"] & & A \times_B \Big(\sm{y:B}Q(y)\Big) \\
\phantom{A \times_B \Big(\sm{y:B}Q(y)\Big)} & \sm{x:A}Q(f(x)) \arrow[ur,swap,"{\gap'\,\defeq\,\lam{(x,q)}(x,(f(x),q),\refl{f(x)})}"]
\end{tikzcd}
\end{equation*}
Since $\gap'$ is an equivalence, it follows by \cref{thm:fib_equiv} that the gap map is an equivalence if and only if $g$ is a fiberwise equivalence.
\end{proof}

\begin{lem}
Consider a commuting square
\begin{equation*}
\begin{tikzcd}
C \arrow[r,"q"] \arrow[d,swap,"p"] & B \arrow[d,"g"] \\
A \arrow[r,swap,"f"] & X
\end{tikzcd}
\end{equation*}
with $H:f\circ p\htpy g\circ q$, and consider the fiberwise transformation
\begin{equation*}
\fibf{(f,q,H)} : \prd{a:A} \fib{p}{a}\to \fib{g}{f(a)}
\end{equation*}
given by $\lam{a}{(c,u)}(q(c),\ct{H(c)^{-1}}{\ap{f}{u}})$. Then there is an equivalence
\begin{equation*}
\eqv{\fib{\gap(p,q,H)}{(a,b,\alpha)}}{\fib{\fibf{(f,q,H)}(a)}{(b,\alpha^{-1})}}
\end{equation*}
\end{lem}

\begin{proof}
To obtain an equivalence of the desired type we simply concatenate known equivalences:
\begin{align*}
\fib{h}{(a,b,\alpha)} & \jdeq \sm{z:C} (p(z),q(z),H(z))=(a,b,\alpha) \\
& \eqvsym \sm{z:C}{u:p(z)=a}{v:q(z)=b}\ct{H(z)}{\ap{g}{v}}=\ct{\ap{f}{u}}{\alpha} \\
& \eqvsym \sm{(z,u):\fib{p}{a}}{v:q(z)=b} \ct{H(z)^{-1}}{\ap{f}{u}}=\ct{\ap{g}{v}}{\alpha^{-1}} \\
& \eqvsym \fib{\varphi(a)}{(b,\alpha^{-1})}\qedhere
\end{align*}
\end{proof}

\begin{cor}\label{cor:pb_fibequiv}
Consider a commuting square
\begin{equation*}
\begin{tikzcd}
C \arrow[r,"q"] \arrow[d,swap,"p"] & B \arrow[d,"g"] \\
A \arrow[r,swap,"f"] & X
\end{tikzcd}
\end{equation*}
with $H:f\circ p\htpy g\circ q$. The following are equivalent:
\begin{enumerate}
\item The square is a pullback square.\index{pullback square!characterized by fiberwise equivalence|textit}
\item The induced map on fibers
\begin{equation*}
\fibf{(p,q,H)}:\prd{x:A}\fib{p}{x}\to \fib{g}{f(x)}
\end{equation*}
is a fiberwise equivalence.
\end{enumerate}
\end{cor}

\begin{cor}\label{cor:pb_equiv}
Consider a commuting square
\begin{equation*}
\begin{tikzcd}
C \arrow[r,"q"] \arrow[d,swap,"p"] & B \arrow[d,"g"] \\
A \arrow[r,swap,"f"] & X.
\end{tikzcd}
\end{equation*}
and suppose that $g$ is an equivalence. Then the following are equivalent:
\begin{enumerate}
\item The square is a pullback square.
\item The map $p:C\to A$ is an equivalence.\index{equivalence!pullback of|textit}
\end{enumerate}
\end{cor}

\begin{proof}
If the square is a pullback square, then by \cref{thm:pb_fibequiv} the fibers of $p$ are equivalent to the fibers of $g$, which are contractible by \cref{thm:contr_equiv}. Thus it follows that $p$ is a contractible map, and hence that $p$ is an equivalence.

If $p$ is an equivalence, then by \cref{thm:contr_equiv} both $\fib{p}{x}$ and $\fib{g}{f(x)}$ are contractible for any $x:X$. It follows that the induced map $\fib{p}{x}\to\fib{g}{f(x)}$ is an equivalence. Thus we apply \cref{cor:pb_fibequiv} to conclude that the square is a pullback.
\end{proof}

\section{The univalence axiom}

The univalence axiom characterizes the identity type of the universe. It is considered to be an \emph{extensionality principle}\index{extensionality principle!types} for types. In the following theorem we introduce the univalence axiom and give two more equivalent ways of stating this.

\begin{prp}\label{thm:univalence}
The following are equivalent:
\begin{enumerate}
\item The \define{univalence axiom}\index{univalence axiom|textbf}: for any $A:\UU$ the map\index{equiv_eq@{$\equiveq$}|textbf}
\begin{equation*}
\equiveq\defeq \ind{A=}(\idfunc[A]) : \prd{B:\UU} (\id{A}{B})\to(\eqv{A}{B}).
\end{equation*}
is a fiberwise equivalence.\index{identity type!universe} If this is the case, we write
$\eqequiv$\index{eq equiv@{$\eqequiv$}|textbf}
for the inverse of $\equiveq$.
\item The type
\begin{equation*}
\sm{B:\UU}\eqv{A}{B}
\end{equation*}
is contractible for each $A:\UU$.
\item The principle of \define{equivalence induction}\index{equivalence induction}\index{induction principle!for equivalences}: for every $A:\UU$ and for every type family
\begin{equation*}
P:\prd{B:\UU} (\eqv{A}{B})\to \type,
\end{equation*}
the map
\begin{equation*}
\Big(\prd{B:\UU}{e:\eqv{A}{B}}P(B,e)\Big)\to P(A,\idfunc[A])
\end{equation*}
given by $f\mapsto f(A,\idfunc[A])$ has a section.\qedhere
\end{enumerate}
\end{prp}

It is a trivial observation, but nevertheless of fundamental importance, that by the univalence axiom the identity types of $\UU$ are equivalent to types in $\UU$, because it provides an equivalence $\eqv{(A=B)}{(\eqv{A}{B})}$, and the type $\eqv{A}{B}$ is in $\UU$ for any $A,B:\UU$. Since the identity types of $\UU$ are equivalent to types in $\UU$, we also say that the universe is \emph{locally small}.

\begin{defn}\label{defn:ess_small}
\begin{enumerate}
\item A type $A$ is said to be \define{essentially small}\index{essentially small!type|textbf} if there is a type $X:\UU$ and an equivalence $\eqv{A}{X}$. We write\index{ess_small(A)@{$\esssmall(A)$}|textbf}
\begin{equation*}
\esssmall(A)\defeq\sm{X:\UU}\eqv{A}{X}.
\end{equation*}
\item A map $f:A\to B$ is said to be \define{essentially small}\index{essentially small!map|textbf} if for each $b:B$ the fiber $\fib{f}{b}$ is essentially small.
We write\index{ess_small(f)@{$\esssmall(f)$}|textbf}
\begin{equation*}
\esssmall(f)\defeq\prd{b:B}\esssmall(\fib{f}{b}).
\end{equation*}
\item A type $A$ is said to be \define{locally small}\index{locally small!type} if for every $x,y:A$ the identity type $x=y$ is essentially small.
We write\index{loc_small(A)@{$\locsmall(A)$}|textbf}
\begin{equation*}
\locsmall(A)\defeq \prd{x,y:A}\esssmall(x=y).
\end{equation*}
\item Similarly, a map $f:A\to X$ is said to be \define{locally small} if $\delta_f:A\to A\times_X A$ is essentially small.
\end{enumerate}
\end{defn}

\begin{lem}\label{lem:isprop_ess_small}
The type $\esssmall(X)$ is a proposition for any type $X$.\index{essentially small!is a proposition|textit}
\end{lem}

\begin{proof}
Let $X$ be a type. Our goal is to show that the type
\begin{equation*}
\sm{Y:\UU}\eqv{X}{Y}
\end{equation*}
is a proposition. Suppose there is a type $X':\UU$ and an equivalence $e:\eqv{X}{X'}$, then the map
\begin{equation*}
(\eqv{X}{Y})\to (\eqv{X'}{Y})
\end{equation*}
given by precomposing with $e^{-1}$ is an equivalence. This induces an equivalence on total spaces
\begin{equation*}
\eqv{\Big(\sm{Y:\UU}\eqv{X}{Y}\Big)}{\Big(\sm{Y:\UU}\eqv{X'}{Y}\Big)}
\end{equation*}
However, the codomain of this equivalence is contractible by \cref{thm:univalence}. Thus it follows that the asserted type is a proposition.
\end{proof}

\begin{cor}
For each function $f:A\to B$, the type $\esssmall(f)$ is a proposition, and for each type $X$ the type $\locsmall(X)$ is a proposition.
\end{cor}

\begin{proof}
This follows from the fact that propositions are closed under dependent products, established in \cref{thm:prop_pi}.
\end{proof}

\begin{rmk}
The property of essentially smallness is preserved by $\Pi$, $\Sigma$, and $\mathsf{Id}$. Of course, any contractible type is essentially small, and so is any small type. The property of essentially smallness is preserved by $\Sigma$ and $\mathsf{Id}$, and the exponent $X^A$ of a locally small type $X$ by an essentially small type $A$ is again locally small. Furthermore, any proposition is locally small, and any universe is locally small with respect to itself. 
\end{rmk}

\begin{defn}
Consider two functions $f:A\to X$ and $g:B\to X$. We define the type 
\begin{equation*}
\hom_X(f,g)\defeq \sm{h:A\to B} f\htpy g\circ h.
\end{equation*}
\end{defn}

In other words, the type $\hom_X(f,g)$ is the type of functions $h:A\to B$ equipped with a homotopy witnessing that the triangle
\begin{equation*}
\begin{tikzcd}[column sep=tiny]
A \arrow[dr,swap,"f"] \arrow[rr,"h"] & & B \arrow[dl,"g"] \\
& X
\end{tikzcd}
\end{equation*}

\begin{lem}
Let $P$ and $Q$ be two type families over $X$, and write $\proj 1^P$ and $\proj 1^Q$ for their first projections, respectively. Then the map
\begin{equation*}
\tottriangle:\Big(\prd{x:X} P(x)\to Q(x)\Big)\to \hom_X(\proj 1^P,\proj 1^Q)
\end{equation*}
given by $\tottriangle(f)\defeq (\total{f},\lam{(x,y)}\refl{x})$, is an equivalence.
\end{lem}

\begin{cor}\label{cor:fib_triangle}
For any two maps $f:A\to X$ and $g:B\to X$, the map
\begin{equation*}
\fibtriangle : \hom_X(f,g) \to \prd{x:X}\fib{f}{x}\to \fib{g}{x}
\end{equation*}
given by $\lam{(h,H)}{x}{(a,p)}(h(a),\ct{H(a)^{-1}}{p})$ is an equivalence.
\end{cor}

\begin{thm}\label{thm:fam_proj}
For any small type $A:\UU$ there is an equivalence
\begin{equation*}
\eqv{(A\to \UU)}{\Big(\sm{X:\UU} X\to A\Big)}.
\end{equation*}
\end{thm}

\begin{proof}
Note that we have the function
\begin{equation*}
\varphi :\lam{B} \Big(\sm{x:A}B(x),\proj 1\Big) : (A\to \UU)\to \Big(\sm{X:\UU}X\to A\Big).
\end{equation*}
The fiber of this map at $(X,f)$ is by univalence and function extensionality equivalent to the type
\begin{equation*}
\sm{B:A\to \UU}{e:\eqv{(\sm{x:A}B(x))}{X}} \proj 1\htpy f\circ e.
\end{equation*}
By \cref{cor:fib_triangle} this type is equivalent to the type
\begin{equation*}
\sm{B:A\to \UU}\prd{a:A} \eqv{B(a)}{\fib{f}{a}},
\end{equation*}
and by `type theoretic choice', which was established in \cref{thm:choice}, this type is equivalent to
\begin{equation*}
\prd{a:A}\sm{X:\UU}\eqv{X}{\fib{f}{a}}.
\end{equation*}
We conclude that the fiber of $\varphi$ at $(X,f)$ is equivalent to the type $\esssmall(f)$. However, since $f:X\to A$ is a map between small types it is essentially small. Moreover, since being essentially small is a proposition by \cref{lem:isprop_ess_small}, it follows that $\fib{\varphi}{(X,f)}$ is contractible for every $f:X\to A$. In other words, $\varphi$ is a contractible map, and therefore it is an equivalence.
\end{proof}

\begin{rmk}
The inverse of the map
\begin{equation*}
\varphi : (A\to \UU)\to \Big(\sm{X:\UU}X\to A\Big).
\end{equation*}
constructed in \cref{thm:fam_proj} is the map $(X,f)\mapsto \fibf{f}$.
\end{rmk}

\section{The object classifier}

\begin{defn}
Let $p:E\to B$ and $p':E'\to B'$ be maps. A morphism from $p'$ to $p$ is a triple $(f,g,H)$ consisting of maps $f:B'\to B$ and $g:E'\to E$ and a homotophy $H:f\circ p'\htpy p\circ g$ witnessing that the square
\begin{equation}\label{eq:morphism_arrow}
\begin{tikzcd}
E' \arrow[r,"g"] \arrow[d,swap,"p'"] & E \arrow[d,"p"] \\
B' \arrow[r,"f"] & B
\end{tikzcd}
\end{equation}
commutes. We write $\mathsf{hom}(p',p)$ for the type of such triples $(f,g,H)$, and sometimes we write $\mathsf{hom}_f(p',p)$ for the type of pairs $(g,H)$. A morphism $(f,g,H)$ is said to be \define{cartesian} if the square in \cref{eq:morphism_arrow} is cartesian. We write $\mathsf{cart}(p',p)$ for the type of cartesian morphisms from $p'$ to $p$, and we write $\mathsf{cart}_f(p',p)$ for the type of triples $(g,H,t)$ for the type of triples, where $t:\ispullback(p',g,H)$.  
\end{defn}

\begin{defn}\label{defn:object_classifier}
A morphism $p:E\to B$ is said to be an \define{object classifier} if the type $\mathsf{cart}(p',p)$ is a proposition for each $p':E'\to B'$. If $p:E\to B$ is an object classifier, we also write 
\begin{equation*}
\mathsf{is\usc{}classified}(p')\defeq \mathsf{hom}(p',p).
\end{equation*}
\end{defn}

Our goal in this section is to show that a univalent universe is an object classifier. 

\begin{prp}\label{thm:pb_fibequiv_complete}
Let $\alpha:I\to J$ be a map, and let $A:I\to\UU$ and $B:J\to\UU$ be type families.
Then the map
\begin{equation*}
\Big(\prd{i:I}A_i\to B_{\alpha(i)}\Big) \to \mathsf{hom}_\alpha(\proj 1^A,\proj 1^B)
\end{equation*}
given by $\lam{f}(\total[\alpha]{f},\lam{(x,y)}\refl{x})$ is an equivalence. Furthermore, the map
\begin{equation*}
\Big(\prd{i:I}A_i\eqvsym B_{\alpha(i)}\Big) \to \mathsf{cart}_\alpha(\proj 1^A,\proj 1^B)
\end{equation*}
given by $\lam{e}(\total[\alpha]{e},\lam{(x,y)}\refl{x},t)$ where $t$ is the term constructed in \cref{thm:pb_fibequiv}, is an equivalence. 
\end{prp}

\begin{proof}
We have the equivalences
\begin{align*}
\prd{i:I}A_i\to B_{\alpha(i)} & \eqvsym \prd{i:I}{a:A_i}\sm{j:J}{\gamma : \alpha(i)=j} B_j \\
& \eqvsym \prd{i:I}{a:A_i}\sm{j:J}{b:B_j}\alpha(i)=j \\
& \eqvsym \prd{(i,a):\sm{i:I}A_i}\sm{(j,b):\sm{j:J}B_j}\alpha(i)=j \\
& \eqvsym \sm{f:\big(\sm{i:I}A_i\big)\to\big(\sm{j:J}B_j\big)}\alpha\circ \proj 1^A\htpy \proj 1^B\circ f \\
& \jdeq \mathsf{hom}_\alpha(\proj 1^A,\proj 1^B).
\end{align*}
It is easy to check that this composite is the asserted map. The second claim follows from \cref{thm:pb_fibequiv}.
\end{proof}

\begin{cor}\label{cor:sq_fib}
Consider a diagram of the form
\begin{equation*}
\begin{tikzcd}
A \arrow[d,swap,"f"] & B \arrow[d,"g"] \\
I \arrow[r,swap,"\alpha"] & J.
\end{tikzcd}
\end{equation*}
Then the map
\begin{equation*}
\mathsf{hom}_\alpha(f,g)\to \Big(\prd{i:I}\fib{f}{i}\to\fib{g}{\alpha(i)}\Big)
\end{equation*}
given by $\lam{(h,H)}{i}{(a,p)}(h(a),\ct{H(a)^{-1}}{\ap{\alpha}{p}})$ is an equivalence.
\end{cor}

\begin{thm}\label{thm:classifier}
Let $f:A\to B$ be a map, and let $\UU$ be a univalent universe with universal family $\mathrm{El}$ over $\UU$. Then there is an equivalence
\begin{equation*}
\eqv{\esssmall(f)}{\mathsf{cart}(f,\proj 1^{\mathrm{El}})}.
\end{equation*}
In particular, the type $\mathsf{cart}(f,\proj 1^{\mathrm{El}})$ is a proposition for each map $f$, so the universe is an object classifier in the sense of \cref{defn:object_classifier}.
\end{thm}

\begin{proof}
From \cref{cor:sq_fib} we obtain that the type of pairs $(\tilde{F},H)$ is equivalent to the type of fiberwise transformations
\begin{equation*}
\prd{b:B}\fib{f}{b}\to F(b).
\end{equation*}
By \cref{cor:pb_fibequiv} the square is a pullback square if and only if the induced map
\begin{equation*}
\prd{b:B}\fib{f}{b}\to F(b)
\end{equation*}
is a fiberwise equivalence. Thus the data $(F,\tilde{F},H,pb)$ is equivalent to the type of pairs $(F,e)$ where $e$ is a fiberwise equivalence from $\fibf{f}$ to $F$. By \cref{thm:choice} the type of pairs $(F,e)$ is equivalent to the type $\esssmall(f)$. 
\end{proof}

\begin{rmk}
For any type $A$ (not necessarily small), and any $B:A\to \UU$, the square\index{Sigma-type@{$\Sigma$-type}!as pullback of universal family|textit}
\begin{equation*}
\begin{tikzcd}[column sep=6em]
\sm{x:A}B(x) \arrow[d,swap,"\proj 1"] \arrow[r,"{\lam{(x,y)}(B(x),y)}"] & \sm{X:\UU}X \arrow[d,"\proj 1"] \\
A \arrow[r,swap,"B"] & \UU
\end{tikzcd}
\end{equation*}
is a pullback square. Therefore it follows that for any family $B:A\to\UU$ of small types, the projection map $\proj 1:\sm{x:A}B(x)\to A$ is an essentially small map.
To see that the claim is a direct consequence of \cref{thm:pb_fibequiv} we write the asserted square in its rudimentary form:
\begin{equation*}
\begin{tikzcd}[column sep=6em]
\sm{x:A}\mathrm{El}(B(x)) \arrow[d,swap,"\proj 1"] \arrow[r,"{\lam{(x,y)}(B(x),y)}"] & \sm{X:\UU}\mathrm{El}(X) \arrow[d,"\proj 1"] \\
A \arrow[r,swap,"B"] & \UU.
\end{tikzcd}
\end{equation*}
\end{rmk}

In the following theorem we show that a type is locally small if and only if its diagonal is classified by $\UU$.

\begin{thm}
Let $A$ be a type. The following are equivalent:
\begin{enumerate}
\item $A$ is locally small.\index{locally small|textit}
\item The diagonal $\delta_A : A\to A\times A$ is classified by $\UU$.\index{diagonal!of a type|textit}
\end{enumerate}
\end{thm}

\begin{proof}
The identity type $x=y$ is the fiber of $\delta_A$ at $(x,y):A\times A$. Therefore it follows that $A$ is locally small if and only if the diagonal $\delta_A$ is essentially small.
Now the result follows from \cref{thm:classifier}.
\end{proof}

\chapter{Type theoretic descent}\label{chap:descent}

In this chapter we study homotopy pushouts, which were established as higher inductive types in homotopy type theory in section 6.8 of \cite{hottbook}. From this chapter on, we will assume that universes are closed under homotopy pushouts. This is the last assumption that we will be making in the present work. In particular, we will not assume the existence of higher inductive types with some self-reference in the constructors (e.g.~the propositional truncation).

Our first main result is the descent theorem for homotopy pushouts (\cref{thm:descent,cor:descent_fib}), in which we establish that a cartesian transformation of spans
\begin{equation*}
\begin{tikzcd}
A' \arrow[d]  & S' \arrow[l] \arrow[r] \arrow[d] \arrow[dl,phantom,"\llcorner" very near start] \arrow[dr,phantom,"\lrcorner" very near start] & B' \arrow[d] \\
A & S \arrow[l] \arrow[r] & B
\end{tikzcd}
\end{equation*}
extends uniquely to a cartesian transformation of the pushout squares, i.e. a commuting cube
\begin{equation*}
\begin{tikzcd}
& S' \arrow[dl] \arrow[dr] \arrow[d] \\
A' \arrow[d] & S \arrow[dl] \arrow[dr] & B' \arrow[dl,crossing over] \arrow[d] \\
A \arrow[dr] & A'\sqcup^{S'}B' \arrow[d] \arrow[from=ul,crossing over] & B \arrow[dl] \\
& A\sqcup^S B
\end{tikzcd}
\end{equation*}
of which the vertical sides are pullback squares.

The second main theorem of this chapter, \cref{thm:cartesian_cube}, is an adaption to homotopy type theory of a theorem due to \cite{AnelBiedermanFinsterJoyal}. It is closely related to the descent theorem but can be stated without a universe: for any commuting cube
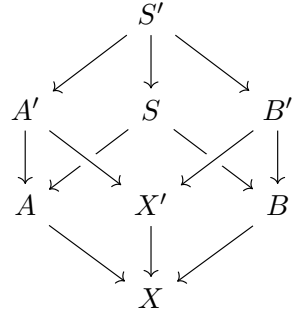
\begin{equation*}
\begin{tikzcd}
& S' \arrow[dl] \arrow[dr] \arrow[d] \\
A' \arrow[d] & S \arrow[dl] \arrow[dr] & B' \arrow[dl,crossing over] \arrow[d] \\
A \arrow[dr] & X' \arrow[d] \arrow[from=ul,crossing over] & B \arrow[dl] \\
& X
\end{tikzcd}
\end{equation*}
of which the two vertical back squares are pullback squares, the two vertical front squares are pullback squares if and only if the square
\begin{equation*}
\begin{tikzcd}
A' \sqcup^{S'} B' \arrow[r] \arrow[d] & X' \arrow[d] \\
A\sqcup^{S} B \arrow[r] & X.
\end{tikzcd}
\end{equation*}
is a pullback square. Even though this statement does not involve a universe, we use the univalence axiom in our proof that this square being pullback implies that the front two vertical squares of the cube are pullback squares. Function extensionality suffices for the converse direction.

\section{Homotopy pushouts}

\subsection{Pushouts as higher inductive types}

\begin{defn}
A \define{span} $\mathcal{S}$ from $A$ to $B$ is a triple $(S,f,g)$ consisting of a type $S$ and maps $f:S\to A$ and $g:S\to B$. We write $\mathsf{span}(A,B)$ for the type of small spans from $A$ to $B$. 
\end{defn}

\begin{defn}
Consider a span $\mathcal{S}\jdeq (S,f,g)$ from $A$ to $B$, and let $X$ be a type. A cocone with vertex $X$ on $\mathcal{S}$ is a triple $(i,j,H)$ consisting of maps $i:A\to X$, $j:B\to X$, and a homotopy $H:i\circ f\htpy j\circ g$ witnessing that the square
\begin{equation*}
\begin{tikzcd}
S \arrow[d,swap,"f"] \arrow[r,"g"] & B \arrow[d,"j"] \\
A \arrow[r,swap,"i"] & X
\end{tikzcd}
\end{equation*}
commutes. We write $\mathsf{cocone}_{\mathcal{S}}(X)$ for the type of cocones with vertex $X$ on $\mathcal{S}$. 
\end{defn}

\begin{defn}
Consider a commuting square
\begin{equation*}
\begin{tikzcd}
S \arrow[d,swap,"f"] \arrow[r,"g"] & B \arrow[d,"j"] \\
A \arrow[r,swap,"i"] & X,
\end{tikzcd}
\end{equation*}
with $H:i\circ f\htpy j\circ g$, and let $Y$ be a type. We define the operation
\begin{equation*}
\mathsf{cocone\usc{}map}((i,j,H),Y) \defeq (X\to Y) \to \mathsf{cocone}_{\mathcal{S}}(Y).
\end{equation*}
by $h\mapsto (h\circ i,h\circ j,h\cdot H)$. 
\end{defn}

\begin{defn}
A commuting square
\begin{equation*}
\begin{tikzcd}
S \arrow[r,"g"] \arrow[d,swap,"f"] & B \arrow[d,"j"] \\
A \arrow[r,swap,"i"] & X
\end{tikzcd}
\end{equation*}
with $H:i\circ f \htpy j\circ g$ is said to be a \define{(homotopy) pushout square}\index{pushout square} if the cocone $(i,j,H)$ with vertex $X$ on the span $\mathcal{S}\jdeq (S,f,g)$
satisfies the \define{universal property of pushouts}\index{universal property!of pushouts|textbf}, which asserts that the map
\begin{equation*}
\mathsf{cocone\usc{}map}(i,j,H):(X\to Y)\to \mathsf{cocone}(Y)
\end{equation*}
is an equivalence for any type $Y$. Sometimes pushout squares are also called \define{cocartesian squares}\index{cocartesian square|textbf}.
\end{defn}

\begin{defn}
Consider a pushout square
\begin{equation*}
\begin{tikzcd}
S \arrow[r,"g"] \arrow[d,swap,"f"] & B \arrow[d,"j"] \\
A \arrow[r,swap,"i"] & X
\end{tikzcd}
\end{equation*}
with $H:i\circ f \htpy j\circ g$, and consider a cocone $(i',j',H')$ with vertex $Y$ on the same span $\mathcal{S}\jdeq(S,f,g)$. Then the unique map $h:X\to Y$ such that 
\begin{equation*}
\mathsf{cocone\usc{}map}((i,j,H),Y,h)= (i',j',H')
\end{equation*}
is called the \define{cogap map} of $(i',j',H')$. We also write $\mathsf{cogap}(i',j',H')$ for the cogap map, and we write
\begin{align*}
\mathsf{left\usc{}comp}(i',j',H') & : i' \htpy \mathsf{cogap}(i',j',H')\circ \inl  \\
\mathsf{right\usc{}comp}(i',j',H') & : j' \htpy \mathsf{cogap}(i',j',H')\circ \inr  \\
\mathsf{coh\usc{}comp}(i',j',H') & : \ct{(\mathsf{left\usc{}comp}(i',j',H')\cdot g)}{H'} \htpy \ct{(\mathsf{cogap}(i',j',H')\cdot \glue)}{(\mathsf{right\usc{}comp}(i',j',H')\cdot f)}.
\end{align*}
for the homotopies determining the uniqueness of $\mathsf{cogap}(i',j',H')$.
\end{defn}

\begin{prp}\label{thm:pushout_up}
Consider a commuting square\index{universal property!of pushouts|textit}
\begin{equation*}
\begin{tikzcd}
S \arrow[r,"g"] \arrow[d,swap,"f"] & B \arrow[d,"j"] \\
A \arrow[r,swap,"i"] & X,
\end{tikzcd}
\end{equation*}
with $H:i\circ f\htpy j\circ g$. The following are equivalent:
\begin{enumerate}
\item The square is a pushout square.
\item The square
\begin{equation*}
\begin{tikzcd}
Y^X \arrow[r,"\blank\circ j"] \arrow[d,swap,"\blank\circ i"] & Y^B \arrow[d,"\blank\circ g"] \\
Y^A \arrow[r,swap,"\blank\circ f"] & Y^S,
\end{tikzcd}
\end{equation*}
which commutes by the homotopy
\begin{equation*}
\lam{h} \mathsf{eq\usc{}htpy}(h\cdot H),
\end{equation*}
is a pullback square, for every type $Y$.
\item For every type family $P$ over $X$, the square
\begin{equation*}
\begin{tikzcd}[column sep=9em]
\prd{x:X}P(x) \arrow[r,"\blank\circ j"] \arrow[d,swap,"\blank\circ i"] & \prd{b:B}P(j(b)) \arrow[d,"\blank\circ g"] \\
\prd{a:A}P(i(a)) \arrow[r,swap,"{\lam{h}{x} \tr_{P}(H(x),h(f(x)))}"] & \prd{x:S}P(j(g(x))),
\end{tikzcd}
\end{equation*}
which commutes by the homotopy
\begin{equation*}
\lam{h}\mathsf{eq\usc{}htpy}(\lam{x}\apd{h}{H(x)})
\end{equation*}
is a pullback square. This property is also called the \define{dependent universal property of pushouts}.
\item The gap map of the square
\begin{equation*}
\begin{tikzcd}
\prd{x:X}P(x) \arrow[r] \arrow[d] & \prd{b:B}P(j(b)) \arrow[d] \\
\prd{a:A}P(i(a)) \arrow[r] & \prd{x:S}P(j(g(x)))
\end{tikzcd}
\end{equation*}
has a section, for any type family $P$ over $X$. This property is also called the \define{induction principle of pushouts}.
\end{enumerate}
\end{prp}

\begin{defn}
From now on we will assume that any span has a pushout, and moreover that universes are closed under pushouts. We will write $A\sqcup^{\mathcal{S}} B$ for the pushout of the span $\mathcal{S}\jdeq(S,f,g)$ from $A$ to $B$. The type $A\sqcup^{\mathcal{S}} B$ comes equipped with a colimiting cocone $(\inl,\inr,\glue)$, as displayed in the pushout square
\begin{equation*}
\begin{tikzcd}
S \arrow[d,swap,"f"] \arrow[r,"g"] & B \arrow[d,"\inr"] \\
A \arrow[r,swap,"\inl"] & A \sqcup^{\mathcal{S}} B.
\end{tikzcd}
\end{equation*}
\end{defn}

\begin{rmk}
We note that if $\mathcal{S}\jdeq (S,f,g)$ is a span of \emph{pointed} types and pointed maps between them, then the pushout $A\sqcup^{\mathcal{S}} B$ of $\mathcal{S}$ is again a pointed type. The cocone $(\inl,\inr,\glue)$ consists of two pointed maps and a pointed homotopy filling the square of pointed maps. Moreover, the pushout $A\sqcup^{\mathcal{S}} B$ satisfies a pointed version of the universal property: for any pointed type $Y$ the square
\begin{equation*}
\begin{tikzcd}
(A\sqcup^{\mathcal{S}} B \to_\ast Y) \arrow[r] \arrow[d] & (B\to_\ast Y) \arrow[d] \\
(A\to_\ast Y) \arrow[r] & (S\to_\ast Y)
\end{tikzcd}
\end{equation*}
is a pullback square. 
\end{rmk}

\subsection{Examples of pushouts}

\begin{defn}
Let $X$ be a type. We define the \define{suspension}\index{suspension|textbf} $\susp X$\index{SX@{$\susp X$}|textbf} of $X$ to be the pushout of the span
\begin{equation*}
\begin{tikzcd}
X \arrow[r] \arrow[d] & \unit \arrow[d,"\inr"] \\
\unit \arrow[r,swap,"\inl"] & \susp X 
\end{tikzcd}
\end{equation*}
We will write $\north\defeq\inl(\ttt)$ and $\south\defeq\inr(\ttt)$. 
\end{defn}

\begin{rmk}By the universal property it follows that the map
\begin{equation*}
(\susp X \to Y) \to \sm{y,y':Y}X\to (y=y')
\end{equation*}
given by $h\mapsto (h(\north),h(\south),h\cdot\glue)$ is an equivalence. 

Moreover, if $X$ is a pointed type, then the suspension is considered to be a pointed type with base point $\north$. By the universal property of $\susp X$ it follows that the square
\begin{equation*}
\begin{tikzcd}
(\susp X\to_\ast Y) \arrow[r] \arrow[d] & (\unit \to_\ast Y) \arrow[d] \\
(\unit\to_\ast Y) \arrow[r] & (X\to_\ast Y)
\end{tikzcd}
\end{equation*}
is a pullback square. Since $\unit\to_\ast Y$ is contractible, it follows that 
\begin{equation*}
(\susp X\to_\ast Y) \eqvsym \loopspace{X\to_\ast Y} \eqvsym X \to_\ast \loopspace Y
\end{equation*}
\end{rmk}

\begin{defn}
Given a map $f:A\to B$, we define the \define{cofiber}\index{cofiber|textbf} $\mathsf{cof}_f$\index{cofib_f@{$\mathsf{cof}_f$}|textbf} of $f$ as the pushout
\begin{equation*}
\begin{tikzcd}
A \arrow[r,"f"] \arrow[d] & B \arrow[d,"\inr"] \\
\unit \arrow[r,swap,"\inl"] & \mathsf{cof}_f. 
\end{tikzcd}
\end{equation*}
The cofiber of a map is sometimes also called the \define{mapping cone}\index{mapping cone|textbf}.
\end{defn}

\begin{defn}
We define the \define{join}\index{join} $\join{X}{Y}$\index{join X Y@{$\join{X}{Y}$}|textbf} of $X$ and $Y$ to be the pushout 
\begin{equation*}
\begin{tikzcd}
X\times Y \arrow[r,"\proj 2"] \arrow[d,swap,"\proj 1"] & Y \arrow[d,"\inr"] \\
X \arrow[r,swap,"\inl"] & X \ast Y. 
\end{tikzcd}
\end{equation*}
\end{defn}

\begin{defn}
We define the \define{$n$-sphere}\index{n-sphere@{$n$-sphere}|textbf} $\sphere{n}$\index{Sn@{$\sphere{n}$}|textbf} for any $n\geq -1$ by induction on $n$, by taking
\begin{align*}
\sphere{-1} & \defeq \emptyt \\
\sphere{0} & \defeq \bool \\
\sphere{n+1} & \defeq \join{\bool}{\sphere{n}}.
\end{align*}
\end{defn}

\begin{defn}
Suppose $A$ and $B$ are pointed types, with base points $a_0$ and $b_0$, respectively. The \define{(binary) wedge}\index{wedge@(binary) wedge|textbf} $A\vee B$ of $A$ and $B$ is defined as the pushout
\begin{equation*}
\begin{tikzcd}
\bool \arrow[r] \arrow[d] & A+B \arrow[d] \\
\unit \arrow[r] & A\vee B.
\end{tikzcd}
\end{equation*}
\end{defn}

\begin{defn}
Given a type $I$, and a family of pointed types $A$ over $i$, with base points $a_0(i)$. We define the \define{(indexed) wedge}\index{wedge@{(indexed) wedge}|textbf} $\bigvee_{(i:I)}A_i$ as the pushout
\begin{equation*}
\begin{tikzcd}[column sep=huge]
I \arrow[d] \arrow[r,"{\lam{i}(i,a_0(i))}"] & \sm{i:I}A_i \arrow[d] \\
\unit \arrow[r] & \bigvee_{(i:I)} A_i.
\end{tikzcd}
\end{equation*}
\end{defn}

\begin{defn}
Suppose $A$ and $B$ are pointed types. We define the \define{wedge inclusion} $\mathsf{wedge\usc{}in}:A\vee B\to A\times B$ to be the unique map obtained via the universal property of pushouts as indicated in the diagram
\begin{equation*}
\begin{tikzcd}
\unit \arrow[d] \arrow[r] &[1em] B \arrow[d] \arrow[ddr,bend left=15,"{\lam{b}(a_0,b)}"] \\
A \arrow[r] \arrow[drr,bend right=15,swap,"{\lam{a}(a,b_0)}"] & A\vee B \arrow[dr,swap,"\mathsf{wedge\usc{}in}" near start] \\
& & A \times B.
\end{tikzcd}
\end{equation*}
We define the \define{smash product} $A\wedge B$ of $A$ and $B$ as the cofiber of the wedge inclusion, i.e.~as a pushout
\begin{equation*}
\begin{tikzcd}[column sep=large]
A\vee B \arrow[r,"\mathsf{wedge\usc{}in}"] \arrow[d] & A\times B \arrow[d] \\
\unit \arrow[r] & A\wedge B.
\end{tikzcd}
\end{equation*}
\end{defn}

\subsection{Properties of iterated pushouts}
The following corollary is also called the \define{pasting property} of pullbacks.\index{pasting property!of pullbacks|textit}\marginnote{edit}

\begin{cor}\label{thm:pb_pasting}
Consider a commuting diagram of the form
\begin{equation*}
\begin{tikzcd}
A \arrow[r,"k"] \arrow[d,swap,"f"] & B \arrow[r,"l"] \arrow[d,"g"] & C \arrow[d,"h"] \\
X \arrow[r,swap,"i"] & Y \arrow[r,swap,"j"] & Z
\end{tikzcd}
\end{equation*}
with homotopies $H:i\circ f\htpy g\circ k$ and $K:j\circ g\htpy h\circ l$, and the homotopy
\begin{equation*}
\ct{(j\cdot H)}{(K\cdot k)}:j\circ i\circ f\htpy h\circ l\circ k
\end{equation*}
witnessing that the outer rectangle commutes. Furthermore, suppose that the square on the right is a pullback square. Then the following are equivalent:
\begin{samepage}%
\begin{enumerate}
\item The square on the left is a pullback square.
\item The outer rectangle is a pullback square.
\end{enumerate}%
\end{samepage}%
\end{cor}

\begin{proof}
The commutativity of the two squares induces fiberwise transformations
\begin{align*}
& \prd{x:X}\fib{f}{x}\to \fib{g}{i(x)} \\
& \prd{y:Y}\fib{g}{y}\to \fib{h}{j(y)}.
\end{align*}
By the assumption that the square on the right is a pullback square, it follows from \cref{cor:pb_fibequiv} that the fiberwise transformation
\begin{equation*}
\prd{y:Y}\fib{g}{y}\to\fib{h}{j(y)}
\end{equation*}
is a fiberwise equivalence. Therefore it follows from 3-for-2 property of equivalences that the fiberwise transformation
\begin{equation*}
\prd{x:X}\fib{f}{x}\to\fib{g}{i(x)}
\end{equation*}
is a fiberwise equivalence if and only if the fiberwise transformation
\begin{equation*}
\prd{x:X}\fib{f}{x}\to\fib{h}{j(i(x))}
\end{equation*}
is a fiberwise equivalence. Now the claim follows from one more application of \cref{cor:pb_fibequiv}.
\end{proof}

\begin{cor}
Consider a commuting cube
\begin{equation*}
\begin{tikzcd}
& S' \arrow[dl] \arrow[dr] \arrow[d] \\
A' \arrow[d] & S \arrow[dl] \arrow[dr] & B' \arrow[dl,crossing over] \arrow[d] \\
A \arrow[dr] & X' \arrow[d] \arrow[from=ul,crossing over] & B \arrow[dl] \\
& X,
\end{tikzcd}
\end{equation*}
of which the two front squares are pullback squares. Then the back left square is a pullback square if and only if the back right square is a pullback square.
\end{cor}

\begin{prp}\label{thm:pushout_pasting}
Consider the following configuration of commuting squares:\index{pushout!pasting property|textit}\index{pasting property!for pushouts|textit}
\begin{equation*}
\begin{tikzcd}
A \arrow[r,"i"] \arrow[d,swap,"f"] & B \arrow[r,"k"] \arrow[d,swap,"g"] & C \arrow[d,"h"] \\
X \arrow[r,swap,"j"] & Y \arrow[r,swap,"l"] & Z
\end{tikzcd}
\end{equation*}
with homotopies $H:j\circ f\htpy g\circ i$ and $K:l\circ g\htpy h\circ k$, and suppose that the square on the left is a pushout square. 
Then the square on the right is a pushout square if and only if the outer rectangle is a pushout square.
\end{prp}

\begin{proof}
Let $T$ be a type. Taking the exponent $T^{(\blank)}$ of the entire diagram of the statement of the theorem, we obtain the following commuting diagram
\begin{equation*}
\begin{tikzcd}
T^Z \arrow[r,"\blank\circ l"] \arrow[d,swap,"\blank\circ h"] & T^Y \arrow[d,swap,"\blank\circ g"] \arrow[r,"\blank\circ j"] & T^X \arrow[d,"\blank\circ f"] \\
T^C \arrow[r,swap,"\blank\circ k"] & T^B \arrow[r,swap,"\blank\circ i"] & T^A.
\end{tikzcd}
\end{equation*}
By the assumption that $Y$ is the pushout of $B\leftarrow A \rightarrow X$, it follows that the square on the right is a pullback square. It follows by \autoref{thm:pb_pasting} that the rectangle on the left is a pullback if and only if the outer rectangle is a pullback. Thus the statement follows by the second characterization in \autoref{thm:pushout_up}.
\end{proof}

\begin{lem}
Consider a map $f:A\to B$. Then the cofiber of the map $\inr:B\to \mathsf{cof}_f$ is equivalent to the suspension $\susp{A}$ of $A$. 
\end{lem}

\begin{prp}
Consider a commuting square
\begin{equation*}
\begin{tikzcd}
A \arrow[r,"i"] \arrow[d,swap,"f"] & B \arrow[d,"g"] \\
X \arrow[r,swap,"j"] & Y
\end{tikzcd}
\end{equation*}
and write $\mathsf{cogap}: X\sqcup^A B\to Y$ for the cogap map. 
Then the square
\begin{equation*}
\begin{tikzcd}
\mathsf{cof}_f \arrow[d] \arrow[r] & \mathsf{cof}_g \arrow[d] \\
\unit \arrow[r] & \mathsf{cof}_{\mathsf{cogap}}
\end{tikzcd}
\end{equation*}
is a pushout square.
\end{prp}

\section{Descent for pushouts}\label{sec:descent}

\subsection{Type families over pushouts}

\begin{defn}
Consider a commuting square
\begin{equation*}
\begin{tikzcd}
S \arrow[r,"g"] \arrow[d,swap,"f"] & B \arrow[d,"j"] \\
A \arrow[r,swap,"i"] & X.
\end{tikzcd}
\end{equation*}
with $H:i\circ f\htpy j\circ g$, where all types involved are in $\UU$. The type $\mathsf{Desc}(\mathcal{S})$\index{Desc@{$\mathsf{Desc}(\mathcal{S})$}|textbf} of \define{descent data}\index{descent data|textbf} for $X$, is defined to be the type of triples $(P_A,P_B,P_S)$ consisting of
\begin{align*}
P_A & : A \to \UU \\
P_B & : B \to \UU \\
P_S & : \prd{x:S} \eqv{P_A(f(x))}{P_B(g(x))}.
\end{align*}
Furthermore, we define the map\index{desc_fam@{$\mathsf{desc\usc{}fam}_{\mathcal{S}}$}|textbf}
\begin{equation*}
\mathsf{desc\usc{}fam}_{\mathcal{S}}(i,j,H) : (X\to \UU)\to \mathsf{Desc}(\mathcal{S})
\end{equation*}
by $P\mapsto (P\circ i,P\circ j,\lam{x}\mathsf{tr}_P(H(x)))$.
\end{defn}

\begin{prp}\label{thm:desc_fam}
Consider a commuting square
\begin{equation*}
\begin{tikzcd}
S \arrow[r,"g"] \arrow[d,swap,"f"] & B \arrow[d,"j"] \\
A \arrow[r,swap,"i"] & X.
\end{tikzcd}
\end{equation*}
with $H:i\circ f\htpy j\circ g$. If the square is a pushout square, then the function\index{desc_fam@{$\mathsf{desc\usc{}fam}_{\mathcal{S}}$}!is an equivalence|textit}
\begin{equation*}
\mathsf{desc\usc{}fam}_{\mathcal{S}}(i,j,H) : (X\to \UU)\to \mathsf{Desc}(\mathcal{S})
\end{equation*}
is an equivalence.
\end{prp}

\begin{proof}
By the 3-for-2 property of equivalences it suffices to construct an equivalence $\varphi:\mathsf{cocone}_{\mathcal{S}}(\UU)\to\mathsf{Desc}(\mathcal{S})$ such that the triangle
\begin{equation*}
\begin{tikzcd}
& \UU^X \arrow[dl,swap,"{\mathsf{cocone\usc{}map}_{\mathcal{S}}(i,j,H)}"] \arrow[dr,"{\mathsf{desc\usc{}fam}_{\mathcal{S}}(i,j,H)}"] & \phantom{\mathsf{cocone}_{\mathcal{S}}(\UU)} \\
\mathsf{cocone}_{\mathcal{S}}(\UU) \arrow[rr,densely dotted,"\eqvsym","\varphi"'] & & \mathsf{Desc}(\mathcal{S})
\end{tikzcd}
\end{equation*}
commutes.

Since we have equivalences
\begin{equation*}
\mathsf{equiv\usc{}eq}:\eqv{\Big(P_A(f(x))=P_B(g(x))\Big)}{\Big(\eqv{P_A(f(x))}{P_B(g(x))}\Big)}
\end{equation*}
for all $x:S$, we obtain an equivalence on the dependent products
\begin{equation*}
\eqv{\Big(\prd{x:S}P_A(f(x))=P_B(g(x))\Big)}{\Big(\prd{x:S}\eqv{P_A(f(x))}{P_B(g(x))}\Big)}.
\end{equation*}
by post-composing with the equivalences $\mathsf{equiv\usc{}eq}$. 
We define $\varphi$ to be the induced map on total spaces. Explicitly, we have
\begin{equation*}
\varphi\defeq \lam{(P_A,P_B,K)}(P_A,P_B,\lam{x}\mathsf{equiv\usc{}eq}(K(x))).
\end{equation*}
Then $\varphi$ is an equivalence by \cref{thm:fib_equiv}, and the triangle commutes because there is a homotopy
\begin{equation*}
\mathsf{equiv\usc{}eq}(\ap{P}{H(x)}) \htpy \mathsf{tr}_P(H(x)). \qedhere
\end{equation*}
\end{proof}

\begin{cor}\label{cor:desc_fam}
Consider descent data $(P_A,P_B,P_S)$ for a pushout square as in \cref{thm:desc_fam}.
Then the type of quadruples $(P,e_A,e_B,e_S)$ consisting of a family $P:X\to\UU$ equipped with fiberwise equivalences
\begin{samepage}
\begin{align*}
e_A & : \prd{a:A}\eqv{P_A(a)}{P(i(a))} \\
e_B & : \prd{b:B}\eqv{P_B(a)}{P(j(b))}
\end{align*}
\end{samepage}%
and a homotopy $e_S$ witnessing that the square
\begin{equation*}
\begin{tikzcd}[column sep=huge]
P_A(f(x)) \arrow[r,"e_A(f(x))"] \arrow[d,swap,"P_S(x)"] & P(i(f(x))) \arrow[d,"\mathsf{tr}_P(H(x))"] \\
P_B(g(x)) \arrow[r,swap,"e_B(g(x))"] & P(j(g(x)))
\end{tikzcd}
\end{equation*}
commutes, is contractible.
\end{cor}

\begin{proof}
The fiber of $\mathsf{desc\usc{}fam}_{\mathcal{S}}(i,j,H)$ map at $(P_A,P_B,P_S)$ is equivalent to the type of quadruples $(P,e_A,e_B,e_S)$ as described in the theorem, which are contractible by \cref{thm:contr_equiv}.
\end{proof}

For the remainder of this subsection we consider a pushout square
\begin{equation*}
\begin{tikzcd}
S \arrow[r,"g"] \arrow[d,swap,"f"] & B \arrow[d,"j"] \\
A \arrow[r,swap,"i"] & X.
\end{tikzcd}
\end{equation*}
with $H:i\circ f\htpy j\circ g$, descent data
\begin{align*}
P_A & : A \to \UU \\
P_B & : B \to \UU \\
P_S & : \prd{x:S} \eqv{P_A(f(x))}{P_B(g(x))},
\end{align*}
and a family $P:X\to\UU$ equipped with 
\begin{align*}
e_A & : \prd{a:A}\eqv{P_A(a)}{P(i(a))} \\
e_B & : \prd{b:B}\eqv{P_B(a)}{P(j(b))}
\end{align*}
and a homotopy $e_S$ witnessing that the square
\begin{equation*}
\begin{tikzcd}[column sep=huge]
P_A(f(x)) \arrow[r,"e_A(f(x))"] \arrow[d,swap,"P_S(x)"] & P(i(f(x))) \arrow[d,"\mathsf{tr}_P(H(x))"] \\
P_B(g(x)) \arrow[r,swap,"e_B(g(x))"] & P(j(g(x)))
\end{tikzcd}
\end{equation*}
commutes.

\begin{defn}
We define the commuting square
\begin{equation*}
\begin{tikzcd}[column sep=6em]
\sm{x:S}P_A(f(x)) \arrow[d,swap,"{f'\,\defeq\,\total[f]{\lam{x}\idfunc[P_A(f(x))]}}"] \arrow[r,"{g'\,\defeq\, \total[g]{e_S}}"] & \sm{b:B}P_B(b) \arrow[d,"{j'\,\defeq\, \total[j]{e_B}}"] \\
\sm{a:A}P_A(a) \arrow[r,swap,"{i'\, \defeq\, \total[i]{e_A}}"] & \sm{x:X}P(x)
\end{tikzcd}
\end{equation*}
with the homotopy $H':i'\circ f'\htpy j'\circ g'$ defined as
\begin{equation*}
\lam{(x,y)}\mathsf{eq\usc{}pair}(H(x),e_S(x,y)^{-1}).
\end{equation*}
Furthermore, we will write $\mathcal{S'}$ for the span
\begin{equation*}
\begin{tikzcd}
\sm{a:A}P_A(a) & \sm{x:S}P_A(f(x)) \arrow[l,swap,"{f'}"] \arrow[r,"{g'}"] & \sm{b:B}P_B(b).
\end{tikzcd}
\end{equation*}
\end{defn}

We now state the flattening lemma for pushouts, which should be compared to the flattening lemma for coequalizers, stated in Lemma 6.12.2 of \cite{hottbook}. We note that, using the dependent universal property of pushouts, our proof is substantially shorter.

\begin{lem}[The flattening lemma]\label{lem:flattening}
The commuting square\index{flattening lemma!for pushouts|textit}
\begin{equation*}
\begin{tikzcd}
\sm{x:S}P_A(f(x)) \arrow[d,swap,"{f'}"] \arrow[r,"{g'}"] & \sm{b:B}P_B(b) \arrow[d,"{j'}"] \\
\sm{a:A}P_A(a) \arrow[r,swap,"{i'}"] & \sm{x:X}P(x)
\end{tikzcd}
\end{equation*}
is a pushout square.
\end{lem}

\begin{proof}
Note that we have a commuting cube
\begin{equation*}
\begin{tikzcd}[row sep=large]
& Y^{\sm{x:X}P(x)} \arrow[dl] \arrow[d,"\mathsf{ev\usc{}pair}"] \arrow[dr] \\
Y^{\sm{a:A}P_A(a)} \arrow[d,swap,"\mathsf{ev\usc{}pair}"] & \prd{x:X}Y^{P(x)} \arrow[dl] \arrow[dr] & Y^{\sm{b:B}P_B(b)} \arrow[dl,crossing over] \arrow[d,"\mathsf{ev\usc{}pair}"] \\
\prd{a:A}Y^{P_A(a)} \arrow[dr] & Y^{\sm{x:S}P_A(f(x))} \arrow[from=ul,crossing over] \arrow[d,swap,"\mathsf{ev\usc{}pair}"] & \prd{b:B}Y^{P_B(b)} \arrow[dl] \\
\phantom{\prd{b:B}Y^{P_B(b)}} & \prd{x:S}Y^{P_A(f(x))} & \phantom{\prd{a:A}Y^{P_A(a)}}
\end{tikzcd}
\end{equation*}
for any type $Y$. In this cube, the bottom square is a pullback square by property (iii) of \cref{thm:pushout_up}. The vertical maps (of the form $\mathsf{ev\usc{}pair}$) are equivalences, so it follows that the top square is a pullback square. We conclude that $\sm{x:X}P(x)$ is a pushout.
\end{proof}

\subsection{The descent property for pushouts}

\begin{defn}
Consider a span $\mathcal{S}$ from $A$ to $B$, and a span $\mathcal{S}'$ from $A'$ to $B'$. A \define{cartesian transformation} of spans\index{cartesian transformation!of spans|textbf} from $\mathcal{S}'$ to $\mathcal{S}$ is a tuple
\begin{equation*}
(h_A,h_S,h_B,F,G,p_f,p_g)
\end{equation*}
consisting of maps $h_A:A'\to A$, $h_S:S'\to S$, and $h_B:B'\to B$, as indicated in the diagram
\begin{equation*}
\begin{tikzcd}
A' \arrow[d,swap,"h_A"]  & S' \arrow[l,swap,"{f'}"] \arrow[r,"{g'}"] \arrow[d,swap,"h_S"] & B' \arrow[d,"h_B"] \\
A & S \arrow[l,"f"] \arrow[r,swap,"g"] & B,
\end{tikzcd}
\end{equation*}
with homotopies $F:f\circ h_S\htpy h_A\circ f'$ and $G:g\circ h_S\htpy h_B\circ g'$, satisfying the conditions
\begin{align*}
p_f & : \mathsf{is\usc{}pullback}(h_S,f',F) \\
p_g & : \mathsf{is\usc{}pullback}(h_S,g',G)
\end{align*}
that both squares are pullback squares. We write $\mathsf{cart}(\mathcal{S}',\mathcal{S})$\index{cart(S,S')@{$\mathsf{cart}(\mathcal{S},\mathcal{S}')$}|textbf} for the type of cartesian transformations from $\mathcal{S}'$ to $\mathcal{S}$, and we write
\begin{equation*}
\mathsf{Cart}(\mathcal{S}) \defeq \sm{A',B':\UU}{\mathcal{S}':\mathsf{span}(A',B')}\mathsf{cart}(\mathcal{S}',\mathcal{S}).
\end{equation*}
\end{defn}

Given descent data $(P_A,P_B,P_S)$ on a span $\mathcal{S}$ from $A$ to $B$, we obtain a cartesian transformation
\begin{equation*}
\begin{tikzcd}[column sep=large]
\sm{a:A}P_A(a) \arrow[d,swap,"\proj 1"] & \sm{x:S}P_A(f(x)) \arrow[d,swap,"\proj 1"] \arrow[l,swap,"{\total[f]{\idfunc}}"] \arrow[r,"{\total[g]{P_S}}"] & \sm{b:B}P_B(b) \arrow[d,"\proj 1"] \\
A & S \arrow[l,"f"] \arrow[r,swap,"g"] & B
\end{tikzcd}
\end{equation*}
with the canonical homotopies witnessing that the squares commute. Note that both the left and right commuting squares are pullback squares by \cref{thm:pb_fibequiv}. Thus we obtain an operation
\begin{equation*}
\mathsf{cart\usc{}desc}_{\mathcal{S}}:\mathsf{Desc}(\mathcal{S})\to \mathsf{Cart}(\mathcal{S}).
\end{equation*}

\begin{lem}\label{lem:cart_desc}
For any span $\mathcal{S}$, the operation\index{cart_desc@{$\mathsf{cart\usc{}desc}_{\mathcal{S}}$}|textit}
\begin{equation*}
\mathsf{cart\usc{}desc}_{\mathcal{S}}:\mathsf{Desc}(\mathcal{S})\to \mathsf{Cart}(\mathcal{S})
\end{equation*}
is an equivalence.
\end{lem}

\begin{proof}
Note that by \cref{thm:pb_fibequiv_complete} it follows that the types of triples $(f',F,p_f)$ and $(g',G,p_g)$ are equivalent to the types of fiberwise equivalences
\begin{align*}
& \prd{x:S}\eqv{\fib{h_S}{x}}{\fib{h_A}{f(x)}} \\
& \prd{x:S}\eqv{\fib{h_S}{x}}{\fib{h_B}{g(x)}}
\end{align*} 
respectively. Furthermore, by \cref{thm:fam_proj} the types of pairs $(S',h_S)$, $(A',h_A)$, and $(B',h_B)$ are equivalent to the types $S\to \UU$, $A\to \UU$, and $B\to \UU$, respectively. Therefore it follows that the type $\mathsf{Cart}(\mathcal{S})$ is equivalent to the type of tuples $(Q,P_A,\varphi,P_B,P_S)$ consisting of
\begin{align*}
Q & : S\to \UU \\
P_A & : A \to \UU \\
P_B & : B \to \UU \\
\varphi & : \prd{x:S}\eqv{Q(x)}{P_A(f(x))} \\
P_S & : \prd{x:S}\eqv{Q(x)}{P_B(g(x))}.
\end{align*}
However, the type of $\varphi$ is equivalent to the type $P_A\circ f=Q$. Thus we see that the type of pairs $(Q,\varphi)$ is contractible, so our claim follows.
\end{proof}

\begin{defn}
Consider a commuting square
\begin{equation*}
\begin{tikzcd}
S \arrow[r,"g"] \arrow[d,swap,"f"] & B \arrow[d,"j"] \\
A \arrow[r,swap,"i"] & X
\end{tikzcd}
\end{equation*}
with $H:i\circ f\htpy j\circ g$. 
We define an operation\index{cart map!{$\mathsf{cart\usc{}map}_{\mathcal{S}}$}|textbf}
\begin{equation*}
\mathsf{cart\usc{}map}_{\mathcal{S}}:{\Big(\sm{X':\UU}X'\to X\Big)}\to \mathsf{Cart}(\mathcal{S}).
\end{equation*}
\end{defn}

\begin{proof}[Construction]
Let $X':\UU$ and $h_X:X'\to X$. Then we define $A'$, $B'$, and $S'$ as the pullbacks
\begin{align*}
A' & \defeq A\times_X X' \\
B' & \defeq B\times_X X' \\
S' & \defeq S\times_A A',
\end{align*}
resulting in a diagram of the form
\begin{equation*}
\begin{tikzcd}
& S' \arrow[dl] \arrow[dr,densely dotted] \arrow[d] \\
A' \arrow[d] & S \arrow[dl] \arrow[dr] & B' \arrow[dl,crossing over] \arrow[d] \\
A \arrow[dr] & X' \arrow[d] \arrow[from=ul,crossing over] & B \arrow[dl] \\
& X
\end{tikzcd}
\end{equation*}
By the universal property of $B'$ it follows that there is a unique map $g':S'\to B'$ making the cube commute. 
Moreover, since the two front squares and the back left squares are pullback squares by construction, it follows by \cref{thm:pb_pasting} that also the back right square is a pullback square. Thus we obtain a cartesian transformation of spans.
\end{proof}

The following theorem is analogous to \cref{thm:desc_fam}.

\begin{thm}[The descent theorem for pushouts]\label{thm:descent}\index{descent theorem!for pushouts|textit}
Consider a commuting square
\begin{equation*}
\begin{tikzcd}
S \arrow[r,"g"] \arrow[d,swap,"f"] & B \arrow[d,"j"] \\
A \arrow[r,swap,"i"] & X
\end{tikzcd}
\end{equation*}
with $H:i\circ f\htpy j\circ g$. 
If this square is a pushout square, then the operation $\mathsf{cart\usc{}map}_{\mathcal{S}}$\index{cart map!{$\mathsf{cart\usc{}map}_{\mathcal{S}}$}!is an equivalence|textit} is an equivalence
\begin{equation*}
\eqv{\Big(\sm{X':\UU}X'\to X\Big)}{\mathsf{Cart}(\mathcal{S})}
\end{equation*}
\end{thm}

\begin{proof}
It suffices to show that the square
\begin{equation*}
\begin{tikzcd}[column sep=huge]
\UU^X \arrow[r,"{\mathsf{desc\usc{}fam}_{\mathcal{S}}(i,j,H)}"] \arrow[d,swap,"\mathsf{map\usc{}fam}_X"] & \mathsf{Desc}(\mathcal{S}) \arrow[d,"\mathsf{cart\usc{}desc}_{\mathcal{S}}"] \\
\sm{X':\UU}X^{X'} \arrow[r,swap,"\mathsf{cart\usc{}map}_{\mathcal{S}}"] & \mathsf{Cart}(\mathcal{S})
\end{tikzcd}
\end{equation*}
commutes. To see that this suffices, note that the operation $\mathsf{map\usc{}fam}_X$ is an equivalence by \cref{thm:fam_proj}, the operation $\mathsf{desc\usc{}fam}_{\mathcal{S}}(i,j,H)$ is an equivalence by \cref{thm:desc_fam}, and the operation $\mathsf{cart\usc{}desc}_{\mathcal{S}}$ is an equivalence by \cref{lem:cart_desc}.

To see that the square commutes, note that the composite
\begin{equation*}
\mathsf{cart\usc{}map}_{\mathcal{S}}\circ \mathsf{map\usc{}fam}_X
\end{equation*}
takes a family $P:X\to \UU$ to the cartesian transformation of spans
\begin{equation*}
\begin{tikzcd}
A\times_X\tilde{P} \arrow[d,swap,"\pi_1"] & S\times_A\Big(A\times_X\tilde{P}\Big) \arrow[l] \arrow[r] \arrow[d,swap,"\pi_1"] & B\times_X\tilde{P} \arrow[d,"\pi_1"] \\
A & S \arrow[l] \arrow[r] & B,
\end{tikzcd}
\end{equation*}
where $\tilde{P}\defeq\sm{x:X}P(x)$.

The composite 
\begin{equation*}
\mathsf{cart\usc{}desc}_{\mathcal{S}}\circ \mathsf{desc\usc{}fam}_X
\end{equation*}
takes a family $P:X\to \UU$ to the cartesian transformation of spans
\begin{equation*}
\begin{tikzcd}
\sm{a:A}P(i(a)) \arrow[d] & \sm{s:S}P(i(f(s))) \arrow[l] \arrow[r] \arrow[d] & \sm{b:B}P(j(b)) \arrow[d] \\
A & S \arrow[l] \arrow[r] & B
\end{tikzcd}
\end{equation*}
These cartesian natural transformations are equal by \cref{thm:pb_fibequiv}.
\end{proof}

Since $\mathsf{cart\usc{}map}_{\mathcal{S}}$ is an equivalence it follows that its fibers are contractible. 

\begin{cor}\label{cor:descent_fib}
Consider a diagram of the form 
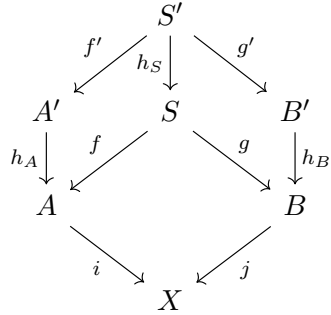
\begin{equation*}
\begin{tikzcd}
& S' \arrow[d,swap,"h_S"] \arrow[dl,swap,"{f'}"] \arrow[dr,"{g'}"] \\
A' \arrow[d,swap,"h_A"] & S \arrow[dl,swap,"f"] \arrow[dr,"g"] & B' \arrow[d,"{h_B}"] \\
A \arrow[dr,swap,"i"] & & B \arrow[dl,"j"] \\
& X
\end{tikzcd}
\end{equation*}
with homotopies
\begin{align*}
F & : f\circ h_S \htpy h_A\circ f' \\
G & : g\circ h_S \htpy h_B\circ g' \\
H & : i\circ f \htpy j\circ g,
\end{align*}
and suppose that the bottom square is a pushout square, and the top squares are pullback squares.
Then the type of tuples $((X',h_X),(i',I,p),(j',J,q),(H',C))$ consisting of
\begin{enumerate}
\item A type $X':\UU$ together with a morphism
\begin{equation*}
h_X : X'\to X,
\end{equation*}
\item A map $i':A'\to X'$, a homotopy $I:i\circ h_A\htpy h_X\circ i'$, and a term $p$ witnessing that the square
\begin{equation*}
\begin{tikzcd}
A' \arrow[d,swap,"h_A"] \arrow[r,"{i'}"] & X' \arrow[d,"h_X"] \\
A \arrow[r,swap,"i"] & X
\end{tikzcd}
\end{equation*}
is a pullback square.
\item A map $j':B'\to X'$, a homotopy $J:j\circ h_B\htpy h_X\circ j'$, and a term $q$ witnessing that the square
\begin{equation*}
\begin{tikzcd}
B' \arrow[d,swap,"h_B"] \arrow[r,"{j'}"] & X' \arrow[d,"h_X"] \\
B \arrow[r,swap,"j"] & X
\end{tikzcd}
\end{equation*}
is a pullback square,
\item A homotopy $H':i'\circ f'\htpy j'\circ g'$, and a homotopy
\begin{equation*}
C : \ct{(i\cdot F)}{(\ct{(I\cdot f')}{(h_X\cdot H')})} \htpy \ct{(H\cdot h_S)}{(\ct{(j\cdot G)}{(J\cdot g')})}
\end{equation*}
witnessing that the cube
\begin{equation*}
\begin{tikzcd}
& S' \arrow[dl] \arrow[dr] \arrow[d] \\
A' \arrow[d] & S \arrow[dl] \arrow[dr] & B' \arrow[dl,crossing over] \arrow[d] \\
A \arrow[dr] & X' \arrow[d] \arrow[from=ul,crossing over] & B \arrow[dl] \\
& X,
\end{tikzcd}
\end{equation*}
commutes,
\end{enumerate}
is contractible.
\end{cor}

The following theorem should be compared to the flattening lemma, \cref{lem:flattening}.\index{flattening lemma!for pushouts}

\begin{thm}\label{cor:descent}
Consider a commuting cube
\begin{equation*}
\begin{tikzcd}
& S' \arrow[dl,swap,"{f'}"] \arrow[dr,"{g'}"] \arrow[d,"h_S"] \\
A' \arrow[d,swap,"h_A"] & S \arrow[dl,swap,"f" near start] \arrow[dr,"g" near start] & B' \arrow[dl,crossing over,"{j'}" near end] \arrow[d,"h_B"] \\
A \arrow[dr,swap,"i"] & X' \arrow[d,"h_X" near start] \arrow[from=ul,crossing over,"{i'}"' near end] & B \arrow[dl,"j"] \\
& X
\end{tikzcd}
\end{equation*}
in which the bottom square is a pushout, and the two vertical squares in the back are pullbacks. Then the following are equivalent:
\begin{enumerate}
\item The two vertical squares in the front are pullback squares.
\item The top square is a pushout square.
\end{enumerate}
\end{thm}

\begin{proof}
By \cref{cor:pb_fibequiv} we have fiberwise equivalences
\begin{align*}
F & : \prd{x:S}\eqv{\fib{h_S}{x}}{\fib{h_A}{f(x)}} \\
G & : \prd{x:S}\eqv{\fib{h_S}{x}}{\fib{h_B}{g(x)}} \\
I & : \prd{a:A}\eqv{\fib{h_A}{a}}{\fib{h_X}{i(a)}} \\
J & : \prd{b:B}\eqv{\fib{h_B}{b}}{\fib{h_X}{j(b)}}. 
\end{align*}
Moreover, since the cube commutes we obtain a fiberwise homotopy
\begin{equation*}
K : \prd{x:S} I(f(x))\circ F(x) \htpy J(g(x))\circ G(x).
\end{equation*}
We define the descent data $(P_A,P_B,P_S)$ consisting of $P_A:A\to\UU$, $P_B:B\to\UU$, and $P_S:\prd{x:S}\eqv{P_A(f(x))}{P_B(g(x))}$ by
\begin{align*}
P_A(a) & \defeq \fib{h_A}{a} \\
P_B(b) & \defeq \fib{h_B}{b} \\
P_S(x) & \defeq G(x)\circ F(x)^{-1}.
\end{align*}
We have
\begin{align*}
P & \defeq \fibf{h_X} \\
e_A & \defeq I \\
e_B & \defeq J \\
e_S & \defeq K.
\end{align*}
Now consider the diagram
\begin{equation*}
\begin{tikzcd}
\sm{s:S}\fib{h_S}{s} \arrow[r] \arrow[d] & \sm{s:S}\fib{h_A}{f(s)} \arrow[r] \arrow[d] & \sm{b:B}\fib{h_B}{b} \arrow[d] \\
\sm{a:A}\fib{h_A}{a} \arrow[r] & \sm{a:A}\fib{h_A}{a} \arrow[r] & \sm{x:X}\fib{h_X}{x}
\end{tikzcd}
\end{equation*}
Since the top and bottom map in the left square are equivalences, we obtain that the left square is a pushout square. Moreover, the right square is a pushout by \cref{lem:flattening}. Therefore it follows by \cref{thm:pushout_pasting} that the outer rectangle is a pushout square.

Now consider the commuting cube
\begin{equation*}
\begin{tikzcd}
& \sm{s:S}\fib{h_S}{s} \arrow[dl] \arrow[dr] \arrow[d] \\
\sm{a:A}\fib{h_A}{a} \arrow[d] & S' \arrow[dl] \arrow[dr] & \sm{b:B}\fib{h_B}{b} \arrow[dl,crossing over] \arrow[d] \\
A' \arrow[dr,swap] & \sm{x:X}\fib{h_X}{x} \arrow[d] \arrow[from=ul,crossing over] & B' \arrow[dl] \\
& X'.
\end{tikzcd}
\end{equation*}
We have seen that the top square is a pushout. The vertical maps are all equivalences, so the vertical squares are all pushout squares. Thus it follows from one more application of \cref{thm:pushout_pasting} that the bottom square is a pushout.
\end{proof}

\begin{thm}\label{thm:cartesian_cube}
Consider a commuting cube of types 
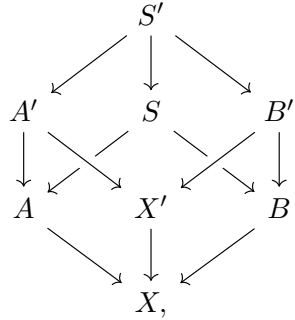
\begin{equation*}\label{eq:cube}
\begin{tikzcd}
& S' \arrow[dl] \arrow[dr] \arrow[d] \\
A' \arrow[d] & S \arrow[dl] \arrow[dr] & B' \arrow[dl,crossing over] \arrow[d] \\
A \arrow[dr] & X' \arrow[d] \arrow[from=ul,crossing over] & B \arrow[dl] \\
& X,
\end{tikzcd}
\end{equation*}
and suppose the two vertical squares in the back are pullback squares. Then the following are equivalent:
\begin{enumerate}
\item The two vertical squares in the front are pullback squares.
\item The commuting square
\begin{equation*}
\begin{tikzcd}
A' \sqcup^{S'} B' \arrow[r] \arrow[d] & X' \arrow[d] \\
A\sqcup^{S} B \arrow[r] & X
\end{tikzcd}
\end{equation*}
is a pullback square.
\end{enumerate}
\end{thm}

\begin{proof}
To see that (i) implies (ii), it suffices to show that the pullback 
\begin{equation*}
(A\sqcup^{S} B)\times_{X}X'
\end{equation*}
has the universal property of the pushout. This follows by the descent theorem, since the vertical squares in the cube
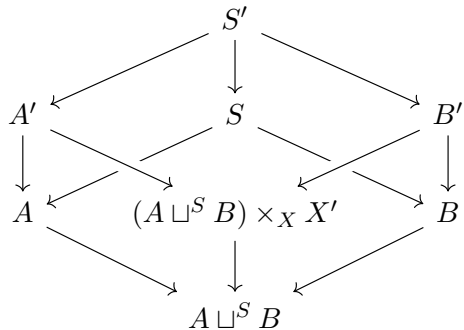
\begin{equation*}
\begin{tikzcd}
& S' \arrow[dl] \arrow[dr] \arrow[d] \\
A' \arrow[d] & S \arrow[dl] \arrow[dr] & B' \arrow[dl,crossing over] \arrow[d] \\
A \arrow[dr] & (A\sqcup^{S} B)\times_{X}X' \arrow[d] \arrow[from=ul,crossing over] & B \arrow[dl] \\
& A\sqcup^{S} B
\end{tikzcd}
\end{equation*}
are pullback squares by \cref{thm:pb_pasting}.

To prove that (ii) implies (i), we note that in the cube
\begin{equation*}
\begin{tikzcd}
& S' \arrow[dl] \arrow[dr] \arrow[d] \\
A' \arrow[d] & S \arrow[dl] \arrow[dr] & B' \arrow[dl,crossing over] \arrow[d] \\
A \arrow[dr] & A'\sqcup^{S'}B' \arrow[d] \arrow[from=ul,crossing over] & B \arrow[dl] \\
& A\sqcup^S B,
\end{tikzcd}
\end{equation*}
the two back squares are pullback squares, and the top and bottom squares are pushout squares. Therefore it follows from \cref{cor:descent} that the two front squares are pullback squares. Now we obtain (i) from the pasting lemma for pushouts.
\end{proof}

\section{Applications of the descent theorem for pushouts}
\sectionmark{Applications of the descent theorem}

\subsection{Fiber sequences}

\begin{defn}
A \define{pointed type} is a pair $(X,x)$ consisting of a type $X$ equipped with a \define{base point} $x:X$. We will write $\UU_\ast$ for the type $\sm{X:\UU}X$ of all pointed types.
\end{defn}

In the following lemma we characterize the identity type of $\UU_\ast$. 

\begin{lem}\label{lem:equiv_of_ptdtype}
For any $(A,a),(B,b):\UU_\ast$ we have an equivalence
\begin{equation*}
\eqv{\Big(\pairr{A,a}=\pairr{B,b}\Big)}{\Big(\sm{e:\eqv{A}{B}}e(a)=b\Big)}.
\end{equation*}
\end{lem}

\begin{proof}[Construction]
By \cref{thm:eq_sigma} the type on the left hand side is
equivalent to the type $\sm{p:A=B}\tr_{\universalfam}({p},{a})=b$.
By the univalence axiom, the map 
\begin{equation*}
\equiveq_{A,B}:(A=B)\to (\eqv{A}{B})
\end{equation*}
is an equivalence for each $B:\UU$. 
Therefore, we have an equivalence of type
\begin{equation*}
\eqv{\Big(\sm{p:A=B}\tr_{\universalfam}({p},{a})=b\Big)}{\Big(\sm{e:\eqv{A}{B}}\tr_{\universalfam}({\eqequiv(e)},{a})=b\Big)}
\end{equation*} 
Moreover, by equivalence induction (the analogue of path induction for 
equivalences), we can compute the transport:
\begin{equation*}
\tr_{\universalfam}({\eqequiv(e)},{a})=e(a).
\end{equation*}
It follows that $\eqv{(\tr_{\universalfam}({\eqequiv(e)},{a})=b)}
{(e(a)=b)}$.
\end{proof}

\begin{defn}
\begin{enumerate}
\item Let $(X,\ast_X)$ be a pointed type. A \define{pointed family} over $(X,\ast_X)$ consists of a type family $P:X\to \UU$ equipped with a base point $\ast_P:P(\ast_X)$. 
\item Let $(P,\ast_P)$ be a pointed family over $(X,\ast_X)$. A \define{pointed section} of $(P,\ast_P)$ consists of a dependent function $f:\prd{x:X}P(x)$ and an identification $p:f(\ast_X)=\ast_P$. We define the \define{pointed $\Pi$-type} to be the type of pointed sections:
\begin{equation*}
\Pi^\ast_{(x:X)}P(x) \defeq \sm{f:\prd{x:X}P(x)}f(\ast_X)=\ast_P
\end{equation*}
In the case of two pointed types $X$ and $Y$, we may also view $Y$ as a pointed family over $X$. In this case we write $X\to_\ast Y$ for the type of pointed functions.
\item Given any two pointed sections $f$ and $g$ of a pointed family $P$ over $X$, we define the type of pointed homotopies
\begin{equation*}
f\htpy_\ast g \defeq \Pi^\ast_{(x:X)} f(x)=g(x),
\end{equation*}
where the family $x\mapsto f(x)=g(x)$ is equipped with the base point $\ct{p}{q^{-1}}$. 
\end{enumerate}
\end{defn}

\begin{defn}
\begin{enumerate}
\item For any pointed type $X$, we define the \define{pointed identity function} $\idp_X\defeq (\idfunc[X],\refl{\ast})$. 
\item For any two pointed maps $f:X\to_\ast Y$ and $g:Y\to_\ast Z$, we define the \define{pointed composite}
\begin{equation*}
g\mathbin{\circ_\ast} f \defeq (g\circ f,\ct{\ap{g}{p_f}}{p_g}).
\end{equation*}
\end{enumerate}
\end{defn}

\begin{defn}
Let $X$ be a pointed type with base point $x$. We define the \define{loop space} $\loopspace{X,x}$ of $X$ at $x$ to be the pointed type $x=x$ with base point $\refl{x}$. 
\end{defn}

\begin{defn}
The loop space operation $\loopspacesym$ is \emph{functorial} in the sense that
\begin{enumerate}
\item For every pointed map $f:X\to_\ast Y$ there is a pointed map
\begin{equation*}
\loopspace{f}:\loopspace{X}\to_\ast \loopspace{Y},
\end{equation*}
defined by $\loopspace{f}(\omega)\defeq \ct{p_f}{\ap{f}{\omega}}{p_f^{-1}}$, which is base point preserving by $\rightinv(p_f)$. 
\item For every pointed type $X$ there is a pointed homotopy
\begin{equation*}
\loopspace{\idp_X}\htpy_\ast \idp_{\loopspace{X}}.
\end{equation*}
\item For any two pointed maps $f:X\to_\ast Y$ and $g:Y\to_\ast X$, there is a pointed homotopy witnessing that the triangle
\begin{equation*}
\begin{tikzcd}
& \loopspace{Y} \arrow[dr,"\loopspace{g}"] \\
\loopspace{X} \arrow[rr,swap,"\loopspace{g\circ_\ast f}"] \arrow[ur,"\loopspace{f}"] & & \loopspace{Z}
\end{tikzcd}
\end{equation*}
of pointed types commutes.
\end{enumerate}
\end{defn}

\begin{lem}\label{lem:equiv_of_ptdequiv}
For any $\pairr{e,p},\pairr{f,q}:\sm{e:\eqv{A}{B}}e(a)=b$, we have an equivalence of type
\begin{equation*}
\eqv{\Big(\pairr{e,p}=\pairr{f,q}\Big)}{\Big(\sm{h:e\htpy f} p=\ct{h(a)}{q}\Big)}.
\end{equation*}
\end{lem}

\begin{proof}[Construction]
The type $\pairr{e,p}=\pairr{f,q}$ is equivalent
to the type $\sm{h:e=f}\tr({h},{p})=q$.
Note that by the principle of function extensionality,
the map $\htpyeq:(e=f)\to(e\htpy f)$
is an equivalence. Furthermore, it follows by homotopy induction that for any 
$h:e\htpy f$ we have an equivalence of type
\begin{equation*}
\eqv{(\tr({\eqhtpy(h)},{p})=q)}
    {(p= \ct{h(a)}{q})}.\qedhere
\end{equation*}
\end{proof}

\begin{defn}
A \define{fiber sequence} $F\hookrightarrow E \twoheadrightarrow B$ consists of:
\begin{enumerate}
\item Pointed types $F$, $E$, and $B$, with base points $x_0$, $y_0$, and $b_0$ respectively, 
\item Base point preserving maps $i:F\to_\ast E$ and $p:E\to_\ast B$, with $\alpha:i(x_0)=y_0$ and $\beta:p(y_0)=b_0$,
\item A pointed homotopy $H:\const_{b_0}\htpy_\ast p\circ_\ast i$ witnessing that the square
\begin{equation*}
\begin{tikzcd}
F \arrow[r,"i"] \arrow[d] & E \arrow[d,"p"] \\
\unit \arrow[r,swap,"\const_{b_0}"] & B,
\end{tikzcd}
\end{equation*}
commutes and is a pullback square.
\end{enumerate}
We will write $\FibSeq$ for the type of all fiber sequences in $\UU$.
\end{defn}

\begin{prp}
The type of all fiber sequences is equivalent to the type
\begin{equation*}
\sm{(B,b):\UU_\ast}{P:B\to\UU}P(b).
\end{equation*}
\end{prp}

\subsection{Fiber sequences obtained by the descent property}

\begin{defn}
Let $f:A\to B$ be a map. The \define{codiagonal}\index{codiagonal}\index{nabla@{$\nabla_f$}} $\nabla_f$ of $f$ is the map obtained from the universal property of the pushout, as indicated in the diagram
\begin{equation*}
\begin{tikzcd}
A \arrow[d,swap,"f"] \arrow[r,"f"] \arrow[dr, phantom, "\ulcorner", very near end] & B \arrow[d,"\inr"] \arrow[ddr,bend left=15,"{\idfunc[B]}"] \\
B \arrow[r,"\inl"] \arrow[drr,bend right=15,swap,"{\idfunc[B]}"] & B\sqcup^{A} B \arrow[dr,densely dotted,near start,swap,"\nabla_f"] \\
& & B
\end{tikzcd}
\end{equation*}
\end{defn}

\begin{prp}
For any map $f:A\to B$ and any $y:B$, there is an equivalence $\eqv{\fib{\nabla_f}{y}}{\susp(\fib{f}{y})}$. 
\end{prp}

\begin{proof}
For any $b:B$ we have the commuting cube 
\begin{equation*}
\begin{tikzcd}
& \fib{f}{b} \arrow[dl] \arrow[d] \arrow[dr] \\
\unit \arrow[d] & A \arrow[dl] \arrow[dr] & \unit \arrow[dl,crossing over] \arrow[d] \\
B \arrow[dr] & \unit \arrow[from=ul,crossing over] \arrow[d,swap,"b"] & B \arrow[dl] \\
& B
\end{tikzcd}
\end{equation*}
of which the vertical sides are pullback squares. Hence we obtain the pullback square
\begin{equation*}
\begin{tikzcd}
\susp{\fib{f}{b}} \arrow[r] \arrow[d] & \unit \arrow[d,"b"] \\
B\sqcup^{A} B \arrow[r] & B
\end{tikzcd}
\end{equation*}
from \cref{thm:cartesian_cube}, from which the claim follows.
\end{proof}

\begin{defn}
Consider two maps $f:A\to B$ and $g:C\to D$.
The \define{pushout-product}\index{pushout-product}
\begin{equation*}
f\square g : (A\times D)\sqcup^{A\times C} (B\times C)\to B\times D
\end{equation*}
of $f$ and $g$ is defined by the universal property of the pushout as the unique map rendering the diagram
\begin{equation*}
\begin{tikzcd}
A\times C \arrow[r,"{f\times \idfunc[C]}"] \arrow[d,swap,"{\idfunc[A]\times g}"] & B\times C \arrow[d,"\inr"] \arrow[ddr,bend left=15,"{\idfunc[B]\times g}"] \\
A\times D \arrow[r,"\inl"] \arrow[drr,bend right=15,swap,"{f\times\idfunc[D]}"] & (A\times D)\sqcup^{A\times C} (B\times C) \arrow[dr,densely dotted,swap,near start,"f\square g"] \\
& & B\times D
\end{tikzcd}
\end{equation*}
commutative.
\end{defn}

\begin{prp}
For any two maps $f:A\to B$ and $g:C\to D$, and any $(b,d):B\times D$, there is an equivalence
\begin{equation*}
\eqv{\fib{f\square g}{b,d}}{\join{\fib{f}{b}}{\fib{g}{d}}}.
\end{equation*}
\end{prp}

\begin{proof}
Let $b:B$ and $d:D$. Then we have the commuting cube 
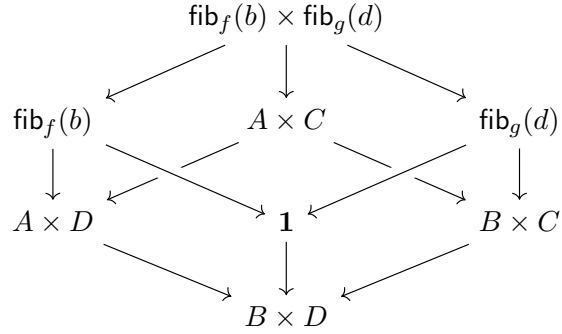
\begin{equation*}
\begin{tikzcd}
& \fib{f}{b}\times \fib{g}{d} \arrow[dl] \arrow[d] \arrow[dr] \\
\fib{f}{b} \arrow[d] & A\times C \arrow[dl] \arrow[dr] & \fib{g}{d} \arrow[dl,crossing over] \arrow[d] \\
A\times D \arrow[dr] & \unit \arrow[from=ul,crossing over] \arrow[d] & B\times C \arrow[dl] \\
& B\times D
\end{tikzcd}
\end{equation*}
of which the vertical sides are pullback squares. Hence the claim follows from \cref{thm:cartesian_cube}.
\end{proof}

\begin{defn}\label{defn:fib_join}
Let $f:A\to X$ and $g:B\to X$ be maps into $X$. We define the \define{fiberwise join} $\join[X]{A}{B}$ and the \define{join}\footnote{\emph{Warning}: By $\join{f}{g}$ we do \emph{not} mean the functorial action of the
join, applied to $(f,g)$.} $\join{f}{g}:\join[X]{A}{B}\to X$ of
$f$ and $g$, as indicated in the following diagram:
\begin{equation*}
\begin{tikzcd}
A\times_X B \arrow[r,"\pi_2"] \arrow[d,swap,"\pi_1"] \arrow[dr, phantom, "{\ulcorner}", at end] & B \arrow[d,"\inr"] \arrow[ddr,bend left=15,"g"] \\
A \arrow[r,swap,"\inl"] \arrow[drr,bend right=15,swap,"f"] & \join[X]{A}{B} \arrow[dr,densely dotted,swap,near start,"\join{f}{g}" xshift=1ex] \\
& & X.
\end{tikzcd}
\end{equation*}
\end{defn}

\begin{thm}\label{defn:join-fiber}
Let $f:A\to X$ and $g:B\to X$ be maps into $X$, and let $x:X$. Then there is
an equivalence
\begin{equation*}
\eqv{\fib{\join{f}{g}}{x}}{\join{\fib{f}{x}}{\fib{g}{x}}}.
\end{equation*}
\end{thm}

\begin{proof}
We have the following commuting cube
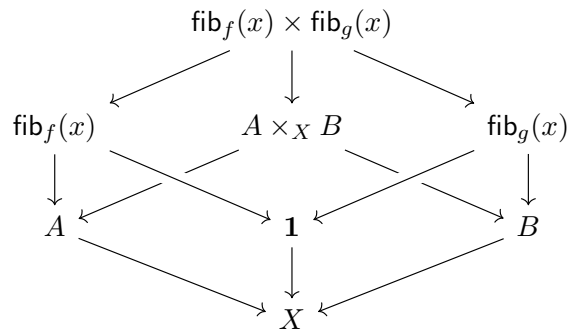
\begin{equation*}
\begin{tikzcd}
& \fib{f}{x}\times\fib{g}{x} \arrow[dl] \arrow[d] \arrow[dr] \\
\fib{f}{x} \arrow[d] & A\times_X B \arrow[dl] \arrow[dr] & \fib{g}{x} \arrow[dl,crossing over] \arrow[d] \\
A \arrow[dr] & \unit \arrow[d] \arrow[from=ul,crossing over] & B \arrow[dl] \\
& X
\end{tikzcd}
\end{equation*}
in which the vertical squares are pullback squares. Therefore it follows by \cref{thm:cartesian_cube} that the square
\begin{equation*}
\begin{tikzcd}
\join{\fib{f}{x}}{\fib{g}{x}} \arrow[r] \arrow[d] & \unit \arrow[d] \\
\join[X]{A}{B} \arrow[r] & X
\end{tikzcd}
\end{equation*}
is a pullback square.
\end{proof}

\begin{rmk}
The join operation on maps with a common codomain is associative up to homotopy (this was formalized by Brunerie, see Proposition 1.8.6 of \cite{BruneriePhD}), and it is a commutative operation on the generalized elements of a type $X$. Furthermore, the unique map of type $\emptyt\to X$ is a unit for the join operation.
\end{rmk}

\begin{defn}
Let $A$ and $B$ be pointed types with base points $a_0:A$ and $b_0:B$. The \define{wedge inclusion}\index{wedge inclusion} is defined as follows by the universal property of the wedge:
\begin{equation*}
\begin{tikzcd}[column sep=huge]
\unit \arrow[r] \arrow[d] & B \arrow[d,"\inr"] \arrow[ddr,bend left=15,"{\lam{b}(a_0,b)}"] \\
A \arrow[r,"\inl"] \arrow[drr,bend right=15,swap,"{\lam{a}(a,b_0)}"] & A\vee B \arrow[dr,densely dotted,swap,"{\mathsf{wedge\usc{}in}_{A,B}}"{near start,xshift=1ex}] \\
& & A\times B
\end{tikzcd}
\end{equation*}
\end{defn}

\begin{prp}
There is a fiber sequence 
\begin{equation*}
\join{\loopspace{A}}{\loopspace{B}}\hookrightarrow A\vee B\twoheadrightarrow A\times B.
\end{equation*}
\end{prp}

\begin{proof}
We have the commuting cube 
\begin{equation*}
\begin{tikzcd}
& \loopspace{B}\times\loopspace{A} \arrow[dl] \arrow[d] \arrow[dr] \\
\loopspace{B} \arrow[d,swap,"\mathsf{const}_a"] & \unit \arrow[dl] \arrow[dr] & \loopspace{A} \arrow[dl,crossing over] \arrow[d,"\mathsf{const}_b"] \\
A \arrow[dr] & \unit \arrow[from=ul,crossing over] \arrow[d] & B \arrow[dl] \\
& A\times B
\end{tikzcd}
\end{equation*}
of which the vertical sides are pullback squares. Hence the claim follows from \cref{thm:cartesian_cube}.
\end{proof}

\begin{defn}
Consider a pointed type $X$. We define the map $\mathsf{fold}:X\vee X\to X$ by the universal property of the wedge as indicated in the diagram
\begin{equation*}
\begin{tikzcd}
\unit \arrow[d,swap,"x_0"] \arrow[r,"x_0"] \arrow[dr, phantom, "\ulcorner", very near end] & X \arrow[d,"\inr"] \arrow[ddr,bend left=15,"{\idfunc[X]}"] \\
X \arrow[r,"\inl"] \arrow[drr,bend right=15,swap,"{\idfunc[X]}"] & X\vee X \arrow[dr,densely dotted,near start,swap,"\mathsf{fold}"] \\
& & X.
\end{tikzcd}
\end{equation*}
\end{defn}

\begin{prp}
There is a fiber sequence
\begin{equation*}
\susp\loopspace{X} \hookrightarrow X\vee X \twoheadrightarrow X.
\end{equation*}
\end{prp}

\begin{proof}
We have the commuting cube 
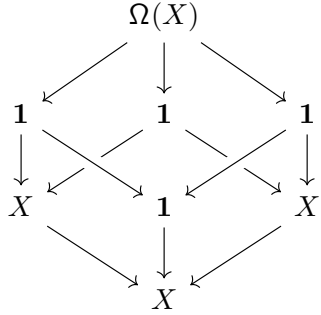
\begin{equation*}
\begin{tikzcd}
& \loopspace{X} \arrow[dl] \arrow[d] \arrow[dr] \\
\unit \arrow[d] & \unit \arrow[dl] \arrow[dr] & \unit \arrow[dl,crossing over] \arrow[d] \\
X \arrow[dr] & \unit \arrow[from=ul,crossing over] \arrow[d] & X \arrow[dl] \\
& X
\end{tikzcd}
\end{equation*}
of which the vertical sides are pullback squares. Hence the claim follows from \cref{thm:cartesian_cube}.
\end{proof}

\begin{rmk}
As a corollary, there are fiber sequences
\begin{align*}
\sphere{1} \hookrightarrow \rprojective{\infty}\vee \rprojective{\infty} & \twoheadrightarrow \rprojective{\infty} \\
\sphere{2} \hookrightarrow \cprojective{\infty}\vee \cprojective{\infty} & \twoheadrightarrow \cprojective{\infty}.
\end{align*}
Here we take $\rprojective{\infty}\defeq K(\Z/2,1)$ and $\cprojective{\infty}\defeq K(\Z,2)$, where the Eilenberg-Mac Lane space $K(G,n)$ is defined in \cite{FinsterLicata}.
\end{rmk}

\begin{cor}
There is a fiber sequence
\begin{equation*}
(\susp\loopspace{X})^{\vee n} \hookrightarrow X^{\vee (n+1)} \twoheadrightarrow X.
\end{equation*}
\end{cor}

\begin{defn}\label{defn:coh_hspace}
A \define{coherent H-space} consists of a type $X$ equipped with a unit $1:X$, a multiplication operation $\mu:X \to (X \to X)$ such that the function $\mu(x,\blank)$ and $\mu(\blank,y)$ are equivalences for each $x:X$ and $y:X$, respectively, and \emph{coherent} unit laws
\begin{align*}
\mathsf{right\usc{}unit} & : \prd{x: X} \mu(x,1)= x \\
\mathsf{left\usc{}unit} & : \prd{y:X} \mu(1,y) = y \\
\mathsf{coh\usc{}unit} & : \mathsf{left\usc{}unit}(1)=\mathsf{right\usc{}unit}(1).
\end{align*}
\end{defn}

The following theorem is also known as the Hopf-construction.

\begin{thm}\label{thm:hopf_construction}
For any coherent H-space $X$ there is a fiber sequence
\begin{equation*}
X \hookrightarrow \join{X}{X} \twoheadrightarrow \susp X.
\end{equation*}
The map $\eta_X:\join{X}{X}\to \susp X$ is called the \define{Hopf fibration} for $X$.
\end{thm}

\begin{proof}
We have the commuting cube
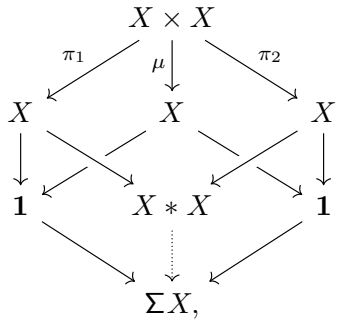
\begin{equation*}
\begin{tikzcd}
& X \times X \arrow[dl,swap,"\pi_1"] \arrow[d,swap,"\mu"] \arrow[dr,"\pi_2"] \\
X \arrow[d] & X \arrow[dl] \arrow[dr] & X \arrow[dl,crossing over] \arrow[d] \\
\unit \arrow[dr] & \join{X}{X} \arrow[from=ul,crossing over] \arrow[d,densely dotted] & \unit \arrow[dl] \\
& \susp X,
\end{tikzcd}
\end{equation*}
where the front map is obtained by the universal property of the pushout.

In this cube, the two vertical squares in the back are pullback squares by \cref{thm:pb_fibequiv}, since $\mu(\blank,y)$ and $\mu(x,\blank)$ are equivalences for every $x:X$ and $y:X$, respectively. Since the bottom and top squares are pushout squares it follows by \cref{thm:cartesian_cube} that the front two squares are pullback squares.
\end{proof}

\begin{rmk}
The classical Hopf fibration
\begin{equation*}
\begin{tikzcd}
\sphere{1} \arrow[r,hookrightarrow] & \sphere{3} \arrow[r,->>] & \sphere{2}
\end{tikzcd}
\end{equation*}
is now obtained from the fact that $\join{\sphere{1}}{\sphere{1}}\eqvsym \sphere{3}$, which was established in \cite{BruneriePhD}. Indeed, the circle is a coherent H-space.
\end{rmk}

\chapter{Reflexive coequalizers}\label{chap:rcoeq}

The material of this chapter is joint work with Bas Spitters, which we started in the academic year 2012-2013 while I was a research assistant at the Radboud University of Nijmegen.

We begin this chapter with a proof of the type theoretic Yoneda lemma, \cref{lem:yoneda}. The Yoneda lemma is used to show in \cref{thm:emb_disc} that the discrete functor $\Delta:\UU\to\mathsf{rGph}$ from small types into the type of small reflexive graphs, is an embedding. In \cref{thm:discrete_rgraph} we show that a reflexive graph $\mathcal{A}$ is discrete if and only if it satisfies any one (and hence all) of the unique extension properties
\begin{equation*}
\begin{tikzcd}
\mathcal{I} \arrow[r] \arrow[d] & \mathcal{A} & \unit \arrow[d,swap,"0"] \arrow[r] & \mathcal{A} & \unit \arrow[d,swap,"1"] \arrow[r] & \mathcal{A}, \\
\unit \arrow[ur,densely dotted] & & \mathcal{I} \arrow[ur,densely dotted] & & \mathcal{I} \arrow[ur,densely dotted]
\end{tikzcd}
\end{equation*}
where $\mathcal{I}$ is the walking edge. 

The (homotopy) reflexive coequalizer satisfies the universal property of the left adjoint of $\Delta:\UU\to\mathsf{rGph}$. We show in \cref{thm:rcoeq_is_pushout} that the reflexive coequalizer of a reflexive graph $\mathcal{A}$ is just a pushout
\begin{equation*}
\begin{tikzcd}
\sm{x,y:\pts{A}}\edg{A}(x,y) \arrow[r,"\pi_2"] \arrow[d,swap,"\pi_1"] & \pts{A} \arrow[d] \\
\pts{A} \arrow[r] & \mathsf{rcoeq}(\mathcal{A}).
\end{tikzcd}
\end{equation*}
In particular it follows that if a universe is closed under pushouts (which is our running assumption), then it is also closed under reflexive coequalizers. More practically, this characterization of reflexive coequalizers as pushouts allows us to compute in \cref{eg:rcoeq} many reflexive coequalizers in terms of previously defined operations.

Our next purpose is to study the morphisms $f:\mathcal{B}\to\mathcal{A}$ of reflexive graphs that are \emph{fibrations} in the sense that they satisfy the right orthogonality conditions
\begin{equation*}
\begin{tikzcd}
\mathcal{I} \arrow[r] \arrow[d] & \mathcal{B} \arrow[d,"f"] & \unit \arrow[d,swap,"0"] \arrow[r] & \mathcal{B} \arrow[d,"f"] & \unit \arrow[d,swap,"1"] \arrow[r] & \mathcal{B} \arrow[d,"f"] \\
\unit \arrow[ur,densely dotted] \arrow[r] & \mathcal{A} & \mathcal{I} \arrow[ur,densely dotted] \arrow[r] & \mathcal{A} & \mathcal{I} \arrow[ur,densely dotted] \arrow[r] & \mathcal{A}.
\end{tikzcd}
\end{equation*}

We show in \cref{prp:fib_cart} that the class of fibrations is precisely the class of \emph{cartesian} morphisms of reflexive graphs, i.e.~the morphisms $f:\mathcal{B}\to\mathcal{A}$ for which the naturality squares
\begin{equation*}
\begin{tikzcd}
\pts{B} \arrow[d] & \sm{i,j:\pts{B}}\edg{B}(i,j) \arrow[d] \arrow[r] \arrow[l] & \pts{B} \arrow[d] & \pts{B} \arrow[r] \arrow[d] & \sm{i,j:\pts{B}}\edg{B}(i,j) \arrow[d] \\
\pts{A} & \sm{i,j:\pts{A}}\edg{A}(i,j) \arrow[l] \arrow[r] & \pts{A} & \pts{A} \arrow[r] & \sm{i,j:\pts{A}}\edg{A}(i,j)
\end{tikzcd}
\end{equation*}
are pullback squares. It follows that $f:\mathcal{B}\to\mathcal{A}$ is a fibration whenever it appears as a pullback
\begin{equation*}
\begin{tikzcd}
\mathcal{B} \arrow[d,swap,"f"] \arrow[r] & \Delta Y \arrow[d] \\
\mathcal{A} \arrow[r] & \Delta X
\end{tikzcd}
\end{equation*}
of some map between discrete reflexive graphs. Pulling back along $h:\mathcal{A}\to\Delta X$ for a fixed type $X$ therefore provides an operation
\begin{equation*}
h^\ast : \Big(\sm{Y:\UU}(Y\to X)\Big)\to \Big(\sm{\mathcal{B}:\mathsf{rGph}}\mathsf{fib}(\mathcal{B},\mathcal{A})\Big).
\end{equation*}
One way of stating the descent property for reflexive coequalizers, is that this map is an equivalence whenever $f:\mathcal{A}\to\Delta X$ is a reflexive coequalizer. As a consequence of the descent property we obtain in \cref{thm:rcoeq_cartesian} that for any fibration $f:\mathcal{B}\to\mathcal{A}$ of reflexive graphs, a commuting square of the form
\begin{equation*}
\begin{tikzcd}
\mathcal{B} \arrow[d,->>,swap,"f"] \arrow[r] & \Delta Y \arrow[d] \\
\mathcal{A} \arrow[r] & \Delta X
\end{tikzcd}
\end{equation*}
is a pullback square of reflexive graphs if and only if the square
\begin{equation*}
\begin{tikzcd}
\mathsf{rcoeq}(\mathcal{B}) \arrow[d,swap,"\mathsf{rcoeq}(f)"] \arrow[r] & Y \arrow[d] \\
\mathsf{rcoeq}(\mathcal{A}) \arrow[r] & X
\end{tikzcd}
\end{equation*}
is a pullback square. A minor note about the forward direction is that the hypothesis that $f$ is a fibration is implied by the assumption that the square is a pullback square, so this assumption is unrestrictively superfluous. However, this assumption is necessary for the converse direction.

The situation here is that we have an adjunction $F\dashv G$ (in the present case $\mathsf{rcoeq}\dashv \Delta$) satisfying the `descent condition' that if $f$ is a fibration, i.e. the naturality square of the unit $\eta:\idfunc\Rightarrow GF$ at $f$ is a pullback square, then any square of the form
\begin{equation*}
\begin{tikzcd}
A \arrow[d,swap,"f"] \arrow[r] & GX \arrow[d,"Gg"] \\
B \arrow[r] & GY
\end{tikzcd}
\end{equation*}
is a pullback square if and only if the square
\begin{equation*}
\begin{tikzcd}
FA \arrow[r] \arrow[d,swap,"Ff"] & X \arrow[d,"g"] \\
FB \arrow[r] & Y
\end{tikzcd}
\end{equation*}
is a pullback square. We will investigate this situation further in \cref{sec:modal_descent}, where we study the descent condition for general modalities.

In the final section of this chapter we apply the previous results to colimits of diagrams indexed by a (reflexive) graph. Most useful to us are the results on sequential colimits, which we will rely on in later chapters. Here we apply our results on diagrams over (reflexive) graphs to obtain results on sequential colimits. In \cite{DoornRijkeSojakova} these results are obtained directly, and moreover they are formalized.

\section{The Yoneda lemma}
The universal property of identity types is sometimes called the type theoretic Yoneda lemma: families of maps out of the identity type are uniquely determined by their action on the reflexivity identification.

\begin{lem}\label{lem:yoneda}
Let $B$ be a type family over $A$, and let $a:A$. Then the map
\begin{equation*}
\mathsf{ev\usc{}refl}:\Big(\prd{x:A} (a=x)\to B(x)\Big)\to B(a)
\end{equation*}
given by $\lam{f} f(a,\refl{a})$ is an equivalence. 
\end{lem}

\begin{proof}
The inverse of $\mathsf{ev\usc{}refl}$ is $\ind{a=}:B(a)\to \prd{x:A}(a=x)\to B(x)$. We have the homotopy $\lam{b}\refl{b}:\mathsf{ev\usc{}refl}\circ\ind{a=}\htpy \idfunc[B(a)]$ by the computation rule for identity types, so it is indeed the case that $\ind{a=}$ is a section of $\mathsf{ev\usc{}refl}$.

To see that $\ind{a=}\circ \mathsf{ev\usc{}refl}\htpy\idfunc$, let $f:\prd{x:A}(a=x)\to B(x)$. To show that $\ind{a=}(f(a,\refl{a}))=f$ we use function extensionality (twice), so it suffices to show that
\begin{equation*}
\prd{x:A}{p:a=x} \varphi(f(a,\refl{a}),x,p)=f(x,p).
\end{equation*}
This follows by path induction on $p$, since $\ind{a=}(f(a,\refl{a}),a,\refl{a})\jdeq f(a,\refl{a})$.
\end{proof}

\begin{cor}\label{cor:yoneda}
Let $B$ be a type family over $A$, and let $a:A$. Then there is an equivalence
\begin{equation*}
\eqv{\Big(\idtypevar{A}(a)=B\Big)}{\Big(B(a)\times\iscontr\Big(\sm{x:A}B(x)\Big)\Big)}.
\end{equation*}
Furthermore, for $b:B(a)$ there is an equivalence
\begin{equation*}
\eqv{\Big((\idtypevar{A}(a),\refl{a})=(B,b)\Big)}{\iscontr\Big(\sm{x:A}B(x)\Big)}.
\end{equation*}
\end{cor}

\begin{proof}
First we show that there is an equivalence
\begin{equation*}
\eqv{\Big(\prd{x:A} (a=x)\eqvsym B(x)\Big)}{B(a)\times\iscontr\Big(\sm{x:A}B(x)\Big)}.
\end{equation*}
To see this, note that the type of fiberwise equivalences $\prd{x:A} (a=x)\eqvsym B(x)$ is equivalent to the type
\begin{equation*}
\sm{f:\prd{x:A}(a=x)\to B(x)}\prd{x:A}\isequiv(f)
\end{equation*}
By \cref{thm:id_fundamental}, the type $\prd{x:A}\isequiv(f)$ is equivalent to the contractibility of the total space of $B$. Now the claim follows from \cref{lem:yoneda}. The claims in the statement are now easy consequences.
\end{proof}

\begin{defn}\label{defn:rrel}
A (small) \define{reflexive relation} $\mathcal{R}$ on $A$ is a pair $(R,\rho)$ consisting of a binary type-valued relation $R:A\to (A\to \UU)$ equipped with a proof of reflexivity
\begin{equation*}
\rho : \prd{x:A}R(x,x).
\end{equation*}
We will write $\mathsf{rRel}(A)$ for the type of all small reflexive relations on $A$, i.e.
\begin{equation*}
\mathsf{rRel}(A) \defeq \sm{R:A\to (A\to \UU)}\prd{x:A}R(x,x).
\end{equation*}
\end{defn}

\begin{cor}\label{cor:yoneda_rel}
Let $\mathcal{R}\jdeq (R,\rho)$ be a reflexive relation on $A$.
Then there is an equivalence
\begin{equation*}
\eqv{\Big((\idtypevar,\refl{})=(R,\rho)\Big)}{\prd{a:A}\iscontr\Big(\sm{x:A}R(a,x)\Big)}.
\end{equation*}
\end{cor}

The following theorem, which was proven independently by Escardó \cite{Escardo2016} around the same time, shows that the canonical map
\begin{equation*}
(x=y)\to (\idtypevar{A}(x)=\idtypevar{A}(y))
\end{equation*}
is an equivalence, for any $x,y:A$. This will be particularly relevant to us once we study $\infty$-equivalence relations.

\begin{prp}
For any type $A:\UU$, the map
\begin{equation*}
\idtypevar{A}:A\to (A\to\UU)
\end{equation*}
is an embedding.\index{identity type!is an embedding|textit}
\end{prp}

\begin{proof}
Let $a:A$. Then we calculate the fiber of $\idtypevar{A}$ at a type family $B:A\to \UU$ as follows:
\begin{align*}
\fib{\idtypevar{A}}{B} & \jdeq \sm{a:A}\idtypevar{A}(a)=B \\
& \eqvsym \sm{a:A} \Big(B(a)\times\iscontr\Big(\sm{x:A}B(x)\Big)\Big) \tag{by \cref{cor:yoneda}}\\
& \eqvsym \Big(\sm{a:A}B(a)\Big)\times \iscontr\Big(\sm{x:A}B(x)\Big) \\
& \eqvsym \iscontr\Big(\sm{x:A}B(x)\Big).
\end{align*}
The last equivalence follows since $X\times \iscontr(X)\eqvsym \iscontr(X)$ for every type $X$. We conclude that the fibers of $\idtypevar{A}$ are propositions, so $\idtypevar{A}$ is an embedding by \cref{thm:prop_emb}.
\end{proof}

\section{Discrete reflexive graphs}

\begin{defn}\label{defn:graphs_ctx}
A \define{non-reflexive graph} $\mathcal{A}$ in $\UU$ is a pair $\pairr{\pts{A},\edg{A}}$ consisting of
\begin{align*}
\pts{A} & : \UU \\
\edg{A} & : \pts{A}\to\pts{A}\to\UU.
\end{align*}
We write $\mathsf{Gph}$ for the type of all graphs in $\UU$. When $\mathcal{A}$ is a graph, we say that $\pts{A}$ is its type of \define{vertices}, and that $\edg{A}$ is its family of \define{edges}.

A \define{reflexive graph} $\mathcal{A}$ in $\UU$ consists of a graph $(\pts{A},\edg{A})$ equipped with a \define{reflexivity} term
\begin{align*}
\rfx{\mathcal{A}} & : \prd{i:\pts{A}}\edg{A}(i,i).
\end{align*}
We write
\begin{equation*}
\mathsf{rGph}\defeq \sm{V:\UU}{E:V\to (V\to\UU)}\prd{v:V}E(v,v)
\end{equation*}
for the type of reflexive graphs in $\UU$. 
\end{defn}

\begin{eg}
\begin{enumerate}
\item A \define{reflexive pair} consists of types $V$ and $E$, and maps $s$, $t$, and $r$ between $V$ and $E$ as indicated in the diagram
\begin{equation*}
\begin{tikzcd}[column sep=large]
E \arrow[r,yshift=1.5ex,"s"] \arrow[r,yshift=-1.5ex,swap,"t"] & V, \arrow[l,"r" description]
\end{tikzcd}
\end{equation*}
equipped with homotopies $H_{sr}:s\circ r\htpy \idfunc[V]$ and $H_{tr}:t\circ r\htpy \idfunc[V]$.
Given a reflexive pair as above, we obtain a reflexive graph $\mathcal{A}$ by taking
\begin{align*}
\pts{A} & \defeq V \\
\edg{A}(v,w) & \defeq \sm{e:E} (s(e)=v)\times (t(e)=w) \\
\rfx{\mathcal{A}}(v) & \defeq \pairr{r(v),H_{sr}(v),H_{tr}(v)}.
\end{align*}
By a routine construction it can be shown that the type of small reflexive pairs is equivalent to the type of small reflexive graphs.
\item \label{eg:disc_codisc} Given a type $X$, the \define{discrete graph} $\Delta(X)$ on $X$ is the reflexive graph consisting of
\begin{align*}
\pts{\Delta(X)} & \defeq X \\
\edg{\Delta(X)} & \defeq \idtypevar{X} \\
\rfx{\Delta(X)} & \defeq \refl{}.
\end{align*}
\item Given a map $f:A\to X$, the \define{pre-kernel} $k(f)$ of $f$ is the reflexive relation on $A$ given by $x,y\mapsto f(x)=f(y)$.
\item The \define{indiscrete graph} $\nabla(X)$ on $X$ is the reflexive graph consisting of
\begin{align*}
\pts{\nabla(X)} & \defeq X \\
\edg{\nabla(X)} & \defeq \lam{x}{y}\unit \\
\rfx{\nabla(X)} & \defeq \ttt.
\end{align*}
The reflexive pair corresponding to $\nabla(X)$ is (equivalent to) the reflexive pair
\begin{equation*}
\begin{tikzcd}
X\times X \arrow[r,yshift=1ex,"\pi_1"] \arrow[r,yshift=-1ex,"\pi_2"] & X. \arrow[l,"\delta" description]
\end{tikzcd}
\end{equation*}
In particular, we have the \define{unit} graph $\unit\defeq\nabla(\unit)$, which happens to also be $\Delta(\unit)$. 
\item The \define{walking edge} $\mathcal{I}$ is an example of a reflexive graph. One way of defining it is by taking
\begin{align*}
\pts{\mathcal{I}} & \defeq \bool \\
\edg{\mathcal{I}}(x,y) & \defeq \tau(x)\to \tau(y) \\
\rfx{\mathcal{I}}(x) & \defeq \idfunc[\tau(x)]
\end{align*}
where $\tau:\bool\to\UU$ is the \define{tautological family} on the type $\bool$ of booleans, given by $\tau(\bfalse)\defeq\emptyt$ and $\tau(\btrue)\defeq\unit$. We write $a$ for the (unique) edge from $\bfalse$ to $\btrue$. It may be helpful to think of the walking edge $\mathcal{I}$ as an interval, hence the choice of notation.
\end{enumerate}
\end{eg}

\begin{defn}
A \define{morphism} of graphs $f$ from $\mathcal{A}$ to $\mathcal{B}$ is a pair
\begin{align*}
\pts{f} & : \pts{A} \to \pts{B} \\
\edg{f} & : \prd{i,j:\pts{A}} \edg{A}(i,j)\to \edg{B}(\pts{f}(i),\pts{f}(j)).
\end{align*}
We write $\mathsf{Gph}(\mathcal{A},\mathcal{B})$ for the type of graph morphisms from $\mathcal{A}$ to $\mathcal{B}$.

A \define{morphism} of reflexive graphs $f$ from $\mathcal{A}$ to $\mathcal{B}$ is a morphism of graphs equipped with a term
\begin{align*}
\rfx{f} & : \prd{i:\pts{A}} \edg{f}(\rfx{\mathcal{A}}(i))= \rfx{\mathcal{B}}(\pts{f}(i))
\end{align*}
witnessing that reflexivity is preserved. We write $\mathsf{rGph}(\mathcal{A},\mathcal{B})$ for the type of reflexive graph morphisms from $\mathcal{A}$ to $\mathcal{B}$.
\end{defn}

For any reflexive graph there is an identity morphism, and for any composable pair of morphisms there is a composite. Furthermore, these operations are associative, and satisfy the unit laws, both up to homotopy. It should be noted, however, that composition of reflexive graph morphisms is not expected to be associative on the nose, since the reflexivity is only preserved up to higher identification.

\begin{eg}\label{eg:rgraph_morphism}
\begin{enumerate}
\item Given a function $f:X\to Y$ we obtain a morphism $\Delta f:\mathsf{rGph}(\Delta X,\Delta Y)$ given by
\begin{align*}
\pts{\Delta f} & \defeq f \\
\edg{\Delta f} & \defeq \apfunc{f} \\
\rfx{\Delta f} & \defeq \refl{\refl{f(x)}}
\end{align*}
The action of $\Delta$ on morphisms preserves identity morphisms and compositions, and moreover it preserves the unit laws and associativity. 
\item The graph $\nabla(\unit)$ is the terminal reflexive graph in the sense that for any graph $\mathcal{A}$, the type $\mathsf{rGph}(\mathcal{A},\unit)$ is contractible.
\item For any reflexive graph $\mathcal{A}$, the map 
\begin{equation*}
\mathsf{ev\usc{}pt}:\mathsf{rGph}(\unit,\mathcal{A})\to\pts{A}
\end{equation*}
given by $f\mapsto \pts{f}(\ttt)$, is an equivalence. The analogous statement is false for non-reflexive graphs.
\item The universal property of the walking edge $\mathcal{I}$ is that every edge $e$ in a graph $\mathcal{A}$, the type of reflexive graph morphisms that map the edge $a$ of $\mathcal{I}$ to $e$ is contractible. In other words, the map
\begin{equation*}
\mathsf{ev\usc{}edge} : \mathsf{rGph}(\mathcal{I},\mathcal{A})\to \sm{i,j:\pts{A}}\edg{A}(i,j)
\end{equation*}
given by $f\mapsto (\pts{f}(0),\pts{f}(1),\edg{f}(a))$ is an equivalence.
\item For the walking edge $\mathcal{I}$ there are reflexive graph morphisms
\begin{equation*}
\begin{tikzcd}
\mathcal{I} \arrow[r] & \unit \arrow[l,yshift=1ex,swap,"1"] \arrow[l,yshift=-1ex,"0"]
\end{tikzcd}
\end{equation*}
This is a \emph{cograph} object of reflexive graphs, since the morphism $\mathcal{I}\to\unit$ is a common \emph{retraction} of the end-point inclusions $0,1:\unit\to \mathcal{I}$, whereas in a reflexive pair (a graph object) the morphism $r$ is a common \emph{section} of the source and target maps.
\item \label{eg:freerfx} Since we do not have the technology available to establish that graphs and reflexive graphs form $\infty$-categories, a comparison between the two structures is limited to what we can say directly using the basic categorical operations such as composition. However, this is just enough to establish the universal property of an adjunction.

Given a non-reflexive graph $\mathcal{A}\jdeq\pairr{\pts{A},\edg{A}}$, we can obtain a reflexive graph $F\mathcal{A}$ by freely adjoining reflexivity:
\begin{align*}
\pts{F\mathcal{A}} & \defeq \pts{A} \\
\edg{F\mathcal{A}}(i,j) & \defeq \edg{A}(i,j)+(i=j) \\
\rfx{F\mathcal{A}}(i) & \defeq \inr(\refl{i}).
\end{align*}
On the other hand, there is the projection $U:\mathsf{rGph}\to\mathsf{Gph}$ which forgets the reflexivity structure, and for each graph $\mathcal{A}$ there is a graph morphism $\eta:\mathcal{A}\to UF\mathcal{A}$. Both $F$ and $U$ are functorial in the sense that they act on morphisms, and preserve units and composition in the obvious way, and $\eta$ is natural in $\mathcal{A}$. 

The universal property of the construction $F$ of freely adjoining reflexivity, is that the map
\begin{equation*}
U(\blank)\circ \eta : \mathsf{rGph}(F\mathcal{A},\mathcal{B})\to \mathsf{Gph}(\mathcal{A},U\mathcal{B})
\end{equation*}
is an equivalence. Indeed, this is the universal property that establishes $F$ as a left adjoint to $U$, even though we cannot manifest $F$ and $U$ as functors.

The construction $F$ of freely adjoining reflexivity is not surjective on morphisms. For example, there are no morphisms from $\unit$ to the graph $\mathcal{A}$ with a contractible type of vertices but no edges. However $F\mathcal{A}$ is the terminal reflexive graph.
\end{enumerate}
\end{eg}

\begin{prp}\label{thm:emb_disc}
The operation $\Delta:\UU\to\mathsf{rGph}$ is an embedding in the following sense:
\begin{enumerate}
\item As a function, $\Delta:\UU\to\mathsf{rGph}$ is an embedding.
\item For every two types $X$ and $Y$, the action on morphisms
\begin{equation*}
\Delta : (X\to Y)\to \mathsf{rGph}(\Delta X,\Delta Y)
\end{equation*}
is an equivalence.
\end{enumerate}
\end{prp}

\begin{proof}
Let $\mathcal{A}$ be a reflexive graph. Then the fiber of $\Delta:\UU\to\mathsf{rGph}$ is calculated as follows:
\begin{align*}
\fib{\Delta}{\mathcal{A}} & \defeq \sm{X:\UU} \Delta X = \mathcal{A} \\
& \eqvsym (\idtypevar{\pts{A}},\refl{})=(\edg{A},\rfx{A}) \\
& \eqvsym \prd{x:\pts{A}}\iscontr\Big(\sm{y:\pts{A}}\edg{A}(x,y)\Big),\tag{by \cref{cor:yoneda_rel}}
\end{align*}
which is a proposition.

The map $(X\to Y)\to \mathsf{rGph}(\Delta X,\Delta Y)$ is an equivalence, since the type
\begin{equation*}
\sm{\varphi:\prd{x,y:X}(x=y)\to (f(x)=f(y))}\varphi(\refl{x})=\refl{f(x)}
\end{equation*}
is equivalent to the fiber of the map
\begin{equation*}
\mathsf{ev\usc{}refl} : \Big(\prd{x,y:X}(x=y)\to (f(x)=f(y))\Big)\to \Big(f(x)=f(x)\Big),
\end{equation*}
which is contractible by \cref{lem:yoneda}.
\end{proof}

\begin{thm}\label{thm:discrete_rgraph}
Let $\mathcal{A}$ be a reflexive graph. The following are equivalent:
\begin{enumerate}
\item The canonical map
\begin{equation*}
\prd{i,j:\pts{A}} (i=j)\to \edg{A}(i,j)
\end{equation*}
given by $\refl{i}\mapsto \rfx{\mathcal{A}}(i)$ is a fiberwise equivalence. In particular, $\mathcal{A}$ is in the image of $\Delta:\UU\to\mathsf{rGph}$. 
\item The graph $\mathcal{A}$ is \define{$\mathcal{I}$-null} in the sense that the map
\begin{equation*}
\mathsf{rGph}(\unit,\mathcal{A}) \to \mathsf{rGph}(\mathcal{I},\mathcal{A})
\end{equation*}
given by pre-composition by the unique morphism $\mathsf{rGph}(\mathcal{I},\unit)$, is an equivalence.
\item The map
\begin{equation*}
\mathsf{rGph}(\mathcal{I},\mathcal{A})\to \mathsf{rGph}(\unit,\mathcal{A})
\end{equation*}
given by pre-composing with the end-point inclusion $0:\mathsf{rGph}(\unit,\mathcal{I})$, is an equivalence.
\item The map
\begin{equation*}
\mathsf{rGph}(\mathcal{I},\mathcal{A})\to \mathsf{rGph}(\unit,\mathcal{A})
\end{equation*}
given by pre-composing with the end-point inclusion $1:\mathsf{rGph}(\unit,\mathcal{I})$, is an equivalence.
\end{enumerate}
If any of these conditions hold, we say that $\mathcal{A}$ is \define{discrete}.
\end{thm}

\begin{proof}
The outline of our argument is as follows:
\begin{equation*}
\begin{tikzcd}[row sep=tiny]
& & (iii) \arrow[dr,Rightarrow] \\
(i) \arrow[r,Rightarrow] & (ii) \arrow[ur,Rightarrow] \arrow[dr,Rightarrow] & & (i). \\
& & (iv) \arrow[ur,Rightarrow]
\end{tikzcd}
\end{equation*}
Suppose (i) holds. We have a commuting square
\begin{equation*}
\begin{tikzcd}
\mathsf{rGph}(\unit,\mathcal{A}) \arrow[d,swap,"\mathsf{ev\usc{}pt}"] \arrow[r] & \mathsf{rGph}(\mathcal{I},\mathcal{A}) \arrow[d,"\mathsf{ev\usc{}edge}"] \\
\pts{A} \arrow[r] & \sm{i,j:\pts{A}}\edg{A}(i,j)
\end{tikzcd}
\end{equation*}
where both vertical maps are equivalences, and the bottom map is an equivalence by assumption. Therefore the top map is an equivalence. This proves that (i) implies (ii).

Since the reflexive graph morphism $\mathsf{rGph}(\mathcal{I},\unit)$ is a common retraction of both end-point inclusions, it follows that the pre-composition map
\begin{equation*}
\mathsf{rGph}(\unit,\mathcal{A}) \to \mathsf{rGph}(\mathcal{I},\mathcal{A})
\end{equation*}
is a common section of both pre-composition maps
\begin{equation*}
0^\ast,1^\ast:\mathsf{rGph}(\mathcal{I},\mathcal{A})\to \mathsf{rGph}(\unit,\mathcal{A})
\end{equation*}
However, assuming (ii) it follows that both $0^\ast$ and $1^\ast$ are equivalences, so (ii) implies both (iii) and (iv).

Now suppose that (iii) holds; we will show that (i) follows. We have the commuting square
\begin{equation*}
\begin{tikzcd}
\mathsf{rGph}(\mathcal{I},\mathcal{A}) \arrow[d,swap,"\mathsf{ev\usc{}edge}"] \arrow[r,"0^\ast"] & \mathsf{rGph}(\unit,\mathcal{A}) \arrow[d,"\mathsf{ev\usc{}pt}"] \\
\sm{i,j:\pts{A}}\edg{A}(i,j) \arrow[r,swap,"\pi_1"] & \pts{A}
\end{tikzcd}
\end{equation*}
in which both vertical maps are equivalences. Therefore the fibers of $\pi_1$ are equivalent to the fibers of $0^\ast$. Note that the fibers of $\pi_1$ are of the form
\begin{equation*}
\sm{j:\pts{A}}\edg{A}(i,j),
\end{equation*}
so it follows from (iii) that these are contractible. Then (i) follows by the fundamental theorem of identity types, \cref{thm:id_fundamental}.

The argument that (i) follows from (iv) is similar, using $\pi_2$ and $1^\ast$ instead of $\pi_1$ and $0^\ast$ in the square.
\end{proof}

\section{Reflexive coequalizers}

\begin{defn}
Consider a reflexive graph $\mathcal{A}$ and a type $X$ equipped with a morphism $f:\mathsf{rGph}(\mathcal{A},\Delta(X))$. We say that $f$ is a \define{reflexive coequalizer of $\mathcal{A}$} if the map
\begin{equation*}
\Delta(\blank) \circ f: (X\to Y)\to \mathsf{rGph}(\mathcal{A},\Delta(Y))
\end{equation*}
is an equivalence.
\end{defn}

Our goal in this section is to show that reflexive coequalizers can be constructed as pushouts. We will use the following lemma, which was discovered jointly with Simon Boulier.

\begin{lem}\label{lem:coh_red}
Consider a type $A$ with a type family $B$, and $a:A$. Furthermore, suppose that
\begin{equation*}
\alpha:\prd{x:A}B(x)\to (a=x).
\end{equation*}
Then the \define{coherence reduction} map
\begin{equation*}
\mathsf{coh\usc{}red} : \Big(\sm{y:B(a)}\alpha(y)=\refl{a}\Big)\to\Big(\sm{x:A}B(x)\Big)
\end{equation*}
given by $(y,q)\mapsto (a,y)$ is an equivalence.
\end{lem}

\begin{rmk}
A quick way to see that there is an equivalence
\begin{equation*}
\sm{x:A}B(x) \eqvsym \sm{b:B(a)} \alpha_{a}(b)=\refl{a}.
\end{equation*}
is to use the contractibility of the total space of identity types twice:
\begin{align*}
\sm{x:A}B(x) & \eqvsym \sm{x:A}{y:B(x)}{p:a=x}\alpha(y)=p \\
& \eqvsym \sm{y:B(a)}\alpha(y)=\refl{a}.\qedhere
\end{align*}
However, it is not clear at once that the underlying map of this composite of equivalences is indeed the map $\mathsf{coh\usc{}red}$ defined in \cref{lem:coh_red}.
\end{rmk}

\begin{proof}[Proof of \cref{lem:coh_red}]
We show that the fibers are contractible:
\begin{align*}
\fib{\mathsf{coh\usc{}red}}{(x,y)} & \eqvsym \sm{y':B(a)}{q:\alpha(y')=\refl{a}} (a,y')=(x,y) \\
& \eqvsym \sm{y':B(a)}{q:\alpha(y')=\refl{a}}{p:a=x} \mathsf{tr}_B(p,y')=y \\
& \eqvsym \sm{y':B(a)}{q:\alpha(y')=\refl{a}}{p:a=x} y'=\mathsf{tr}_B(p^{-1},y) \\
& \eqvsym \sm{p:a=x}\alpha(\mathsf{tr}_B(p^{-1},y))=\refl{a} \\
& \eqvsym \sm{p:a=x}p=\alpha(y).\qedhere
\end{align*}
\end{proof}

\begin{cor}
Consider a type $A$ with a relation $R:A\to A\to\UU$ such that
\begin{equation*}
\alpha: \prd{x,y:A}R(x,y)\to (x=y).
\end{equation*}
Then the map
\begin{equation*}
\mathsf{coh\usc{}red}:\Big(\sm{x:A}{r:R(x,x)}\alpha(r)=\refl{x}\Big)\to \Big(\sm{x,y:A}R(x,y)\Big).
\end{equation*}
given by $(x,r,c)\mapsto (x,x,r)$ is an equivalence.\qed
\end{cor}

\begin{defn}
For any reflexive graph $\mathcal{A}$ we define the span $\tilde{\mathcal{A}}$ of $\mathcal{A}$ to consist of
\begin{equation*}
\begin{tikzcd}
\tilde{A}_0 & \tilde{A}_1 \arrow[l,swap,"\pi_1"] \arrow[r,"\pi_2"] & \tilde{A}_0
\end{tikzcd}
\end{equation*}
where $\tilde{A}_0\defeq A_0$ and $\tilde{A}_1\defeq\sm{x,y:A_0}A_1(x,y)$.
\end{defn}

\begin{lem}\label{lem:rcoeq_into_disc}
For any reflexive graph $\mathcal{A}$, and any type $X$, the map
\begin{equation*}
f\mapsto (f_0,f_0,\lam{(x,y,e)}\edg{f}(e)) : \mathsf{rGph}(\mathcal{A},\Delta X) \to \mathsf{cocone}_{\tilde{\mathcal{A}}}(X),
\end{equation*}
is an equivalence, where $\tilde{\mathcal{A}}$ is the span $(\tilde{A}_1,\pi_1,\pi_2)$ from $A_0$ to $A_0$. 
\end{lem}

\begin{proof}
First observe that $\mathsf{rGph}(\mathcal{A},\Delta X)$ is equivalent to the type
\begin{equation*}
\sm{\pts{f}:\pts{A}\to X}{\edg{f}:\prd{x,y:\pts{A}}\edg{A}(x,y)\to f(x)=f(y)}\mathsf{eq\usc{}htpy}(\lam{x}\edg{f}(\rfx{\mathcal{A}}(x)))=\refl{\pts{f}}.
\end{equation*}
We write $\mathsf{rGph}'(\mathcal{A},\Delta X)$ for the latter type. Furthermore, we observe that the type $\mathsf{cocone}_{\tilde{\mathcal{A}}}(X)$ is equivalent to the type
\begin{equation*}
\mathsf{cocone}'_{\tilde{A}}(X) \defeq \sm{f,g:\pts{A}\to X}\prd{x,y:\pts{A}}\edg{A}(x,y)\to f(x)=f(y).
\end{equation*}
Now we note that we have a commuting square
\begin{equation*}
\begin{tikzcd}[column sep=large]
\mathsf{rGph}(\mathcal{A},\Delta X) \arrow[d] \arrow[r,"\varphi"] & \mathsf{cocone}_{\tilde{\mathcal{A}}}(X) \arrow[d] \\
\mathsf{rGph}'(\mathcal{A},\Delta X) \arrow[r,swap,"\mathsf{coh\usc{}red}"] & \mathsf{cocone}'_{\tilde{\mathcal{A}}}(X).
\end{tikzcd}
\end{equation*}
where the map $\varphi$ is the map in the statement, and the coherence reduction map uses the homotopy
\begin{equation*}
(f,g,H)\mapsto \mathsf{eq\usc{}htpy}(\lam{x}H(\rfx{A}(x))).
\end{equation*}
Both vertical maps and the coherence reduction map are equivalences, so it follows that the asserted map is an equivalence.
\end{proof}

\begin{prp}\label{thm:rcoeq_is_pushout}
Let $\mathcal{A}$ be a reflexive graph, and let $X$ be a type equipped with $f:\mathsf{rGph}(\mathcal{A},\Delta X)$. Then the following are equivalent:
\begin{enumerate}
\item $X$ is a reflexive coequalizer of $\mathcal{A}$.
\item The square
\begin{equation*}
\begin{tikzcd}
\sm{x,y:\pts{A}}\edg{A}(x,y) \arrow[r,"\pi_2"] \arrow[d,swap,"\pi_1"] & \pts{A} \arrow[d,"\pts{f}"] \\
\pts{A} \arrow[r,swap,"\pts{f}"] & X
\end{tikzcd}
\end{equation*}
which commutes by $\lam{(i,j,e)}\edg{f}(e)$, is a pushout square.
\end{enumerate}
In particular, there is a reflexive coequalizer for every reflexive graph $\mathcal{A}$, for which we write
\begin{equation*}
\mathsf{constr}:\mathsf{rGph}(\mathcal{A},\Delta(\mathsf{rcoeq}(\mathcal{A}))).
\end{equation*}
\end{prp}

\begin{proof}
The triangle
\begin{equation*}
\begin{tikzcd}
& Y^X \arrow[dl,swap,"{\Delta(\blank)\circ f}"] \arrow[dr,"\mathsf{cocone\usc{}map}"] \\
\mathsf{rGph}(\mathcal{A},\Delta Y) \arrow[rr,"\eqvsym"] & & \mathsf{cocone}(Y)
\end{tikzcd}
\end{equation*}
commutes, for any type $Y$. In this triangle, the bottom map is the map defined in \cref{lem:rcoeq_into_disc}, which is an equivalence, so the claim follows by the 3-for-2 property of equivalences.
\end{proof}

\begin{eg}\label{eg:rcoeq} {}~
\begin{enumerate}
\item The reflexive graph quotient of the discrete graph $\Delta(X)$ of a type $X$ is just $X$ itself. It also follows that $\mathsf{constr} : \mathsf{rGph}(\mathcal{A},\Delta(\mathsf{rcoeq}(\mathcal{A})))$ is an equivalence of reflexive graphs if and only if $\mathcal{A}$ is a discrete graph.
\item The reflexive graph quotient of the indiscrete graph $\nabla(X)$ on a type $X$ is the join square $\join{X}{X}$.
\item Let $\mathcal{A}$ be a non-reflexive graph, and let $F(\mathcal{A})$ be the reflexive graph obtained by freely adding a proof of reflexivity, as in \autoref{eg:freerfx} of \autoref{eg:rgraph_morphism}. Then the non-reflexive graph quotient of $\mathcal{A}$ is the reflexive graph quotient of $F(\mathcal{A})$. 
\item Let $f:A\to B$ be a map. Then the reflexive coequalizer of the reflexive graph $(A,k(f))$ is the fiberwise join $\join[X]{A}{A}$, which was introduced in \cref{defn:fib_join}.
\item Let $X$ be a type with base point $x_0:X$. Define the reflexive graph $SX$ by
\begin{align*}
\pts{SX} & \defeq \unit \\
\edg{SX} & \defeq \lam{\nameless}{\nameless}X \\
\rfx{SX} & \defeq \lam{\nameless}x_0.
\end{align*} 
The reflexive graph quotient of $SX$ is the suspension of $X$.
\item The reflexive coequalizer of the walking edge $\mathcal{I}$ is the interval, which is contractible.
\end{enumerate}
\end{eg}

\section{Descent for reflexive coequalizers}\label{sec:descent_rcoeq}

Recall from \cite{AnelBiedermanFinsterJoyal} that morphism $f:X\to Y$ is right orthogonal to a map $i:A\to B$ if and only if the square
\begin{equation*}
\begin{tikzcd}
X^B \arrow[r,"\blank\circ i"] \arrow[d,swap,"f\circ \blank"] & X^A \arrow[d,"f\circ \blank"] \\
Y^B \arrow[r,swap,"\blank\circ i"] & Y^A
\end{tikzcd}
\end{equation*}
is a pullback square. We use this way of stating the orthogonality condition in our definition of fibrations of reflexive graphs: a morphism $f:\mathsf{rGph}(\mathcal{X},\mathcal{Y})$ of reflexive graphs is said to be right orthogonal to a morphism $i:\mathsf{rGph}(\mathcal{A},\mathcal{B})$ if the square
\begin{equation*}
\begin{tikzcd}
\mathsf{rGph}(\mathcal{X},\mathcal{B}) \arrow[r,"\blank\circ i"] \arrow[d,swap,"f\circ\blank"] & \mathsf{rGph}(\mathcal{X},\mathcal{A}) \arrow[d,"f\circ\blank"] \\
\mathsf{rGph}(\mathcal{Y},\mathcal{B}) \arrow[r,swap,"\blank\circ i"] & \mathsf{rGph}(\mathcal{Y},\mathcal{A})
\end{tikzcd}
\end{equation*}
is a pullback square.%
\footnote{One may note that the category of reflexive graphs is also locally cartesian closed (has $\Pi$-types), so that we could also state the orthogonality condition internally. Although this is straightforward, it is extra work and we will not need it in this thesis.}

\begin{defn}\label{defn:graph_fibration}
A morphism $f:\mathsf{rGph}(\mathcal{A},\mathcal{B})$ is said to be a \define{left fibration} of reflexive graphs if it is right orthogonal to the morphism $0:\mathsf{rGph}(\unit,\mathcal{I})$. Similarly, we say that $f$ is a \define{right fibration} of reflexive graphs if it is right orthogonal to the morphism $1:\mathsf{rGph}(\unit,\mathcal{I})$, and we say that $f$ is a \define{fibration} if it is both a left and a right fibration.
\end{defn}

\begin{lem}\label{lem:leftfib_Inull}
Suppose $f:\mathsf{rGph}(\mathcal{B},\mathcal{A})$ is a left or a right fibration. Then $f$ is right orthogonal to the terminal projection $t:\mathsf{rGph}(\mathcal{I},\unit)$.
\end{lem}

\begin{proof}
We prove the assertion assuming that $f$ is a left fibration, the case of a right fibration being similar. Consider the diagram
\begin{equation*}
\begin{tikzcd}
\mathsf{rGph}(\unit,\mathcal{A}) \arrow[d] \arrow[r] & \mathsf{rGph}(\mathcal{I},\mathcal{A}) \arrow[r] \arrow[d] & \mathsf{rGph}(\unit,\mathcal{A}) \arrow[d] \\
\mathsf{rGph}(\unit,\mathcal{B}) \arrow[r] & \mathsf{rGph}(\mathcal{I},\mathcal{B}) \arrow[r] & \mathsf{rGph}(\unit,\mathcal{B})
\end{tikzcd}
\end{equation*}
The square on the right is a pullback. Since the composite $\unit \to \mathcal{I}\to \unit$ is the identity morphism of reflexive graphs, the outer rectangle is also a pullback. Therefore the square on the left is a pullback.
\end{proof}

\begin{defn}
A morphism $f:\mathsf{rGph}(\mathcal{A},\mathcal{B})$ of reflexive graphs is said to be \define{left cartesian} if the naturality square
\begin{equation*}
\begin{tikzcd}
\sm{i,j:\pts{A}}\edg{A}(i,j) \arrow[d] \arrow[r,"\pi_1"] & \pts{A} \arrow[d] \\
\sm{i,j:\pts{B}}\edg{B}(i,j) \arrow[r,swap,"\pi_1"] & \pts{B}
\end{tikzcd}
\end{equation*}
is a pullback square. Similarly we say that $f$ is \define{right cartesian} if the naturality square
\begin{equation*}
\begin{tikzcd}
\sm{i,j:\pts{A}}\edg{A}(i,j) \arrow[d] \arrow[r,"\pi_2"] & \pts{A} \arrow[d] \\
\sm{i,j:\pts{B}}\edg{B}(i,j) \arrow[r,swap,"\pi_2"] & \pts{B}
\end{tikzcd}
\end{equation*}
is a pullback square, and we say that $f$ is \define{cartesian} if it is both left and right cartesian.
\end{defn}

\begin{eg}
A common way to obtain a cartesian morphism of reflexive graphs is via equifibered families.
An \define{equifibered family} $\mathcal{E}$ over $\mathcal{A}$ consists of
\begin{align*}
\pts{E} & : \pts{A}\to\UU \\
\edg{E} & : \prd{i,j:\pts{A}} \edg{A}(i,j)\to (\eqv{\pts{E}(i)}{\pts{E}(j)}) \\
\rfx{E} & : \prd{i:\pts{A}} \edg{E}(\rfx{\mathcal{A}}(i))\htpy \idfunc[\pts{E}(i)]
\end{align*}
Given an equifibered family $\mathcal{E}$, we form the reflexive graph $\msm{\mathcal{A}}{\mathcal{E}}$ by
\begin{align*}
\pts{\msm{\mathcal{A}}{\mathcal{E}}} & \defeq \sm{x:\pts{A}}\pts{E}(x) \\
\edg{\msm{\mathcal{A}}{\mathcal{E}}}((x,u),(y,v)) & \defeq \sm{e:\edg{A}(x,y)} \edg{E}(e,u) = v \\
\rfx{\msm{\mathcal{A}}{\mathcal{E}}}(x,u) & \defeq (\rfx{\mathcal{A}}(x),\rfx{\mathcal{E}}(x,u)).
\end{align*}
There is an obvious projection morphism $\proj 1 : \mathsf{rGph}(\msm{\mathcal{A}}{\mathcal{E}},\mathcal{A})$, which is cartesian, because the squares in the diagram
\begin{equation*}
\begin{tikzcd}
\pts{\msm{\mathcal{A}}{\mathcal{E}}} \arrow[d,"\proj 1"] & \sm{s,t:\msm{\mathcal{A}}{\mathcal{E}}}\edg{\msm{\mathcal{A}}{\mathcal{E}}} \arrow[l] \arrow[r] \arrow[d,swap,"\proj 1"] & \pts{\msm{\mathcal{A}}{\mathcal{E}}} \arrow[d] \\
\pts{\mathcal{A}} & \sm{x,y:\pts{A}}\edg{A}(x,y) \arrow[l] \arrow[r] & \pts{\mathcal{A}}
\end{tikzcd}
\end{equation*}
are pullback squares by \cref{thm:pb_fibequiv}.
\end{eg}

\begin{lem}
Consider $f:\mathsf{rGph}(\mathcal{B},\mathcal{A})$ and $g:\mathsf{rGph}(\mathcal{C},\mathcal{B})$, and suppose that $f$ is (left/right) cartesian. Then $g$ is (left/right) cartesian if and only if $f\circ g$ is (left/right) cartesian, respectively.
\end{lem}

\begin{proof}
Immediate by \cref{thm:pb_pasting}.
\end{proof}

\begin{prp}\label{prp:fib_cart}
Let $f:\mathsf{rGph}(\mathcal{B},\mathcal{A})$ be a morphism of reflexive graphs. The following are equivalent:
\begin{enumerate}
\item $f$ is a (left/right) fibration.
\item $f$ is (left/right) cartesian.
\end{enumerate}
\end{prp}

\begin{proof}
We only show that $f$ is a left fibration if and only if $f$ is left cartesian, the right case being similar.

For any morphism $f:\mathsf{rGph}(\mathcal{B},\mathcal{A})$ we have the commuting cube
\begin{equation*}
\begin{tikzcd}
& \mathsf{rGph}(\mathcal{I},\mathcal{B}) \arrow[dl] \arrow[d] \arrow[dr] \\
\mathsf{rGph}(\unit,\mathcal{B}) \arrow[d] & \sm{i,j:\pts{B}}\edg{B}(i,j) \arrow[dl] \arrow[dr] & \mathsf{rGph}(\mathcal{I},\mathcal{A}) \arrow[dl,crossing over] \arrow[d] \\
\pts{B} \arrow[dr] & \mathsf{rGph}(\unit,\mathcal{A}) \arrow[d] \arrow[from=ul,crossing over] & \sm{i,j:\pts{A}}\edg{A}(i,j) \arrow[dl] \\
\phantom{\sm{i,j:\pts{A}}\edg{A}(i,j)} & \pts{A}
\end{tikzcd}
\end{equation*}
in which all vertical maps are equivalences. Therefore the top square is a pullback if and only if the bottom square is a pullback, which proves that (ii) holds if and only if (iii) holds.
\end{proof}

\begin{cor}
Suppose that $f:\mathsf{rGph}(\mathcal{B},\mathcal{A})$ is left or right cartesian. Then the square
\begin{equation*}
\begin{tikzcd}
\pts{B} \arrow[r] \arrow[d] & \sm{i,j:\pts{B}}\edg{B}(i,j) \arrow[d] \\
\pts{A} \arrow[r] & \sm{i,j:\pts{A}}\edg{A}(i,j)
\end{tikzcd}
\end{equation*}
is a pullback square.
\end{cor}

\begin{prp}
Consider a commuting square
\begin{equation*}
\begin{tikzcd}
\mathcal{C} \arrow[r] \arrow[d] & \mathcal{B} \arrow[d] \\
\mathcal{A} \arrow[r] & \mathcal{X}
\end{tikzcd}
\end{equation*}
of reflexive graphs. The following are equivalent:
\begin{enumerate}
\item The square is a pullback square.
\item The squares
\begin{equation*}
\begin{tikzcd}
\tilde{C}_0 \arrow[r] \arrow[d] & \tilde{B}_0 \arrow[d] & \tilde{C}_1 \arrow[r] \arrow[d] & \tilde{B}_1 \arrow[d] \\
\tilde{A}_0 \arrow[r] & \tilde{X}_0 & \tilde{A}_1 \arrow[r] & \tilde{X}_1
\end{tikzcd}
\end{equation*}
are pullback squares.
\end{enumerate}
\end{prp}

\begin{proof}
Straightforward.
\end{proof}

The following proposition is true more generally:

\begin{prp}\label{thm:fibration_pullback}
Consider a pullback square
\begin{equation*}
\begin{tikzcd}
\mathcal{B}' \arrow[r] \arrow[d,swap,"{f'}"] & \mathcal{B} \arrow[d,"f"] \\
\mathcal{A}' \arrow[r] & \mathcal{A}
\end{tikzcd}
\end{equation*}
of reflexive graphs, and let $h:\mathsf{rGph}(\mathcal{Y},\mathcal{X})$ be a morphism of reflexive graphs. If $f$ is right orthogonal to $h$, then so is $f'$. In particular, if $f$ is a (left/right) fibration, then so is $f'$. 
\end{prp}

\begin{proof}
Consider the commuting cube
\begin{equation*}
\begin{tikzcd}
& \mathsf{rGph}(\mathcal{Y},\mathcal{B}') \arrow[dl] \arrow[d] \arrow[dr] \\
\mathsf{rGph}(\mathcal{Y},\mathcal{A}') \arrow[d] & \mathsf{rGph}(\mathcal{X},\mathcal{B}') \arrow[dl] \arrow[dr] & \mathsf{rGph}(\mathcal{Y},\mathcal{B}) \arrow[dl,crossing over] \arrow[d] \\
\mathsf{rGph}(\mathcal{X},\mathcal{A}') \arrow[dr] & \mathsf{rGph}(\mathcal{Y},\mathcal{A}) \arrow[from=ul,crossing over] \arrow[d] & \mathsf{rGph}(\mathcal{X},\mathcal{B}) \arrow[dl] \\
& \mathsf{rGph}(\mathcal{X},\mathcal{A}). & \phantom{\mathsf{rGph}(\mathcal{X},\mathcal{A}')}
\end{tikzcd}
\end{equation*}
Note $\mathsf{rGph}(\mathcal{X},\blank)$ preserves pullbacks for any reflexive graph $\mathcal{X}$. Therefore it follows that the top and bottom squares are pullback squares. Moreover, the square on the front right is a pullback square by the assumption that $f$ is right orthogonal to $h:\mathsf{rGph}(\mathcal{Y},\mathcal{X})$. Now it follows by the pasting property of pullbacks that the square on the back left is a pullback square. In other words, $f'$ is right orthogonal to $h$. 
\end{proof}

\begin{prp}\label{prp:fibration_discrete}
For any map $f:X\to Y$, the morphism $\Delta f : \Delta X \to \Delta Y$ is a fibration.
\end{prp}

\begin{proof}
The commuting square
\begin{equation*}
\begin{tikzcd}
\mathsf{rGph}(\mathcal{I},\Delta X) \arrow[d] \arrow[r,"0^\ast/1^\ast"] & \mathsf{rGph}(\unit,\Delta X) \arrow[d] \\
\mathsf{rGph}(\mathcal{I},\Delta Y) \arrow[r,swap,"0^\ast/1^\ast"] & \mathsf{rGph}(\unit,\Delta Y)
\end{tikzcd}
\end{equation*}
is a pullback square, since both horizontal maps are equivalences.
\end{proof}

\begin{thm}\label{thm:rcoeq_cartesian}
Consider a commuting square
\begin{equation}\label{eq:rcoeq_descent}
\begin{tikzcd}
\mathcal{B} \arrow[r] \arrow[d,swap,"f"] & \Delta Y \arrow[d] \\
\mathcal{A} \arrow[r] & \Delta X 
\end{tikzcd}
\end{equation}
of reflexive graphs, where $f:\mathcal{B}\to\mathcal{A}$ is a fibration of reflexive graphs. Then the following are equivalent:
\begin{enumerate}
\item The square is a pullback square of reflexive graphs.
\item The square
\begin{equation*}
\begin{tikzcd}
\mathsf{rcoeq}(\mathcal{B}) \arrow[r] \arrow[d,swap,"\mathsf{rcoeq}(f)"] & Y \arrow[d] \\
\mathsf{rcoeq}(\mathcal{A}) \arrow[r] & X 
\end{tikzcd}
\end{equation*}
of types is a pullback square.
\end{enumerate}
\end{thm}

\begin{proof}
The commuting square \cref{eq:rcoeq_descent} induces a commuting cube
\begin{equation*}
\begin{tikzcd}
& \tilde{B}_1 \arrow[dl] \arrow[dr] \arrow[d] \\
\tilde{B}_0 \arrow[d] & \tilde{A}_1 \arrow[dl] \arrow[dr] & \tilde{B}_0 \arrow[d] \arrow[dl,crossing over] \\
\tilde{A}_0 \arrow[dr] & Y \arrow[from=ul,crossing over] \arrow[d] & \tilde{A}_0 \arrow[dl] \\
& X
\end{tikzcd}
\end{equation*}
in which the two vertical squares in the back are pullback squares. Moreover, the square in \cref{eq:rcoeq_descent} is a pullback square if and only if the front two squares in the cube are pullback squares. The front two squares are pullback squares if and only if the square
\begin{equation*}
\begin{tikzcd}
\tilde{B}_0 \sqcup^{\tilde{B}_1} \tilde{B}_0 \arrow[r] \arrow[d] & Y \arrow[d] \\
\tilde{A}_0 \sqcup^{\tilde{A}_1} \tilde{A}_0 \arrow[r] & Y
\end{tikzcd}
\end{equation*}
is a pullback square. Since the pushouts on the left are reflexive coequalizers, the claim follows.
\end{proof}

\section{Colimits of diagrams over graphs}

\begin{defn}
Let $\mathcal{A}$ be a reflexive graph. A diagram $\mathcal{D}$ over $\mathcal{A}$ is a triple $(\pts{D},\edg{D},\rfx{\mathcal{D}})$ consisting of
\begin{align*}
\pts{D} & : \pts{A}\to \UU \\
\edg{D} & : \prd{i,j:\pts{A}} \edg{A}(i,j)\to (\pts{D}(i)\to \pts{D}(j)) \\
\rfx{\mathcal{D}} & : \prd{i:\pts{A}} \edg{D}(\rfx{\mathcal{A}}(i))\htpy \idfunc[\pts{D}(i)].
\end{align*}
\end{defn}

\begin{defn}
Let $\mathcal{D}$ be a diagram over $\mathcal{A}$. We form the \define{total graph} $\msm{\mathcal{A}}{\mathcal{D}}$ of $\mathcal{D}$ by
\begin{align*}
\pts{\msm{\mathcal{A}}{\mathcal{D}}} & \defeq \sm{i:\pts{A}}\pts{D}(i) \\
\edg{\msm{\mathcal{A}}{\mathcal{D}}}((i,x),(j,y)) & \defeq \sm{e:\edg{A}(i,j)} \edg{D}(e,x)=y \\
\rfx{\msm{\mathcal{A}}{\mathcal{D}}}((i,x)) & \defeq (\rfx{\mathcal{A}}(i),\rfx{\mathcal{D}}(i,x)).
\end{align*}
There is an obvious projection $\proj 1 : \mathsf{rGph}(\msm{\mathcal{A}}{\mathcal{D}},\mathcal{A})$.
\end{defn}

\begin{lem}
For any diagram $\mathcal{D}$ over $\mathcal{A}$, the projection $\proj 1 : \mathsf{rGph}(\msm{\mathcal{A}}{\mathcal{D}},\mathcal{A})$ is a left fibration.
\end{lem}

\begin{proof}
It suffices to show that
\begin{equation*}
\begin{tikzcd}
\sm{(i,x),(j,y):\pts{\msm{\mathcal{A}}{\mathcal{D}}}} \edg{\msm{\mathcal{A}}{\mathcal{D}}}((i,x),(j,y)) \arrow[r,"\pi_1"] \arrow[d] & \msm{\mathcal{A}}{\mathcal{D}} \arrow[d] \\
\sm{i,j:\pts{A}}\edg{A}(i,j) \arrow[r,swap,"\pi_1"] & \pts{A}
\end{tikzcd}
\end{equation*}
is a pullback square. Note that this square is equivalent to the square
\begin{equation*}
\begin{tikzcd}[column sep=6em]
\sm{i,j:\pts{A}}{e:\edg{A}(i,j)}\pts{D}(i) \arrow[r,"{\lam{(i,j,e,x)}(i,x)}"] \arrow[d] & \sm{i:\pts{A}}\pts{D}(i) \arrow[d] \\
\sm{i,j:\pts{A}}\edg{A}(i,j) \arrow[r,swap,"\pi_1"] & \pts{A},
\end{tikzcd}
\end{equation*}
which is clearly a pullback square.
\end{proof}

\begin{rmk}
It can be shown that the total graph operation is an equivalence from diagrams over $\mathcal{A}$ to left fibrations over $\mathcal{A}$.
\end{rmk}

\begin{defn}
Let $D$ be a diagram over $\mathcal{A}$, and let $X$ be a type. A $D$-cocone on $X$ is a morphism $f:\mathsf{rGph}(\msm{\mathcal{A}}{\mathcal{D}},\Delta X)$ of reflexive graphs. A $D$-cocone $f$ on $X$ is colimiting if the map
\begin{equation*}
\Delta(\blank)\circ f : (X\to Y)\to \mathsf{rGph}(\msm{\mathcal{A}}{\mathcal{D}},\Delta Y)
\end{equation*}
is an equivalence for every type $Y$. 
\end{defn}

\begin{rmk}
By \cref{thm:rcoeq_is_pushout} it follows that every diagram $D$ over any reflexive graph $\mathcal{A}$ has a colimit. 
\end{rmk}

\begin{defn}
Let $\mathcal{D}$ and $\mathcal{D}'$ be diagrams over $\mathcal{A}$. A \define{natural transformation} $\tau : \mathcal{D}'\to \mathcal{D}$ of diagrams consists of
\begin{align*}
\pts{\tau} & : \prd{i:\pts{A}}\pts{D'}(i)\to\pts{D}(i) \\
\edg{\tau} & : \prd*{i,j:\pts{A}}{e:\edg{A}(i,j)}\prd{x:\pts{D}(i)}\pts{\tau}(\edg{D'}(e,x))=\edg{D}(e,\pts{\tau}(x)) \\
\rfx{\tau} & : \prd{i:\pts{A}}{x:\pts{D}(i)} \dpath{}{}{\edg{\tau}(\rfx{A}(i),x)}{\refl{\pts{\tau}(x)}}
\end{align*}
A natural transformation $\tau:\mathcal{D}'\to \mathcal{D}$ of diagrams over $\mathcal{A}$ is said to be \define{cartesian} if the commutative squares
\begin{equation*}
\begin{tikzcd}[column sep=large]
\pts{D'}(i) \arrow[r,"{\edg{D'}(e)}"] \arrow[d,swap,"\pts{\tau}(i)"] & \pts{D'}(j) \arrow[d,"\pts{\tau}(j)"] \\
\pts{D}(i) \arrow[r,swap,"{\edg{D}(e)}"] & \pts{D}(i)
\end{tikzcd}
\end{equation*}
are all pullback squares. 
\end{defn}

We will use the following general fact about pullbacks.

\begin{prp}\label{lem:pb_total}
Let $I$ be a type, and consider for each $i:I$ a commuting square
\begin{equation*}
\begin{tikzcd}
C_i \arrow[r,"q_i"] \arrow[d,swap,"p_i"] & B_i \arrow[d,"g_i"] \\
A_i \arrow[r,swap,"f_i"] & X_i.
\end{tikzcd}
\end{equation*}
with $H_i:f_i\circ p_i\htpy g_i\circ q_i$. Then the following are equivalent:
\begin{enumerate}
\item For each $i:I$, the square is a pullback square.
\item The induced square on total spaces
\begin{equation*}
\begin{tikzcd}
\sm{i:I}C_i \arrow[r,"\total{q}"] \arrow[d,swap,"\total{p}"] & \sm{i:I}B_i \arrow[d,"\total{g}"] \\
\sm{i:I}A_i \arrow[r,swap,"\total{f}"] & \sm{i:I}X_i.
\end{tikzcd}
\end{equation*}
which commutes via the homotopy $\total{H}:\total{f}\circ\total{p}\htpy\total{g}\circ\total{q}$, is a pullback square.
\end{enumerate}
\end{prp}

\begin{proof}
The gap map of the square in assertion (ii) factors as follows:
\begin{equation*}
\begin{tikzcd}[column sep=small]
\phantom{\Big(\sm{i:I}A_i\Big)\times_{\big(\sm{i:I}X_i\big)} \Big(\sm{i:I}B_i\Big)} & \sm{i:I}C_i \arrow[dl,swap,"\total{\mathsf{gap}(p_i,q_i,H_i)}"] \arrow[dr,"{\mathsf{gap}(\total{p},\total{q},\total{H})}"] \\
\sm{i:I}A_i\times_{X_i}B_i \arrow[rr,swap,"{\lam{(i,a,b,p)}((i,a),(i,b),(\refl{i},p))}"] & & \Big(\sm{i:I}A_i\Big)\times_{\big(\sm{i:I}X_i\big)} \Big(\sm{i:I}B_i\Big)
\end{tikzcd}
\end{equation*}
and the bottom map is an equivalence. Therefore it follows by the 3-for-2 property and an application of \cref{thm:fib_equiv} that $\mathsf{gap}(p_i,q_i,H_i)$ is an equivalence for each $i:I$, if and only if $\mathsf{gap}(\total{p},\total{q},\total{H})$ is an equivalence.
\end{proof}

\begin{prp}\label{prp:nattrans_cartesian}
Let $\tau:\mathsf{nat}(\mathcal{E},\mathcal{D})$ be a natural transformation of diagrams over $\mathcal{A}$. Then $\tau$ is cartesian if and only if the induced morphism of reflexive graphs
\begin{equation*}
\total{\tau} : \mathsf{rGph}(\msm{\mathcal{A}}{\mathcal{E}},\msm{\mathcal{A}}{\mathcal{D}})
\end{equation*}
is cartesian.
\end{prp}

\begin{proof}
Straightforward consequence of \cref{lem:pb_total}.
\end{proof}

\begin{cor}
Let $\tau:\mathsf{cart}(\mathcal{E},\mathcal{D})$ be a cartesian morphisms of diagrams over $\mathcal{A}$, and consider a commuting square of reflexive graphs of the following form
\begin{equation*}
\begin{tikzcd}
\msm{\mathcal{A}}{\mathcal{E}} \arrow[d,swap,"\mathsf{tot}(\tau)"] \arrow[r] & \Delta Y \arrow[d] \\
\msm{\mathcal{A}}{\mathcal{D}} \arrow[r] & \Delta X.
\end{tikzcd}
\end{equation*}
Then the following are equivalent:
\begin{enumerate}
\item The square is a pullback square.
\item The square
\begin{equation*}
\begin{tikzcd}
\mathsf{colim}(\mathcal{E}) \arrow[r] \arrow[d,swap,"\mathsf{colim}(\tau)"] & Y \arrow[d] \\
\mathsf{colim}(\mathcal{D}) \arrow[r] & X
\end{tikzcd}
\end{equation*}
is a pullback square.
\end{enumerate}
\end{cor}

\section{Sequential colimits}

Type sequences are diagrams of the following form.
\begin{equation*}
\begin{tikzcd}
A_0 \arrow[r,"f_0"] & A_1 \arrow[r,"f_1"] & A_2 \arrow[r,"f_2"] & \cdots.
\end{tikzcd}
\end{equation*}
Their formal specification is as follows.

\begin{defn}
An \define{(increasing) type sequence} $\mathcal{A}$ consists of
\begin{align*}
A & : \N\to\UU \\
f & : \prd{n:\N} A_n\to A_{n+1}. 
\end{align*}
\end{defn}

Sequential colimits are characterized by their universal property.

\begin{defn}
\begin{enumerate}
\item A \define{(sequential) cocone} on a type sequence $\mathcal{A}$ with vertex $X$ consists of
\begin{align*}
h & : \prd{n:\N} A_n\to X \\
H & : \prd{n:\N} f_n\htpy f_{n+1}\circ H_n.
\end{align*}
We write $\mathsf{cocone}(X)$ for the type of cones with vertex $X$.
\item Given a cone $(h,H)$ with vertex $X$ on a type sequence $\mathcal{A}$ we define the map
\begin{equation*}
\mathsf{cocone\usc{}map}(h,H) : (X\to Y)\to \mathsf{cocone}(X)
\end{equation*}
given by $f\mapsto (f\circ h,\lam{n}{x}\mathsf{ap}_f(H_n(x)))$. 
\item We say that a cone $(h,H)$ with vertex $X$ is \define{colimiting} if the map $\mathsf{cocone\usc{}map}(h,H)$ is an equivalence for any type $Y$. 
\end{enumerate}
\end{defn}

In order to study sequential colimits, we first note that type sequences appear as diagrams over the graph $(\N,\prec)$ with type of vertices $\N$, and with
\begin{equation*}
(n\prec m) \defeq (n+1=m).
\end{equation*}
Note that this is a non-reflexive graph. A reflexive graph is obtained by the adjunction explained in \cref{eg:freerfx} of \cref{eg:rgraph_morphism}, which transforms the relation $\prec$ into the reflexive relation $\preceq$ given by
\begin{equation*}
(n\preceq m) \defeq (n+1=m)+ (n=m).
\end{equation*}
A diagram $\mathcal{A}$ on $(\N,\prec)$ consists of
\begin{equation*}
\pts{A} : \prd{n:\N} A_n \\
\edg{A} : \prd{n,m:\N}{e:n+1=m} A_n\to A_m.
\end{equation*}
Note that the type of $\edg{A}$ is equivalent to the type $\prd{n:\N}A_n\to A_{n+1}$, so we see indeed that type sequences are equivalently described as diagrams over $(\N,\prec)$. Similarly, a sequential cocone on $\mathcal{A}$ with vertex $X$ is equivalently described as a cocone on the diagram $\mathcal{A}$ with vertex $X$. Thus, we also have the following:

\begin{lem}\label{thm:sequential_up}
Consider a cocone $(h,H)$ with vertex $B$ for a type sequence $\mathcal{A}$. The following are equivalent:
\begin{enumerate}
\item The cocone $(h,H)$ is colimiting.
\item The cocone $(h,H)$ is inductive in the sense that for every type family $P:B\to \UU$, the map
\begin{align*}
\Big(\prd{b:B}P(b)\Big)\to {}& \sm{h:\prd{n:\N}{x:A_n}P(h_n(x))}\\ 
& \qquad \prd{n:\N}{x:A_n} \mathsf{tr}_P(H_n(x),h_n(x))={h_{n+1}(f_n(x))}
\end{align*}
given by
\begin{equation*}
s\mapsto (\lam{n}s\circ h_n,\lam{n}{x} \mathsf{apd}_{s}(H_n(x)))
\end{equation*}
has a section.
\item The map in (ii) is an equivalence.\qedhere
\end{enumerate}
\end{lem}

\begin{defn}
The type of \define{descent data} on a type sequence $\mathcal{A}\jdeq (A,f)$ is defined to be
\begin{equation*}
\mathsf{Desc}(\mathcal{A}) \defeq \sm{B:\prd{n:\N}A_n\to\UU}\prd{n:\N}{x:A_n}\eqv{B_n(x)}{B_{n+1}(f_n(x))}.
\end{equation*}
Equivalently, if $\mathcal{A}$ is seen as a diagram over $(\N,\prec)$, 
\end{defn}

\begin{defn}
We define a map
\begin{equation*}
\mathsf{desc\usc{}fam} : (A_\infty\to\UU)\to\mathsf{Desc}(\mathcal{A})
\end{equation*}
by $B\mapsto (\lam{n}{x}B(\mathsf{seq\usc{}in}(n,x)),\lam{n}{x}\mathsf{tr}_B(\mathsf{seq\usc{}glue}(n,x)))$.
\end{defn}

\begin{thm}
The map 
\begin{equation*}
\mathsf{desc\usc{}fam} : (A_\infty\to\UU)\to\mathsf{Desc}(\mathcal{A})
\end{equation*}
is an equivalence.
\end{thm}

\begin{defn}
A \define{cartesian transformation} of type sequences from $\mathcal{A}$ to $\mathcal{B}$ is a pair $(h,H)$ consisting of
\begin{align*}
h & : \prd{n:\N} A_n\to B_n \\
H & : \prd{n:\N} g_n\circ h_n \htpy h_{n+1}\circ f_n,
\end{align*}
such that each of the squares in the diagram
\begin{equation*}
\begin{tikzcd}
A_0 \arrow[d,swap,"h_0"] \arrow[r,"f_0"] & A_1 \arrow[d,swap,"h_1"] \arrow[r,"f_1"] & A_2 \arrow[d,swap,"h_2"] \arrow[r,"f_2"] & \cdots \\
B_0 \arrow[r,swap,"g_0"] & B_1 \arrow[r,swap,"g_1"] & B_2 \arrow[r,swap,"g_2"] & \cdots
\end{tikzcd}
\end{equation*}
is a pullback square. We define
\begin{align*}
\mathsf{cart}(\mathcal{A},\mathcal{B}) & \defeq\sm{h:\prd{n:\N}A_n\to B_n} \\
& \qquad\qquad \sm{H:\prd{n:\N}g_n\circ h_n\htpy h_{n+1}\circ f_n}\prd{n:\N}\mathsf{is\usc{}pullback}(h_n,f_n,H_n),
\end{align*}
and we write
\begin{equation*}
\mathsf{Cart}(\mathcal{B}) \defeq \sm{\mathcal{A}:\mathsf{Seq}}\mathsf{cart}(\mathcal{A},\mathcal{B}).
\end{equation*}
\end{defn}

\begin{defn}
We define a map
\begin{equation*}
\mathsf{cart\usc{}map}(\mathcal{B}) : \Big(\sm{X':\UU}X'\to X\Big)\to\mathsf{Cart}(\mathcal{B}).
\end{equation*}
which associates to any morphism $h:X'\to X$ a cartesian transformation of type sequences into $\mathcal{B}$.
\end{defn}

\begin{thm}
The operation $\mathsf{cart\usc{}map}(\mathcal{B})$ is an equivalence.
\end{thm}

The flattening lemma for sequential colimits essentially states that sequential colimits commute with $\Sigma$. 

\begin{lem}
Consider
\begin{align*}
B & : \prd{n:\N}A_n\to\UU \\
g & : \prd{n:\N}{x:A_n}\eqv{B_n(x)}{B_{n+1}(f_n(x))}.
\end{align*}
and suppose $P:A_\infty\to\UU$ is the unique family equipped with
\begin{align*}
e & : \prd{n:\N}\eqv{B_n(x)}{P(\mathsf{seq\usc{}in}(n,x))}
\end{align*}
and homotopies $H_n(x)$ witnessing that the square
\begin{equation*}
\begin{tikzcd}[column sep=7em]
B_n(x) \arrow[r,"g_n(x)"] \arrow[d,swap,"e_n(x)"] & B_{n+1}(f_n(x)) \arrow[d,"e_{n+1}(f_n(x))"] \\
P(\mathsf{seq\usc{}in}(n,x)) \arrow[r,swap,"{\mathsf{tr}_P(\mathsf{seq\usc{}glue}(n,x))}"] & P(\mathsf{seq\usc{}in}(n+1,f_n(x)))
\end{tikzcd}
\end{equation*}
commutes. Then $\sm{t:A_\infty}P(t)$ satisfies the universal property of the sequential colimit of the type sequence
\begin{equation*}
\begin{tikzcd}
\sm{x:A_0}B_0(x) \arrow[r,"{\total[f_0]{g_0}}"] & \sm{x:A_1}B_1(x) \arrow[r,"{\total[f_1]{g_1}}"] & \sm{x:A_2}B_2(x) \arrow[r,"{\total[f_2]{g_2}}"] & \cdots.
\end{tikzcd}
\end{equation*}
\end{lem}

In the following theorem we rephrase the flattening lemma in using cartesian transformations of type sequences.

\begin{thm}
Consider a commuting diagram of the form
\begin{equation*}
\begin{tikzcd}[column sep=small,row sep=small]
A_0 \arrow[rr] \arrow[dd] & & A_1 \arrow[rr] \arrow[dr] \arrow[dd] &[-.9em] &[-.9em] A_2 \arrow[dl] \arrow[dd] & & \cdots \\
& & & X \arrow[from=ulll,crossing over] \arrow[from=urrr,crossing over] \arrow[from=ur,to=urrr] \\
B_0 \arrow[rr] \arrow[drrr] & & B_1 \arrow[rr] \arrow[dr] & & B_2 \arrow[rr] \arrow[dl] & & \cdots \arrow[dlll] \\
& & & Y \arrow[from=uu,crossing over] 
\end{tikzcd}
\end{equation*}
If each of the vertical squares is a pullback square, and $Y$ is the sequential colimit of the type sequence $B_n$, then $X$ is the sequential colimit of the type sequence $A_n$. 
\end{thm}

\begin{cor}
Consider a commuting diagram of the form
\begin{equation*}
\begin{tikzcd}[column sep=small,row sep=small]
A_0 \arrow[rr] \arrow[dd] & & A_1 \arrow[rr] \arrow[dr] \arrow[dd] &[-.9em] &[-.9em] A_2 \arrow[dl] \arrow[dd] & & \cdots \\
& & & X \arrow[from=ulll,crossing over] \arrow[from=urrr,crossing over] \arrow[from=ur,to=urrr] \\
B_0 \arrow[rr] \arrow[drrr] & & B_1 \arrow[rr] \arrow[dr] & & B_2 \arrow[rr] \arrow[dl] & & \cdots \arrow[dlll] \\
& & & Y \arrow[from=uu,crossing over] 
\end{tikzcd}
\end{equation*}
If each of the vertical squares is a pullback square, then the square
\begin{equation*}
\begin{tikzcd}
A_\infty \arrow[r] \arrow[d] & X \arrow[d] \\
B_\infty \arrow[r] & Y
\end{tikzcd}
\end{equation*} 
is a pullback square.
\end{cor}

\chapter{Homotopy images}\label{chap:image}

We have observed in \cref{eg:rcoeq} that the reflexive coequalizer of the pre-kernel of a map $f:A\to X$ is the fiberwise join $\join[X]{A}{A}$, i.e.~we have a reflexive coequalizer diagram
\begin{equation*}
\begin{tikzcd}
A\times_X A \arrow[r,yshift=1ex] \arrow[r,yshift=-1ex] & A \arrow[l] \arrow[r] & \join[X]{A}{A}.
\end{tikzcd}
\end{equation*}
It can also be shown that the colimit of the 2-pre-kernel of a map $f:A\to X$ is the triple fiberwise join $\join[X]{A}{\join[X]{A}{A}}$, i.e.~we have a colimiting 
\begin{equation*}
\begin{tikzcd}
A\times_{X} A \times_{X} A \arrow[r,yshift=2ex] \arrow[r] \arrow[r,yshift=-2ex] & A \times_X A \arrow[l,yshift=1ex] \arrow[l,yshift=-1ex] \arrow[r,yshift=1ex] \arrow[r,yshift=-1ex] & A \arrow[l] \arrow[r] & \join[X]{A}{\join[X]{A}{A}},
\end{tikzcd}
\end{equation*}
although one has to take care to use the right amount of coherence data.
These results suggest that geometric realization of the Cech nerve of a map $f:A\to X$, i.e.~the homotopy image of $f$, is the sequential colimit of the type sequence
\begin{equation*}
\begin{tikzcd}
A \arrow[r] & \join[X]{A}{A} \arrow[r] & \join[X]{A}{(\join[X]{A}{A})} \arrow[r] & \cdots
\end{tikzcd}
\end{equation*}
and its sequential colimit. We do not have a way of presenting the Cech nerve of a map in type theory, due to the infinite coherence problem of presenting simplicial types in type theory. Nevertheless, we \emph{can} analyze this type sequence and its colimit in homotopy type theory, since it is constructed entirely in terms of known operations. We will show in \cref{thm:image} that for any map $f:A\to X$, the infinite fiberwise join-power $\join[X]{\cdots}{\join[X]{A}{\join[X]{A}{A}}}$ is the image of $f$. We conclude that the image of a map always exists in univalent type theory with homotopy pushouts. We note that an earlier construction of the propositional truncation in a similar setting is due to van Doorn \cite{vanDoorn2016}, and another one is due to Kraus \cite{Kraus2016}. The present construction of the image of $f$ is called the join construction implies, as we show in \cref{thm:image_small}, that the image of an essentially small type mapping into a locally small type is again essentially small. In particular, the image of a map from a small type into the universe is essentially small. This corollary should be viewed as a type theoretic replacement axiom. This fact has the important consequence that any connected component of the universe is essentially small. We also note that the join construction leads to new constructions of set quotients and of Rezk completions, see \cite{joinconstruction}, where also a construction of the $n$-truncations is given as an application of the join construction.

\section{The universal property of the image}
\begin{defn}\label{defn:image_up}
Consider a commuting triangle
\begin{equation*}
\begin{tikzcd}[column sep=small]
A \arrow[rr,"i"] \arrow[dr,swap,"f"] & & U \arrow[dl,"m"] \\
& X
\end{tikzcd}
\end{equation*}
with $I:f\htpy m\circ i$, and where $m$ is an embedding\index{embedding}.
We say that $m$ has the \define{universal property of the image of $f$}\index{universal property!of the image|textit} if the map
\begin{equation*}
(i,I)^\ast : \mathrm{hom}_X(m,m')\to\mathrm{hom}_X(f,m')
\end{equation*}
defined by $(i,I)^\ast(h,H)\defeq (h\circ i,\ct{I}{(i\cdot H)})$,
is an equivalence for every embedding $m':U'\to X$. 
\end{defn}

\begin{lem}
For any $f:A\to X$ and any embedding\index{embedding} $m:U\to X$, the type $\mathrm{hom}_X(f,m)$ is a proposition.
\end{lem}

\begin{proof}
From \cref{cor:fib_triangle} we obtain that the type $\mathrm{hom}_X(f,m)$ is equivalent to the type
\begin{equation*}
\prd{x:X}\fib{f}{x}\to\fib{m}{x},
\end{equation*}
so it suffices to show that this is a proposition. 
Recall from \cref{thm:prop_emb} that a map is an embedding if and only if its fibers are propositions.
Thus we see that the type $\prd{x:X}\fib{f}{x}\to\fib{m}{x}$ is a product of propositions, so it is a proposition by \cref{thm:prop_pi}.
\end{proof}

\begin{cor}\label{cor:image_up}
Consider a commuting triangle
\begin{equation*}
\begin{tikzcd}[column sep=small]
A \arrow[rr,"i"] \arrow[dr,swap,"f"] & & U \arrow[dl,"m"] \\
& X
\end{tikzcd}
\end{equation*}
with $I:f\htpy m\circ i$, and where $m$ is an embedding. Then $m$ satisfies the universal property of the image of $f$ if and only if the implication
\begin{equation*}
\mathrm{hom}_X(f,m')\to\mathrm{hom}_X(m,m')
\end{equation*}
holds for every embedding $m':U'\to X$. 
\end{cor}

Recall that embeddings into the unit type are just propositions.
Therefore, the universal property of the image of the map $A\to\unit$ is a proposition $P$ satisfying the universal property of the propositional truncation:

\begin{defn}
Let $A$ be a type, and let $P$ be a proposition that comes equipped with a map $f:A\to P$. We say that $f:A\to P$ satisfies the \define{universal property of propositional truncation}\index{universal property!of propositional truncation|textit} if for every proposition $Q$, the precomposition map
\begin{equation*}
\blank\circ f:(P\to Q)\to (A\to Q)
\end{equation*}
is an equivalence.
\end{defn}

\section{The join construction}
\subsection{Step one: constructing the propositional truncation}

\begin{lem}\label{lem:extend_join_prop}
Suppose $f:A\to P$, where $A$ is any type, and $P$ is a proposition.
Then the map
\begin{equation*}
(\join{A}{B}\to P)\to (B\to P)
\end{equation*}
given by $h\mapsto h\circ \inr$ is an equivalence, for any type $B$.
\end{lem}

\begin{proof}
Since both types are propositions by \cref{thm:prop_pi} it suffices to construct a map
\begin{equation*}
(B\to P)\to (\join{A}{B}\to P).
\end{equation*}
Let $g:B\to P$. Then the square
\begin{equation*}
\begin{tikzcd}
A\times B \arrow[r,"\proj 2"] \arrow[d,swap,"\proj 1"] & B \arrow[d,"g"] \\
A \arrow[r,swap,"f"] & P
\end{tikzcd}
\end{equation*}
commutes since $P$ is a proposition. Therefore we obtain a map $\join{A}{B}\to P$ by the universal property of the join.
\end{proof}

The idea of the construction of the propositional truncation is that if we are given a map $f:A\to P$, where $P$ is a proposition, then it extends uniquely along $\inr:A\to \join{A}{A}$ to a map $\join{A}{A}\to P$. This extension again extends uniquely along $\inr:\join{A}{A}\to \join{A}{(\join{A}{A})}$ to a map $\join{A}{(\join{A}{A})}\to P$ and so on, resulting in a diagram of the form
\begin{equation*}
\begin{tikzcd}
A \arrow[dr] \arrow[r,"\inr"] & \join{A}{A} \arrow[d,densely dotted] \arrow[r,"\inr"] & \join{A}{(\join{A}{A})} \arrow[dl,densely dotted] \arrow[r,"\inr"] & \cdots \arrow[dll,densely dotted,bend left=10] \\
& P
\end{tikzcd}
\end{equation*}

\begin{defn}
The \define{join powers} $A^{\ast n}$ of a type $X$ are defined by
\begin{align*}
A^{\ast 0} & \defeq \emptyt \\
A^{\ast 1} & \defeq A \\
A^{\ast (n+1)} & \defeq \join{A}{A^{\ast n}}.
\end{align*}
Furthermore, we define $A^{\ast\infty}$ to be the sequential colimit of the type sequence
\begin{equation*}
\begin{tikzcd}
A^{\ast 0} \arrow[r] & A^{\ast 1} \arrow[r,"\inr"] & A^{\ast 2} \arrow[r,"\inr"] & \cdots.
\end{tikzcd}
\end{equation*}
\end{defn}

Our goal is now to show that $A^{\ast\infty}$ is a proposition and satisfies the universal property of the propositional truncation.

\begin{lem}
Consider a type sequence
\begin{equation*}
\begin{tikzcd}
A_0 \arrow[r,"f_0"] & A_1 \arrow[r,"f_1"] & A_2 \arrow[r,"f_2"] & \cdots
\end{tikzcd}
\end{equation*}
with sequential colimit $A_\infty$, and let $P$ be a proposition. Then the map
\begin{equation*}
\seqin^\ast: (A_\infty\to P)\to \Big(\prd{n:\N}A_n\to P\Big)
\end{equation*}
given by $h\mapsto \lam{n}(h\circ \seqin_n)$ is an equivalence. 
\end{lem}

\begin{proof}
By the universal property of sequential colimits established in \cref{thm:sequential_up} we obtain that $\coconemap$ is an equivalence. Note that we have a commuting triangle
\begin{equation*}
\begin{tikzcd}[column sep=tiny]
\phantom{\Big(\prd{n:\N}A_n\to P\Big).} & P^{A_\infty} \arrow[dl,swap,"\coconemap"] \arrow[dr,"\seqin^\ast"] & \phantom{\cocone(P)} \\
\cocone(P) \arrow[rr,swap,"\proj 1"] & & \Big(\prd{n:\N}A_n\to P\Big)
\end{tikzcd}
\end{equation*}
Note that for any $g:\prd{n:\N}A_n\to P$ the type 
\begin{equation*}
\prd{n:\N} g_n\htpy g_{n+1}\circ f_n
\end{equation*}
is a product of contractible types, since $P$ is a proposition. Therefore it is contractible by \cref{thm:funext_wkfunext}, and it follows that the projection is an equivalence. We conclude by the 3-for-2 property of equivalences that $\seqin^\ast$ is an equivalence.
\end{proof}

\begin{lem}\label{lem:infjp_up}
Let $A$ be a type, and let $P$ be a proposition. Then the function
\begin{equation*}
\blank\circ \seqin_0: (A^{\ast\infty}\to P)\to (A\to P)
\end{equation*}
is an equivalence. 
\end{lem}

\begin{proof}
We have the commuting triangle
\begin{equation*}
\begin{tikzcd}[column sep=0]
& P^{A^{\ast\infty}} \arrow[dl,swap,"\seqin^\ast"] \arrow[dr,"\blank\circ\seqin_0"] & \phantom{\Big(\prd{n:\N}A^{\ast n} \to P\Big)} \\
\Big(\prd{n:\N}A^{\ast n} \to P\Big) \arrow[rr,swap,"\lam{h}h_0"] & & P^A.
\end{tikzcd}
\end{equation*}
Therefore it suffices to show that the bottom map is an equivalence. Since this is a map between propositions, it suffices to construct a map in the converse direction. Let $f:A\to P$. We will construct a term of type
\begin{equation*}
\prd{n:\N}A^{\ast n} \to P
\end{equation*}
by induction on $n:\N$. The base case is trivial. Given a map $g:A^{\ast n}\to P$, we obtain a map $g:A^{\ast(n+1)}\to P$ by \cref{lem:extend_join_prop}.
\end{proof}

\begin{lem}\label{lem:seqcolim_contr}
Consider a type sequence
\begin{equation*}
\begin{tikzcd}
A_0 \arrow[r,"f_0"] & A_1 \arrow[r,"f_1"] & A_2 \arrow[r,"f_2"] & \cdots
\end{tikzcd}
\end{equation*}
and suppose that each $A_n$ is equipped with a base point $a_n:A_n$, and each $f_n$ is equipped with a homotopy $H_n:\mathsf{const}_{a_{n+1}}\htpy f_n$. Then the sequential colimit $A_\infty$ is contractible.\qed
\end{lem}

\begin{lem}\label{lem:isprop_infjp}
The type $A^{\ast\infty}$ is a proposition for any type $A$.
\end{lem}

\begin{proof}
By \cref{lem:prop_char} it suffices to show that $A^{\ast\infty}\to \iscontr(A^{\ast\infty})$, and by \cref{lem:infjp_up} it suffices to show that
\begin{equation*}
A\to \iscontr(A^{\ast\infty}),
\end{equation*}
because $\iscontr(A^{\ast\infty})$ is a proposition. 

Let $x:A$. To see that $A^{\ast\infty}$ is contractible it suffices by \cref{lem:seqcolim_contr} to show that $\inr:A^{\ast n}\to A^{\ast(n+1)}$ is homotopic to the constant function $\const_{\inl(x)}$. However, we get a homotopy $\const_{\inl(x)}\htpy \inr$ immediately from the path constructor $\glue$.  
\end{proof}

\begin{thm}
For any type $A:\UU$ there is a proposition $\brck{A}:\UU$ that comes equipped with a map $\eta:A\to \brck{A}$, and satisfies the universal property of propositional truncation.
\end{thm}

\begin{proof}
Let $A$ be a type. Then we define $\brck{A}\defeq A^{\ast\infty}$, and we define $\eta\defeq \seqin_0:A\to A^{\ast\infty}$. Then $\brck{A}$ is a proposition by \cref{lem:isprop_infjp}, and $\eta:A\to \brck{A}$ satisfies the universal property of propositional truncation by \cref{lem:infjp_up}.
\end{proof}

\subsection{Step two: constructing the image of a map}\label{sec:join_stage2}
Following Definition 7.6.3 of \cite{hottbook}, we recall that the image of a map $f:A\to X$ can be defined using the propositional truncation:
\begin{defn}
For any map $f:A\to X$ we define the \define{image}\index{image|textbf} of $f$ to be the type
\begin{equation*}
\im(f) \defeq \sm{x:X}\brck{\fib{f}{x}}
\end{equation*}
and we define the \define{image inclusion} to be the projection $\proj 1 :\im(f)\to X$. 
\end{defn}
However, the construction of the fiberwise join in \cref{defn:fib_join} suggests that we can also define the image of $f$ as the infinite join power $f^{\ast\infty}$, where we repeatedly take the fiberwise join of $f$ with itself. Our reason for defining the image in this way is twofold: 
\begin{itemize}
\item We use this construction to show that the image of a map $f:A\to B$ from an essentially small type $A$ into a locally small type $B$ is again essentially small.
\item Some interesting types, such as the real and complex projective spaces, appear in specific instances of this construction.
\end{itemize}

\begin{lem}
Consider a map $f:A\to X$, an embedding $m:U\to X$, and $h:\mathrm{hom}_X(f,m)$. Then the map
\begin{equation*}
\mathrm{hom}_X(\join{f}{g},m)\to \mathrm{hom}_X(g,m)
\end{equation*}
is an equivalence for any $g:B\to X$.
\end{lem}

\begin{proof}
Note that both types are propositions, so any equivalence can be used to prove the claim. Thus, we simply calculate
\begin{align*}
\mathrm{hom}_X(\join{f}{g},m) & \eqvsym \prd{x:X}\fib{\join{f}{g}}{x}\to \fib{m}{x} \\
& \eqvsym \prd{x:X}\join{\fib{f}{x}}{\fib{g}{x}}\to\fib{m}{x} \\
& \eqvsym \prd{x:X}\fib{g}{x}\to\fib{m}{x} \\
& \eqvsym \mathrm{hom}_X(g,m).
\end{align*}
The first equivalence holds by \cref{cor:fib_triangle}; the second equivalence holds by \cref{defn:join-fiber}; the third equivalence holds by \cref{lem:extend_join_prop}; the last equivalence again holds by \cref{cor:fib_triangle}.
\end{proof}

For the construction of the image of $f:A\to X$ we observe that if we are given an embedding $m:U\to X$ and a map $(i,I):\mathrm{hom}_X(f,m)$, then $(i,I)$ extends uniquely along $\inr:A\to \join[X]{A}{A}$ to a map $\mathrm{hom}_X(\join{f}{f},m)$. This extension again extends uniquely along $\inr:\join[X]{A}{A}\to \join[X]{A}{(\join[X]{A}{A})}$ to a map $\mathrm{hom}_X(\join{f}{(\join{f}{f})},m)$ and so on, resulting in a diagram of the form
\begin{equation*}
\begin{tikzcd}
A \arrow[dr] \arrow[r,"\inr"] & \join[X]{A}{A} \arrow[d,densely dotted] \arrow[r,"\inr"] & \join[X]{A}{(\join[X]{A}{A})} \arrow[dl,densely dotted] \arrow[r,"\inr"] & \cdots \arrow[dll,densely dotted,bend left=10] \\
& U
\end{tikzcd}
\end{equation*}

\begin{defn}
Suppose $f:A\to X$ is a map. Then we define the \define{fiberwise join powers} 
\begin{equation*}
f^{\ast n}\defeq A_X^{\ast n}\to X.
\end{equation*}
\end{defn}

\begin{proof}[Construction]
Note that the operation $(B,g)\mapsto (\join[X]{A}{B},\join{f}{g})$ defines an endomorphism on the type
\begin{equation*}
\sm{B:\UU}B\to X.
\end{equation*}
We also have $(\emptyt,\ind{\emptyt})$ and $(A,f)$ of this type. For $n\geq 1$ we define
\begin{align*}
A_X^{\ast (n+1)} & \defeq \join[X]{A}{A_X^{\ast n}} \\
f^{\ast (n+1)} & \defeq \join{f}{f^{\ast n}}.\qedhere
\end{align*}
\end{proof}

\begin{defn}
We define $A_X^{\ast\infty}$ to be the sequential colimit of the type sequence
\begin{equation*}
\begin{tikzcd}
A_X^{\ast 0} \arrow[r] & A_X^{\ast 1} \arrow[r,"\inr"] & A_X^{\ast 2} \arrow[r,"\inr"] & \cdots.
\end{tikzcd}
\end{equation*}
Since we have a cocone
\begin{equation*}
\begin{tikzcd}
A_X^{\ast 0} \arrow[r] \arrow[dr,swap,"f^{\ast 0}" near start] & A_X^{\ast 1} \arrow[r,"\inr"] \arrow[d,swap,"f^{\ast 1}" near start] & A_X^{\ast 2} \arrow[r,"\inr"] \arrow[dl,swap,"f^{\ast 2}" xshift=1ex] & \cdots \arrow[dll,bend left=10] \\
& X
\end{tikzcd}
\end{equation*}
we also obtain a map $f^{\ast\infty}:A_X^{\ast\infty}\to X$ by the universal property of $A_X^{\ast\infty}$. 
\end{defn}

\begin{lem}\label{lem:finfjp_up}
Let $f:A\to X$ be a map, and let $m:U\to X$ be an embedding. Then the function
\begin{equation*}
\blank\circ \seqin_0: \mathrm{hom}_X(f^{\ast\infty},m)\to \mathrm{hom}_X(f,m)
\end{equation*}
is an equivalence. 
\end{lem}

\begin{thm}\label{thm:image}
For any map $f:A\to X$, the map $f^{\ast\infty}:A_X^{\ast\infty}\to X$ is an embedding that satisfies the universal property of the image inclusion of $f$.
\end{thm}

\subsection{Step three: establishing the smallness of the image}

Recall from \cref{defn:ess_small} that a type is said to be locally small if its identity types are equivalent to small types.

\begin{lem}
Consider a commuting square
\begin{equation*}
\begin{tikzcd}
A \arrow[r] \arrow[d] & B \arrow[d] \\
C \arrow[r] & D.
\end{tikzcd}
\end{equation*}
\begin{enumerate}
\item If the square is cartesian, $B$ and $C$ are essentially small, and $D$ is locally small, then $A$ is essentially small.
\item If the square is cocartesian, and $A$, $B$, and $C$ are essentially small, then $D$ is essentially small. 
\end{enumerate}
\end{lem}

\begin{cor}
Suppose $f:A\to X$ and $g:B\to X$ are maps from essentially small types $A$ and $B$, respectively, to a locally small type $X$. Then $A\times_X B$ is again essentially small. 
\end{cor}

\begin{lem}
Consider a type sequence
\begin{equation*}
\begin{tikzcd}
A_0 \arrow[r,"f_0"] & A_1 \arrow[r,"f_1"] & A_2 \arrow[r,"f_2"] & \cdots
\end{tikzcd}
\end{equation*}
where each $A_n$ is essentially small. Then its sequential colimit is again essentially small. 
\end{lem}

\begin{thm}\label{thm:image_small}
For any map $f:A\to X$ from a small type $A$ into a locally small type $X$, the image $\im(f)$ is an essentially small type.
\end{thm}

Recall that in set theory, the replacement axiom asserts that for any family of sets $\{X_i\}_{i\in I}$ indexed by a set $I$, there is a set $X[I]$ consisting of precisely those sets $x$ for which there exists an $i\in I$ such that $x\in X_i$. In other words: the image of a set-indexed family of sets is again a set. Without the replacement axiom, $X[I]$ would be a class. In the following corollary we establish a type-theoretic analogue of the replacement axiom: the image of a family of small types indexed by a small type is again (essentially) small.

\begin{cor}\label{cor:im_small}
For any small type family $B:A\to\UU$, where $A$ is small, the image $\im(B)$ is essentially small. We call $\im(B)$ the \define{univalent completion} of $B$. 
\end{cor}

\chapter{Reflective subuniverses}\label{chap:reflective}

In this chapter we study reflective subuniverses. Reflective subuniverses were first introduced in section 7.7 of \cite{hottbook}, and were studied in much more detail in \cite{RijkeShulmanSpitters}. 

In \cref{sec:prop-rfsu} we establish the basic closure properties of reflective subuniverses. In particular, we show in \cref{prp:local_pb} that pullbacks of $L$-local types are again $L$-local. It follows that cartesian products of $L$-local types and identity types of $L$-local types are again $L$-local. In \cref{lem:modal-Pi} we show that any dependent product of $L$-local types is also $L$-local, regardless of whether the indexing type is $L$-local or not. We use this fact in \cref{thm:modal-pres-prod} to show that the map $\modalunit\times\modalunit:X\times Y\to LX\times LY$ is an $L$-localization, and that the $L$-localization of a proposition is again a proposition.

In \cref{sec:accessible} we study accessible reflective subuniverses.

In \cref{sec:separated} we introduce the notion of $L$-separated type, for any reflective subuniverse $L$, and we show that the subuniverse of $L$-separated types is again a reflective subuniverse. The contents of \cref{sec:separated} are joint work with Dan Christensen, Morgan Opie, and Luis Scoccola. As a corollary we obtain in \cref{thm:truncation} that the $k$-truncation can be constructed, for any $k\geq -2$. A more elementary way of obtaining this result appears in \cite{joinconstruction}.

\section{Localizations}
\label{sec:prop-rfsu}

A \define{subuniverse} is simply a subtype of the universe. Note that we do not require that subuniverses are closed under any type constructors\footnote{In particular, we will see that reflective subuniverses aren't necessarily closed under $\Sigma$.}
For any subuniverse $P:\UU\to\prop$, we write $\UU_P\defeq \sm{X:\UU}P(X)$, and we say that $X:\UU$ is a $P$-type if $X$ is in $\UU_P$, i.e.~if $P(X)$ holds.

\begin{defn}
Let $P:\UU\to\prop$ be a subuniverse, and let $X$ be a type. A \define{$P$-localization} of $X$ is a triple $(Y,l,p)$ consisting of a $P$-type $Y:\UU_P$, a map $l:X\to Y$, and a term $p$ witnessing that the map
\begin{equation*}
l^\ast : (Y\to Z)\to (X\to Z)
\end{equation*}
is an equivalence, for every $Z:\UU_P$. This property is also called the \define{universal property} of the $P$-localization of $X$.
\end{defn}

In other words, a $P$-localization of $X$ is a map $l:X\to Y$ into a $P$-type $Y$, such that every map $f:X\to Z$ into a $P$-type $Z$ extends uniquely along $l$, as indicated in the diagram
\begin{equation*}
\begin{tikzcd}
X \arrow[r,"f"] \arrow[d,swap,"l"] & Z. \\
Y \arrow[ur,densely dotted]
\end{tikzcd}
\end{equation*}

\begin{prp}\label{lem:reflective_uniqueness}
For any subuniverse $P:\UU\to \prop$ and any type $X$, the type of $P$-localizations of $X$ is a proposition.
\end{prp}

\begin{proof}
Consider $(Y,f,p)$ and $(Y',f',p')$ of the described type. Since $p$ and $p'$
are terms of a proposition, it suffices to show that $(Y,f)=(Y',f')$. In
other words, we have to find an equivalence $g:Y\to Y'$ such that $g\circ f'=f$.

By $p(Y')$, the type of
pairs $(g,h)$ consisting of a function $g:Y\to Y'$ such that $h:g\circ f=f'$ is contractible. By
$p'(Y)$, the type of pairs $(g',h')$ consisting of a function $g':Y'\to Y$
such that $h':g'\circ f'=f$ is contractible.

Now $g'\circ g$ is a function such that $g'\circ g\circ f=g'\circ f'=f$, as
is $\idfunc[Y]$. By contractibility, it follows that $g'\circ g=\idfunc[Y]$.
Similarly, $g\circ g'=\idfunc[Y']$.
\end{proof}

\begin{prp}\label{thm:subuniv-modal}
Consider a subuniverse $P:\UU\to\prop$ and a $P$-localization $l:X\to Y$. The following are equivalent:
\begin{enumerate}
\item $X$ is a $P$-type (i.e.~$X$ is in $\UU_P$).
\item The $P$-localization $l:X\to Y$ is an equivalence.
\item The map $\precomp{l}:(Y\to X)\to (X\to X)$ is an equivalence.
\item The $P$-localization $l:X\to Y$ has a retraction.
\end{enumerate}
\end{prp}

\begin{proof}
  Certainly if $l$ is an equivalence, then $X$ is in $P$ since it is equivalent to the type $Y:\UU_P$.
  Conversely, if $X$ is in $P$ then $\idfunc[X]$ has the same universal property of $l$; so by \cref{lem:reflective_uniqueness} they are equivalent and hence $l$ is an equivalence. This shows that (i) holds if and only if (ii) holds.

  It is clear that (ii) implies (iii). Furthermore, (iii) implies (iv) because the fiber of $\precomp{l}$ at $\idfunc[X]$ is contractible. In particular, there is a function $g:Y\to X$ equipped with a homotopy $g\circ l\htpy\idfunc$. In other words,$l$ has a retraction. To see that (iv) implies (ii), suppose $g$ is a retraction of $l$, i.e.\ $g\circ l = \idfunc[X]$.
  Then $l\circ g\circ l = l$, so $l\circ g$ is a factorization of $l$ through itself.
  By uniqueness of such factorizations, $l\circ g = \idfunc[Y]$.
  Thus $g$ is also a section of $l$, hence $l$ is an equivalence.
\end{proof}

\begin{cor}\label{cor:unit_local}
For any subuniverse $P$, the unit type $\unit$ has a $P$-localization if and only if $\unit$ is already a $P$-type.
\end{cor}

\begin{proof}
Immediate from the fact that $\unit$ is a retract of any pointed type.
\end{proof}

The universal property of $P$-localization is by itself not sufficient to imply a dependent universal property. However, we have the following `constrained' dependent universal property.

\begin{prp}\label{theorem:generalized-induction}
Let $P$ be a subuniverse, and let $l:X\to Y$ be a $P$-localization. Furthermore, consider a type family $Z:Y\to\UU$ such that the total space $\sm{y:Y}Z(y)$ is a $P$-type.
Then the precomposition map
\[
  \precomp{l} : \Big(\prd{y:Y}Z(y)\Big) \lra \Big(\prd{x:X}Z(l(x))\Big)
\]
is an equivalence.
\end{prp}

\begin{proof}
Since $Y$ is a $P$-type and $\sm{y:Y}Z(y)$ is a $P$-type, the precomposition maps $\precomp{l}$ in the commuting square
\[
  \begin{tikzcd}
    (Y\to \sm{y:Y}Z(y)) \arrow[r,"\precomp{l}"] \arrow[d,swap,"\proj 1\circ\blank"] & (X\to \sm{y:Y}Z(y)) \arrow[d,"\proj 1\circ\blank"] \\
    (Y\to Y) \arrow[r,swap,"\precomp{l}"] & (X\to Y)
  \end{tikzcd}
\]
are equivalences. It follows that they induce an equivalence from the fiber of the left-hand map $\proj 1\circ\blank$ at $\idfunc[Y]$ to the fiber of the right-hand map $\proj 1\circ\blank$ at $l$. In other words, we have an equivalence
\[
  \precomp{l} : \Big(\prd{y:Y}Z(y)\Big) \lra \Big(\prd{x:X}Z(l(y))\Big).\qedhere
\]
\end{proof}

\begin{prp}
Let $P:\UU\to\prop$ be a subuniverse, and write $\tilde{\UU}_P \defeq\sm{X:\UU_P}X$. The projection
$\proj1:\tilde{\UU}_P \to\UU_P$ classifies the small maps whose fibers satisfy $P$.
\end{prp}

\begin{proof}
Let $f:Y\to X$ be any map into $X$. Then $\fibf{f}:X\to\UU$ factors through
$\UU_P$ if and only if all the fibers of $f$ satisfy $P$. Let us write
$P(f)$ for $\prd{x:X}P(\fib{f}{x})$. Then we see that the equivalence
$\chi$ of Theorem 4.8.3 of \cite{TheBook} restricts to an
equivalence
\begin{equation*}
\chi^P:(\sm{Y:\UU}{f:Y\to X}P(f))\to(X\to\UU_P).
\end{equation*}
Now observe that the outer square and the square on the right in the diagram
\begin{equation*}
\begin{tikzcd}[column sep=6em]
Y \arrow[d,swap,"f"] \arrow[rr,"{\lam{y}\pairr{\fib{f}{f(y)},\blank,\pairr{y,\refl{f(y)}}}}"] & & \tilde{\UU}_P \arrow[r] \arrow[d] & \tilde{\UU} \arrow[d] \\
X \arrow[rr,swap,"{\fibf{f}}"] & & \UU_P \arrow[r] & \UU
\end{tikzcd}
\end{equation*}
are pullback squares. Hence the square on the left is a pullback square.
\end{proof}

\begin{defn}
A \define{reflective subuniverse} $L$ is a subuniverse $\UU_L\to \UU$ equipped with an $L$-localization 
\begin{equation*}
\modalunit[X]:X\to LX
\end{equation*}
for every type $X:\UU$. The $L$-localization $\modalunit[X]:X\to LX$ is sometimes also called the \define{unit} of the localization. A type in $\UU_L$ is also said to be \define{local}.
\end{defn}

\begin{thm}\label{thm:subuniverse-rs}
The data of any two reflective subuniverses with the same local types are the same.
\end{thm}

\begin{proof}
Immediate from the fact that the type of localizations is a proposition.
\end{proof}

\begin{lem}
  Any reflective subuniverse is a functor up to homotopy: given $f:A\to B$ we have an induced map $L f : L A \to L B$, preserving identities and composition up to homotopy.
  Moreover, $\modalunit$ is a natural transformation up to homotopy, i.e.\ for any $f$ we have $L f \circ \modalunit[A] = \modalunit[B] \circ f$.
\end{lem}

\begin{proof}
  Define $L f$ to be the unique function such that $L f \circ \modalunit[A] = \modalunit[B] \circ f$, using the universal property of $\modalunit[A]$.
  The rest is easy to check using further universal properties.
\end{proof}

\begin{cor}\label{cor:local_retract}
The subuniverse of $L$-local types is closed under retracts.
\end{cor}

\begin{proof}
Consider an $L$-local type $Y$, and suppose that $i:X\to Y$ has a retraction $r:Y\to X$. By the functoriality of $L$ it follows that $LX$ is also a retract of $LY$, since we have $Lr\circ Li\htpy \idfunc$. Therefore we see that the $L$-localization $\modalunit :X\to LX$ is a retract of the $L$-localization $\modalunit : Y\to LY$, as indicated in the diagram
\begin{equation*}
\begin{tikzcd}
X \arrow[r,"i"] \arrow[d,swap,"\modalunit"] & Y \arrow[d,swap,"\modalunit"] \arrow[r,"r"] & X \arrow[d,"\modalunit"] \\
LX \arrow[r,swap,"Li"] & LY \arrow[r,swap,"Lr"] & LX.
\end{tikzcd}
\end{equation*}
Since $\modalunit : Y\to LY$ is an equivalence, and equivalences are closed under retracts, it follows that $\modalunit :X\to LX$ is an equivalence, hence $X$ is $L$-local.
\end{proof}

\begin{defn}
A map $f:A\to B$ is said to be an \define{$L$-equivalence} if $Lf:LA\to LB$ is an equivalence.
\end{defn}

\begin{prp}\label{lem:local_equivalence}
For a map $f : A \to B$ the following are equivalent:
\begin{enumerate}
\item $f$ is an $L$-equivalence.
\item For any $L$-local type $X$, the precomposition map
\begin{equation*}
\precomp{f} : (B \to X) \to (A \to X)
\end{equation*}
is an equivalence.
\end{enumerate}
\end{prp}

\begin{proof} 
Suppose first that $f$ is an $L$-equivalence, and let $X$ be $L$-local. Then the square
\begin{equation*}
\begin{tikzcd}
X^{LB} \arrow[r,"\precomp{Lf}"] \arrow[d,swap,"\precomp{\eta}"] & X^{LA} \arrow[d,"\precomp{\eta}"] \\
X^{B} \arrow[r,swap,"\precomp{f}"] & X^{A}
\end{tikzcd}
\end{equation*}
commutes. In this square the two vertical maps are equivalences by the universal property of localization, and the top map is an equivalence since $L f$ is an equivalence. Therefore the map $\precomp{f}:X^B\to X^A$ is an equivalence, as desired.

Conversely, assume that $\precomp{f} : X^B \to X^A$ is an equivalence for every $L$-local type $X$. By the square above it follows that $\precomp{L f}:X^{L B}\to X^{L A}$ is an equivalence for every $L$-local type $X$. The fiber of $L A^{L B}\to L A^{L A}$ at $\idfunc:L A\to L A$ is contractible, so we obtain a retraction $g$ of $L f$. To see that $g$ is also a section observe that the fiber of $L B^{L B}\to L B^{L A}$ at $L f$ is contractible. This fiber contains $(\idfunc[L B],\refl{L f})$. However, we also have an identification $p:\precomp{L f}(L f\circ g)=L f$, since
\begin{equation*}
\precomp{L f}(L f\circ g)\jdeq (L f \circ g)\circ L f\jdeq L f \circ (g\circ L f) = L f. 
\end{equation*}
Therefore $(L f\circ g,p)$ is in the fiber of $\precomp{L f}:L B^{L B}\to L B^{L A}$ at $L f$. By the contractibility of the fibers it follows that $(L f\circ g,p)=(\idfunc[L B],\refl{L f})$, so it follows that $L f\circ g=\idfunc[L B]$. In other words, $g$ is both a retraction and a section of $L f$, so $L f$ is an equivalence.
\end{proof}

\begin{cor}\label{cor:localization_lequiv}
Let $f:X\to Y$ be a map into an $L$-local type $Y$. Then the following are equivalent:
\begin{enumerate}
\item $f$ is an $L$-localization.
\item $f$ is an $L$-equivalence.
\end{enumerate}
In particular, the map $\modalunit[X]:X\to LX$ is an $L$-equivalence.
\end{cor}

\begin{cor}\label{cor:localization_retract}
Any retract of an $L$-localization is again an $L$-localization.
\end{cor}

\begin{proof}
Suppose that $f:A\to B$ is a retract of an $L$-localization $l:X\to Y$. Then $B$ is $L$-local by the previous claim.
Moreover, $Lf$ is a retract of $Ll$, which is an equivalence, so $f$ is an $L$-equivalence. Therefore $f$ is an $L$-localization by \cref{cor:localization_lequiv}.
\end{proof}

\begin{prp}\label{thm:rsu-galois}
  Given a reflective subuniverse $L$, a type $X$ is $L$-local if and only if $(\blank \circ f) : (B\to X) \to (A\to X)$ is an equivalence, for any $L$-equivalence $f:A\to B$.
\end{prp}

\begin{proof}
  If $L f$ is an equivalence and $X$ is $L$-local, then by the universal property of $\modalunit$, we have a commutative square
  \[
  \begin{tikzcd}[column sep=large]
    X^{LB} \ar[r,"\blank\circ Lf"] & X^{LA} \\
    X^{B} \ar[r,"\blank\circ f"'] \ar[from=u,"{\blank\circ \modalunit[B]}"'] &
    X^{A} \ar[from=u,"{\blank\circ \modalunit[A]}"]
  \end{tikzcd}
  \]
  in which all but the top map are equivalences; thus so is the top map.

  Conversely, since $L\modalunit[X]$ is an equivalence, the hypothesis implies that
  $(\blank \circ \modalunit[X]) : (L X\to X) \to (X\to X)$
  is an equivalence.
  In particular, its fiber over $\idfunc[X]$ is inhabited, i.e.\ $\modalunit[X]$ has a retraction; hence $X$ is $L$-local.
\end{proof}

\begin{prp}\label{lem:Lequiv_total}
For any family
\begin{equation*}
f:\prd{x:X}P(x)\to Q(x)
\end{equation*}
of $L$-equivalences, the induced map on total spaces
\begin{equation*}
\total{f}:\Big(\sm{x:X}P(x)\Big)\to \Big(\sm{x:X} Q(x)\Big)
\end{equation*}
is an $L$-equivalence.
\end{prp}

\begin{proof}
Note that we have a commuting square
\begin{equation*}
\begin{tikzcd}[column sep=7em]
Y^{\sm{x:X}Q(x)} \arrow[d,swap,"\mathsf{ev\usc{}pair}"] \arrow[r,"\precomp{\total{f}}"] & Y^{\sm{x:X}P(x)} \arrow[d,"\mathsf{ev\usc{}pair}"] \\
\prd{x:X}Y^{Q(x)} \arrow[r,swap,"\lam{g}{x}g(x)\circ f(x)"] & \prd{x:X}Y^{P(x)}
\end{tikzcd}
\end{equation*}
in which all but the top map are known to be equivalences. Therefore the top map is an equivalence, so the claim follows by \cref{lem:local_equivalence}.
\end{proof}

\begin{cor}[Theorem 1.24 of \cite{RijkeShulmanSpitters}]\label{lem:sum_idempotent}
For any family $P:X\to\UU$, the map
\begin{equation*}
\lam{(x,y)}\eta(x,\eta(y)) : \Big(\sm{x:X}P(x)\Big)\to L\Big(\sm{x:X}LP(x)\Big)
\end{equation*}
is a localization.
\end{cor}

\begin{lem}\label{lem:localization_uphtpy}
For any two maps $f,g:LX\to Y$ into an $L$-local type $Y$, the map
\begin{equation*}
(f\htpy g) \to (f\circ \modalunit\htpy g\circ\modalunit)
\end{equation*}
given by $H\mapsto H\cdot\modalunit$, is an equivalence. 
\end{lem}

\begin{proof}
The square
\begin{equation*}
\begin{tikzcd}[column sep=huge]
(f=g) \arrow[r,"\apfunc{\precomp{\modalunit}}"] \arrow[d,swap,"\mathsf{htpy\usc{}eq}"] & (f\circ \modalunit=g\circ\modalunit) \arrow[d,"\mathsf{htpy\usc{}eq}"] \\
(f\htpy g) \arrow[r,swap,"\lam{H}H\cdot\modalunit"] & (f\circ\modalunit\htpy g\circ\modalunit)
\end{tikzcd}
\end{equation*}
commutes, and all but one of the maps are known equivalences. Therefore it follows that the bottom map is an equivalence, as claimed.
\end{proof}

\begin{prp}\label{prp:local_pb}
Consider a pullback square
\begin{equation*}
\begin{tikzcd}
C \arrow[d,swap,"p"] \arrow[r,"q"] & B \arrow[d,"g"] \\
A \arrow[r,swap,"f"] & X
\end{tikzcd}
\end{equation*}
with $H:f\circ p\htpy g\circ q$. If $A$, $B$, and $X$ are $L$-local types, then so is $C$. 
\end{prp}

\begin{proof}
We have the commuting square
\begin{equation*}
\begin{tikzcd}
(LC\to C) \arrow[d,swap,"\mathsf{cone\usc{}map}"] \arrow[r,"\precomp{\modalunit}"] & (C\to C) \arrow[d,"\mathsf{cone\usc{}map}"] \\
\mathsf{cone}(LC) \arrow[r,densely dotted] & \mathsf{cone}(C)
\end{tikzcd}
\end{equation*}
where the bottom map is the equivalence given by $(\tilde{p},\tilde{q},\tilde{H})\mapsto (\tilde{p}\circ\eta,\tilde{q}\circ\eta,\tilde{H}\cdot \eta)$. This is an equivalence by the assumption that $A$, $B$, and $X$ are local, and an application of \cref{lem:localization_uphtpy}. The two vertical maps are equivalences by the assumption that $C$ is a pullback. Therefore it follows that the top map is an equivalence. By \cref{thm:subuniv-modal}(iii) this suffices to show that $C$ is $L$-local.
\end{proof}

\begin{cor}\label{cor:local_prod}
Cartesian products of $L$-local types are $L$-local.
\end{cor}

\begin{proof}
Suppose that $X$ and $Y$ are $L$-local. Then their cartesian product is a pullback
\begin{equation*}
\begin{tikzcd}
X\times Y \arrow[d,swap,"\proj 1"] \arrow[r,"\proj 2"] & Y \arrow[d] \\
X \arrow[r] & \unit,
\end{tikzcd}
\end{equation*}
Since the unit type is $L$-local for any reflective subuniverse by \cref{cor:unit_local}, the claim follows.
\end{proof}

\begin{cor}\label{lem:rs_idstable}
If $X$ is $L$-local, then so is the identity type $x=y$ for any $x,y:X$.
\end{cor}

\begin{proof}
This follows at once from the pullback square
\begin{equation*}
\begin{tikzcd}
(x=y) \arrow[d] \arrow[r] & \unit \arrow[d,"\mathsf{const}_y"] \\
\unit \arrow[r,swap,"\mathsf{const}_x"] & X,
\end{tikzcd}
\end{equation*}
noting that the unit type is $L$-local for any reflective subuniverse by \cref{cor:unit_local}, and $X$ is $L$-local by assumption.
\end{proof}

\begin{prp}\label{lem:modal-Pi}
Given a reflective subuniverse,
if $P(x)$ is $L$-local for every $x:X$, then so is $\prd{x:X}P(x)$. In particular, the type $Y^X$ is $L$-local whenever $Y$ is $L$-local.
\end{prp}

\begin{proof}
Consider the commuting square
\begin{equation*}
\begin{tikzcd}
\Big(L\Big(\prd{y:X}P(y)\Big)\to \prd{x:X}P(x)\Big) \arrow[r,"\precomp{\modalunit}"] \arrow[d,swap,"\mathsf{swap}"] & \Big(\Big(\prd{y:X}P(y)\Big)\to \prd{x:X}P(x)\Big) \arrow[d,"\mathsf{swap}"] \\
\Big(\prd{x:X} L\Big(\prd{y:X}P(y)\Big)\to P(x)\Big) \arrow[r,"\eqvsym"] & \Big(\prd{x:X}\Big(\prd{y:X}P(y)\Big)\to P(x)\Big)
\end{tikzcd}
\end{equation*}
The vertical maps swap the order of the arguments, and are therefore equivalences. The bottom map is an equivalence by the assumption that each $P(x)$ is $L$-local. By \cref{thm:subuniv-modal}(iii) this suffices to show that $\prd{x:X}P(x)$ is $L$-local.
\end{proof}

\begin{cor}\label{cor:local_equiv}
For any two $L$-local types $X$ and $Y$, the type of equivalences $\eqv{X}{Y}$ is again $L$-local.
\end{cor}

\begin{proof}
The type $\eqv{X}{Y}$ is equivalent to the pullback
  \[
    \begin{tikzcd}[column sep=8em]
      (\eqv{X}{Y}) \arrow[r] \arrow[d] & \unit \arrow[d,"\mathsf{const}_{(\idfunc[X],\idfunc[Y])}"] \\
      Y^X\times X^Y\times X^Y
         \arrow[r,swap,"{(f,g,h) \mapsto (hf,fg)}"] & X^X\times Y^Y.
    \end{tikzcd}
  \]
  of $L$-local types, so it is $L$-local.
\end{proof}

\begin{rmk}
Similarly it follows that $\mathsf{is\usc{}trunc}_k(X)$ is $L$-local for any $L$-local type $X$, and $\mathsf{is\usc{}trunc}_k(f)$ is $L$-local for any map $f:X\to Y$ between $L$-local types.
\end{rmk}

\begin{prp}\label{thm:modal-pres-prod}
For any two types $X$ and $Y$, the map
\begin{equation*}
\modalunit\times\modalunit : X\times Y \to LX \times LY
\end{equation*}
is an $L$-localization.
Thus $L$-localization preserves finite products, for any reflective subuniverse $L$.
\end{prp}

\begin{proof}
First we note that the product $LX\times LY$ is indeed $L$-local by \cref{cor:local_prod}. To see that $\modalunit\times\modalunit$ is an $L$-localization, consider an $L$-local type $Z$. Then we have the commuting square
\begin{equation*}
\begin{tikzcd}[column sep=9em]
(LX\times LY\to Z) \arrow[r,"\precomp{(\modalunit\times\modalunit)}"] \arrow[d,swap,"\mathsf{ev\usc{}pair}"] & (X\times Y\to Z) \arrow[d,"\mathsf{ev\usc{}pair}"] \\
(LX\to (LY\to Z)) \arrow[r,swap,"{\lam{f}{x}{y}f(\modalunit(x),\modalunit(y))}"] & (X\to (Y\to Z)).
\end{tikzcd}
\end{equation*}
The bottom map is an equivalence by the fact that $Z$ and $LY\to Z$ are both $L$-local types, and the vertical maps are equivalences too. Therefore $\modalunit\times\modalunit$ is an $L$-localization.
\end{proof}

\begin{cor}\label{lem:modal-pres-prop}
Given any reflective subuniverse, the modal operator preserves propositions.
\end{cor}
\begin{proof}
  A type $P$ is a proposition if and only if the diagonal $P\to P\times P$ is an equivalence.
  The result then follows from \cref{thm:modal-pres-prod}.
\end{proof}

By contrast, localizations, and even modalities, do not generally preserve $n$-types for any $n\ge 0$.
For instance, the ``shape'' modality of~\cite{shulman2015brouwer} takes the topological circle, which is a 0-type, to the homotopical circle, which is a 1-type, and the topological 2-sphere, which is also a 0-type, to the homotopical 2-sphere, which is (conjecturally) not an $n$-type for any finite $n$.

\section{The reflective subuniverse of separated types}\label{sec:separated}

\begin{defn}
Consider a subuniverse $P:\UU\to\prop$. We say that a type $X$ is \define{$P$-separated} if the identity types of $X$ are $P$-types. We write $P'$ for the subuniverse of $P$-separated types.
\end{defn}

\begin{eg}
We define $\istrunc{} : \Z_{\geq-2}\to\UU\to\UU$ by induction on $k:\Z_{\geq -2}$, taking
\begin{align*}
\istrunc{-2}(A) & \defeq \iscontr(A) \\
\istrunc{k+1}(A) & \defeq \prd{x,y:A}\istrunc{k}(\id{x}{y}).\qedhere
\end{align*}
For any type $A$, we say that $A$ is \define{$k$-truncated}, or a \define{$k$-type}, if there is a term of type $\istrunc{k}(A)$. We say that a map $f:A\to B$ is $k$-truncated if its fibers are $k$-truncated.

In other words, the subuniverse of $(n+1)$-truncated types is precisely the subuniverse of \define{$n$-separated types}, i.e. the subuniverse of types whose identity types are $n$-truncated.
\end{eg}

\begin{defn}
Let $L$ be a reflective subuniverse and let $X : \UU$ be a type. 
An \define{$L'$-localization} of a type $X$ is a localization with respect to the subuniverse of $L$-separated types.
\end{defn}

In other words, a type $X$ is $L$-separated if its diagonal $\Delta:X\to X\times X$ is classified by $\UU_L$.

\begin{eg}\label{example:truncationisseparated}
Given $n \geq -2$, the subuniverse of $(n+1)$-truncated types is precisely the subuniverse of separated types for the reflective subuniverse of $n$-truncated types.
\end{eg}

\begin{lem}
Any $L$-local type is $L$-separated.
\end{lem}

\begin{proof}
Immediate by \cref{lem:rs_idstable}.
\end{proof}

\begin{lem}
Any small subtype of an $L$-separated type is again $L$-separated. In particular, any small proposition is $L$-separated.
\end{lem}

The following lemma can be proven directly. However, it also follows once we have shown that the subuniverse of $L$-separated types is a reflective subuniverse, so we will omit the proof.

\begin{lem}
The subuniverse of $L$-separated types is closed under pullbacks, retracts, and dependent products of families of $L$-separated types.
\end{lem}

\begin{rmk}[Move to modalities]
If $L$ is closed under dependent sums, then $L'$ is also closed under dependent sums,
    by the characterization of identity types of dependent sums~\cite[Theorem~2.7.2]{hottbook}.
    So, given that separated types form a reflective subuniverse, it will follow that if $L$ is a modality,
    then so is $L'$.
\end{rmk}

\begin{prp}\label{lemma:separatedpluslocalisseparated}
If $X$ is an $L$-separated type and $P:X\to \UU_L$ is a family of $L$-local types, then the type
$\sm{x:X}P(x)$ is $L$-separated.
\end{prp}

\begin{proof}
For any $(x,p)$ and $(y,q)$ in $\sm{x:X}P(x)$, the type $(x,p)=(y,q)$ is equivalent to the pullback
\[
  \begin{tikzcd}[column sep=huge]
    (x,p)=(y,q) \arrow[r] \arrow[d] & \unit \arrow[d,"q"] \\
    (x=y) \arrow[r,swap,"\transfib{P}{-}{p}"] & P(y)
  \end{tikzcd}
\]
of $L$-local types, so it is $L$-local. 
\end{proof}

\begin{cor}\label{proposition:inductionLseparated}
Suppose $l':X\to Y'$ is an $L'$-localization, and let $P:L'X \to \UU_L$ be a family of $L$-local types.
Then the precomposition map
\[
    \precomp{l'}: \prd{y':Y'} P(y') \simeq \prd{x : X} P(l'(x)).
\]
is an equivalence.
\end{cor}

\begin{proof}
This follows immediately from \cref{theorem:generalized-induction}.
\end{proof}

\begin{prp}\label{prop:UU_L-is-L-separated}
Any small subtype of the subuniverse $\UU_L$ is $L$-separated.
\end{prp}

\begin{proof}
    Note that for any two $L$-local types $A$ and $B$ we have $\eqv{(A=B)}{(\eqv{A}{B})}$ by univalence and the
    fact that being $L$-local is a mere proposition. Therefore the claim follows from \cref{cor:local_equiv}.
\end{proof}

The only thing that prevents $\UU_L$ from actually being $L$-separated is the fact that $\UU_L$ is not small. In other words, we could say that $\UU_L$ is \define{essentially $L$-separated}. Using the fact that the image of a small type into $\UU_L$ is essentially small, the condition of being essentially $L$-separated suffices to eliminate from $L'X$ into $\UU_L$.

\begin{lem}\label{lemma:extendtoUL}
Consider an $L'$-localization $l':X\to Y'$. Then any type family $P:X\to\UU_L$ of $L$-local types has a unique extension along $l'$
\begin{equation*}
\begin{tikzcd}
X \arrow[d,swap,"{l'}"] \arrow[r,"P"] & \UU_L. \\
Y' \arrow[ur,densely dotted]
\end{tikzcd}
\end{equation*}
\end{lem}

\begin{proof}
    We prove the first form of the statement.
    By \cref{prop:UU_L-is-L-separated}, the identity types of $\UU_L$ are
    equivalent to small types, i.e., $\UU_L$ is a locally small type.
    By the join construction~\cite{joinconstruction}, the image of $P$
    can be taken to be a small type $I$ in $\UU$, so there is a factorization
    of $P$ into a surjection $\hat{P} : X \to I$ followed by
    an embedding $i : I \to \UU_L$:
\begin{equation*}
\begin{tikzcd}
X \arrow[d,swap,"\modalunit'"] \arrow[r,"P"] \arrow[dr,swap,"\hat{P}"] & \UU_L \\
L'X & I \arrow[u,swap,"i"] 
\end{tikzcd}
\end{equation*}

    Since the identity types of $I$ are equivalent to identity types
    of $\UU_L$, and $I$ is small, it follows that the identity types of $I$ are actually $L$-local.
    This means that $I$ is an $L$-separated type, so we can extend $\hat{P}$ to $L' X$ 
    giving us the desired extension of $P$ by composing with $i$.

    Since $X \to L'X$ is surjective (\cref{thm:separation_characterization}),
    any such extension must factor through the image $I$.
    So uniqueness follows from the universal property of $L'$-localization.
\end{proof}

Before we show that $L'$-localizations exist for any type $X$, we characterize them.
To establish our characterization of $L'$-localizations, we need the following simple lemma, that allows us to construct unique extensions.

\begin{lem}\label{lem:unique_extension}
Let $g:A\to B$ and $f:A\to C$ be maps for which we have a unique extension
\[
  \begin{tikzcd}
    \fib{g}{b} \arrow[r,"f\circ\proj 1"] \arrow[d] & C \\
    \unit \arrow[ur,densely dotted]
  \end{tikzcd}
\]
for every $b:B$.
Then $f$ extends uniquely along $g$.
\end{lem}

\begin{proof}
By assumption we have
\[
  \prd{b:B}\iscontr\Big(\sm{c:C}\prd{a:A}{p:g(a)=b} f(a)=c\Big).
\]
The center of contraction gives us an extension
\[
  \begin{tikzcd}
    A \arrow[r,"f"] \arrow[d,swap,"g"] & C \\
    B \arrow[ur,densely dotted,swap,"\tilde{f}"]
  \end{tikzcd}
\]
and its uniqueness follows from the contraction.
\end{proof}

\begin{thm}\label{thm:separation_characterization}
Consider a map $l':X\to Y'$, where $Y'$ is assumed to be $L$-separated. Then the following are equivalent:
\begin{enumerate}
\item The map $l':X\to Y'$ is an $L'$-localization.
\item The map $l':X\to Y'$ is surjective, and for each $x,y:X$, the map
\begin{equation*}
\apfunc{l'}:(x=y)\to (l'(x)=l'(y))
\end{equation*}
is an $L$-localization. 
\end{enumerate}
\end{thm}

\begin{proof}
First, suppose that $l':X\to Y'$ is an $L'$-localization. To see that $l'$ is surjective, we note that $\im(l')$ is $L$-separated since it is a subtype of $Y'$, so the surjective map $q:X\to\im(l')$ extends uniquely along $l'$. 
\[
  \begin{tikzcd}
    X \arrow[r,"q"] \arrow[d,swap,"l'"] & \im(l') \\
    Y' \arrow[ur,densely dotted,swap,"h"]
  \end{tikzcd}
\]
By the universal property of $L'$-localization it follows that $h$ is a section of the image inclusion $\im(l')\to Y'$. In particular, the image inclusion is both surjective and an embedding, so it must be an equivalence. It follows that $l'$ is surjective.

Next, we need to show that for each $x,y:X$, the map
\begin{equation*}
\apfunc{l'}:(x=y)\to (l'(x)=l'(y))
\end{equation*}
is an $L$-localization. Fix $x : X$.
Since $l'$ is an $L'$-localization, there is a unique extension
\begin{equation*}
\begin{tikzcd}[column sep=huge]
X \arrow[r,"y\mapsto L(x=y)"] \arrow[d,swap,"{l'}"] & \UU_L. \\
Y' \arrow[ur,densely dotted,swap,"P"]
\end{tikzcd}
\end{equation*}
The family $P$ comes equipped with a point $p_0:P(l'(x))$ that is induced by $\eta(\refl{x}):L(x=x)$. Moreover, by the fact that $P$ extends $y\mapsto L(x=y)$ we have a pullback square
\begin{equation*}
\begin{tikzcd}
\sm{y:X}L(x=y) \arrow[d] \arrow[r] & \sm{y':Y'}P(y') \arrow[d] \\
X \arrow[r,swap,"{l'}"] & Y'
\end{tikzcd}
\end{equation*}
Thus we see that our claim follows by \cref{thm:id_fundamental}, once we show that the total space of $P$ is contractible.

For the center of contraction of $\sm{y:L'X}P(y)$ we take $(l'(x),\eta(\refl{x}))$.
It remains to construct a contraction
\begin{equation*}
\prd{y':Y'}{p:P(y')}(l'(x),\eta(\refl{x}))=(y',p).
\end{equation*}
Since the fibers of $P$ are $L$-local, it follows by \cref{lemma:separatedpluslocalisseparated} that the total space of $P$ is $L$-separated. Therefore we obtain by \cref{lem:modal-Pi} that the type
\[
  \prd{p:P(y')} (l'(x),\eta(\refl{x}))=(y',p)
\]
is $L$-local for every $y':Y'$. Thus \cref{proposition:inductionLseparated} reduces the problem to
constructing a term of type
\begin{equation*}
\prd{y:X}{p:L(x=y)}(l'(x),\eta(\refl{x}))=(l'(y),p).
\end{equation*}

Furthermore, for $y : X$ we have equivalences
\begin{align*}
    \Big(\sm{p:L(x=y)}(l'(x),\eta(\refl{x}))=(l'(y),p)\Big)
 & \eqvsym \Big(\sm{p:L(x=y)}{\alpha:l'(x)=l'(y)} \trans{\alpha}{\eta(\refl{x})} = p \Big) \\
 & \eqvsym \sm{\alpha:l'(x) = l'(y)} \unit \\
 & \eqvsym l'(x) = l'(y),
\end{align*}
where the last type is clearly $L$-local.
So we can apply \cref{theorem:generalized-induction} to reduce the problem to
the problem of constructing a term of type
\begin{equation*}
    \prd{y:X}{p:x=y}(l'(x),\eta(\refl{x}))=(l'(y),\eta(p)).
\end{equation*}
This can be done by a simple application of path induction. This completes the proof that (i) implies (ii).

To show that (ii) implies (i), assume that $l'$ is surjective, and that for every $x,y:X$ the map
\begin{equation*}
\apfunc{l'}:(x=y)\to (l'(x)=l'(y))
\end{equation*}
is an $L$-localization. Our goal to show that $l'$ satisfies the universal property of $L'$-localization, so assume $f : X \to Z$ is a map into an $L$-separated type $Z$.

By \cref{lem:unique_extension}, it is enough to show that $f$ restricts to a unique constant map
on the fibers of $l'$. This means that we must show that
\[
  \sm{z:Z}\prd{x:X} (l'(x)=y') \lra (f(x)=z)
\]
is contractible for every $y':Y'$.
Since this is a mere proposition, and $l'$ is surjective, we can assume that
$y' = l'(y)$.  In other words, it suffices to show that
\[
  \sm{z:Z}\prd{x:X} (l'(x)=l'(y)) \lra (f(x)=z)
\]
is contractible for every $y:X$.

Since $Z$ is assumed to be $L$-separated and $\mathsf{ap}_{l'}$ is assumed to be an $L$-localization,
this type is equivalent to
\[
  \sm{z:Z}\prd{x:X} (x=y) \lra (f(x)=z)
\]
and it is easy to see that this is a contractible type by applying the contractibility of
the total space of the path fibration twice.
\end{proof}

Our final goal for this section is to show that $L'$ is a reflective subuniverse, i.e.~that there is an $L'$-localization for every type $X:\UU$.
We will use a `local version' of the type theoretic Yoneda Lemma.

\begin{lem}\label{lemma:local_yoneda}
For each $y:X$ and each $P:X\to \UU_L$, the map
\[
  \mathsf{ev\usc{}locrefl}:\Big(\prd{z:X}L(y=z)\to P(z)\Big) \lra P(y)
\]
given by $f\mapsto f(y,\eta(\refl{y}))$ is an equivalence.
\end{lem}

\begin{proof}
By the universal property of $L(y=z)$ and identity elimination,
the map in the statement can be factored as follows:
\begin{align*}
\Big(\prd{z:X}L(y=z)\to P(z)\Big) & \simeq \Big(\prd{z:X}(y=z)\to P(z)\Big) \\
& \simeq P(y).\qedhere
\end{align*}
\end{proof}

\begin{thm}\label{thm:Lsep}
For any reflective subuniverse $L$, the subuniverse of $L$-separated types is again reflective. We will write
\begin{equation*}
\modalunit':X\to L'X
\end{equation*}
for the $L'$-localization of a type $X$.
\end{thm}

\begin{proof}
Fix a type $X : \UU$. Let $\mathcal{Y}_L:X\to (X\to\UU)$ be given by
\[
\mathcal{Y}_L(x,y)\defeq L(x=y).
\]
We would like to define $L' X$ to be $\im(\mathcal{Y}_L)$, but this is a subtype of
$X \to \UU$, so it is not small (i.e., it does not live in $\UU$).
However, since $\UU$ is locally small, so is $X \to \UU$.
Thus the join construction~\cite{joinconstruction} implies that the image
is equivalent to a small type which we denote $L' X$.
This comes equipped with a surjective map
\[
\eta':X \lra L'X,
\]
which we take to be the unit of the reflective subuniverse.

To show that $\eta'$ is a localization, we apply \cref{thm:separation_characterization}.
First we show that $L' X$ is $L$-separated.  Since $\eta'$ is surjective
and being $L$-local is a proposition, it is enough to show that
$\eta'(x) = \eta'(y)$ is $L$-local for $x$ and $y$ in $X$.
Since $L' X$ embeds in $X \to \UU$, we have an equivalence between
$\eta'(x) = \eta'(y)$ and $(\lambda z . L(x = z)) = (\lambda z . L(y = z))$.
The latter is equivalent to $\prd{z:X} L(x = z) = L(y = z)$, which
is $L$-local by \cref{prop:UU_L-is-L-separated}.

It remains to show that the canonical map $L(x=y)\to (\eta'(x)=\eta'(y))$ is an equivalence.
By the above argument, combined with univalence,
the problem reduces to showing that the canonical map
\[
    L(x=y) \lra \left(\prd{z:X} L(x=z) \simeq L(y=z)\right)\]
is an equivalence.
Using symmetry of equivalences, it suffices to show that the map
\[
    L(x=y) \lra \left(\prd{z:X} L(y=z) \simeq L(x=z)\right)
\]
is an equivalence.
Moreover, since the forgetful map from equivalences to maps is an embedding,
it is enough to show that the composite map
\[
    \alpha_{x,y} : L(x=y) \lra \left(\prd{z:X} L(y=z) \to L(x=z)\right)
\]
is an equivalence.
Indeed, if $i$ is an embedding and $i \circ g$ is an equivalence,
then $i$ is surjective.  Therefore $i$ is an equivalence and hence so is $g$.

By the local Yoneda Lemma~\ref{lemma:local_yoneda}, with $P(z) \defeq L(x = z)$, there is an equivalence
\[
\beta_{x,y}: L(x=y) \lra \Big(\prd{z:X} L(y=z) \to L(x=z)\Big)
\]
which sends $p : L(x = y)$ to the unique function $f$ such that
$f(x, \eta(\refl{x})) = p$.
So it suffices to show that $\alpha_{x,y} = \beta_{x,y}$.
By the universal property of $L(x=y)$, it is enough to show that
$\alpha_{x,y} \circ \eta = \beta_{x,y} \circ \eta$ as maps
$(x=y) \to \big(\prd{z:X} L(y=z) \to L(x=z)\big)$.
Letting $x$ and $y$ vary and using path induction, we reduce the
problem to showing that $\alpha_{x,x}(\eta(\refl{x})) = \beta_{x,x}(\eta(\refl{x}))$.

Since $\alpha_{x,y}$ is defined by path induction, it is easy to see
that $\alpha_{x,x}(\eta(\refl{x}))$ is equal to $\lam{z}\idfunc[L(x=z)]$.
On the other hand, $\beta_{x,x}(\eta(\refl{x}))$ is the unique function $f$ such that
$f(x, \eta(\refl{x})) = \eta(\refl{x})$.
Therefore, this $f$ must also equal $\lam{z}\idfunc[L(x=z)]$,
showing that $\alpha_{x,x}(\eta(\refl{x})) = \beta_{x,x}(\eta(\refl{x}))$.
\end{proof}

\begin{thm}\label{thm:truncation}
For each $k\geq -2$, the subuniverse of $k$-truncated types is reflective.
\end{thm}

\begin{proof}
The subuniverse of contractible types is obviously reflective, it's localization is just the constant function $\lam{X}\unit$ mapping every type to the unit type. Since the subuniverse of $(k+1)$-truncated types is precisely the subuniverse of $k$-separated types, the claim follows inductively by \cref{thm:Lsep}.
\end{proof}

\section{\texorpdfstring{$L$}{L}-local maps}

\begin{defn}
A map $f:A\to B$ is said to be \define{$L$-local} if its fibers are $L$-local.
\end{defn}

\begin{thm}
Let $B$ be a type family over $A$. Then the following are equivalent:
\begin{enumerate}
\item For each $x:A$ the type $B(x)$ is $L$-local.
\item The projection map
\begin{equation*}
\proj 1 : \Big(\sm{x:A}B(x)\Big)\to A
\end{equation*}
is an $L$-local map.
\end{enumerate}
\end{thm}

\begin{proof}
Immediate from the equivalences $\eqv{B(x)}{\fib{\proj 1}{x}}$.
\end{proof}

\begin{thm}\label{thm:trunc_ap}
Let $f:A\to B$ be a map. The following are equivalent:
\begin{enumerate}
\item The map $f$ is $L'$-local.
\item For each $x,y:A$, the map
\begin{equation*}
\apfunc{f} : (x=y)\to (f(x)=f(y))
\end{equation*}
is $L$-local.
\item The diagonal $\delta_f:A\to A\times_B A$ of $f$ is $L$-local. 
\end{enumerate}
\end{thm}

\begin{proof}
First we show that for any $s,t:\fib{f}{b}$ there is an equivalence
\begin{equation*}
\eqv{(s=t)}{\fib{\apfunc{f}}{\ct{\proj 2(s)}{\proj 2(t)^{-1}}}}
\end{equation*}
We do this by $\Sigma$-induction on $s$ and $t$, and then we calculate using basic manipulations of identifications that
\begin{align*}
(\pairr{x,p}=\pairr{y,q}) & \eqvsym \sm{r:x=y} \mathsf{tr}_{f(\blank)=b}(r,p)=q \\
& \eqvsym \sm{r:x=y} \ct{\ap{f}{r}^{-1}}{p}=q \\
& \eqvsym \sm{r:x=y} \ap{f}{r}=\ct{p}{q^{-1}} \\
& \jdeq \fib{\apfunc{f}}{\ct{p}{q^{-1}}}.
\end{align*}
By these equivalences, it follows that if $\apfunc{f}$ is $L$-local, then for each $s,t:\fib{f}{b}$ the identity type $s=t$ is an $L$-local type.

For the converse, note that we have equivalences
\begin{align*}
\fib{\apfunc{f}}{p} & \eqvsym ((x,p)=(y,\refl{f(y)})).
\end{align*}
Therefore it follows that if $f$ is $L'$-local, then the identity type $(x,p)=(y,\refl{f(y)})$ in $\fib{f}{f(y)}$ is $L$-local for any $p:f(x)=f(y)$, and therefore $\fib{\apfunc{f}}{p}$ is $L$-local. 
\end{proof}

\begin{thm}
Let $f:\prd{x:A}B(x)\to C(x)$ be a fiberwise transformation. Then the following are equivalent:
\begin{enumerate}
\item For each $x:A$ the map $f(x)$ is $L$-local.
\item The induced map 
\begin{equation*}
\total{f}:\Big(\sm{x:A}B(x)\Big)\to\Big(\sm{x:A}C(x)\Big)
\end{equation*}
is $L$-local.
\end{enumerate}
\end{thm}

\begin{proof}
This follows directly from \cref{lem:fib_total}.
\end{proof}

\section{Quasi-left-exactness of \texorpdfstring{$L'$}{L'}-localization}\label{ss:lex}

We now explain how $L$ and $L'$ together behave similarly to a \define{lex modality},
i.e., a modality that preserves pullbacks.
Theorem~3.1 of \cite{RijkeShulmanSpitters} gives 13 equivalent characterizations of a lex modality,
and it turns out that these hold for any reflective subuniverse
if the modal operator is replaced by $L$ and $L'$ in the appropriate way.
The propositions in this section show this for parts (ix), (x), (xii) and (xi)
of Theorem~3.1, respectively.
The proofs use the dependent elimination of $L'$ in a crucial way,
but do not use the specific construction of $L'$-localization, just the existence.

Before proving the next result,
we need a lemma, which follows directly from the dependent elimination of $L'$.

\begin{lem}\label{lemma:Lequivalencetotalspaces}
Let $P:L'X\to \UU$ be a type family over $L'X$. 
Then the map
\begin{equation*}
f:\Big(\sm{x:X} P(\eta'(x))\Big)\to \Big(\sm{y:L'X}P(y)\Big)
\end{equation*}
given by $(x,p)\mapsto (\eta'(x),p)$ is an $L$-equivalence. 
\end{lem}

\begin{proof}
For any $L$-local type $Z$ we have the commuting square
\begin{equation*}
\begin{tikzcd}[column sep=9em]
\Big(\Big(\sm{y : L'X} P(y)\Big)\to Z\Big) \arrow[r,"{\lam{f}{(x,p)}f(\eta'(x),p)}"] \arrow[d,swap,"\mathsf{ev\usc{}pair}"] & \left(\sm{x:X} P(\eta' x)\right) \to Z \arrow[d,"\mathsf{ev\usc{}pair}"] \\
\prd{y : L'X} (P(y) \to Z) \arrow[r,swap,"{\lam{f}{x}{y}f(\eta'(x),y)}"] & \prd{x : X} (P(\eta' x) \to Z)
\end{tikzcd}
\end{equation*}
The bottom map is an equivalence by \cref{proposition:inductionLseparated}, using that $Z$ and $P(y) \to Z$ are $L$-local.
The two vertical maps are also equivalences, so it follows that the top map is an equivalence. The claim now follows from \cref{lem:local_equivalence}.
\end{proof}

\begin{prp}\label{remark:preservationpullbacks} %
Consider a commuting cube of the form
\begin{equation*}
\begin{tikzcd}
& A\times_X B \arrow[d,densely dotted] \arrow[dl] \arrow[dr] \\
A \arrow[d,swap,"\eta'"] & L'A\times_{L'X} L'B \arrow[dl] \arrow[dr] & B \arrow[dl,crossing over] \arrow[d,"\eta'"] \\
L'A \arrow[dr] & X \arrow[from=ul,crossing over] \arrow[d,swap,"\eta'"] & L'B \arrow[dl] \\
& L'X.
\end{tikzcd}
\end{equation*}
Then the map $A\times_X B\to L'A\times_{L'X} L'B$ is an $L$-equivalence.
\end{prp}

\begin{proof}
Consider the following commuting square
\begin{equation*}
\begin{tikzcd}
\sm{x:A}{y:B}f(x)=g(y) \arrow[r,densely dotted] \arrow[d] & \sm{x':L'A}{y':L'B} L'f(x')=L'g(y') \\
\sm{x:A}{y:B}\eta'(f(x))=\eta'(g(y)) \arrow[r] & \sm{x:A}{y:B}L'f(\eta'(x))=L'g(\eta'(y)) \arrow[u]
\end{tikzcd}
\end{equation*}
In this square, the downwards morphism on the left is the induced map on total spaces of the map $\apfunc{\eta'}:(f(x)=g(y))\to (\eta'(f(x))=\eta'(g(y)))$, which is an $L$-equivalence by \cref{thm:separation_characterization} and \cref{lem:Lequiv_total}. The bottom map is an equivalence, obtained from the naturality squares $\eta'\circ f\htpy L'f\circ \eta'$ and $\eta'\circ g\htpy L'g\circ \eta'$. In particular, it is an $L$-equivalence. The upwards map on the right is an $L$-equivalence by \cref{lemma:Lequivalencetotalspaces}. Therefore, the asserted map is a composite of $L$-equivalences, so it is also an $L$-equivalence.
\end{proof}

As a consequence we get a result about the preservation of certain fiber sequences.

\begin{cor}\label{corollary:preservationfibersequences2}
    Given a fiber sequence $F \hookrightarrow E \twoheadrightarrow B$, there is a map of fiber sequences
\begin{equation*}
\begin{tikzcd}
F \arrow[r,hookrightarrow] \arrow[d] & E \arrow[r,->>,"p"] \arrow[d,swap,"\eta'"] & B \arrow[d,"\eta'"] \\
F' \arrow[r,hookrightarrow] & L'E \arrow[r,->>,swap,"L'p"] & L'B
\end{tikzcd}
\end{equation*}
    in which the left vertical map is an $L$-equivalence.\qed
\end{cor}

\section{Connected maps}
\label{sec:connected-maps}

\begin{defn}
A map $f:A\to B$ is said to be \define{$L$-connected} if $L(\fib{f}{b})$ is contractible for every $b:B$. In particular, a type $A$ is $L$-connected if and only if $LA$ is contractible.
\end{defn}

In the following proposition we characterize $L'$-connected types. We will establish a similar claim for maps in \cref{cor:L'connected_maps}

\begin{prp}\label{lem:L'connected_types}
A type $A$ is $L'$-connected if and only if $A$ is merely inhabited (i.e.~$\brck{A}$ holds), and the identity types of $A$ are $L$-connected.
\end{prp}

\begin{proof}
$A$ is $L'$-connected if and only if $A\to\unit$ is an $L'$-localization. By \cref{thm:separation_characterization} this holds if and only if $A\to\unit$ is surjective, and the maps $(x=y)\to \loopspace\unit$ are $L$-localizations --- in other words: if and only if $A$ is merely inhabited and the identity types of $A$ are $L$-connected.
\end{proof}

In the following proposition we provide two equivalent conditions to a map being $L$-connected.

\begin{prp}\label{prop:nconnected_tested_by_lv_n_dependent types}
Consider a map $f:A\to B$. The following are equivalent:
\begin{enumerate}
\item $f$ is $L$-connected.\label{item:conntest1}
\item For every family $P:B\to \UU_L$ of $L$-local types, the map 
\begin{equation*}
\lam{s} s\circ f :\Parens{\prd{b:B} P(b)}\to\Parens{\prd{a:A}P(f(a))}.
\end{equation*}
is an equivalence.\label{item:conntest2}
\item For every family $P:B\to \UU_L$ of $L$-local types, the map 
\begin{equation*}
\lam{s} s\circ f :\Parens{\prd{b:B} P(b)}\to\Parens{\prd{a:A}P(f(a))}.
\end{equation*}
has a section.\label{item:conntest3}
\end{enumerate}
\end{prp}

\begin{proof}
First suppose $f$ is $L$-connected and let $P:B\to \UU_L$. Then the map
\begin{equation*}
\lam{p}\mathsf{const}_p : P(b)\to (\fib{f}{b}\to P(b))
\end{equation*}
is an equivalence for every $b:B$, since $L(\fib{f}{b})$ is assumed to be contractible.
Therefore we obtain a commuting square
\begin{equation*}
\begin{tikzcd}[column sep=large]
\prd{b:B}P(b) \arrow[r] \arrow[d] & \prd{a:A}P(f(a)) \\
\prd{b:B}\big(\fib{f}{b}\to P(b)\big) \arrow[r,swap,"\mathsf{evpair}"] & \prd{b:B}{a:A}(f(a)=b)\to P(b) \arrow[u,swap,"\evrefl"]
\end{tikzcd}
\end{equation*}
in which three out of four maps are known equivalences. The remaining map must therefore also be an equivalence.

Thus,~\ref{item:conntest1}$\Rightarrow$\ref{item:conntest2}, and clearly~\ref{item:conntest2}$\Rightarrow$\ref{item:conntest3}.
To show~\ref{item:conntest3}$\Rightarrow$\ref{item:conntest1}, let
$P(b)\defeq L(\fib{f}b)$.
Then~\ref{item:conntest3} yields a map $c:\prd{b:B} L(\fib{f}b)$ with
$c(f(a))=\modalunit{\pairr{a,\refl{f(a)}}}$. To show that each $L(\fib{f}b)$ is contractible, we will show that $c(b)=w$ for any $b:B$ and $w:L(\fib{f}b)$.
In other words, we must show that the identity function $L(\fib{f}b) \to L(\fib{f}b)$ is equal to the constant function at $c(b)$.
By the universal property of $L(\fib{f}b)$, it suffices to show that they become equal when precomposed with $\modalunit[\fib{f}b]$, i.e.\ we may assume that $w = \modalunit\pairr{a,p}$ for some $a:A$ and $p:f(a)=b$.
But now path induction on $p$ reduces our goal to the given $c(f(a))=\modalunit{\pairr{a,\refl{f(a)}}}$.
\end{proof}

\begin{cor}\label{connectedtotruncated}
A type $A$ is $L$-connected if and only if the ``constant functions'' map
$
  B \to (A\to B)
$
is an equivalence for every modal type $B$.\qed
\end{cor}

Dually, we will prove in \cref{thm:detect-right-by-fibers} that when $L$ is a modality, if this holds for all $L$-connected $A$ then $B$ is $L$-local.

\begin{cor}
If $f:A\to B$ is an $L$-connected map into an $L$-local type $B$, then $f$ is an $L$-localization. (The converse only holds when $L$ is a modality.)
\end{cor}

\begin{prp}\label{cor:L'equivalenceisLconnected}
Consider a map $f:A\to B$.
\begin{enumerate}
\item If $f$ is an $L'$-equivalence, then $f$ is $L$-connected.
\item If $f$ is $L$-connected, then $f$ is an $L$-equivalence.
\end{enumerate}
\end{prp}

\begin{proof}
The first statement is a direct consequence of \cref{corollary:preservationfibersequences2}, since we get for each $b:B$ a map of fiber sequences
\begin{equation*}
\begin{tikzcd}
\fib{f}{b} \arrow[r] \arrow[d] & A \arrow[r,"f"] \arrow[d,swap,"\eta'"] & B \arrow[d,"\eta'"] \\
\fib{L'f}{\eta'b} \arrow[r] & L'A \arrow[r,swap,"L'f"] & L'B
\end{tikzcd}
\end{equation*}
in which the induced map $\fib{f}{b}\to\fib{L'f}{\eta'b}$ is an $L$-equivalence. Since $L'f$ is an equivalence, it follows that $\fib{L'f}{\eta'b}$ is contractible, so we conclude that $\fib{f}{b}$ is $L$-connected. It follows that $f$ is $L$-connected.

For the second statement, we observe that if $f$ is $L$-connected, then by \cref{prop:nconnected_tested_by_lv_n_dependent types} it induces an equivalence
\begin{equation*}
\precomp{f} : (B\to X)\to (A\to X)
\end{equation*}
for every $L$-local type $X$. Therefore we conclude by \cref{lem:local_equivalence} that $f$ is an $L$-equivalence.
\end{proof}

\begin{prp}\label{thm:rsu-compose-cancel} %
Consider a commuting triangle
\begin{equation*}
\begin{tikzcd}[column sep=tiny]
A \arrow[rr,"h"] \arrow[dr,swap,"f"] & & B \arrow[dl,"g"] \\
& X
\end{tikzcd}
\end{equation*}
\begin{enumerate}
\item If $h$ is $L$-connected, then the following are equivalent:
\begin{enumerate}
\item $f$ is $L$-connected.
\item $g$ is $L$-connected.
\end{enumerate}
\item If $f$ is $L'$-connected, then the following are equivalent:
\begin{enumerate}
\item $g$ is $L'$-connected.
\item $h$ is $L$-connected.
\end{enumerate}
In particular, $g$ is $(n+1)$-connected if and only if $h$ is $n$-connected, assuming that $f$ is $(n+1)$-connected.
\end{enumerate}
\end{prp}

\begin{proof}
For the first statement suppose that $h$ is $L$-connected.
  For any $z:C$ we have
  \begin{align*}
    L(\fib{g\circ h}{z})
    & \eqvsym
      L(\sm{p:\fib{g}{z}}\fib{h}{\proj1(p)}) \\
    & \eqvsym
      L(\sm{p:\fib{g}{z}}L(\fib{h}{\proj1(p)})) \\
    & \eqvsym
      L(\sm{p:\fib{g}{z}}\unit) \\
    & \eqvsym
      L(\fib{g}{z}).
  \end{align*}
  using the fact that $h$ is $L$-connected.
  Thus, one is contractible if and only if the other is. We conclude that $f$ is $L$-connected if and only if $g$ is.

For the second statement, suppose that $f$ is $L'$-connected.
    If $g$ is $L'$-connected, then both $g\circ h$ and $g$ are $L'$-equivalences,
    and thus $h$ is an $L'$-equivalence. Then \cref{cor:L'equivalenceisLconnected}
    implies that $h$ is $L$-connected.

    For the converse, notice that taking fibers over each $x:X$ reduces the problem
    to showing that given an $L$-connected map $h : A \to B$
    such that $A$ is $L'$-connected, it follows that $B$ is $L'$-connected.

    Since $L'A$ is contractible, it follows that $L'B$ is contractible if and only if $L'h:L'A\to L'B$ is an equivalence.
    Therefore it suffices to show that the fibers of $L'h$ are contractible. Moreover, since $\eta':B\to L'B$ is a surjective map by \cref{thm:separation_characterization}, it is enough to show that $\fib{L'h}{\eta'(b)}$ is contractible for every $b:B$.
First we observe that, $\eqv{\fib{L'h}{\eta'(b)}}{\loopspace{L'B,\eta'(b)}}$ since $L'A$ is contractible. 
    Now we observe using \cref{corollary:preservationfibersequences2} that we have for every $b:B$ a morphism of fiber sequences
\begin{equation*}
\begin{tikzcd}
\fib{h}{b} \arrow[r,hookrightarrow] \arrow[d,densely dotted] & A \arrow[d,swap,"\eta'"] \arrow[r,"h"] & B \arrow[d,"\eta'"] \\
\loopspace{L'B,\eta'(b)} \arrow[r,hookrightarrow] & L'A \arrow[r,swap,"L'h"] & L'B
\end{tikzcd}
\end{equation*}
in which the map on the left is an $L$-equivalence. However, the type $\loopspace{L'B,\eta'(b)}$ is $L$-local, so it follows that the map
\begin{equation*}
\fib{h}{b}\to \loopspace{L'B,\eta(b)}
\end{equation*}
is an $L$-localization. Now it follows by our assumption that $h$ is $L$-connected that $\loopspace{L'B,\eta(b)}$ is contractible. We conclude that $\fib{L'h}{\eta'(b)}$ is contractible, and therefore that $L'B$ is contractible.
\end{proof}

\begin{rmk}
In general it is not true that if $g$ and $g\circ h$ are $\modal$-connected then $h$ is; this is one of the equivalent characterizations of lex modalities (Theorem 3.1 of \cite{RijkeShulmanSpitters}).

The above proposition almost gives us a 3-for-2 property that combines $L$ and $L'$.
However the map $\emptyt \to \unit$ is $(-2)$-connected, and $\unit$ is $(-1)$-connected,
whereas $\emptyt$ is not $(-1)$-connected. So the remaining implication of the 3-for-2 property
does not hold. One can show the weaker result that the composite of an $L$-connected map followed by
an $L'$-connected map is $L$-connected.
\end{rmk}

\begin{rmk}
\cref{thm:separation_characterization} allows us to give a concrete description of
the extension defined in \cref{lemma:extendtoUL}.
Using an argument similar to the one used in the proof of \cref{remark:preservationpullbacks},
one can show that given an $L'$-localization $\eta' : X \to L' X$
and a map $f: Y \to X$ with $L$-local fibers, $f$ is the pullback of
the fiberwise $L$-localization of $\eta' \circ f$.
\end{rmk}

\begin{defn}
A commuting square
\begin{equation*}
\begin{tikzcd}
A\arrow[d,swap,"f"] \arrow[r,"h"] & B \arrow[d,"g"] \\
X \arrow[r,swap,"i"] & Y
\end{tikzcd}
\end{equation*}
is said to be \define{$L$-cartesian} if its gap map is $L$-connected.
\end{defn}

\begin{prp}
Consider a commuting square
\begin{equation*}
\begin{tikzcd}
A\arrow[d,swap,"f"] \arrow[r,"h"] & B \arrow[d,"g"] \\
X \arrow[r,swap,"i"] & Y.
\end{tikzcd}
\end{equation*}
\begin{enumerate}
\item Suppose that $g$ is surjective, and that the square is $L$-cartesian. Then the following are equivalent:
\begin{enumerate}
\item The map $h:A\to B$ is $L$-connected.
\item The map $i:X\to Y$ is $L$-connected.
\end{enumerate}
\item Suppose that $i:X\to Y$ is $L'$-connected. Then the following are equivalent:
\begin{enumerate}
\item The map $h:A\to B$ is $L'$-connected.
\item The square is $L$-cartesian.
\end{enumerate}
\end{enumerate}
\end{prp}

\begin{proof}
For the first statement we observe that by \cref{thm:rsu-compose-cancel} $h$ is an $L$-connected if and only if $g^\ast i$ is an $L$-connected. Since $g$ is assumed to be surjective, it follows that $g^\ast i$ is $L$-connected if and only if $i$ is $L$-connected.

For the second statement we observe that, since $i:X\to Y$ is assumed to be $L'$-connected, the map $g^\ast i:X\times_Y B\to B$. Now it follows by \cref{thm:rsu-compose-cancel} that the gap map is $L$-connected if and only if $h$ is $L'$-connected.
\end{proof}

\begin{cor}\label{cor:L'connected_maps}
A map $f:A\to B$ is a $L'$-connected if and only if $f$ is surjective and $\delta_f:A\to A\times_B A$ is a $L$-connected.
\end{cor}

\begin{proof}
Since every proposition is $L'$-local, it follows that $\brck{X}=\brck{L'X}$ for any type $X$. In particular, if $X$ is $L'$-connected, then $\brck{X}=\unit$. From this observation it follows that if $f$ is $L'$-connected, then $f$ is surjective. Furthermore, since the identity function $\idfunc[A]:A\to A$ is obviously $L'$-connected, it follows that the square
\begin{equation*}
\begin{tikzcd}
A \arrow[d] \arrow[r] & A \arrow[d,"f"] \\
A \arrow[r,swap,"f"] & B
\end{tikzcd}
\end{equation*}
is $L$-cartesian. In other words: $\delta_f:A\to A\times_B A$ is $L$-connected.

Now suppose that $f$ is surjective and that $\delta_f$ is $L$-connected. Then the fibers of $\delta_f$ are $L$-connected. Since we have equivalences
\begin{equation*}
((x,p)=(y,q))\eqvsym \fib{\delta_f}{x,y,\ct{p}{q^{-1}}}
\end{equation*}
for any $(x,p),(y,q):\fib{f}{b}$ it follows by \cref{lem:L'connected_types} that $\fib{f}{b}$ is $L'$-connected for every $b:B$. In other words: $f$ is $L'$-connected.
\end{proof}

A commuting square is said to be a \define{quasi-pullback} if its gap map is surjective. In other words, quasi-pullback squares are the same as $(-1)$-cartesian squares.

\begin{cor}
A map $f:A\to B$ is $n$-connected if and only if the gap map of the commuting square
\begin{equation*}
\begin{tikzcd}
A \arrow[r,"f"] & B \arrow[d] \\
A^{\sphere{k}} \arrow[r,swap,"f^{\sphere{k}}"] & B^{\sphere{k}}
\end{tikzcd}
\end{equation*}
is a quasi-pullback square, for each $-1\leq k\leq n$. A map is therefore $\infty$-connected if this square is a quasi-pullback for all $k$
\end{cor}

\begin{cor}\label{cor:ptd_connected}
A pointed connected type $A$ is $L'$-connected if and only if the map $\unit\to A$ is $L$-connected.
\end{cor}

\begin{proof}
If $A$ is $L'$-connected, then its loop space is $L$-connected. It follows that the square
\begin{equation*}
\begin{tikzcd}
\unit \arrow[r] \arrow[d] & \unit \arrow[d] \\
\unit \arrow[r] & A
\end{tikzcd}
\end{equation*}
is $L$-cartesian. Since the map $\unit\to\unit$ is obviously $L$-connected and the map $\unit\to A$ is surjective by the assumption that $A$ is connected, it follows that the map $\unit\to A$ is $L$-connected.

For the converse, suppose that the map $\unit\to A$ is $L$-connected. Then the identity types of $A$ are $L$-connected, since all identity types of $A$ are merely equal to $\loopspace{A}$. Therefore it follows by \cref{lem:L'connected_types} that $A$ is $L'$-connected.
\end{proof}

\begin{cor}\label{lem:nconnected_postcomp_variation}
Let $P:A\to\UU$ and $Q:B\to\UU$ be type families, and let
\begin{equation*}
g:\prd{x:A}P(x)\to Q(f(x))
\end{equation*}
be a fiberwise transformation over $f:A\to B$.
\begin{enumerate}
\item Suppose that each $g_x:P(x)\to Q(f(x))$ is $L$-connected. Then we have the following:
\begin{enumerate}
\item If $f$ is $L$-connected, then $\total[f]{g}$ is $L$-connected.
\item If each $Q(x)$ is merely inhabited and $\total[f]{g}$ is $L$-connected, then $f$ is $L$-connected.
\end{enumerate}
\item Suppose that $f$ is $L'$-connected. Then the following are equivalent:
\begin{enumerate}
\item Each $g_x:P(x)\to Q(f(x))$ is $L$-connected.
\item The map $\total[f]{g}$ is $L'$-connected.
\end{enumerate}  
\end{enumerate}
\end{cor}

\begin{prp}\label{prop:nconn_fiber_to_total}
Let $P,Q:A\to\type$ be type families and $f:\prd{a:A} \Parens{P(a)\to Q(a)}$. Then the following are equivalent
\begin{enumerate}
\item Each $f(a):P(a)\to Q(a)$ is $L$-connected.
\item The map $\total{f}: \sm{a:A}P(a) \to \sm{a:A} Q(a)$ is an $L$-connected.
\end{enumerate}
\end{prp}

\begin{proof}
By \cref{lem:fib_total} we have $\fib{\total f}{\pairr{x,v}}\eqvsym\fib{f(x)}v$
for each $x:A$ and $v:Q(x)$. Hence $L(\fib{\total f}{\pairr{x,v}})$ is contractible if and only if
$L(\fib{f(x)}v)$ is contractible.
\end{proof}

Of course, the ``if'' direction of \cref{prop:nconn_fiber_to_total} is a special case of \cref{lem:nconnected_postcomp_variation}.
This suggests a similar generalization of the ``only if'' direction of \cref{prop:nconn_fiber_to_total}, which would be a version of \cref{lem:nconnected_postcomp_variation} asserting that if $f$ and $\varphi$ are $L$-connected then so is each $g_a$.

\section{Modalities}\label{sec:modal-refl-subun}

In this section we will introduce the following four notions of modality
and prove that they are all equivalent:
\begin{enumerate}
\item Higher modalities
\item Uniquely eliminating modalities
\item $\Sigma$-closed reflective subuniverses
\item Stable orthogonal factorization systems
\end{enumerate}
After their equivalence has been established, we will call all of them simply \emph{modalities}.

The first three definitions have the following data in common: by a \define{modal operator} we mean a function $\modal:\UU\to\UU$, and by a \define{modal unit} we mean a family of functions $\modalunit^\modal:\prd*{A:\UU}A\to\modal A$.
Given these data, we say a type $X$ \define{is modal} if $\modalunit[X]:X\to\modal X$ is an equivalence, and we write $\UU_\modal \defeq \sm{X:\UU} \ismodal(X)$ for the \define{subuniverse of modal types}.

\begin{defn}\label{defn:highermod}
A \define{higher modality} consists of a modal operator and modal unit together with
\begin{enumerate}
\item for every $A:\UU$ and every dependent type $P:\modal A\to\UU$, a
function
\begin{equation*}
\mathsf{ind}_{\modal A}:\big(\prd{a:A}\modal(P(\eta(a)))\big)\to\prd{z:\modal A}\modal(P(z)).
\end{equation*}
\item An identification
\begin{equation*}
\mathsf{comp}_{\modal A}(f,x):\id{\mathsf{ind}_{\modal A}(f)(\eta(x))}{f(x)}
\end{equation*}
for each $f:\prd{x:A}\modal(P(\eta(x)))$.
\item For any $x,y:\modal A$ the modal unit $\modalunit[(\id{x}{y})]:\id{x}{y}\to \modal(\id{x}{y})$ is an equivalence.
\end{enumerate}
\end{defn}

\begin{defn}\label{defn:modunique}
A \define{uniquely eliminating modality} consists of
a modal operator and modal unit such that the function
\begin{equation*}
\lam{f} f\circ\modalunit[A] : (\prd{x:\modal A}\modal(P(x)))\to(\prd{a:A}\modal(P(\modalunit[A](a))))
\end{equation*}
is an equivalence for any $A$ and any $P:\modal A\to\UU$.
\end{defn}

\begin{defn}\label{defn:ssrs}
A reflective subuniverse $L$ is said to be \define{$\Sigma$-closed} if $\sm{x:X}P(x)$ is $L$-local for every family $P:X\to \UU_L$ of $L$-local types over an $L$-local type $X$.
\end{defn}

Note that unlike \cref{defn:highermod,defn:modunique}, in \cref{defn:ssrs} the notion of ``modal type'' is part of the data.
However, we will show in \cref{thm:subuniv-modal} that $\ismodal(A)$ if and only if $\modalunit[A]$ is an equivalence.

\begin{defn}\label{defn:sofs}
An \define{orthogonal factorization system} consists of
predicates $\mathcal{L},\mathcal{R}:\prd*{A,B:\UU} (A\to B)\to\prop$ such that
\begin{enumerate}
\item $\mathcal{L}$ and $\mathcal{R}$ are closed under composition and contain all identities (i.e.\ they are subcategories of the category of types that contain all the objects), and
\item the type $\fact_{\mathcal{L},\mathcal{R}}(f)$ of factorizations
\begin{equation*}
\begin{tikzcd}
A \arrow[rr,"f"] \arrow[dr,swap,"f_{\mathcal{L}}"] & & B \\
& \im_{\mathcal{L},\mathcal{R}}(f) \arrow[ur,swap,"f_{\mathcal{R}}"]
\end{tikzcd}
\end{equation*}
of $f$, with $f_{\mathcal{L}}$ in $\mathcal{L}$ and $f_{\mathcal{R}}$ in $\mathcal{R}$, is contractible.
\end{enumerate}
More precisely, the type $\fact_{\mathcal{L},\mathcal{R}}(f)$ is defined to
be the type of
tuples
\begin{equation*}
(\im_{\mathcal{L},\mathcal{R}}(f),(f_{\mathcal{L}},p),(f_{\mathcal{R}},q),h)
\end{equation*}
consisting of a type $\im_{\mathcal{L},\mathcal{R}}(f)$, a function $f_{\mathcal{L}}:A\to \im_{\mathcal{L},\mathcal{R}}(f)$ with
$p:\mathcal{L}(f_{\mathcal{L}})$, a function $f_{\mathcal{R}}:\im_{\mathcal{L},\mathcal{R}}(f)\to B$ with $q:\mathcal{R}(f_{\mathcal{R}})$, and an identification $h:\id{f}{f_{\mathcal{R}}\circ f_{\mathcal{L}}}$. The type $\im_{\mathcal{L},\mathcal{R}}(f)$ is called
the \define{$(\mathcal{L},\mathcal{R})$-image of $f$}.

A type $X$ is said to be \define{$(\mathcal{L},\mathcal{R})$-modal} if
the map $!:X\to\unit$ is in $\mathcal{R}$ (and hence $!_\mathcal{L}$
is an equivalence).

An orthogonal factorization system is said to be \define{stable} if the class
$\mathcal{L}$ is stable under pullbacks (By
\autoref{lem:ofs_rightstable}, $\mathcal{R}$ is always stable under pullbacks).
\end{defn}

\begin{rmk}
  By univalence, the fact that $\mathcal{L}$ and $\mathcal{R}$ contain all identities implies that they each contain all equivalences.
  Conversely, if $f\in \mathcal{L}\cap\mathcal{R}$, then $(\idfunc,f)$ and $(f,\idfunc)$ are both $(\mathcal{L},\mathcal{R})$-factorizations of $f$, and hence equal; which implies that $f$ is an equivalence.
  Thus, $\mathcal{L}\cap\mathcal{R}$ consists exactly of the equivalences.
\end{rmk}

We now consider a few examples.
Since we will eventually prove all the definitions to be equivalent, we can use any one of them to describe any particular example.

\begin{eg}
  The prime example is the \textbf{$n$-truncation modality} $\truncf n$ as studied in~\cite[Chapter 7]{hottbook}.
  This can be given as a higher modality, using its induction principle and the fact that $\trunc n A$ is an $n$-type and the identity types of an $n$-type are again $n$-types (indeed, $(n-1)$-types).
  The corresponding stable orthogonal factorization system, consisting of $n$-connected and $n$-truncated maps, is also constructed in~\cite[Chapter 7]{hottbook}; our construction in \cref{thm:sofs_from_ssrs} will be a generalization of this.
\end{eg}

\begin{eg}\label{eg:open}
  Let $Q$ be a mere proposition.
  The \textbf{open modality} determined by $Q$ is defined by $\open Q A = (Q\to A)$, with unit $\modalunit[A](x) = \lam{\nameless}x : A \to (Q \to A)$.
  To show that this is a higher modality, suppose we have $P: (Q\to A) \to \UU$ and $f:\prd{a:A} Q \to P(\lam{\nameless} a)$.
  Then for any $z:Q\to A$ and $q:Q$ we have $f(z(q),q) : P(\lam{\nameless} z(q))$.
  And since $Q$ is a mere proposition, we have $z(q) = z(q')$ for any $q':Q$, hence $e(z,q) : (\lam{\nameless} z(q)) = z$ by function extensionality.
  This gives
  \[ \lam{z}{q} \trans{e(z,q)}{(f(z(q),q))} : \prd{z:Q\to A} Q \to P(z) \]
  For the computation rule, we have
  \begin{align*}
    (\lam{z}{q} \trans{e(z,q)}{(f(z(q),q))})(\lam{\nameless} a) &= \lam{q} \trans{e(\lam{\nameless} a,q)}{(f(a,q))}\\
    &= \lam{q} f(a,q) = f(a)
  \end{align*}
  by function extensionality, since $e(\lam{\nameless} a,q) = \refl{}$.
  Finally, if $x,y:Q\to A$, then $(x=y) \simeq \prd{q:Q} x(q) = y(q)$, and the map
  \[ \Big(\prd{q:Q} x(q) = y(q)\Big) \to \Big( Q \to \prd{q:Q} x(q) = y(q)\Big) \]
  is (by currying) essentially precomposition with a product projection $Q\times Q\to Q$, and that is an equivalence since $Q$ is a mere proposition.
\end{eg}

\begin{eg}\label{eg:closed}
  Again, let $Q$ be a mere proposition.
  The \textbf{closed modality} determined by $Q$ is defined by $\closed Q A = Q \ast A$, the \emph{join} of $Q$ and $A$ (the pushout of $Q$ and $A$ under $Q\times A$).
  We show that this is a $\Sigma$-closed reflective subuniverse.
  Define a type $B$ to be modal if $Q \to \iscontr(B)$, and note that it is indeed the case that $Q\to\iscontr(Q\ast A)$, for any type $A$.
  By the universal property of pushouts, a map $Q \ast A \to B$ consists of a map $f:A\to B$ and a map $g:Q\to B$ and for any $a:A$ and $q:Q$ an identification $p:f(a)=g(q)$.
  But if $Q \to \iscontr(B)$, then $g$ and $p$ are uniquely determined, so this is just a map $A\to B$.
  Thus $(\closed Q A \to B) \to (A\to B)$ is an equivalence, so we have a reflective subuniverse.
  It is $\Sigma$-closed since the dependent sum of a contractible family of types over a contractible base is contractible.
\end{eg}

\begin{eg}\label{eg:dneg}
  The \textbf{double negation modality} is defined by $A\mapsto \neg\neg A$, i.e.\ $(A\to \emptyt)\to \emptyt$, with $\modalunit(a) = \lam{g} g(a)$.
  We show that this is a uniquely eliminating modality.
  Since the map $\lam{f}f\circ \modalunit[A]$ that must be an equivalence has mere propositions as domain and codomain, it suffices to give a map in the other direction.
  Thus, let $P: \neg\neg A \to \UU$ and $f:\prd{a:A} \neg \neg P(\lam{g} g(a))$; given $z:\neg\neg A$ we must derive a contradiction from $g:\neg P(z)$.
  Since we are proving a contradiction, we can strip the double negation from $z$ and assume given an $a:A$.
  And since $\neg\neg A$ is a mere proposition, we have $z = \lam{g} g(a)$, so that we can transport $f(a)$ to get an element of $\neg\neg P(z)$, contradicting $g$.
\end{eg}

\begin{eg}
  The \textbf{trivial modality} is the identity function on $\UU$.
  It coincides with $\open \top$ and with $\closed\bot$.

  Dually, the \textbf{zero modality} sends all types to $\unit$.
  It is equivalently the $(-2)$-truncation, and coincides with $\open\bot$ and with $\closed \top$.
\end{eg}

\paragraph*{Summary.}
In each of \autoref{defn:highermod,defn:modunique,defn:ssrs,defn:sofs}
we have defined what it means for a type to be modal. In each case, being
modal is a family of mere propositions indexed by the universe, i.e.~a subuniverse.
We will show in \autoref{thm:subuniv-highermod,thm:subuniv-modunique,thm:subuniverse-rs,thm:subuniv-sofs} that each kind of structure is completely determined by this subuniverse.
(\autoref{thm:subuniverse-rs} is more general, not requiring $\Sigma$-closedness.)

It follows that the type of all modalities of each
kind is a subset of the set $\UU\to\prop$ of all subuniverses, and in particular is a set.
This makes it easier to establish
the equivalences of the different kinds of modalities.
It suffices
to show that any modality of one kind determines a modality of the next kind
with the same modal types, which we will do as follows:
\begin{center}
\begin{tikzcd}
  & \text{higher modality} \ar[dr,bend left,"\text{\autoref{thm:modunique_from_highermod}}"] \\
  \parbox{3cm}{\centering stable factorization system} \ar[ur,bend left,"\text{\autoref{thm:highermod_from_sofs}}"] &&
  \parbox{3cm}{\centering uniquely eliminating modality} \ar[dl,bend left,"\text{\autoref{thm:ssrs_from_modunique}}"] \\
  & \parbox{3cm}{\centering $\Sigma$-closed reflective subuniverse} \ar[ul,bend left,"\text{\autoref{thm:sofs_from_ssrs}}"]
\end{tikzcd}
\end{center}
Before \autoref{thm:sofs_from_ssrs} we take the opportunity to develop a bit more theory of reflective subuniverses, including closure under identity types (\autoref{lem:rs_idstable}) and dependent products
(\autoref{lem:modal-Pi}), along with several equivalent characterizations of $\Sigma$-closedness (\autoref{thm:ssrs-characterize}).

Of these equivalences, the most surprising is that a stable factorization system is uniquely determined by its underlying reflective subuniverse of types.
This is false for stable factorization systems on arbitrary categories; the reason it holds here is that we are talking \emph{in type theory} about factorization systems \emph{on the category of types}.
An analogous fact is true in classical set-based mathematics for stable factorization systems on the category of sets (although in that case there are much fewer interesting examples).
We \cite{RijkeShulmanSpitters} we also observe that when type theory is interpreted in a higher category, the data of a reflective subuniverse or modality has to be interpreted ``fiberwise'', giving a richer structure than a single reflective subcategory.

\subsection{Higher modalities}
\label{sec:higher-modalities}

We start by showing that a higher modality is determined by its modal types, and gives rise to a uniquely eliminating modality.

\begin{lem}
If $\modal$ is a higher modality, then any type of the form $\modal X$ is modal.
\end{lem}

\begin{proof}
  We want to show that the modal unit $\modalunit[\modal X]:\modal X\to\modal\modal X$
is an equivalence. By the induction principle and the computation rule for
higher modalities, we find a function $f:\modal \modal X\to\modal X$ with
the property that $f\circ \modalunit[\modal X]\htpy\idfunc[\modal X]$. We wish to
show that we also have $\modalunit[\modal X]\circ f\htpy\idfunc$. Since identity
types of types of the form $\modal Y$ are declared to be modal, it is
equivalent to find a term of type
\begin{equation*}
\prd{x:\modal \modal X}\modal(\modalunit[\modal X](f(x))=x).
\end{equation*}
Now we are in the position to use the induction principle of higher modalities
again, so it suffices to show that $\modalunit(f(\modalunit(x)))=\modalunit(x)$
for any $x:\modal X$. This follows from the fact that $f\circ\modalunit=\idfunc$.
\end{proof}

\begin{thm}\label{thm:subuniv-highermod}
The data of two higher modalities $\modal$ and $\modal'$
are identical if and only if they have the same modal types.
\end{thm}

\begin{proof}
Another way of stating this, is that the function from the type of \emph{all}
modalities on $\UU$ to the type $\UU\to\prop$ of predicates on $\UU$, given
by mapping a modality to the predicate $\ismodal$, is an embedding. Thus, we
need to show that for any predicate $\mathcal{M}:\UU\to\prop$, we can find at
most one modality for which $\mathcal{M}$ is the class of modal types. This
follows, once we demonstrate that,
\begin{quote}
for any $\mathcal{M}:\UU\to\prop$ closed under identity types,
and for any type $X$, the type of tuples $(Y,p,\pi,I,C)$ ---
consisting of a type $Y$ with $p$ witnessing that $Y$
satisfies $\mathcal{M}$, a function $\pi:X\to Y$, a function
\begin{equation*}
I_P:(\prd{x:X} P(\pi(x)))\to(\prd{y:Y} P(y))
\end{equation*}
for every $P:Y\to\UU_{\mathcal{M}}$,
which is a right inverse of precomposing with $\pi$, as is witnessed by the
term $C$ --- is a mere proposition.
\end{quote}

We prove this statement in two parts. First, we show that the
type of pairs $(I,C)$, with $I$ and $C$ of the indicated types, is a mere
proposition for any $(Y,p,\pi)$. After that, we show that the type of triples
$(Y,p,\pi)$ is also a mere proposition. These two facts combined prove the
statement.

Consider a type $Y$ satisfying $\mathcal{M}$, and a function $\pi:X\to Y$, and
let $(I,C)$ and $(I',C')$ be two terms witnessing that $Y$ satisfies an induction
principle with a computation rule. We want to show that $(I,C)=(I',C')$, and of
course it suffices to show that $(I(s),C(s))=(I'(s),C(s))$ for any
$P:Y\to\UU_{\mathcal{M}}$ and $s:\prd{x:X}P(\pi(x))$.

To show that $I(s,y)=I'(s,y)$ for any $y:Y$, we use
the induction principle $(I,C)$. So it suffices to show that
$I(s,\pi(x))=I'(s,\pi(x))$. Both of these terms are equal to $s(x)$. Thus,
we obtain a proof $J(s,y)$ that $I(s,y)=I'(s,y)$, with the property that
$J(s,\pi(x))=\ct{C(s,x)}{\inv{C'(s,x)}}$.
Now we need to show that $\trans{J(s)}{C(s)}=C'(s)$, which is equivalent
to the property we just stated. This finishes the proof that the type of
the induction principle and computation rule is a mere proposition.

It remains to show that $(Y,\pi)=(Y',\pi')$, provided that $Y$ and $Y'$ are both
in $\mathcal{M}$, and that both sides satisfy
the induction principle and computation rule. It suffices to find an equivalence
$f:Y\to Y'$ such that $f\circ \pi=\pi'$.

From the induction principles of $Y$ resp. $Y'$, we obtain a function
$f:Y\to Y'$ with the property that $f\circ \pi=\pi'$, and a function
$f':Y'\to Y$ with the property that $f'\circ \pi'=\pi$.
To show that $f'\circ f=\idfunc$ we use the induction principle
of $Y$. Since the type $f'(f(y))=y$ is in $\mathcal{M}$, it suffices to show that
$f'(f(\pi(y)))=\pi(y)$. This readily follows from the defining properties of $f$
and $f'$. Similarly, we have $f\circ f'=\idfunc$.
\end{proof}

\begin{thm}\label{thm:modunique_from_highermod}
A higher modality is a uniquely eliminating modality, with the
same modal types.
\end{thm}

\begin{proof}
Let $\modal$ be a modality with modal units $\modalunit[A]$. Our goal is to show
that the pre-composition map
\begin{equation*}
\lam{s}s\circ\modalunit[A]:(\prd{x:\modal A}\modal(P(x)))\to(\prd{a:A}\modal(P(\modalunit[A](a))))
\end{equation*}
is an equivalence for each $A:\UU$ and $P:\modal A\to\UU$.
By the given induction principle and computation rule, we obtain a
right inverse $\mathsf{ind}_{\modal A}$ of $\blank\circ\modalunit[A]$.

To show that it is a left inverse, consider $s:\prd{x:\modal A}\modal(P(x))$.
We need to find a homotopy
\begin{equation*}
\prd{x:\modal A}\id{s(x)}{\mathsf{ind}_{\modal A}(s\circ \modalunit_A)(x)}.
\end{equation*}
By assumption we have that $P(x)$ is
modal for each $z:\modal A$ and hence it follows that $\id{s(x)}{\mathsf{ind}_{\modal A}(s\circ \modalunit_A)(x)}$
is modal for each $x$. Hence it suffices to find a function of type
\begin{equation*}
\prd{a:A}\id{s(\modalunit_A(a))}{\mathsf{ind}_{\modal A}(s\circ \modalunit_A)(\modalunit_A(a))}.
\end{equation*}
This follows straight from the computation rule of higher modalities.
\end{proof}

\subsection{Uniquely eliminating modalities}
\label{sec:uniq-elim}

Next, we show that a uniquely eliminating modality is determined by its modal types, and gives rise to a $\Sigma$-closed reflective subuniverse.

\begin{lem}
Given a uniquely eliminating modality, $\modal X$ is modal for any type $X$.
\end{lem}

\begin{proof}
Using the elimination principle of $\modal \modal X$, we find a function
$f:\modal \modal X\to\modal X$ and an identification $f\circ\modalunit[\modal X]=\idfunc[\modal X]$.
By the uniqueness property, the type
\begin{equation*}
\sm{g:\modal \modal X\to\modal \modal X} g\circ\modalunit[\modal X]=\modalunit[\modal X]
\end{equation*}
is contractible. Since both $\idfunc[\modal \modal X]$ and $\modalunit[\modal X]\circ f$
are in this type (with suitable identifications), we find that $f$ is also the
right inverse of $\modalunit[\modal X]$. This shows that $\modalunit[\modal X]$ is an
equivalence, so $\modal X$ is modal.
\end{proof}

\begin{thm}\label{thm:subuniv-modunique}
The data of two uniquely eliminating modalities $\modal$ and $\modal'$ are equivalent if and only if both have the same modal types.
\end{thm}

\begin{proof}
We need to show that the type of uniquely eliminating modalities
with a given class $\mathcal{M}:\UU\to\prop$ of modal types
is a mere proposition. Since the types of the form $\modal X$ are modal,
it suffices to show for any class $\mathcal{M}
:\UU\to\prop$ and any type $X$, that
\begin{quote}
the type of tuples $(Y,p,\pi,H)$ --- consisting of a type $Y$
with $p$ witnessing that $Y$ is in $\mathcal{M}$, a function $\pi:X\to Y$, and
for each $P:Y\to\UU$ a term $H_P$ witnessing that the function
\begin{equation*}
\lam{s}s\circ \pi:(\prd{y:Y}\modal(P(y)))\to(\prd{x:X}\modal(P(\pi(x))))
\end{equation*}
is an equivalence --- is a mere proposition.
\end{quote}
Let $(Y,p,\pi,H)$ and $(Y',p',\pi',H')$ be such tuples. To show that they are
equal, it suffices to show that $(Y,\pi)=(Y',\pi')$ because the other things
in the list are terms of mere propositions. Furthermore, showing that
$(Y,\pi)=(Y',\pi')$ is equivalent to finding an equivalence $f:\eqv{Y}{Y'}$ with
the property that $f\circ\pi=\pi'$. By $H$, there is such a function, and by
$H'$ there is a function $f':Y'\to Y$ such that $f'\circ\pi'=\pi$. Now the
uniqueness gives that $f'\circ f$ is the only function from $Y$ to $Y$ such
that $f'\circ f\circ\pi=\pi$ and of course $\idfunc[Y]$ is another such function.
Therefore it follows that $f'\circ f=\idfunc$, and similarly it follows that
$f\circ f'=\idfunc$.
\end{proof}

\begin{thm}\label{thm:ssrs_from_modunique}
Any uniquely eliminating modality determines a $\Sigma$-closed reflective
subuniverse with the same modal types.
\end{thm}

\begin{proof}
It is immediate from the definition of uniquely eliminating modalities
that every map $f:A\to B$ into a modal type $B$ has a homotopy unique extension to $\modal A$
along the modal unit:
\begin{equation*}
\begin{tikzcd}
A \arrow[dr,"f"] \arrow[d,swap,"\modalunit_A"] \\ \modal A \arrow[r,densely dotted,swap,"\tilde f"] & B
\end{tikzcd}
\end{equation*}
Since the types of the form $\modal X$ are modal, we obtain a reflective subuniverse.
It remains to verify  that the type $\sm{x:\modal X}\modal(P(x))$ is modal for
any type $X$ and $P:X\to\UU$. We have the function
\begin{equation*}
\varphi\defeq\lam{m}\pairr{f(m),g(m)}:\modal(\sm{x:\modal X}\modal(P(x)))\to\sm{x:\modal X}\modal(P(x)),
\end{equation*}
where
\begin{align*}
f & \defeq \ind{\modal}(\lam{x}{u} x) & & : \modal(\sm{x:\modal X}\modal(P(x)))\to \modal X \\
g & \defeq \ind{\modal}(\lam{x}{u} u) & & : \prd{w:\modal(\sm{x:\modal X}\modal(P(x)))} \modal(P(f(w)))
\end{align*}
Our goal is to show that $\varphi$ is an inverse to the modal unit.

Note that
\begin{equation*}
\varphi(\modalunit(x,y)) \jdeq \pairr{f(\modalunit(x,y)),g(\modalunit(x,y))} \jdeq \pairr{x,y},
\end{equation*}
so we see immediately that $\varphi$ is a left inverse of $\modalunit$.

To show that $\varphi$ is a right inverse of $\modalunit$, note that the type
of functions $h$ fitting in a commuting square of the form
\begin{equation*}
\begin{tikzcd}[column sep=-3em]
\modal(\sm{x:\modal X}\modal(P(x))) \arrow[rr,densely dotted,"h"] & & \modal(\sm{x:\modal X}\modal(P(x))) \\
& \sm{x:\modal X}\modal(P(x)) \arrow[ul,"\modalunit"] \arrow[ur,swap,"\modalunit"]
\end{tikzcd}
\end{equation*}
is contractible, and it contains the identity function. Therefore, it suffices
to show that $(\modalunit\circ\varphi)\circ\modalunit=\modalunit$, but this follows
from the fact that $\varphi$ is a left inverse of the modal unit.
\end{proof}

\subsection{\texorpdfstring{$\Sigma$}{Σ}-closed reflective subuniverses}
\label{sec:ssrs}

Now we study reflective subuniverses in a bit more detail, and end by
showing that $\Sigma$-closed ones give rise to stable factorization
systems. $\Sigma$-closure is used in \autoref{thm:sofs_from_ssrs} to
show that left maps and right maps are closed under composition.

\subsubsection{\texorpdfstring{$\Sigma$}{Σ}-closed reflective subuniverses}
\label{sec:sigma-closed}

\begin{defn}\label{defn:connected}
Let $\mathcal{M}:\UU\to\prop$ be a reflective subuniverse with modal
operator $\modal$. We say
that a type $X$ is \define{$\modal$-connected} if $\modal X$ is contractible,
and we say that a function $f:X\to Y$ is \define{$\modal$-connected} if each
of its fibers is. Similarly, we say that $f$ is \define{modal} if each of its
fibers is.
\end{defn}

Note that a type $X$ is modal or $\modal$-connected just when the map $X\to\unit$ is.

\begin{eg}\label{eg:closed-connected}
  Recall from \cref{eg:open} that the open modality associated to a proposition $Q$ is defined by $\open Q(A) \defeq (Q\to A)$.
  We claim that $A$ is $\open Q$-connected if and only if $Q \to\iscontr(A)$.
  In other words, $(Q \to\iscontr(A))\eqvsym \iscontr(Q\to A)$.
  For on the one hand, if $Q\to \iscontr(A)$, then $Q\to A$; while any two $f,g:Q\to A$ can be shown equal by function extensionality, since if $Q$ then $A$ is contractible.
  But on the other hand, if $\iscontr(Q\to A)$ and $Q$, then $\eqv{(Q\to A)}{A}$, hence $\iscontr(A)$.

  Note that $Q \to\iscontr(A)$ is also the defining condition for the $\closed Q$-modal types from \cref{eg:closed}.
  That is, the $\open Q$-connected types coincide with the $\closed Q$-modal types.
\end{eg}

The following theorem combines Lemma 7.5.7 and Theorem 7.7.4 of \cite{hottbook}.

\begin{thm}\label{thm:ssrs-characterize}
Given a reflective universe with modal operator $\modal$,
the following are equivalent:
\begin{enumerate}
\item It is $\Sigma$-closed.
\item It is uniquely eliminating.
\item The modal units are $\modal$-connected.
\end{enumerate}
\end{thm}

\begin{proof}
Suppose first that $\modal$ is $\Sigma$-closed, let $X$ be a type and let
$P:\modal X\to\UU_\modal$, i.e.\ $P(x)$ is modal for each $x:\modal X$.
To show that $\modal$ is uniquely eliminating, we want
\begin{equation*}
\lam{s}s\circ\modalunit[X]:(\prd{x:\modal X}P(x))\to(\prd{x:X}P(\modalunit(x)))
\end{equation*}
to be an equivalence. Since the type $\prd{a:A}B(a)$ is equivalent to the type
of functions $f:A\to\sm{a:A}B(a)$ such that $\proj1\circ f=\idfunc[A]$, we
get the desired equivalence if the pre-composition map $\lam{j}j\circ\modalunit$
gives an equivalence from diagonal fillers of the square
\begin{equation*}
\begin{tikzcd}
X \arrow[r,densely dotted] \arrow[d,swap,"\modalunit"] & \sm{x:\modal X}P(x) \arrow[d,"\proj1"] \\
\modal X \arrow[ur,densely dotted,"j"] \arrow[r,equals] & \modal X
\end{tikzcd}
\end{equation*}
to the type of maps $X\to\sm{x:\modal X}P(x)$ such that the indicated square
commutes.  But this is true by the universal property of $\modalunit$, since $\sm{x:\modal X}P(x)$ is modal by $\Sigma$-closedness.

Now suppose that $\modal$ is uniquely eliminating.
To show that the modal units are connected, we want a term of type
\begin{equation*}
\prd{x:\modal X}\iscontr(\modal(\fib{\modalunit}{x})).
\end{equation*}
Using the dependent eliminators, it is easy to find a term
$s:\prd{x:\modal X}\modal(\fib{\modalunit}{x})$ with the property that
$s\circ\modalunit(x)=\modalunit(x,\refl{\modalunit(x)})$. Now we need to show
that
\begin{equation*}
\prd{x:\modal X}{w:\modal(\fib{\modalunit}{x})}w=s(x).
\end{equation*}
Since the type $w=s(x)$ is modal, this is equivalent to
\begin{equation*}
\prd{x:\modal X}{x':X}{p:\modalunit(x')=x} \modalunit(x',p)=s(x).
\end{equation*}
Moreover, the type $\sm{x:\modal X}\modalunit(x')=x$ is contractible, so this
is equivalent to
\begin{equation*}
\prd{x':X} \modalunit(x',\refl{\modalunit(x')})=s(\modalunit(x')),
\end{equation*}
of which we have a term by the defining property of $s$.

Finally, suppose that all the modal units are $\modal$-connected, let $X$ be modal and let $P:X\to\UU_\modal$.
To show that $\sm{x:X}P(x)$ is modal, we show that
$\modalunit:(\sm{x:X}P(x))\to\modal(\sm{x:X}P(x))$ is an equivalence.
Since $X$ is modal, we can extend $\proj 1$ along $\modalunit$ as indicated
in the diagram
\begin{equation*}
\begin{tikzcd}
\sm{x:X}P(x) \arrow[d,"\modalunit"] \arrow[dr,"\proj 1"] \\
\modal(\sm{x:X}P(x)) \arrow[r,densely dotted,swap,"p"] & X
\end{tikzcd}
\end{equation*}
The type of maps
\begin{equation*}
f:\modal(\sm{x:X}P(x))\to \sm{x:X}P(x)
\end{equation*}
such that $\proj1\circ f=p$ is equivalent to the type $\prd{z:\modal(\sm{x:X}P(x))}P(p(z))$.
Using the assumption that $\modalunit$ is connected, we calculate
\begin{align*}
\prd{z:\modal(\sm{x:X}P(x))}P(p(z))
& \eqvsym \prd{z:\modal(\sm{x:X}P(x))} \modal(\fib{\modalunit}{z})\to P(p(z)) \\
& \eqvsym \prd{z:\modal(\sm{x:X}P(x))} \fib{\modalunit}{z}\to P(p(z)) \\
& \eqvsym \prd{\pairr{x,y}:\sm{x:X}P(x)} P(x)
\end{align*}
We have the second projection $\proj 2$ of the latter type. We obtain a term
\begin{equation*}
q : \prd{z:\modal(\sm{x:X}P(x))}P(p(z))
\end{equation*}
such that $q(\eta(x,y))=y$. Therefore, we get the map $\pairr{p,q}:\modal(\sm{x:X}P(x))\to \sm{x:X}P(x)$ for which the diagram
\begin{equation*}
\begin{tikzcd}
\sm{x:X}P(x) \arrow[dr,"\modalunit"] \arrow[ddr,bend right=15,swap,"\proj 1"] \arrow[rr,"\idfunc"] & & \sm{x:X}P(x) \arrow[ddl,bend left=15,"\proj 1"] \\
& \modal(\sm{x:X}P(x)) \arrow[d,swap,"p"] \arrow[ur,densely dotted,"\pairr{p,q}"] \\
& X
\end{tikzcd}
\end{equation*}
commutes. In particular, $\pairr{p,q}$ is a left inverse of the modal unit.
To see that it is also a right inverse, note that $\modalunit\circ\pairr{p,q}\circ\modalunit=\modalunit=\idfunc\circ\modalunit$; thus $\modalunit\circ\pairr{p,q}=\idfunc$ follows by uniqueness.
Hence $\sm{x:X}P(x)$ is modal.
\end{proof}

\begin{thm}\label{thm:sofs_from_ssrs}
A $\Sigma$-closed reflective subuniverse determines a stable orthogonal factorization system with the same
modal types.
\end{thm}

\begin{proof}
Define $\mathcal{L}$ to be the class of $\modal$-connected
maps and $\mathcal{R}$ to be the class of modal maps.
We first show that both $\mathcal{L}$ and $\mathcal{R}$ are closed under
composition. Recall that for $f:X\to Y$ and $g:Y\to Z$, one has
$\fib{g\circ f}{z}=\sm{p:\fib{g}{z}}\fib{f}{\proj1(p)}$.
Thus, by $\Sigma$-closedness, if $f$ and $g$ are both in $\mathcal{R}$ then so is $g\circ f$, so $\cR$ is closed under composition; while \cref{thm:rsu-compose-cancel} implies that $\cL$ is closed under composition.
And since the fibers of an identity map are contractible, and contractible types are both modal and $\modal$-connected, both $\mathcal{L}$ and $\mathcal{R}$ contain all identities.

To obtain a factorization system,
it remains to show that the type of
$(\mathcal{L},\mathcal{R})$-factorizations of any function $f$ is contractible.
Since $\pairr{X,f}=\pairr{\sm{y:Y}\fib{f}{y},\proj1}$, it is sufficient to
show that $\fact_{\mathcal{L},\mathcal{R}}(\proj1)$ is contractible for any
$\proj1:\sm{y:Y}P(y)\to Y$. But $\proj1$ factors as
\begin{equation*}
\begin{tikzcd}
\sm{y:Y}P(y) \arrow[r,"p_\mathcal{L}"] & \sm{y:Y}\modal(P(y)) \arrow[r,"p_\mathcal{R}"] & Y
\end{tikzcd}
\end{equation*}
where $p_\mathcal{L}\defeq\total{\modalunit[P(\blank)]}$ and $p_\mathcal{R}\defeq\proj1$.
The fibers of $p_\mathcal{R}$ are $\modal(P(\blank))$, so it follows
immediately that $p_\mathcal{R}$ is in $\mathcal{R}$.
Moreover, since
$\eqv{\fib{\total{\modalunit}}{\pairr{y,u}}}{\fib{\modalunit[P(y)]}{u}}$ and each $\modalunit$ is $\modal$-connected, it follows that $p_\mathcal{L}$ is in
$\mathcal{L}$.

Now consider any other factorization $(g,h,H)$ of $\proj1$ into
an $\cL$-map $g:(\sm{y:Y}P(y))\to I$ followed by an $\cR$-map $h:I\to Y$. Since
$I=\sm{y:Y}\fib{h}{y}$, we have a commuting square
\begin{equation*}
\begin{tikzcd}
\sm{y:Y}P(y) \arrow[r,"g"] \arrow[d,swap,"{\total{\gamma}}"]
  & I \arrow[d,"h"] \\
\sm{y:Y}\fib{h}{y} \arrow[ur,equals] \arrow[r,swap,"\proj1"] & Y
\end{tikzcd}
\end{equation*}
in which $\gamma(y,u)\defeq \pairr{g(y,u),H(y,u)}$.
It follows that $(g,h,H)=(\total{\gamma},\proj1,\nameless)$.
Thus suffices to show that there is a commuting triangle
\begin{equation*}
\begin{tikzcd}[column sep=0]
& P(y) \arrow[dl,swap,"\modalunit"] \arrow[dr,"{\gamma_y}"] \\
\modal(P(y)) \arrow[rr,equals] & & \fib{h}{y}
\end{tikzcd}
\end{equation*}
We will do this using \cref{lem:reflective_uniqueness}, by showing that $\gamma_y$ has the same universal property as $\modalunit[P(y)]$.
This follows from the following calculation:
\begin{align*}
(\fib{h}{y}\to Z) & \eqvsym ((\sm{w:\fib{h}{y}}\modal(\fib{g}{\proj1(w)}))\to Z) \\
& \eqvsym ((\sm{w:\fib{h}{y}}\fib{g}{\proj1(w)})\to Z) \\
& \eqvsym (\fib{h\circ g}{y}\to Z) \\
& \eqvsym (P(y)\to Z).
\end{align*}
which we can verify is given by precomposition with $\gamma_y$.

It remains to show that our orthogonal factorization system is stable. Consider a pullback diagram
\begin{equation*}
\begin{tikzcd}
A' \arrow[d,swap,"k"] \arrow[r,"f"] & A \arrow[d,"l"] \\
B' \arrow[r,swap,"g"] & B
\end{tikzcd}
\end{equation*}
in which $l$ is in $\mathcal{L}$. By the pasting lemma for pullbacks, it
follows that $\fib{k}{b}=\fib{l}{g(b)}$ for each $b:B'$. Thus, it follows that
$k$ is in $\mathcal{L}$.
\end{proof}

\subsection{Stable orthogonal factorization systems}

To complete \cref{sec:modal-refl-subun}, we will show that stable orthogonal factorization systems are also determined by their modal types, and give rise to higher modalities.

\subsubsection{Orthogonal factorization systems}

In classical category theory, orthogonal factorization systems are equivalently characterized by a unique lifting property.
We begin with the analogue of this in our context.

\begin{defn}
Let $(\mathcal{L},\mathcal{R})$ be an orthogonal factorization system, and
consider a commutative square
\begin{equation*}
\begin{tikzcd}
A \arrow[r,"f"] \arrow[d,swap,"l"] \ar[dr,phantom,"\scriptstyle S"] & X \arrow[d,"r"] \\
B \arrow[r,swap,"g"] & Y
\end{tikzcd}
\end{equation*}
(i.e.\ paths $S : r\circ f = g\circ l$)
for which $l$ is in $\mathcal{L}$ and $r$ is in $\mathcal{R}$. We define
$\fillers S$ to be the type of \define{diagonal fillers}
of the above diagram, i.e.~the type of tuples $(j,H_f,H_g,K)$ consisting of
$j:B\to X$, $H_f:j\circ l=f$ and $H_g:r\circ j=g$ and an equality $K : r\circ H_f = \ct S{(H_g \circ l)}$.
\end{defn}

\begin{lem}\label{lem:diagonal_fillers}
Let $(\mathcal{L},\mathcal{R})$ be an orthogonal factorization system, and
consider a commutative square
\begin{equation*}
\begin{tikzcd}
A \arrow[r,"f"] \arrow[d,swap,"l"] \ar[dr,phantom,"\scriptstyle S"] & X \arrow[d,"r"] \\
B \arrow[r,swap,"g"] & Y
\end{tikzcd}
\end{equation*}
for which $l$ is in $\mathcal{L}$ and $r$ is in $\mathcal{R}$. Then the type
$\fillers S$ of diagonal fillers is contractible.
\end{lem}

\begin{proof}
By the fact that every morphism factors uniquely as a left map followed by a
right map, we may factorize $f$ and $g$ in $(\mathcal{L},\mathcal{R})$ as $H_f : f = f_\cR \circ f_\cL$ and $H_g : g = g_\cR \circ g_\cL$, obtaining the diagram
\begin{equation*}
\begin{tikzcd}
A \arrow[r,"f_{\mathcal{L}}"] \arrow[d,swap,"l"] & \im(f) \arrow[r,"f_{\mathcal{R}}"] & X \arrow[d,"r"] \\
B \arrow[r,swap,"g_{\mathcal{L}}"] & \im(g) \arrow[r,swap,"g_{\mathcal{R}}"] & Y
\end{tikzcd}
\end{equation*}
Now both $(r\circ f_{\mathcal{R}})\circ f_{\mathcal{L}}$ and
$g_{\mathcal{R}}\circ(g_{\mathcal{L}}\circ l)$ are factorizations
of the same function $r\circ f:A\to Y$.
Since $\fact_{\mathcal{L},\mathcal{R}}(r\circ f)$ is contractible, so is its identity type
\[ (\im(f), f_\cL, r\circ f_\cR, r\circ H_f) = (\im(g), g_\cL \circ l, g_\cR, \ct{S}{(H_g\circ l)}). \]
This identity type is equivalent to
\begin{multline*}
\sm{e:\im(f) \simeq \im(g)}{H_\cL : g_\cL \circ l = e\circ f_\cL}{H_\cR : r\circ f_\cR = g_\cR\circ e}\\
(\ct{(r\circ H_f)}{(H_\cR \circ f_\cL)} = \ct S{\ct{(H_g \circ l)}{(g_\cR \circ H_\cL)}})
\end{multline*}
Now since $\fact_{\cL,\cR}(f)$ and $\fact_{\cL,\cR}(g)$ are also contractible, we can sum over them to get that the following type is contractible:
\begin{multline*}
  \sm{\im(f):\UU}{f_\cL : A \to \im(f)}{f_\cR : \im(f) \to X}{H_f : f = f_\cR \circ f_\cL}\\
  \sm{\im(g):\UU}{g_\cL : B \to \im(g)}{g_\cR : \im(g) \to Y}{H_g : g = g_\cR \circ g_\cL}\\
\sm{e:\im(f) \simeq \im(g)}{H_\cL : g_\cL \circ l = e\circ f_\cL}{H_\cR : r\circ f_\cR = g_\cR\circ e}\\
(\ct{(r\circ H_f)}{(H_\cR \circ f_\cL)} = \ct S{\ct{(H_g \circ l)}{(g_\cR \circ H_\cL)}})
\end{multline*}
(omitting the hypotheses that $f_\cL,g_\cL\in\cL$ and $f_\cR,g_\cR\in\cR$).
Reassociating and removing the contractible type $\sm{\im(g):\UU}(\im(f) \simeq \im(g))$, and renaming $\im(f)$ as simply $I$, this is equivalent to
\begin{multline*}
  \sm{I:\UU}{f_\cL : A \to I}{f_\cR : I \to X}{H_f : f = f_\cR \circ f_\cL}\\
  \sm{g_\cL : B \to I}{g_\cR : I \to Y}{H_g : g = g_\cR \circ g_\cL}{H_\cL : g_\cL \circ l = f_\cL}{H_\cR : r\circ f_\cR = g_\cR}\\
(\ct{(r\circ H_f)}{(H_\cR \circ f_\cL)} = \ct S{\ct{(H_g \circ l)}{(g_\cR \circ H_\cL)}})
\end{multline*}
Removing the contractible $\sm{f_\cL : A \to I} (g_\cL \circ l = f_\cL)$ and $\sm{g_\cR : I \to Y} (r\circ f_\cR = g_\cR)$, this becomes
\begin{multline*}
  \sm{I:\UU}{f_\cR : I \to X}{g_\cL : B \to I}{H_f : f = f_\cR \circ g_\cL \circ l}{H_g : g = r\circ f_\cR \circ g_\cL}\\
(r\circ H_f = \ct S{(H_g \circ l)})
\end{multline*}
Inserting a contractible $\sm{j:B\to X} (f_\cR \circ g_\cL = j)$, and reassociating more, we get
\begin{multline*}
  \sm{j:B\to X}{I:\UU}{f_\cR : I \to X}{g_\cL : B \to I}{H_j:f_\cR \circ g_\cL = j}\\
  \sm{H_f : f = f_\cR \circ g_\cL \circ l}{H_g : g = r\circ f_\cR \circ g_\cL}
  (r\circ H_f = \ct S{(H_g \circ l)})
\end{multline*}
But now $\sm{I:\UU}{f_\cR : I \to X}{g_\cL : B \to I}{H_j:f_\cR \circ g_\cL = j}$ is just $\fact_{\cL,\cR}(j)$, hence contractible.
Removing it, we get
\begin{equation*}
  \sm{j:B\to X}{H_f : f = j \circ l}{H_g : g = r\circ j}(r\circ H_f = \ct S{(H_g \circ l)})
\end{equation*}
which is just $\fillers S$.
Therefore, this is also contractible.
\end{proof}

\begin{defn}\label{defn:orthogonal}
For any class $\mathcal{C}:\prd*{A,B:\UU}(A\to B)\to\prop$ of maps, we define
\begin{enumerate}
\item $^{\bot}\mathcal{C}$ to be the class of maps with \define{(unique) left lifting
property} with respect to all maps in $\mathcal{C}$: the mere proposition
$^\bot\mathcal{C}(l)$ asserts that for every commutative square
\begin{equation*}
\begin{tikzcd}
A \arrow[r,"f"] \arrow[d,swap,"l"] \ar[dr,phantom,"S"] & X \arrow[d,"r"] \\
B \arrow[r,swap,"g"] & Y
\end{tikzcd}
\end{equation*}
with $r$ in $\mathcal{C}$, the type $\fillers S$ of diagonal fillers is contractible.
\item $\mathcal{C}^\bot$ to be the class of maps with the dual \define{(unique) right lifting
property} with respect to all maps in $\mathcal{C}$.
\end{enumerate}
\end{defn}

\begin{lem}\label{lem:ofs_lifting}
In an orthogonal factorization system $(\mathcal{L},\mathcal{R})$, one has
$\mathcal{L}={^\bot\mathcal{R}}$ and $\mathcal{L}^\bot=\mathcal{R}$.
\end{lem}

\begin{proof}
We first show that $\mathcal{L}={^\bot\mathcal{R}}$, i.e.~we show that
$\mathcal{L}(f)\leftrightarrow {^\bot\mathcal{R}}(f)$ for any map $f$. Note
that the implication $\mathcal{L}(f)\to {^\bot\mathcal{R}}(f)$ follows from
\autoref{lem:diagonal_fillers}.

Let $f:A\to B$ be a map in ${^\bot\mathcal{R}}$.
We wish to show that $\mathcal{L}(f)$. Consider the factorization
$(f_{\mathcal{L}},f_{\mathcal{R}})$ of $f$. Then the square
\begin{equation*}
\begin{tikzcd}
A \arrow[r,"f_{\mathcal{L}}"] \arrow[d,swap,"f"] & \mathsf{im}_{\mathcal{L},\mathcal{R}}(f) \arrow[d,"f_{\mathcal{R}}"] \\
B \arrow[r,swap,"\idfunc"] & B
\end{tikzcd}
\end{equation*}
commutes. Since $f$ has the left lifting property, the type of diagonal fillers
of this square is contractible. Thus we have a section $j$ of $f_{\mathcal{R}}$.
The map $j\circ f_\mathcal{R}$ is then a diagonal filler of the square
\begin{equation*}
\begin{tikzcd}
A \arrow[r,"f_{\mathcal{L}}"] \arrow[d,swap,"f_{\mathcal{L}}"] & \mathsf{im}_{\mathcal{L},\mathcal{R}}(f) \arrow[d,"f_{\mathcal{R}}"] \\
\mathsf{im}_{\mathcal{L},\mathcal{R}}(f) \arrow[r,swap,"f_{\mathcal{R}}"] & B
\end{tikzcd}
\end{equation*}
Of course, the identity map $\idfunc[\mathsf{im}_{\mathcal{L},\mathcal{R}}(f)]$
is also a diagonal filler for this square, so the fact that the type of
such diagonal fillers is contractible implies that $j\circ f_{\mathcal{R}}=\idfunc$.
Thus, $j$ and $f_\cR$ are inverse equivalences, and so the pair $(B,f)$ is equal to the pair $(\mathsf{im}_{\mathcal{L},\mathcal{R}}(f),f_\cL)$.
Hence $f$, like $f_\cL$, is in $\cL$.

Similarly, \autoref{lem:diagonal_fillers} also implies that $\mathcal{R}(f)\to \mathcal{L}^\bot(f)$
for any map $f$, while we can prove $\mathcal{L}^\bot(f)\to\mathcal{R}(f)$ analogously to ${^\bot\mathcal{R}}(f)\to\mathcal{L}(f)$.
\end{proof}

\begin{cor}\label{lem:sofs_req}
The data of two orthogonal factorization systems $(\mathcal{L},\mathcal{R})$ and
$(\mathcal{L}',\mathcal{R}')$ are identical if and only if
$\mathcal{R}=\mathcal{R}'$.
\end{cor}
\begin{proof}
  ``Only if'' is obvious.
  Conversely, if $\mathcal{R}=\mathcal{R}'$, then by \cref{lem:ofs_lifting} we have $\cL = \cL'$, and the remaining data of an orthogonal factorization system is a mere proposition.
\end{proof}

\begin{lem}\label{lem:ofs_rightstable}
Let $(\mathcal{L},\mathcal{R})$ be an orthogonal factorization system. Then
the class $\mathcal{R}$ is stable under pullbacks.
\end{lem}

\begin{proof}
Consider a pullback diagram
\begin{equation*}
\begin{tikzcd}
A \arrow[d,swap,"k"] \arrow[r,"g"] & X \arrow[d,"h"] \\
B \arrow[r,swap,"f"] & Y
\end{tikzcd}
\end{equation*}
where $h:X\to Y$ is assumed to be in $\mathcal{R}$, and let $k=k_{\mathcal{R}}\circ k_\mathcal{L}$ be a factorization of $h$.
Then the outer rectangle in the diagram
\begin{equation*}
\begin{tikzcd}
A \arrow[r,equals] \arrow[d,swap,"k_{\mathcal{L}}"] & A \arrow[d,swap,"k"] \arrow[r,"g"] & X \arrow[d,"h"] \\
\im_{\mathcal{L},\mathcal{R}}(k) \arrow[r,swap,"k_{\mathcal{R}}"] & B \arrow[r,swap,"f"] & Y
\end{tikzcd}
\end{equation*}
commutes, so by the universal property of pullbacks we obtain a unique map $j:\im_{\mathcal{L},\mathcal{R}}(k)\to A$ such that $j\circ k_{\mathcal{L}}=\idfunc$ and $k\circ j=k_{\mathcal{R}}$.
It suffices to show that $k_{\mathcal{L}}$ is an equivalence, and since we already have that $j\circ k_{\mathcal{L}}=\idfunc$ we only need to show that $k_{\mathcal{L}}\circ j=\idfunc$.

We do this using the contractibility of the type of diagonal fillers. Consider the square
\begin{equation*}
\begin{tikzcd}
A \arrow[r,"k_{\mathcal{L}}"] \arrow[d,swap,"k_{\mathcal{L}}"] & \im_{\mathcal{L},\mathcal{R}}(k) \arrow[d,"k_{\mathcal{R}}"] \\
\im_{\mathcal{L},\mathcal{R}}(k) \arrow[r,swap,"k_{\mathcal{R}}"] & B,
\end{tikzcd}
\end{equation*}
for which $\idfunc:\im_{\mathcal{L},\mathcal{R}}(k)\to \im_{\mathcal{L},\mathcal{R}}(k)$ (with the trivial homotopies) is a diagonal filler. However, we also have the homotopies $k_{\mathcal{L}}\circ j\circ k_{\mathcal{L}} \htpy k_{\mathcal{L}}$ and $k_{\mathcal{R}}\circ k_{\mathcal{L}}\circ j\htpy k\circ j\htpy k_{\mathcal{R}}$. This shows that we have a second diagonal filler, of which the underlying map is $k_{\mathcal{L}}\circ j$. Since the type of diagonal fillers is contractible, it follows that $k_{\mathcal{L}}\circ j=\idfunc$, as desired.
\end{proof}

\subsubsection{Stable orthogonal factorization systems}

\begin{lem}\label{lem:fill_compute}
Given $l,r,f,g$ and a homotopy $S : r \circ f = g  \circ l$, consider as $b:B$ varies all the diagrams of the form
\begin{equation*}
\begin{tikzcd}
\fib{l}{b} \arrow[r,"i_b"] \arrow[d,"!"'] & A \arrow[d,swap,"l"] \arrow[r,"f"] \ar[dr,phantom,"S"] & X \arrow[d,"r"] \\
\unit \arrow[r,swap,"b"] & B \arrow[r,swap,"g"] & Y
\end{tikzcd}
\end{equation*}
and write $S_b : r \circ (f \circ i_b) = (g\circ b) \circ \mathord !$ for the induced commutative square.
(Note that the square on the left commutes judgmentally.)
Then the map
\begin{equation*}
\fillers{S} \to \prd{b:B}\fillers{S_b},
\end{equation*}
defined by precomposition with $b$, is an equivalence.
\end{lem}

\begin{proof}
It suffices to show that the map on total spaces
\begin{equation}
  \Big(\sm{S:r\circ f = g\circ l} \fillers{S}\Big) \to \Big( \sm{S:r\circ f = g\circ l} \prd{b:B}\fillers{S_b}\Big)\label{eq:fill-total}
\end{equation}
is an equivalence.
The domain of~\eqref{eq:fill-total} can be computed as
\begin{align*}
  &\hspace{-1cm}\sm{S:r\circ f = g\circ l}{j:B\to X}{H_f :j\circ l=f}{H_g:r\circ j=g} \ct{(r\circ H_f)}{(H_g\circ l)^{-1}} = S\\
  &\eqvsym \sm{j:B\to X}(j\circ l=f)\times (r\circ j=g)
\end{align*}
by contracting a based path space.
On the other hand, note that
\begin{align*}
  (r\circ f = g\circ l)
  &\eqvsym
  \prd{a:A} r(f(a)) = g(l(a))\\
  &\eqvsym
  \prd{a:A}{b:B}{l(a)=b} r(f(a)) = g(l(a))\\
  &\eqvsym
  \prd{b:B}{u:\fib l b} r(f(i_b(a))) = g(l(i_b(a)))\\
  &\eqvsym
  \prd{b:B}{u:\fib l b} r(f(i_b(a))) = g(b)\\
  &\eqvsym
  \prd{b:B} (r \circ (f \circ i_b) = (g\circ b) \circ \mathord !)
\end{align*}
That is, to give $S$ is the same as to give each $S_b$.
Thus the codomain of~\eqref{eq:fill-total} can be computed as
\begin{align*}
  &\hspace{-1cm}\sm{S:r\circ f = g\circ l} \prd{b:B}\fillers{S_b}\\
  &\eqvsym \sm{S:\prd{b:B} (r \circ (f \circ i_b) = (g\circ b) \circ \mathord !)} \prd{b:B}\fillers{S_b}\\
  &\eqvsym \prd{b:B}\sm{S_b:r \circ (f \circ i_b) = (g\circ b) \circ \mathord !} \fillers{S_b}\\
  &\eqvsym \prd{b:B}\sm{j_b:\unit\to X}(j_b=f\circ i_b)\times (r\circ j_b=g(b))
\end{align*}
using the same argument as above for $S$.
Now we can compute
\begin{align*}
&\hspace{-1cm}\prd{b:B}\sm{j_b:\unit\to X}(j_b=f\circ i_b)\times(r\circ j_b=g\circ b) \\
& \eqvsym
\prd{b:B}\sm{j_b:X}(\lam{x}j_b=f\circ i_b)\times(r(j_b)=g(b)) \\
& \eqvsym
\sm{j:B\to X}\prd{b:B}(\lam{\nameless}j(b)=f\circ i_b)\times(r(j(b))=g(b)) \\
& \eqvsym
\sm{j:B\to X}(\prd{b:B}\lam{\nameless}j(b)=f\circ i_b)\times(\prd{b:B}r(j(b))=g(b)) \\
& \eqvsym
\sm{j:B\to X}(\prd{b:B}\lam{\nameless}j(b)=f\circ i_b)\times(r\circ j=g) \\
& \eqvsym
\sm{j:B\to X}(\prd{b:B}\prd{a:A}{p:l(a)=b}j(b)=f(a))\times(r\circ j=g) \\
& \eqvsym
\sm{j:B\to X}(\prd{a:A}j(l(a))=f(a))\times(r\circ j=g) \\
& \eqvsym
\sm{j:B\to X}(j\circ l=f)\times(r\circ j=g)
\end{align*}
which is what we computed as the domain of~\eqref{eq:fill-total} above.
\end{proof}

\begin{cor}
In any orthogonal factorization system
$(\mathcal{L},\mathcal{R})$, if
$l:A\to B$ is a map such that $\fib{l}{b} \to \unit$ is in $\cL$ for each $b:B$, then also $l$ itself is in $\cL$.
\end{cor}
\begin{proof}
  By \cref{lem:ofs_lifting}, $l$ is in $\cL$ iff $\fillers S$ is contractible for each $r\in\cR$ and $S$ as in \cref{lem:fill_compute}, while similarly $\fib{l}{b} \to \unit$ is in $\cL$ iff $\fillers {S_b}$ is contractible.
  But the product of contractible types is contractible.
\end{proof}

\begin{cor}\label{thm:detect-right-by-fibers}
  In any stable orthogonal factorization system, if $l\perp r$ for all maps $l\in\cL$ of the form $l:A\to \unit$, then $r\in\cR$.
  In particular, for any modality $\modal$, if $X\to (A\to X)$ is an equivalence for all $\modal$-connected types $A$, then $X$ is modal.
\end{cor}
\begin{proof}
  By \cref{lem:fill_compute}, for any $l\in\cL$ and commutative square $S$ from $l$ to $r$, we have $\fillers{S} \eqvsym \prd{b:B}\fillers{S_b}$.
  Since $(\cL,\cR)$ is stable, each map $\mathord{!}_b:\fib{l}{b}\to \unit$ is also in $\cL$, so that $\mathord{!}_b\perp r$ by assumption.
  Thus $\fillers{S_b}$ is contractible for all $b$, hence so is $\fillers{S}$.

  For the second statement, the type $f:A\to X$ is equivalent to the type of commutative squares
  \[
  \begin{tikzcd}
    A \ar[r,"f"] \ar[d] & X \ar[d] \\ \unit\ar[r] & \unit
  \end{tikzcd}
  \]
  and the type of fillers for such a square is equivalent to the type of $x:X$ such that $f(a) = x$ for all $a:A$, i.e.\ the fiber of $X\to (A\to X)$ over $f$.
  Thus, the assumption ensures that all such types of fillers are contractible, i.e.\ $l\perp r$ for all $\modal$-connected maps of the form $l:A\to \unit$, so the first statement applies.
\end{proof}

\begin{lem}\label{lem:sofs_rfib}
Let $(\mathcal{L},\mathcal{R})$ be a stable orthogonal factorization system.
Then a map $r:X\to Y$ is in $\mathcal{R}$ if and only if $\fib{r}{y}$
is $(\mathcal{L},\mathcal{R})$-modal for each $y:Y$.
\end{lem}

\begin{proof}
The class of right maps is stable under pullbacks by \autoref{lem:ofs_rightstable},
so it suffices to show that any map with modal fibers is in $\mathcal{R}$.

Let $r:X\to Y$ be a map with modal fibers. Our goal is to show that
$r$ is in $\mathcal{R}$. By \autoref{lem:ofs_lifting} it suffices to show that
$r$ has the right lifting property with respect to the left maps.
Consider a diagram of the form
\begin{equation*}
\begin{tikzcd}
A \arrow[d,swap,"l"] \arrow[r,"f"] & X \arrow[d,"r"] \\
B \arrow[r,swap,"g"] & Y
\end{tikzcd}
\end{equation*}
in which $l$ is a map in $\mathcal{L}$.
We wish to show that the type of diagonal fillers is contractible.
By \autoref{lem:fill_compute}, the type of diagonal fillers of the above diagram
is equivalent to the dependent product of the types of fillers of
\begin{equation*}
\begin{tikzcd}
\fib{l}{b} \arrow[d] \arrow[r,"f\circ i_b"] & X \arrow[d,"r"] \\
\unit \arrow[r,swap,"g(b)"] & Y
\end{tikzcd}
\end{equation*}
indexed by $b:B$. Thus, it suffices that the type of diagonal fillers for this
square is contractible for each $b:B$. Since any filler factors uniquely through
the pullback $\unit\times_Y X$, which is $\fib{r}{g(b)}$, the type of diagonal
fillers of the above square is equivalent to the type of diagonal fillers of the
square
\begin{equation*}
\begin{tikzcd}
\fib{l}{b} \arrow[d] \arrow[r,densely dotted] & \fib{r}{g(b)} \arrow[d] \\
\unit \arrow[r,equals] & \unit
\end{tikzcd}
\end{equation*}
where the dotted map, is the unique map into the pullback $\fib{r}{g(b)}$. In
this square, the left map is in $\mathcal{L}$ because $\mathcal{L}$ is assumed
to be stable under pullbacks, and the right map is in $\mathcal{R}$ by assumption,
so the type of diagonal fillers is contractible.
\end{proof}

\begin{thm}\label{thm:subuniv-sofs}
Any two stable orthogonal factorization systems with the same modal types are
equal.
\end{thm}

\begin{proof}
By \autoref{lem:sofs_req} it follows that any orthogonal factorization system
is completely determined by the class of right maps.
By \autoref{lem:sofs_rfib} it follows that in a stable orthogonal factorization
system, the class of right maps is completely determined by the modal types.
\end{proof}

\begin{thm}\label{thm:highermod_from_sofs}
Any stable orthogonal factorization system determines a higher modality with
the same modal types.
\end{thm}

\begin{proof}
For every type $X$ we have the $(\cL,\cR)$-factorization $X\to\modal X\to\unit$ of the
unique map $X\to\unit$. This determines the modal unit
$\modalunit:X\to\modal X$ which is in $\mathcal{L}$, and the
unique map $\modal X\to\unit$ is in $\mathcal{R}$, i.e.\ $\modal X$ is $(\cL,\cR)$-modal.

To show the induction principle, let $P:\modal X\to\UU$ and $f:\prd{x:X} \modal(P(\eta(x)))$.
Then we have a (judgmentally) commutative square
\begin{equation*}
\begin{tikzcd}
X \arrow[r,"f"] \arrow[d,swap,"\modalunit"] & \sm{x:\modal X}\modal(P(x)) \arrow[d,"\proj1"] \\
\modal X \arrow[r,equals] & \modal X.
\end{tikzcd}
\end{equation*}
Note that by \autoref{lem:sofs_rfib},
the projection $\proj1:(\sm{x:\modal X}\modal(P(x)))\to\modal X$ is in $\mathcal{R}$
because its fibers are modal. Also, the modal unit
$\modalunit:X\to\modal X$ is in $\mathcal{L}$.
Thus, by \cref{defn:orthogonal}, the type of fillers of this square is contractible.
Such a filler consists of a function $s$ and homotopies filling the two triangles
\begin{equation*}
\begin{tikzcd}
X \arrow[r,"f"] \arrow[d,swap,"\modalunit"] & \sm{x:\modal X}\modal(P(x)) \arrow[d,"\proj1"] \\
\modal X \arrow[r,equals] \arrow[ur,densely dotted] & \modal X
\end{tikzcd}
\end{equation*}
whose composite is reflexivity, i.e.\ the type
\begin{multline*}
\sm{s:\modal X \to \sm{x:\modal X}\modal(P(x))}{H:\prd{x:\modal X} \proj1(s(x))=x}{K:\prd{x:X} s(\modalunit(x))=f(x)}\\
\prd{x:X} \proj1(K(x)) = H(\modalunit(x)).
\end{multline*}
If we decompose $s$, $f$, and $K$ by their components, we get
\begin{multline*}
\sm{s_1:\modal X \to \modal X}{s_2:\prd{x:\modal X} \modal(P(s_1(x)))}{H:\prd{x:\modal X} s_1(x)=x}\\
\sm{K_1:\prd{x:X} s_1(\modalunit(x))=f_1(x)}{K_2 :\prd{x:X} s_2(\modalunit(x)) =_{K_1(x)} f_2(x)}\\
\prd{x:X} K_1(x) = H(\modalunit(x)).
\end{multline*}
Now we can contract $s_1$ and $H$, and also $K_1$ with the final unnamed homotopy, to get
\begin{equation*}
\sm{s_2:\prd{x:\modal X} \modal(P(x))}  \prd{x:X} s_2(\modalunit(x)) =_{K_1(x)} f_2(x).
\end{equation*}
But this is just the type of extensions of $f$ along $\modalunit$, i.e.\ the fiber of precomposition by $\modalunit$.
Thus, precomposition by $\modalunit$ is an equivalence, so in fact that we have a uniquely eliminating modality.
By \cref{lem:rs_idstable}, the identity types of $\modal X$ are modal, so we have a higher modality as well.
\end{proof}

\section{Accessible reflective subuniverses}\label{sec:accessible}

\begin{defn}
Given a family $f:\prd{i:I}A_i\to B_i$ of maps, a type $X$ is said to be \define{$f$-local} if the precomposition map
\begin{equation*}
\precomp{f_i}:(B_i\to X)\to (A_i\to X)
\end{equation*}
is an equivalence, for each $i:I$. The family $f$ is said to be a \define{presentation} of a reflective subuniverse $L$ if the subuniverses of $f$-local types and $L$-local types coincide. A reflective subuniverse is said to be \define{accessible} if there exists a presentation for it. 
\end{defn}

In \cite{RijkeShulmanSpitters} it is shown that the subuniverse of $f$-local types is always a reflective subuniverse, provided that sufficiently many higher inductive types are available. However, it is not clear whether their construction is possible in our current setting, where the only higher inductive types that are assumed to exist are homotopy pushouts. In this section we will establish general properties of accessible reflective subuniverses. We will show in \cref{chap:compact} that for any family $f$ of maps between \emph{compact} types, the subuniverse of $f$-local types is indeed reflective. 

\begin{rmk}
Note that being accessible is structure; different families can present the same reflective subuniverse or modality.
As a trivial example, note that localizing at the empty
type, and localizing at the type family on $\bool$ defined by
$\bfalse\mapsto \emptyt$ and $\btrue\mapsto \unit$ both map all types to contractible types.

However, we are usually only interested in properties of presentations insofar as they determine properties of subuniverses.
For instance, by \cref{thm:acc-modal}, a reflective subuniverse is a modality exactly when it has a presentation in which each $C(a)=\unit$.
\end{rmk}

\begin{eg}\label{thm:trunc-acc}
The trivial modality $\truncf{(-2)}$ is presented by $\emptyt$, while the propositional truncation modality $\truncf{(-1)}$ is presented by $\bool$.  More generally, the
$n$-truncation modality $\truncf{n}$ is presented by the $(n+1)$-sphere $\Sn^{n+1}$.
\end{eg}

\begin{eg}\label{thm:open-acc}
For every mere proposition $P$, the open modality $\open P (X) \defeq (P\to X)$ from \cref{eg:open} is 
presented by the singleton type family $P$.
To see this, note that $\modalunit[X] : X \to (P\to X)$ is the same as the map in the definition of locality, so that $X$ is modal for the open modality on $P$ if and only if it is $P$-local.
(If $P$ is not a mere proposition, however, then $X\mapsto (P\to X)$ is not a modality, and in particular does not coincide with localization at $P$.)
\end{eg}

\begin{eg}\label{thm:closed-acc}
  The closed modality $\closed P$ from \cref{eg:closed} associated to a mere proposition $P$ is presented by the type family $\lam{x} \emptyt : P \to \UU$.
  For by definition, $A$ is null for this family if and only if for any $p:P$ the map $A \to (\emptyt \to A)$ is an equivalence.
  But $\emptyt \to P$ is contractible, so this says that $P\to\iscontr(A)$, which was the definition of $\closed P$-modal types from \cref{eg:closed}.
\end{eg}

\begin{lem}\label{lemma:characterizationsigmaflocal}
    Let $f:\prd{i:I}A_i\to B_i$ be a family of maps. Denote the family consisting of the suspensions
    of the functions by $\susp{f} : \prd{i:I} \susp{A_i} \to \susp{B_i}$.
    A type $X$ is $\suspsym f$-local if and only if for every $x,y : X$, the type
    $x =_X y$ is $f$-local.
    In other words, $L_{\susp{f}} = (L_{f})'$.
\end{lem}

\begin{proof}
    By the induction principle for suspension and naturality, we obtain for each $i : I$ a commutative square
\[
  \begin{tikzcd}
    (\susp{B_i} \to X) \arrow[r,"\simeq"] \arrow[d] & \left( \sm{x,y:X} (B_i \to x = y) \right) \arrow[d] \\
    (\susp{A_i} \to X) \arrow[r,"\simeq"] & \left( \sm{x,y:X} (A_i \to x = y) \right)
  \end{tikzcd}
\]
in which the horizontal maps are equivalences.
So $X$ is $\suspsym f$-local if and only if the right vertical map is an equivalence
for every $i : I$, if and only if for each $x,y : X$, the type $x = y$ is $f_i$-local
for every $i : I$.
\end{proof}

A general localization is only a reflective subuniverse, but there is a convenient sufficient condition for it to be a modality: if each $C(a)=\unit$.
A localization modality of this sort is called \emph{nullification}.

\begin{thm}\label{thm:nullification_modality}
  If $F:\prd{a:A} B(a) \to C(a)$ is such that each $C(a)=\unit$, then localization at $F$ is a modality, called \define{nullification at $B$}.
\end{thm}
\begin{proof}
  It suffices to show that for any $B:A\to\UU$, the $B$-null types are $\Sigma$-closed.
  Thus, let $X:\UU$ and $Y:X\to \UU$ be such that $X$ and each $Y(x)$ are $B$-null.
  Then
  \begin{align*}
    (B\to \sm{x:X} Y(x))
    &\eqvsym \sm{g:B\to X} \prd{b:B} Y(g(b)) \\
    &\eqvsym \sm{x:X} B \to Y(x) \\
    &\eqvsym \sm{x:X} Y(x)
  \end{align*}
  with the inverse equivalence being given by constant maps.
  Thus, $\sm{x:X} Y(x)$ is $B$-null.
\end{proof}

Of course, it might happen that $\localization{F}$ is a modality even if $F$ doesn't satisfy the condition of \cref{thm:nullification_modality}.
For instance, if $B:A\to \UU$ has a section $s:\prd{a:A} B(a)$, then localizing at the family $s' : \prd{a:A} \unit \to B(a)$ is equivalent to nullifying at $B$, since in a section-retraction pair the section is an equivalence if and only if the retraction is.
However, we can say the following.

\begin{lem}\label{thm:acc-modal}
  If $F:\prd{a:A} B(a)\to C(a)$ is such that $\localization{F}$ is a modality, then there exists a family $E:D\to \UU$ such that $\localization{F}$ coincides with nullification at $E$.
\end{lem}
\begin{proof}
  Write $\modal\defeq\localization{F}$ and $\modalunit$ for its modal unit.
  Define $D = \sm{a:A} (\modal (B(a)) + \modal(C(a)))$, and $E:D\to \UU$ by
  \begin{align*}
    E(a,\inl(b)) &\defeq \fib{\modalunit[B(a)]}{b}\\
    E(a,\inr(c)) &\defeq \fib{\modalunit[C(a)]}{c}.
  \end{align*}
  Then since $\modalunit$ is $\modal$-connected, each $E(d)$ is $\modal$-connected, and hence every $F$-local type is $E$-null.

  On the other hand, suppose $X$ is an $E$-null type.
  Each $\modalunit[B(a)]$ and $\modalunit[C(a)]$ is $\localization{E}$-connected, since their fibers are $\localization{E}$-connected (by definition); thus $X$ is also $\modalunit[B(a)]$-local and $\modalunit[C(a)]$-local.
  But we have the following commutative square:
  \[
  \begin{tikzcd}[column sep=large]
    B(a) \ar[r,"{\modalunit[B(a)]}"] \ar[d,"F(a)"'] & \modal(B(a)) \ar[d,"{\modal(F(a))}"]\\
    C(a) \ar[r,"{\modalunit[C(a)]}"'] & \modal(C(a))
  \end{tikzcd}
  \]
  and ${\modal(F(a))}$ is an equivalence; thus $X$ is also $F(a)$-local.
  So the $F$-local types coincide with the $E$-null types.
\end{proof}

This shows that the following definition of accessible modality is consistent with our terminology of accessible reflective subuniverse.

\begin{defn}\label{defn:accessible}
A modality $\modal$ on $\UU$ is said to be \define{accessible} if it is the nullification at a family of types in $\UU$, indexed by a type in $\UU$.
A \define{presentation} of a modality $\modal$ consists of a family of types $B: A\to\UU$, where $A:\UU$, such that the subuniverse of modal types coincides with the subuniverse of $B$-null types.
\end{defn}

\begin{egs}
  Our characterizations of the truncation and open and closed modalities in \cref{thm:trunc-acc,thm:open-acc,thm:closed-acc} made no reference to the ambient universe.
\end{egs}

\begin{eg}
  By contrast, the double-negation modality $\neg\neg$ \emph{is} defined in a polymorphic way on all universes, but in general there seems no reason for it to be accessible on any of them.
  However, if propositional resizing holds, then it is the nullification at $\bool$ together with all propositions $P$ such that $\neg\neg P$ holds, and hence accessible.

  Whether or not any inaccessible modalities remain after imposing propositional resizing may depend on large-cardinal principles.
  It is shown in~\cite{css:large-cardinal} that this is the case for the analogous question about reflective sub-$(\infty,1)$-categories of the $(\infty,1)$-category of $\infty$-groupoids.
\end{eg}

\begin{rmk}\label{rmk:extend-oops}
  It is tempting to think that \emph{any} reflective subuniverse $\modal$ on $\UU$ could be extended to an accessible one on $\UU'$ by localizing at the family of \emph{all} functions in $\UU$ that are inverted by $\modal$ (or nullifying at the family of all $\modal$-connected types in $\UU$, in the case of modalities), which is a $\UU'$-small family though not a $\UU$-small one.
  This does produce an accessible reflective subuniverse $\modal'$ of $\UU'$ such that the $\modal'$-modal types in $\UU$ coincide with the $\modal$-modal ones, but there seems no reason why the modal \emph{operators} $\modal'$ and $\modal$ should agree on types in $\UU$.
\end{rmk}

\chapter{The equifibrant replacement operation}\label{chap:equifibrant}

We begin this chapter with a generalization of the descent theorem for reflexive coequalizers \cref{thm:rcoeq_cartesian}: the \emph{modal descent theorem}.
Any modality $\modal$ gives rise to a class of maps $f:A\to B$ satisfying the condition, due to Wellen \cite{WellenPhD}, that the naturality square
\begin{equation*}
\begin{tikzcd}
A \arrow[r,"f"] \arrow[d,swap,"\eta"] & B \arrow[d,"\eta"] \\
\modal A \arrow[r,swap,"\modal f"] & \modal B
\end{tikzcd}
\end{equation*}
is a pullback square. Following \cite{WellenPhD} we call such maps $\modal$-\'etale maps. The modal descent theorem asserts that a $\modal$-\'etale map into $A$ is equivalently described as a modal map into $\modal A$. Every $\modal$-\'etale map is certainly modal, but the condition of being $\modal$-\'etale is slightly stronger than the condition of being modal. 

The difference between the notions of $\modal$-\'etale maps and modal maps becomes perhaps most visible when we look at the left orthogonal classes of the $\modal$-\'etale maps and modal maps. The $\modal$-connected maps are left orthogonal to the modal maps, whereas a map is left orthogonal to the $\modal$-\'etale maps if and only if it is a $\modal$-equivalence (i.e.~a map $f$ such that $\modal f$ is an equivalence). In the case of the $n$-truncation there is a clear difference: a map $f:X\to Y$ is an $n$-equivalence if and only if it induces an isomorphism of homotopy groups $\pi_i(X)\to \pi_i(Y)$ for any $i\leq n$, whereas an $n$-connected map is an $n$-equivalence satisfying the further condition that $\pi_{n+1}(X)\to\pi_{n+1}(Y)$ is surjective.

In \cref{thm:modal_rofs} we show that the $\modal$-equivalences and the $\modal$-\'etale maps form an orthogonal factorization system. We call this factorization system the \emph{reflective} factorization system of a modality. The reflective factorization system is not stable, so it does not form a new modality. What does follow is that for the unique factorization $f=f_r\circ f_l$ of $f:A\to X$ as a $\modal$-equivalence $f_l$ followed by a $\modal$-\'etale map $f_r$, the canonical map
\begin{equation*}
\mathrm{hom}_X(f_r,e)\to\mathrm{hom}_X(f,e)
\end{equation*}
is an equivalence for any $\modal$-\'etale map $e$ into $X$. 

We then proceed to apply these ideas to the case of reflexive graphs and the reflexive coequalizer. The class of $\Delta$-\'etale maps is introduces as the class of morphisms $f:\mathsf{rGph}(\mathcal{B},\mathcal{A})$ of reflexive graphs satisfying the condition that the square
\begin{equation*}
\begin{tikzcd}[column sep=huge]
\mathcal{B} \arrow[r,"f"] \arrow[d,swap,"\mathsf{constr}"] & \mathcal{A} \arrow[d,"\mathsf{constr}"] \\
\Delta(\mathsf{rcoeq}(\mathcal{B})) \arrow[r,swap,"\Delta(\mathsf{rcoeq}(f))"] & \Delta(\mathsf{rcoeq}(\mathcal{A}))
\end{tikzcd}
\end{equation*}
is a pullback square, and we show that the class of $\Delta$-\'etale morphisms is precisely the class of fibrations of reflexive graphs, as defined in \cref{defn:graph_fibration}. 

In the case of sequential colimits, another way of obtaining a cartesian transformation from an arbitrary one is to take the sequential colimit fiberwise. We establish this result in \cref{thm:colim_fib}. It has many important consequences. First of all, we show in \cref{thm:colim_id} that sequential colimits commute with identity types. Second, we show in \cref{thm:colim_fiberseq} that the sequential colimit of a fiber sequence is again a fiber sequence. Third, we show in \cref{thm:colim_hlevel} that truncation levels are closed under sequential colimits, and that sequential colimits commute with truncations. Finally, we show in \cref{thm:colim_hgroup} that sequential colimits commute with homotopy groups. 

\section{Modal descent}\label{sec:modal_descent}
\subsection{\texorpdfstring{$\modal$}{○}-\'etale maps}

\begin{defn}[Definition 4.4.1 of \cite{WellenPhD}]
We say that a map $f:A\to B$ is \define{$\modal$-\'etale} if the square
\begin{equation*}
\begin{tikzcd}
A \arrow[r,"\eta"] \arrow[d,swap,"f"] & \modal A \arrow[d,"\modal f"] \\
B \arrow[r,swap,"\eta"] & \modal B
\end{tikzcd}
\end{equation*}
is a pullback square. We will write
\begin{equation*}
\isetale(f)\defeq\ispullback(\eta,f,\natunit_\modal(f)).
\end{equation*}
\end{defn}

It is immediate from the definition that any equivalence is $\modal$-\'etale, and that the $\modal$-\'etale maps are closed under composition, and that every equivalence is $\modal$-\'etale.

\begin{eg}\label{eg:etale_prop}
We claim that a map $f:A\to B$ is $\brck{\blank}$-\'etale if and only if $A\to \isequiv(f)$. Examples of maps that satisfy this condition include equivalences, maps between propositions, and any map of the form $\emptyt\to B$.

To see that if $f:A\to B$ is $\modal$-\'etale, then $A\to\isequiv(f)$, consider the pullback square
\begin{equation*}
\begin{tikzcd}
A \arrow[r] \arrow[d,swap,"f"] & \brck{A} \arrow[d,"\brck{f}"] \\
B \arrow[r] & \brck{B},
\end{tikzcd}
\end{equation*}
and let $a:A$. Then both $\brck{A}$ and $\brck{B}$ are contractible, so $\brck{f}:\brck{A}\to\brck{B}$ is an equivalence. Since equivalences are stable under pullback it follows that $f$ is an equivalence.

Now suppose that $A\to \isequiv(f)$. Since $\isequiv(f)$ is a proposition, we also have $\brck{A}\to\isequiv(f)$. To see that the gap map
\begin{equation*}
A \to B\times_{\brck{B}}\brck{A}
\end{equation*}
is an equivalence, we will show that its fibers are contractible. Let $b:B$, $x:\brck{A}$ and $p:\bproj{b}=\brck{f}(x)$. Since $\brck{A}\to\isequiv(f)$, it follows that $f$ is an equivalence. Then $\brck{f}$ is also an equivalence, from which it follows that the naturality square is a pullback square. We conclude that the fibers of the gap map are contractible. 
\end{eg}

\begin{lem}\label{lem:etale_modal}
Any map between $\modal$-modal types is $\modal$-\'etale.
\end{lem}

\begin{proof}
Suppose $f:X\to Y$ is a map between $\modal$-modal types. Then the top and bottom maps in the square
\begin{equation*}
\begin{tikzcd}
X \arrow[r] \arrow[d] & \modal X \arrow[d] \\
Y \arrow[r] & \modal Y
\end{tikzcd}
\end{equation*}
are equivalences. Therefore this square is a pullback square, so $f$ is $\modal$-\'etale.
\end{proof}

For the following lemma, recall that for a modality all propositions are modal if and only if all units $\eta : X\to \modal X$ are surjective. 

\begin{lem}\label{lem:etale_char}
Let $\modal$ be a modality for which all propositions are modal, and consider a map $f:A\to B$. The following are equivalent:
\begin{enumerate}
\item $f$ is $\modal$-\'etale.
\item The commuting square
\begin{equation*}
\begin{tikzcd}[column sep=large]
A\times_{\modal A} A \arrow[d,swap,"\pi_1"] \arrow[r,"{f\times_{\modal f} f}"] & B\times_{\modal B} B \arrow[d,"\pi_1"] \\
A \arrow[r,swap,"f"] & B
\end{tikzcd}
\end{equation*}
is a pullback square.
\end{enumerate}
\end{lem}

\begin{rmk}
In the special case of $(-1)$-truncation, the characterization of \cref{lem:etale_char} asserts that a map $f:A\to B$ is $(-1)$-\'etale if and only if the square
\begin{equation*}
\begin{tikzcd}
A\times A \arrow[d,swap,"\pi_1"] \arrow[r,"{f\times f}"] & B\times B \arrow[d,"\pi_1"] \\
A \arrow[r,swap,"f"] & B
\end{tikzcd}
\end{equation*}
is a pullback square.
\end{rmk}

\begin{proof}
Suppose first that $f$ is $\modal$-\'etale, and consider the commuting cube
\begin{equation*}
\begin{tikzcd}
& A\times_{\modal A} A \arrow[dl] \arrow[d] \arrow[dr] \\
A \arrow[d] & B\times_{\modal B} B \arrow[dl] \arrow[dr] & A \arrow[dl,crossing over] \arrow[d] \\
B \arrow[dr] & \modal A \arrow[from=ul,crossing over] \arrow[d] & B \arrow[dl] \\
& \modal B
\end{tikzcd}
\end{equation*}
Since the top, bottom, and both front squares are pullback squares, it follows that both back squares are pullback. This proves that (i) implies (ii).

Now suppose that (ii) holds. Then the map
\begin{equation*}
\fib{\modalunit}{\modalunit(a)}\to \fib{\modalunit}{\modalunit(f(a))}
\end{equation*}
is an equivalence for every $a:A$. Since all propositions are assumed to be modal, it follows that
\begin{equation*}
\fib{\modalunit}{t}\to \fib{\modalunit}{\modal f(t)}
\end{equation*}
is an equivalence for every $t:\modal A$. Thus it follows that the square
\begin{equation*}
\begin{tikzcd}
A \arrow[d,swap,"\modalunit"] \arrow[r] & B \arrow[d,"\modalunit"] \\
\modal A \arrow[r] & \modal B
\end{tikzcd}
\end{equation*}
is a pullback square.
\end{proof}

\begin{cor}
If $f:A\to B$ is $\modal$-\'etale, then the square
\begin{equation*}
\begin{tikzcd}[column sep=large]
A \arrow[d,swap,"\delta_{\modalunit}"] \arrow[r,"f"] & B \arrow[d,"\delta_{\modalunit}"] \\
A\times_{\modal A} A \arrow[r,swap,"f\times_{\modal f}f"] & B\times_{\modal B} B
\end{tikzcd}
\end{equation*}
is a pullback square.
\end{cor}

\begin{proof}
Consider the diagram
\begin{equation*}
\begin{tikzcd}[column sep=large]
A \arrow[d,swap,"\delta_{\modalunit}"] \arrow[r,"f"] & B \arrow[d,"\delta_{\modalunit}"] \\
A\times_{\modal A} A \arrow[d,swap,"\pi_1"] \arrow[r,"{f\times_{\modal f} f}"] & B\times_{\modal B} B \arrow[d,"\pi_1"] \\
A \arrow[r,"f"] & B
\end{tikzcd}
\end{equation*}
The bottom square is a pullback square by \cref{lem:etale_char}, and the outer rectangle is a pullback since both vertical composites are homotopic to the respective identity functions. Therefore the top square is a pullback.
\end{proof}

\begin{thm}
A map $f:A\to B$ is $0$-\'etale if and only if for each $a:A$ the restriction
\begin{equation*}
\begin{tikzcd}
\sm{x:A}\brck{a=x} \arrow[d,swap,"\proj 1"] \arrow[r,densely dotted,"f"] & \sm{y:B}\brck{f(a)=y} \arrow[d,"\proj 1"] \\
A \arrow[r,swap,"f"] & B
\end{tikzcd}
\end{equation*}
of $f$ to the connected component at $a$ of $A$ is an equivalence. 
\end{thm}

\begin{proof}
By \cref{lem:etale_char} and the fact that $\eqv{(\tproj{0}{a}=\tproj{0}{x})}{\brck{a=x}}$, it follows that $f$ is $0$-\'etale if and only if the square
\begin{equation*}
\begin{tikzcd}[column sep=huge]
\sm{a,x:A}\brck{a=x} \arrow[d,swap,"\pi_1"] \arrow[r,"\total{\brck{\apfunc{f}}}"] & \sm{b,y:B}\brck{b=y} \arrow[d,"\pi_1"] \\
A \arrow[r,swap,"f"] & B
\end{tikzcd}
\end{equation*}
is a pullback square. Furthermore, this square is a pullback if and only if the induced map
\begin{equation*}
\Big(\sm{x:A}\brck{a=x}\Big)\to\Big(\sm{y:B}\brck{f(a)=y}\Big)
\end{equation*}
is an equivalence, for each $a:A$.
\end{proof}

\subsection{Modal descent}

The following theorem can be seen as a `modal flattening lemma', since it is analogous to \cref{thm:rcoeq_cartesian}.

\begin{thm}\label{thm:etale_flattening}
Consider a pullback square
\begin{equation*}
\begin{tikzcd}
E' \arrow[d,swap,"{p'}"] \arrow[r,"g"] & E \arrow[d,"p"] \\
B' \arrow[r,swap,"f"] & B
\end{tikzcd}
\end{equation*}
with $H:f\circ p'\htpy p\circ g$, where $E$ and $B$ are modal types. Then the square
\begin{equation*}
\begin{tikzcd}
\modal E' \arrow[r,"\tilde{g}"] \arrow[d,swap,"{\modal p'}"] & E \arrow[d,"p"] \\
\modal B \arrow[r,swap,"\tilde{f}"] & B
\end{tikzcd}
\end{equation*}
is a pullback square, where $\tilde{f}$ and $\tilde{g}$ are the unique extensions of $f$ and $g$ along the modal units of $B'$ and $E'$, respectively.
\end{thm}

\begin{proof}
Consider the diagram
\begin{equation*}
\begin{tikzcd}
E' \arrow[r,"{\mathsf{gap}(p',g,H)}"] \arrow[d,swap,"{p'}"] &[2.5em] \modal B'\times_{B} E \arrow[r,"\pi_2"] \arrow[d,swap,"\pi_1"] & E \arrow[d,"p"] \\
B' \arrow[r,swap,"\modalunit"] & \modal B' \arrow[r,swap,"\tilde{f}"] & B
\end{tikzcd}
\end{equation*}
In this diagram, the square on the right is a pullback by definition, and the outer rectangle is a pullback by assumption, so the square on the left is also a pullback. Therefore the gap map $E'\to \modal B'\times_B E$ is $\modal$-connected. Moreover, since the modal types are closed under pullbacks it follows that $\modal B'\times_B E$ is modal, and therefore it follows that $\pi_2:\modal B'\times_B E\to E$ is a modal map. Therefore the composite
\begin{equation*}
\begin{tikzcd}
E' \arrow[r,"{\mathsf{gap}(p',g,H)}"] &[2.5em] \modal B'\times_{B} E \arrow[r,"\pi_2"] & E 
\end{tikzcd}
\end{equation*}
factors $g$ as a $\modal$-connected map followed by a $\modal$-modal map. Of course, another such factorization is the composite $g\htpy \tilde{g}\circ\modalunit$. Since factorizations are unique, the claim follows.
\end{proof}

Using modal flattening we establish partial left exactness of the modality.

\begin{cor}\label{cor:etale_lex}
Consider a pullback square
\begin{equation*}
\begin{tikzcd}
A' \arrow[d,swap,"{f'}"] \arrow[r] & A \arrow[d,"f"] \\
B' \arrow[r] & B,
\end{tikzcd}
\end{equation*}
where $f$ is assumed to be $\modal$-\'etale. Then the square
\begin{equation*}
\begin{tikzcd}
\modal A' \arrow[d,swap,"{\modal f'}"] \arrow[r] & \modal A \arrow[d,"\modal f"] \\
\modal B' \arrow[r] & \modal B,
\end{tikzcd}
\end{equation*}
is again a pullback square.
\end{cor}

\begin{proof}
Since $f$ is assumed to be $\modal$-\'etale, the square on the right in the diagram
\begin{equation*}
\begin{tikzcd}
A' \arrow[r] \arrow[d,swap,"{f'}"] & A \arrow[r] \arrow[d,swap,"f"] & \modal A \arrow[d,"\modal f"] \\
B' \arrow[r] & B \arrow[r] & \modal B
\end{tikzcd}
\end{equation*}
is a pullback square. Therefore the outer rectangle is a pullback square by the pullback pasting lemma. Now the claim follows from modal flattening \cref{thm:etale_flattening}, using the outer rectangle.
\end{proof}

\begin{cor}\label{cor:etale_pb}
Consider a pullback square
\begin{equation*}
\begin{tikzcd}
E' \arrow[d,swap,"{p'}"] \arrow[r,"g"] & E \arrow[d,"p"] \\
B' \arrow[r,swap,"f"] & B
\end{tikzcd}
\end{equation*} 
and suppose that $p:E\to B$ is $\modal$-\'etale. Then $p':E'\to B'$ is $\modal$-\'etale.
\end{cor}

\begin{proof}
Consider the commuting cube
\begin{equation*}
\begin{tikzcd}
& E' \arrow[dr] \arrow[d] \arrow[dl] \\
\modal E' \arrow[d] & B' \arrow[dl] \arrow[dr] & E \arrow[dl,crossing over] \arrow[d] \\
\modal B' \arrow[dr] & \modal E \arrow[d] \arrow[from=ul,crossing over] & B \arrow[dl] \\
& \modal B. & \phantom{\modal E'}
\end{tikzcd}
\end{equation*}
The vertical squares on the back right and front right are pullback squares by assumption.
Then it follows from \cref{cor:etale_lex} that the vertical square on the front left is a pullback square.
Therefore the square on the back left is a pullback square by the pullback pasting property.
\end{proof}

\begin{defn}
Let $X$ be a type. We will define an operation
\begin{equation*}
\etmap:\Big(\sm{A:\UU_\modal}A\to\modal X\Big)\to\Big(\sm{Y:\UU}{g:Y\to X}\isetale(g)\Big)
\end{equation*}
\end{defn}

\begin{proof}[Construction]
Given a map $f:A\to \modal X$ we take the pullback
\begin{equation*}
\begin{tikzcd}
X\times_{\modal X}A \arrow[d,swap,"\pi_1"] \arrow[r,"\pi_2"] & A \arrow[d,"f"] \\
X \arrow[r,swap,"\modalunit"] & \modal X.
\end{tikzcd}
\end{equation*}
Then the map $\pi_1:X\times_{\modal X}A\to X$ is $\modal$-\'etale by \cref{lem:etale_modal,cor:etale_pb}.
\end{proof}

The following is a descent theorem for $\modal$-\'etale maps.

\begin{thm}[Modal descent]\label{thm:modal_descent}
For any modality $\modal$, and any type $X$, the operation
\begin{equation*}
\etmap:\Big(\sm{A:\UU_\modal}A\to\modal X\Big)\to\Big(\sm{Y:\UU}{g:Y\to X}\isetale(g)\Big)
\end{equation*}
is an equivalence.
\end{thm}

\begin{proof}
If $g:Y\to X$ is $\modal$-\'etale, then the square
\begin{equation*}
\begin{tikzcd}
Y \arrow[d,swap,"g"] \arrow[r,"\modalunit"] & \modal Y \arrow[d,"\modal g"] \\
X \arrow[r,swap,"\modalunit"] & \modal X
\end{tikzcd}
\end{equation*}
is a pullback square. Therefore $g:Y\to X$ is in the fiber of $\etmap$ at $\modal g : \modal Y\to\modal X$. 

It remains to show that for any map $f:A\to\modal X$ with modal domain, there is an equivalence $\eqv{A}{\modal (X\times_{\modal X} A)}$ such that the triangle
\begin{equation*}
\begin{tikzcd}[column sep=tiny]
A \arrow[dr,swap,"f"] \arrow[rr,"\eqvsym"] & & \modal (X\times_{\modal X} A) \arrow[dl,"\modal(\etmap(f))"] \\
\phantom{\modal (X\times_{\modal X} A)} & \modal X
\end{tikzcd}
\end{equation*}
commutes. To see this, note that both $f\circ \pi_2$ and $\modal(\etmap(f))\circ \modalunit$ factor the same map as a $\modal$-connected map followed by a modal map, so the claim follows from uniqueness of factorizations.
\end{proof}

\begin{cor}
Suppose $P:X\to\UU_\modal$ is a family of modal types such that the projection map $\proj 1:\big(\sm{x:X}P(x)\big)\to X$ is $\modal$-\'etale. Then $P$ has a unique extension
\begin{equation*}
\begin{tikzcd}
X \arrow[d,swap,"\modalunit"] \arrow[r,"P"] & \UU_\modal. \\
\modal X \arrow[ur,densely dotted,swap,"\tilde{P}"] 
\end{tikzcd}
\end{equation*}
It follows that the square commuting square
\begin{equation*}
\begin{tikzcd}
\sm{x:X}P(x) \arrow[d,swap,"\proj 1"] \arrow[r] & \sm{t:\modal X}\tilde{P}(t) \arrow[d,"\proj 1"] \\
X \arrow[r,swap,"\modalunit"] & \modal X
\end{tikzcd}
\end{equation*}
is a pullback square. In particular the top map is $\modal$-connected, so this square is in fact a $\modal$-naturality square.
\end{cor}

We conclude with a slightly more economic rephrasing of \cref{thm:etale_flattening}, which is an easy corollary of the results in this section.

\begin{thm}
Consider a commuting square
\begin{equation*}
\begin{tikzcd}
E' \arrow[d,swap,"{p'}"] \arrow[r,"g"] & E \arrow[d,"p"] \\
B' \arrow[r,swap,"f"] & B
\end{tikzcd}
\end{equation*}
with $H:f\circ p'\htpy p\circ g$, where $E$ and $B$ are modal types, where $p':E'\to B'$ is $\modal$-\'etale. Then the following are equivalent:
\begin{enumerate}
\item The square is a pullback square.
\item The square
\begin{equation*}
\begin{tikzcd}
\modal E' \arrow[r,"\tilde{g}"] \arrow[d,swap,"{\modal p'}"] & E \arrow[d,"p"] \\
\modal B \arrow[r,swap,"\tilde{f}"] & B
\end{tikzcd}
\end{equation*}
is a pullback square, where $\tilde{f}$ and $\tilde{g}$ are the unique extensions of $f$ and $g$ along the modal units of $B'$ and $E'$, respectively.
\end{enumerate}
\end{thm}

\subsection{The reflective factorization system of a modality}

In this subsection we investigate the reflective factorization system associated to a modality, of which the right class is the class of $\modal$-\'etale maps. The left class is the class of \emph{$\modal$-equivalences}.
\begin{defn}
We say that a map $f:A\to B$ is an \define{$\modal$-equivalence} if $L f:\modal A\to \modal B$ is an equivalence.
\end{defn}

\begin{rmk}
The difference between the notions of $\modal$-equivalences and $\modal$-connected maps is best explained by an example. In the case of $n$-truncation, the $n$-equivalences are precisely the maps that induce isomorphisms on the first $n$ homotopy groups. The $n$-connected maps are the maps that induce isomorphisms on the first $n$ homotopy groups, and moreover induce an epimorphism on the $(n+1)$-st homotopy group. 

We also note that the $n$-equivalences are not stable under pullbacks, whereas the $n$-connected maps are. Consider for instance the pullback square
\begin{equation*}
\begin{tikzcd}
\loopspace {\sphere{n+1}} \arrow[r] \arrow[d] & \unit \arrow[d] \\
\unit\arrow[r] & \sphere{n+1}
\end{tikzcd}
\end{equation*}
Here the map on the right is an $n$-equivalence, since $\sphere{n+1}$ is $n$-connected. However, the map on the left is not an $n$-equivalence, since the $n$-th homotopy group of $\loopspace{\sphere{n+1}}$ is not trivial: it is the $(n+1)$-st homotopy group of $\sphere{n+1}$, which is $\Z$.
\end{rmk}

\begin{defn}
The \define{reflective factorization system} associated to a modality $\modal$ consists of the $\modal$-equivalences as the left class, and the $\modal$-\'etale maps as the right class.
\end{defn}

Our goal in this section is to show that the reflective factorization system associated to a modality is an orthogonal factorization system.

\begin{lem}\label{lem:3for2_mequiv}
The $\modal$-equivalences satisfy the 3-for-2 property: given a commuting triangle
\begin{equation*}
\begin{tikzcd}[column sep=tiny]
A \arrow[rr,"h"] \arrow[dr,swap,"f"] & & B \arrow[dl,"g"] \\
& C,
\end{tikzcd}
\end{equation*}
if any two of $f$, $g$, and $h$ are $\modal$-equivalences, then so is the third.
\end{lem}

\begin{proof}
Apply $\modal$ to the commuting triangle, and use the 3-for-2 property of equivalences.
\end{proof}

\begin{lem}\label{lem:modal_equivalence}
For a map $f : A \to B$ the following are equivalent:
\begin{enumerate}
\item $f$ is an $\modal$-equivalence.
\item For any modal type $X$, the precomposition map
\begin{equation*}
\precomp{f} : (B \to X) \to (A \to X)
\end{equation*}
is an equivalence.
\end{enumerate}
\end{lem}

\begin{proof} 
Suppose first that $f$ is an $\modal$-equivalence, and let $X$ be $\modal$-modal. Then the square
\begin{equation*}
\begin{tikzcd}
X^B \arrow[r,"\precomp{f}"] \arrow[d,swap,"\precomp{\eta}"] & X^A \arrow[d,"\precomp{\eta}"] \\
X^{\modal B} \arrow[r,swap,"\precomp{\modal f}"] & X^{\modal A}
\end{tikzcd}
\end{equation*}
commutes. In this square the two vertical maps are equivalences by the universal property of modalization, and the bottom map is an equivalence since $\modal f$ is an equivalence. Therefore the map $\precomp{f}:X^B\to X^A$ is an equivalence, as desired.

Conversely, assume that $\precomp{f} : X^B \to X^A$ is an equivalence for every $\modal$-modal type $X$. By the square above it follows that $\precomp{\modal f}:X^{\modal B}\to X^{\modal A}$ is an equivalence for every $\modal$-modal type $X$. The fiber of $\modal A^{\modal B}\to \modal A^{\modal A}$ at $\idfunc:\modal A\to \modal A$ is contractible, so we obtain a retraction $g$ of $\modal f$. To see that $g$ is also a section observe that the fiber of $\modal B^{\modal B}\to \modal B^{\modal A}$ at $\modal f$ is contractible. This fiber contains $(\idfunc[\modal B],\refl{\modal f})$. However, we also have an identification $p:\precomp{\modal f}(\modal f\circ g)=\modal f$, since
\begin{equation*}
\precomp{\modal f}(\modal f\circ g)\jdeq (\modal f \circ g)\circ \modal f\jdeq \modal f \circ (g\circ \modal f) = \modal f. 
\end{equation*}
Therefore $(\modal f\circ g,p)$ is in the fiber of $\precomp{\modal f}:\modal B^{\modal B}\to \modal B^{\modal A}$ at $\modal f$. By the contractibility of the fibers it follows that $(\modal f\circ g,p)=(\idfunc[\modal B],\refl{\modal f})$, so it follows that $\modal f\circ g=\idfunc[\modal B]$. In other words, $g$ is both a retraction and a section of $\modal f$, so $\modal f$ is an equivalence.
\end{proof}

\begin{cor}\label{cor:mequiv_mconn}
Every $\modal$-connected map is a $\modal$-equivalence.
\end{cor}

\begin{defn}
Let $f:A\to B$ be a map. We define
\begin{align*}
[A]_f & \defeq B \times_{\modal B}\modal A
\intertext{and we define the maps}
\mathsf{et}(f) & : [A]_f \to B \\
\bar{\eta} : A \to [A]_f
\end{align*}
by the universal property of pullbacks, as indicated in the following diagram
\begin{equation*}
\begin{tikzcd}
A \arrow[ddr,bend right=15,swap,"f"] \arrow[drr,bend left=15,"\modalunit"] \arrow[dr,"\bar{\eta}" description] \\
& {[A]_f} \arrow[d,"\mathsf{et}(f)"] \arrow[r,"\pi_2"] & \modal A \arrow[d,"\modal f"] \\
& B \arrow[r,swap,"\modalunit"] & \modal B.
\end{tikzcd}
\end{equation*}
\end{defn}

\begin{lem}\label{lem:rfs_factor}
For every map $f:A\to B$, the map $\bar{\eta}:A\to[A]$ is a $\modal$-equivalence, and the map $\mathsf{et}(f)$ is $\modal$-\'etale.
\end{lem}

\begin{proof}
The map $\mathsf{et}(f)$ is a pullback of a map between modal types, so it is $\modal$-\'etale by \cref{cor:etale_pb}. Furthermore, the map $\pi_2:[A] \to \modal A$ is a pullback of a $\modal$-connected map, so it is $\modal$-connected. It follows from \cref{cor:mequiv_mconn} that it is a $\modal$-equivalence. Since the modal unit $\modalunit :A\to\modal A$ is also $\modal$-connected, and therefore a $\modal$-equivalence, we obtain by the 3-for-2 property of $\modal$-equivalences established in \cref{lem:3for2_mequiv} that the gap map is also a $\modal$-equivalence.
\end{proof}

\begin{lem}\label{lem:rfs_orthogonal}
The class of $\modal$-equivalences is left orthogonal to the class of $\modal$-\'etale maps.
\end{lem}

\begin{proof}
We have to show that for every $\modal$-equivalence $i:A\to B$, and every $\modal$-\'etale map $f:X\to Y$, the square
\begin{equation*}
\begin{tikzcd}
X^B \arrow[r] \arrow[d] & Y^B \arrow[d] \\
X^A \arrow[r] & Y^A
\end{tikzcd}
\end{equation*}
is a pullback square. Consider the commuting cube
\begin{equation*}
\begin{tikzcd}
& X^B \arrow[dl] \arrow[d] \arrow[dr] \\
(\modal X)^B \arrow[d] & X^A \arrow[dl] \arrow[dr] & Y^B \arrow[d] \arrow[dl,crossing over] \\
(\modal X)^A \arrow[dr] & (\modal Y)^B \arrow[from=ul,crossing over] \arrow[d] & Y^A \arrow[dl] \\
& (\modal Y)^A & \phantom{(\modal X)^B}
\end{tikzcd}
\end{equation*}
In this cube the top and bottom squares are pullback by the assumption that $f$ is $\modal$-\'etale and the fact that exponents of pullback squares are again pullback squares. Furthermore, the square in the front left is pullback, because the two vertical maps are equivalences by the assumption that $i:A\to B$ is a $\modal$-equivalence. Therefore we conclude that the square in the back right is also a pullback square, as desired.
\end{proof}

\begin{cor}
For any map $f:X\to Y$, the type of factorizations into a $\modal$-connected map followed by a $\modal$-\'etale map is contractible.
\end{cor}

The class of $\modal$-\'etale morphisms into a given type $A$, thought of as objects of the slice category $\UU/A$, form a reflective subuniverse in the following sense.

\begin{thm}\label{thm:et_up}
Let $f:A\to X$ be a map. Then the pre-composition function
\begin{equation*}
\mathrm{hom}_X(\mathsf{et}(f),e)\to \mathrm{hom}_X(f,e)
\end{equation*}
is an equivalence for every \'etale map $e:B\to X$. 
\end{thm}

\begin{proof}
Let $e:B\to X$ be a $\modal$-\'etale map. Then the square
\begin{equation*}
\begin{tikzcd}
B^{[A]} \arrow[r,"\blank\circ\eta_f"] \arrow[d,swap,"e\circ\blank"] & B^A \arrow[d,"e\circ\blank"] \\
X^{[A]} \arrow[r,swap,"\blank\circ\eta_f"] & X^A
\end{tikzcd}
\end{equation*}
is a pullback square by \cref{lem:rfs_factor,lem:rfs_orthogonal}. Therefore we have a fiberwise equivalence
\begin{equation*}
\prd{i:[A]\to X} \fib{e\circ\blank}{i}\to \fib{e\circ\blank}{i\circ\eta_f}
\end{equation*}
by \cref{cor:pb_fibequiv}. Now the claim follows, since we have a commuting square
\begin{equation*}
\begin{tikzcd}
\fib{e\circ\blank}{i} \arrow[r] \arrow[d,swap,"\eqvsym"] & \fib{e\circ\blank}{i\circ\eta_f} \arrow[d,"\eqvsym"] \\
\mathrm{hom}_X(\mathsf{et}(f),e) \arrow[r] & \mathrm{hom}_X(f,e)
\end{tikzcd}
\end{equation*}
with equivalences on both sides, for each $i:[A]\to X$.
\end{proof}
 
\section{The reflective factorization system for the reflexive coequalizer}

\subsection{\texorpdfstring{$\Delta$}{Δ}-\'etale maps}

\begin{defn}
Let $f:\mathsf{rGph}(\mathcal{A},\mathcal{B})$ be a morphism of reflexive graphs. We say that $f$ is \define{$\Delta$-\'etale} if the square
\begin{equation*}
\begin{tikzcd}
\mathcal{A} \arrow[d,swap,"f"] \arrow[r] & \Delta(\rcoeq(\mathcal{A})) \arrow[d,"\Delta(\rcoeq(f))"] \\
\mathcal{B} \arrow[r] & \Delta(\rcoeq(\mathcal{B}))
\end{tikzcd}
\end{equation*}
is a pullback square of reflexive graphs. We write $\mathsf{is\usc{}etale}_\Delta(f)$ for the proposition that $f$ is $\Delta$-\'etale, and we also write $\mathcal{R}^\Delta$ for the class of $\Delta$-\'etale morphisms of reflexive graphs.
\end{defn}

\begin{thm}\label{thm:etale_fibration}
Consider a morphism $f:\mathsf{rGph}(\mathcal{A},\mathcal{B})$ of reflexive graphs. The following are equivalent:
\begin{enumerate}
\item The morphism $f$ is a fibration in the sense of \cref{defn:graph_fibration}.
\item The morphism $f$ is $\Delta$-\'etale.
\end{enumerate}
\end{thm}

\begin{proof}
Consider the commuting diagram
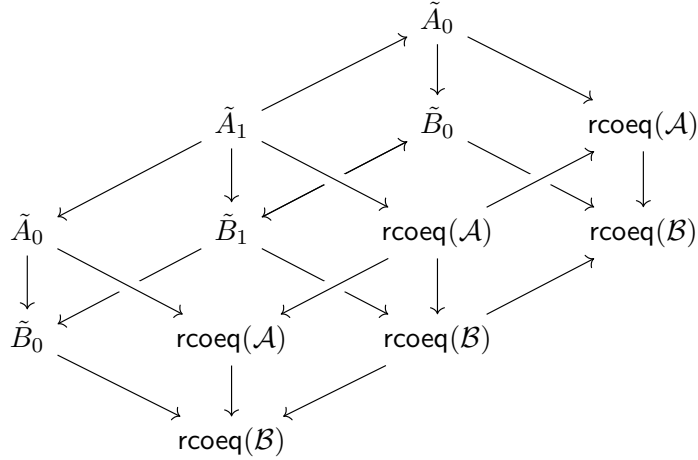
\begin{equation*}
\begin{tikzcd}
& & \pts{\tilde A} \arrow[d] \arrow[dr] \\
& \edg{\tilde A} \arrow[dl] \arrow[d] \arrow[ur] & \pts{\tilde B} \arrow[dl] \arrow[dr] & \rcoeq(\mathcal{A}) \arrow[d] \\
\pts{\tilde A} \arrow[d] & \edg{\tilde B} \arrow[dl] \arrow[dr] \arrow[ur] & \rcoeq(\mathcal{A}) \arrow[from=ul,crossing over] \arrow[dl,crossing over] \arrow[d] \arrow[ur,crossing over] & \rcoeq(\mathcal{B}) \\
\pts{\tilde B} \arrow[dr] & \rcoeq(\mathcal{A}) \arrow[d] \arrow[from=ul,crossing over] & \rcoeq(\mathcal{B}) \arrow[dl] \arrow[ur] \\
\phantom{\rcoeq(\mathcal{B})} & \rcoeq(\mathcal{B})
\end{tikzcd}
\end{equation*}
If $f$ is $\Delta$-\'etale, then the three parallel vertical squares are pullback squares, hence so are the two squares on the back left side. This shows that (ii) implies (i). 

Now suppose that $f$ is a fibration, or equivalently, that $f$ is cartesian. Then the map $\rcoeq(f)$ is the unique map such that the naturality squares are pullback squares. In particular, $f$ is a $\Delta$-\'etale map.
\end{proof}

\begin{cor}
Let $\mathcal{B}$ be a family of reflexive graphs over $\mathcal{A}$. The following are equivalent:
\begin{enumerate}
\item The family $\mathcal{B}$ is equifibered.
\item The morphism $\proj 1 : \mathsf{rGph}(\msm{\mathcal{A}}{\mathcal{B}},\mathcal{A})$ is $\Delta$-\'etale.
\end{enumerate}
\end{cor}

The following proposition is analogous to \cref{cor:etale_lex}.

\begin{prp}
Consider a pullback square
\begin{equation*}
\begin{tikzcd}
\mathcal{A}' \arrow[r] \arrow[d,swap,"{f'}"] & \mathcal{A} \arrow[d,"f"] \\
\mathcal{B}' \arrow[r] & \mathcal{B}
\end{tikzcd}
\end{equation*}
of reflexive graphs, and suppose that $f$ is a fibration. Then the square
\begin{equation*}
\begin{tikzcd}
\rcoeq(\mathcal{A}') \arrow[r] \arrow[d,swap,"\rcoeq({f'})"] & \rcoeq(\mathcal{A}) \arrow[d,"\rcoeq(f)"] \\
\rcoeq(\mathcal{B}') \arrow[r] & \rcoeq(\mathcal{B})
\end{tikzcd}
\end{equation*}
is again a pullback square.
\end{prp}

\begin{proof}
Consider the diagram
\begin{equation*}
\begin{tikzcd}
\mathcal{A}' \arrow[r] \arrow[d,swap,"{f'}"] & \mathcal{A} \arrow[d,"f"] \arrow[r,"\mathsf{constr}"] &[1ex] \Delta(\rcoeq(\mathcal{A})) \arrow[d,"\Delta(\rcoeq(f))"] \\
\mathcal{B}' \arrow[r] & \mathcal{B} \arrow[r,"\mathsf{constr}"'] & \Delta(\rcoeq(\mathcal{B}))
\end{tikzcd}
\end{equation*}
of reflexive graphs. Since $f$ is assumed to be a fibration, we obtain by \cref{thm:etale_fibration} that the square on the right is a pullback square. Furthermore, since the left square is a pullback square by assumption, it follows that the outer rectangle is again a pullback square. Hence the assertion follows from \cref{thm:rcoeq_cartesian}.
\end{proof}

The following proposition is analogous to \cref{cor:etale_pb}.

\begin{prp}
Consider a pullback square
\begin{equation*}
\begin{tikzcd}
\mathcal{A}' \arrow[r] \arrow[d,swap,"{f'}"] & \mathcal{A} \arrow[d,"f"] \\
\mathcal{B}' \arrow[r] & \mathcal{B}
\end{tikzcd}
\end{equation*}
of reflexive graphs, and suppose that $f$ is a fibration. Then $f'$ is a fibration.
\end{prp}

\begin{proof}
Consider the cube
\begin{equation*}
\begin{tikzcd}
& \mathcal{A}' \arrow[dr] \arrow[d] \arrow[dl] \\
\Delta(\rcoeq(\mathcal{A}')) \arrow[d] & \mathcal{B}' \arrow[dl] \arrow[dr] & \mathcal{A} \arrow[dl,crossing over] \arrow[d] \\
\Delta(\rcoeq(\mathcal{B}')) \arrow[dr] & \Delta(\rcoeq(\mathcal{A})) \arrow[d] \arrow[from=ul,crossing over] & \mathcal{B} \arrow[dl] \\
& \Delta(\rcoeq(\mathcal{B})). & \phantom{\Delta(\rcoeq(\mathcal{A}'))}
\end{tikzcd}
\end{equation*}
Then the two squares in the front and the square in the back right are pullback squares, so it follows that the square in the back left is a pullback square.
\end{proof}

\subsection{The reflective factorization system of discrete graphs}

\begin{defn}
Let $f:\mathsf{rGph}(\mathcal{B},\mathcal{A})$ be a morphism of reflexive graphs. We say that $f$ is a \define{$\Delta$-equivalence} if the map
\begin{equation*}
\begin{tikzcd}
\rcoeq(f):\rcoeq(\mathcal{B})\to \rcoeq(\mathcal{A})
\end{tikzcd}
\end{equation*}
is an equivalence. We also write $\mathcal{L}^\Delta$ for the class of $\Delta$-equivalences of reflexive graphs.
\end{defn}

\begin{lem}
The $\Delta$-equivalences satisfy the 3-for-2 property.\hfill$\square$
\end{lem}

\begin{prp}
Let $f:\mathsf{rGph}(\mathcal{B},\mathcal{A})$ be a morphism of reflexive graphs. The following are equivalent:
\begin{enumerate}
\item $f$ is a $\Delta$-equivalence.
\item For any type $X$, the map
\begin{equation*}
\mathsf{rGph}(\mathcal{A},\Delta(X))\to\mathsf{rGph}(\mathcal{B},\Delta(X))
\end{equation*}
is an equivalence.
\end{enumerate}
\end{prp}

\begin{proof}
We have a commuting square
\begin{equation*}
\begin{tikzcd}
\mathsf{rGph}(\mathsf{rcoeq}(\mathcal{A}),X) \arrow[r,"\blank\circ \mathsf{rcoeq}(f)"] \arrow[d] & \mathsf{rGph}(\mathsf{rcoeq}(\mathcal{B}),X) \arrow[d] \\
\mathsf{rGph}(\mathcal{A},\Delta(X)) \arrow[r,"\blank\circ f"] & \mathsf{rGph}(\mathcal{B},\Delta(X))
\end{tikzcd}
\end{equation*}
in which the two vertical maps are equivalences. Therefore it follows that if $\mathsf{rcoeq}(f)$ is an equivalence, then so is the bottom map in the square. Conversely, if the bottom map in the square is an equivalence for every type $X$, then $\blank\circ\mathsf{rcoeq}(f)$ is an equivalence for any type $X$, which implies by \cref{prp:equiv_precomp} that $\mathsf{rcoeq}(f)$ is an equivalence.
\end{proof}

\begin{defn}
Let $f:\mathsf{rGph}(\mathcal{B},\mathcal{A})$ be a morphism of reflexive graphs. We define the morphisms
\begin{align*}
\mathsf{et}_\Delta(f) & : \mathsf{rGph}(\mathcal{B}^{\eqvsym},\mathcal{A}) \\
\mathsf{\eta}_f & : \mathsf{rGph}(\mathcal{B},\mathcal{B}^{\eqvsym})
\end{align*}
by the universal property of pullbacks, as indicated in the following diagram
\begin{equation*}
\begin{tikzcd}
\mathcal{B} \arrow[ddr,bend right=15] \arrow[drr,bend left=15] \arrow[dr,densely dotted,"\eta_f" near end] \\
& \mathcal{B}^{\eqvsym} \arrow[dr,phantom,"\lrcorner" {very near start,xshift=-1ex}] \arrow[d,"\mathsf{et}_\Delta(f)"] \arrow[r] & \Delta(\rcoeq(\mathcal{B})) \arrow[d] \\[1ex]
& \mathcal{A} \arrow[r,swap,"\mathsf{constr}"] & \Delta(\rcoeq(\mathcal{A}))
\end{tikzcd}
\end{equation*}
The morphism $\mathsf{et}_\Delta(f)$ is called the \define{\'etale factor} of $f$, and the morphism $\eta_f$ is called the \define{unit} of the \'etale factor.
\end{defn}

\begin{lem}\label{lem:D_equiv_pullback}
For any morphism $f:\mathsf{rGph}(\mathcal{B},\mathcal{A})$, the unit $\eta_f$ of the \'etale factor of $f$ is a $\Delta$-equivalence. 
\end{lem}

\begin{proof}
By \cref{thm:rcoeq_cartesian}, it follows that the square
\begin{equation*}
\begin{tikzcd}
\rcoeq(\mathcal{B}^{\eqvsym}) \arrow[d] \arrow[r] & \rcoeq(\Delta(\rcoeq(\mathcal{B}))) \arrow[d] \\
\rcoeq(\mathcal{A}) \arrow[r] & \rcoeq(\Delta(\rcoeq(\mathcal{A})))
\end{tikzcd}
\end{equation*}
is a pullback square. Since the bottom map is an equivalence of reflexive graphs, it follows that the top map is an equivalence of reflexive graphs. In other words, we have shown that the map
\begin{equation*}
\begin{tikzcd}
\mathcal{B}^{\eqvsym} \arrow[r] & \Delta(\rcoeq(\mathcal{B}))
\end{tikzcd}
\end{equation*}
is a $\Delta$-equivalence. Of course, the morphism $\mathsf{constr}:\mathcal{B}\to\Delta(\rcoeq(\mathcal{B}))$ is also a $\Delta$-equivalence, so the claim follows by the 3-for-2 property of $\Delta$-equivalences.
\end{proof}

\begin{thm}\label{thm:modal_rofs}
The pair $(\mathcal{L}^\Delta,\mathcal{R}^\Delta)$ forms an orthogonal factorization system of reflexive graphs.
\end{thm}

\begin{proof}
Since any pullback of a morphism between discrete reflexive graphs is a fibration of reflexive graphs, it follows from \cref{lem:D_equiv_pullback} that every morphism of reflexive graphs factors as a $\Delta$-equivalence followed by a $\Delta$-\'etale map. Therefore it remains to show that the class of $\Delta$-equivalences is left orthogonal to the class of $\Delta$-\'etale maps.

We have to show that for every $\Delta$-equivalence $i:\mathcal{A}\to\mathcal{B}$, and every $\Delta$-\'etale morphism $f:\mathcal{X}\to\mathcal{Y}$, the square
\begin{equation*}
\begin{tikzcd}
\mathcal{X}^{\mathcal{B}} \arrow[r] \arrow[d] & \mathcal{Y}^{\mathcal{B}} \arrow[d] \\
\mathcal{X}^{\mathcal{A}} \arrow[r] & \mathcal{Y}^{\mathcal{A}}
\end{tikzcd}
\end{equation*}
is a pullback square. Consider the commuting cube
\begin{equation*}
\begin{tikzcd}
& \mathcal{X}^{\mathcal{B}} \arrow[dl] \arrow[d] \arrow[dr] \\
\Delta(\rcoeq(\mathcal{X}))^{\mathcal{B}} \arrow[d] & \mathcal{X}^{\mathcal{A}} \arrow[dl] \arrow[dr] & \mathcal{Y}^{\mathcal{B}} \arrow[d] \arrow[dl,crossing over] \\
\Delta(\rcoeq(\mathcal{X}))^{\mathcal{A}} \arrow[dr] & \Delta(\rcoeq(\mathcal{Y}))^{\mathcal{B}} \arrow[from=ul,crossing over] \arrow[d] & \mathcal{Y}^{\mathcal{A}} \arrow[dl] \\
& \Delta(\rcoeq(\mathcal{Y}))^{\mathcal{A}} & \phantom{\Delta(\rcoeq(\mathcal{X}))^{\mathcal{B}}}
\end{tikzcd}
\end{equation*}
In this cube the top and bottom squares are pullback by the assumption that $f$ is $\Delta$-\'etale and the fact that exponents of pullback squares are again pullback squares. Furthermore, the square in the front left is pullback, because the two vertical maps are equivalences by the assumption that $i:\mathcal{A}\to\mathcal{B}$ is a $\Delta$-equivalence. Therefore we conclude that the square in the back right is also a pullback square, as desired.
\end{proof}

\begin{thm}\label{thm:rgph_et_up}
For any morphism $f:\mathsf{rGph}(\mathcal{B},\mathcal{A})$, the pre-composition map
\begin{equation*}
\mathsf{rGph}_{\mathcal{A}}(\mathsf{et}_\Delta(f),e)\to \mathrm{hom}_{\mathcal{A}}(f,e)
\end{equation*}
is an equivalence for every $\Delta$-\'etale map $e:\mathcal{C}\to\mathcal{A}$.
\end{thm}

\begin{proof}
Analogous to \cref{thm:et_up}.
\end{proof}

\subsection{Equifibrant replacement}\label{sec:equifibrant_replacement}

\begin{defn}
Let $\mathcal{B}$ and $\mathcal{C}$ be families of reflexive graphs over $\mathcal{A}$. We define the type
\begin{equation*}
\mathsf{rGph}_{\mathcal{A}}(\mathcal{B},\mathcal{C})
\end{equation*}
of morphisms of reflexive graphs \define{over} $\mathcal{A}$ to consist of triples $(\pts{f},\edg{f},\rfx{f})$ consisting of
\begin{align*}
\pts{f}(x) & : \pts{B}(x)\to \pts{C}(x) \\
\edg{f}(e) & : \prd{u:\pts{B}(x)}{v:\pts{B}(y)} \edg{B}(e,u,v)\to \edg{C}(e,\pts{f}(u),\pts{f}(v)) \\
\rfx{f}(x) & : \prd{u:\pts{B}(x)} \edg{f}(\rfx{\mathcal{B}}(x,u))=\rfx{\mathcal{C}}(x,\pts{f}(u))
\end{align*}
\end{defn}

\begin{lem}
Let $\mathcal{B}$ and $\mathcal{C}$ be families of reflexive graphs over $\mathcal{A}$. Then there is an equivalence
\begin{equation*}
\eqv{\mathsf{rGph}_{\mathcal{A}}(\mathcal{B},\mathcal{C})}{\sm{f:\mathsf{rGph}(\msm{\mathcal{A}}{\mathcal{B}},\msm{\mathcal{A}}{\mathcal{B}})} \proj 1 = \proj 1 \circ f}.
\end{equation*}
\end{lem}

Note that, given a reflexive graph $\mathcal{A}$, any family $B:\mathsf{rcoeq}(\mathcal{A})\to\UU$ determines an equifibered family $\mathsf{equifib\usc{}fam}(B)$ over $\mathcal{A}$ given by
\begin{align*}
\pts{\mathsf{equifib\usc{}fam}(B)}(x) & \defeq B(\mathsf{constr}_0(x)) \\
\edg{\mathsf{equifib\usc{}fam}(B)}(e) & \defeq \mathsf{tr}_B(\mathsf{constr}_1(e)) \\
\rfx{\mathsf{equifib\usc{}fam}(B)}(x,y) & \defeq \mathsf{htpy\usc{}eq}(\mathsf{ap}_{\mathsf{tr}_B}(\rfx{\mathsf{constr}}(x)),y).
\end{align*}
In other words, $\mathsf{equifib\usc{}fam}$ is an operation that takes a family $B:\mathsf{rcoeq}(\mathcal{A})\to\UU$ to an equifibered family over $\mathcal{A}$. Moreover, it is not hard to see that
\begin{equation*}
\mathsf{equifib\usc{}fam} : (\mathsf{rcoeq}(\mathcal{A})\to\UU)\to\mathsf{equifib}(\mathcal{A})
\end{equation*}
is in fact an equivalence, where $\mathsf{equifib}(\mathcal{A})$ is the type of all (small) equifibered families over $\mathcal{A}$. In the following definition we observe that any family over $\mathcal{A}$ induces an equifibered family over $\mathcal{A}$. 

\begin{defn}
Let $\mathcal{B}$ be a family of reflexive graphs over $\mathcal{A}$. We define the \define{equifibrant replacement} $\mathsf{EqF}(\mathcal{B})$ of $\mathcal{B}$ by
\begin{equation*}
\mathsf{EqF}(\mathcal{B}) \defeq \mathsf{equifib\usc{}fam}(\fibf{\rcoeq(\proj 1)}),
\end{equation*}
and we define the morphism $\eta_{\mathcal{B}}:\mathsf{rGph}_{\mathcal{A}}(\mathcal{B},\mathsf{EqF}(\mathcal{B}))$. 
\end{defn}

\begin{thm}\label{thm:eqf_initial}
Consider a family $\mathcal{B}$ of reflexive graphs over $\mathcal{A}$, and let $\mathcal{E}$ be an equifibered family over $\mathcal{A}$. Then the pre-composition operation
\begin{equation*}
\mathsf{rGph}_{\mathcal{A}}(\mathsf{EqF}(\mathcal{B}),\mathcal{E})\to \mathsf{rGph}_{\mathcal{A}}(\mathcal{B},\mathcal{E})
\end{equation*}
is an equivalence for every equifibered family $\mathcal{E}$ over $\mathcal{A}$.
\end{thm}

\begin{proof}
First we observe that the square
\begin{equation*}
\begin{tikzcd}
\msm{\mathcal{A}}{\mathsf{EqF}(\mathcal{B})} \arrow[r] \arrow[d,swap,"p"] & \Delta(\mathsf{rcoeq}(\msm{\mathcal{A}}{\mathcal{B}})) \arrow[d,"\mathsf{rcoeq}(\proj 1)"] \\
\mathcal{A} \arrow[r,swap,"\mathsf{constr}"] & \Delta(\mathsf{rcoeq}(\mathcal{A}))
\end{tikzcd}
\end{equation*}
is a pullback square of reflexive graphs. Thus, the map $p$ is the \'etale factor of the morphism $\proj 1 : \msm{\mathcal{A}}{\mathcal{B}}\to \mathcal{A}$. Now the claim follows by \cref{thm:rgph_et_up}.
\end{proof}

\begin{rmk}
By the universal property of the equifibrant replacement it follows that the family $\pts{\mathsf{EqF}(\mathcal{B})}$ of vertices of the equifibrant replacement of $\mathcal{B}$ can be seen as a `recursive' higher inductive family of types. It comes equipped with
\begin{align*}
\zeta & : \prd{x,y:\pts{A}} \edg{A}(x,y)\to \eqv{\pts{\mathsf{EqF}(\mathcal{B})}(x)}{\pts{\mathsf{EqF}(\mathcal{B})}(y)} \\
\epsilon & : \prd{x:\pts{A}} \zeta(\rfx{\mathcal{A}}(x))\htpy \idfunc[\pts{\mathsf{EqF}(\mathcal{B})}(x)] \\
\pts{\eta} & : \prd{x:\pts{A}} \pts{B}(x) \to \pts{\mathsf{EqF}(\mathcal{B})}(x) \\
\edg{\eta} & : \prd{x,y:\pts{A}}{e:\edg{A}(x,y)}{u:\pts{B}(x)}{v:\pts{B}(y)} \edg{B}(e,u,v) \to \zeta(e,u)=v \\
\rfx{\zeta} & : \prd{x:\pts{A}}{u:\pts{B}(x)} \edg{\eta}(\rfx{\mathcal{A}}(x),\rfx{\mathcal{B}}(u))= \epsilon(x,u).
\end{align*}
In other words, the equifibrant replacement is the homotopy initial family of types over $\pts{A}$ with equivalences over the edges of $\mathcal{A}$ (coherent with reflexivity) and a morphism $\eta : \mathsf{rGph}_{\mathcal{A}}(\mathcal{B},\mathsf{EqF}(\mathcal{B}))$.
\end{rmk}

\subsection{Identity types of reflexive coequalizers}

\begin{defn}
Consider a reflexive graph $\mathcal{A}$ with a base point $a:A_0$. We define the \define{universal $\Delta$-bundle} $\mathcal{E}(a)$ over $\mathcal{A}$ at $a$ to be the equifibered family over $\mathcal{A}$ corresponding to the equifibrant replacement of the morphism $a:\mathsf{rGph}(\unit,\mathcal{A})$ corresponding to $a:A_0$. 
\end{defn}

\begin{prp}
Consider a reflexive graph $\mathcal{A}$ with a base point $a:A_0$. Then there are equivalences
\begin{equation*}
\pts{\mathcal{E}(a)}(x) \eqvsym (\mathsf{constr}_0(a)=\mathsf{constr}_0(x))
\end{equation*}
In particular, we have an equivalence $\eqv{\pts{\mathcal{E}(a)}(a)}{\loopspace{\rcoeq(\mathcal{A})}}$. 
\end{prp}

\begin{proof}
Immediate from the definition of the equifibrant replacement.
\end{proof}

In the following theorem we establish the universal property of the identity type of $\rcoeq(\mathcal{A})$ as the initial reflexive relation $R$ on $\pts{A}$ with certain extra structure. 

\begin{thm}
Consider a reflexive graph $\mathcal{A}$, and let
\begin{align*}
R & : A_0 \to A_0 \to\UU, \\
\rho & : \prd{x:A_0}R(x,x)
\end{align*}
be a reflexive relation on $A_0$ equipped with a `composition' operation
\begin{equation*}
\mu:\prd{a,x,y:\pts{A}} \edg{A}(x,y) \to (\eqv{R(a,x)}{R(a,y)}).
\end{equation*}
satisfying the unit law
\begin{align*}
\mathsf{left\usc{}unit}_\mu & : \prd{a,x:\pts{A}}{r:R(a,x)} \mu(\rfx{\mathcal{A}}(x),r)=r.
\end{align*}
Then there is a unique extension
\begin{equation*}
\begin{tikzcd}
\mathcal{A} \arrow[r,"\bar{\mu}"] \arrow[d] & (A_0,R,\rho) \\
k(\mathsf{constr}_0) \arrow[ur,densely dotted]
\end{tikzcd}
\end{equation*}
of the morphism $\bar{\mu}:\mathsf{rGph}(\mathcal{A},(A_0,R,\rho))$ consisting of
\begin{align*}
\pts{\bar{\mu}} & \defeq \mathsf{constr}_0 \\
\edg{\bar{\mu}}(e) & \defeq \mu(e,\rho(x)) \\
\rfx{\bar{\mu}}(x) & \defeq \mathsf{left\usc{}unit}_\mu(\rho(x)),
\end{align*}
where $k(\mathsf{constr}_0)$ is the pre-kernel of the function $\mathsf{constr}_0:A_0\to\rcoeq(\mathcal{A})$.
\end{thm}

\begin{proof}
The data $(R,\mu,\mathsf{left\usc{}unit}_\mu)$ equivalently gives for every $a:A$ an equifibered family $\mathcal{B}(a)$ over $\mathcal{A}$ consisting of
\begin{align*}
\pts{\mathcal{B}(a)}(x) & \defeq R(a,x) \\
\edg{\mathcal{B}(a)}(e) & \defeq \mu(e) \\
\rfx{\mathcal{B}(a)}(x,r) & \defeq \mathsf{left\usc{}unit}_\mu(r),
\end{align*}
and the reflexivity term $\rho$ provides for every $a:\pts{A}$ a term of type $\pts{\mathcal{B}(a)}$. Thus we have by \cref{thm:eqf_initial} for every $a:\pts{A}$ a unique morphism from the equifibrant replacement of $a:\unit\to \mathcal{A}$ to $\mathcal{B}$ over $\mathcal{A}$. Equivalently, we have a unique extension of $\bar{\mu}$ along $\mathcal{A}\to k(\mathsf{constr}_0)$ as asserted.
\end{proof}

\begin{cor}
The loop space of the suspension $\susp X$ of a pointed type $X$ is the initial pointed type $Y$ equipped with a pointed map 
\begin{equation*}
X\to_\ast (\eqv{Y}{Y}).
\end{equation*}
In particular, the loop space of the $(n+1)$-sphere is the initial pointed type $Y$ equipped with a pointed map $\sphere{n}\to (\eqv{Y}{Y})$, or equivalently, an $(n+1)$-loop $\loopspace[n+1]{\mathrm{BAut}(Y)}$. Even more in particular, the loop space of the $2$-sphere is the initial pointed type $Y$ equipped with a homotopy $\idfunc[Y]\htpy\idfunc[Y]$. 
\end{cor}

\begin{cor}
The type $\Z$ if integers is the initial pointed type equipped with an automorphism.
\end{cor}

\section{Equifibrant replacement for other homotopy colimits}

\subsection{Equifibrant replacement for diagrams over graphs}

Recall that the colimit of a diagram $\mathcal{D}$ over a reflexive graph $\mathcal{A}$ is simply the reflexive coequalizer of the reflexive graph $\msm{\mathcal{A}}{\mathcal{D}}$.

\begin{defn}
Let $\tau:\mathcal{D}'\to\mathcal{D}$ be a natural transformation of diagrams over a reflexive graph $\mathcal{A}$. 
\begin{enumerate}
\item We say that $\tau$ is a \define{weak equivalence} if it induces an equivalence
\begin{equation*}
\tfcolim(\tau):\tfcolim(\mathcal{D}')\to\tfcolim(\mathcal{D}).
\end{equation*}
We write $\mathcal{W}$ for the class of weak equivalences.
\item We say that $\tau$ is \define{\'etale} if the square
\begin{equation*}
\begin{tikzcd}
\msm{\mathcal{A}}{\mathcal{D}'} \arrow[d,swap,"\tau"] \arrow[r] & \Delta(\tfcolim(\mathcal{D}')) \arrow[d,"\tfcolim(\tau)"] \\
\msm{\mathcal{A}}{\mathcal{D}} \arrow[r] & \Delta(\tfcolim(\mathcal{D}))
\end{tikzcd}
\end{equation*}
is a pullback square of reflexive graphs. We write $\mathcal{R}$ for the class of \'etale maps.
\end{enumerate}
\end{defn}

\begin{thm}
A natural transformation $\tau:\mathcal{D}'\to\mathcal{D}$ is \'etale if and only if it is cartesian.
\end{thm}

\begin{proof}
This follows directly by \cref{thm:etale_fibration,prp:nattrans_cartesian}.
\end{proof}

\begin{thm}
The pair $(\mathcal{W},\mathcal{R})$ forms an orthogonal factorization system. 
\end{thm}

\begin{proof}
This follows from \cref{thm:modal_rofs}. 
\end{proof}

\section{Equifibrant replacement for sequential colimits}\label{sec:seqcolim_eqf}%
\sectionmark{Sequential colimits}

\begin{defn}
Let $\mathcal{B}$ be a sequential family over $\mathcal{A}$, and consider $x:A_n$. We write $\mathcal{B}_n[x]$ for the type sequence
\begin{equation*}
\begin{tikzcd}
B_n(x) \arrow[r] & B_{n+1}(f_n(x)) \arrow[r] & B_{n+2}(f_{n+1}(f_n(x))) \arrow[r] & \cdots.
\end{tikzcd}
\end{equation*}
\end{defn}

\begin{defn}
Let $\mathcal{B}$ be a sequential family over $\mathcal{A}$. We will define an equifibered sequential family $\square\mathcal{B}$ over $\mathcal{A}$ equipped with a morphism 
\begin{equation*}
\mathrm{hom}_{\mathcal{A}}(\mathcal{B},\square\mathcal{B})
\end{equation*}
\end{defn}

\begin{proof}[Construction]
We define
\begin{equation*}
\square B : \prd{n:\N}A_n\to\UU
\end{equation*}
by $\square B_n(x)\defeq \tfcolim \mathcal{B}_n[x]$. Next, we have to construct a fiberwise equivalence
\begin{equation*}
g_\infty : \prd{n:\N}{x:A_n} \square B_n(x) \eqvsym \square B_{n+1}(f_n(x))
\end{equation*}
Note that we have the natural transformation
\begin{equation*}
\begin{tikzcd}[column sep=small]
B_n(x) \arrow[d] \arrow[r] & \cdots \arrow[r] & B_{k+n}(f^k(x)) \arrow[r] \arrow[d] & B_{(k+1)+n}(f^{k+1}(x)) \arrow[r] \arrow[d] & \cdots \\
B_{n+1}(x) \arrow[r] & \cdots \arrow[r] & B_{k+(n+1)}(f^k(f(x))) \arrow[r] & B_{(k+1)+(n+1)}(f^{k+1}(f(x))) \arrow[r]
& \cdots
\end{tikzcd}
\end{equation*}
Now we observe that there are identifications
\begin{equation*}
(k+(n+1),f^k(f(x)))=((k+1)+n,f^{k+1}(x)),
\end{equation*}
which we may use to construct equivalences going up diagonally in each naturality square. An induction argument reveals that each square of the form
\begin{equation*}
\begin{tikzcd}
B_{k+(n+1)}(f^k(f(x))) \arrow[r] \arrow[d] & B_{(k+1)+(n+1)}(f^{k+1}(f(x))) \arrow[d] \\
B_{(k+1)+n}(f^{k+1}(x)) \arrow[r] & B_{(k+2)+n}(f^{k+2}(x))
\end{tikzcd}
\end{equation*}
commutes, from which we obtain the desired equivalence.\footnote{A formalization of this argument appears in \url{https://github.com/cmu-phil/Spectral/blob/master/colimit/seq_colim.hlean}}
\end{proof}

\begin{thm}\label{thm:colim_fib}
The equifibered family $\square\mathcal{B}$ over $\mathcal{A}$ satisfies the universal property that
\begin{equation*}
\mathrm{hom}_{\mathcal{A}}(\square\mathcal{B},\mathcal{E})\to\mathrm{hom}_{\mathcal{A}}(\mathcal{B},\mathcal{E})
\end{equation*}
is an equivalence for any equifibered family $\mathcal{E}$ over $\mathcal{A}$. 
\end{thm}

\begin{proof}
This fact is formalized in \cite{DoornRijkeSojakova}. An `informalized' proof is work in progress.
\end{proof}

\begin{cor}\label{thm:colim_sigma}
Let $\mathcal{B}$ be a sequential family over $\mathcal{A}$. Then we have a commuting triangle
\begin{equation*}
\begin{tikzcd}[column sep=tiny]
\mathsf{colim}\Big(\msm{\mathcal{A}}{\mathcal{B}}\Big) \arrow[rr] \arrow[dr,swap,"\mathsf{colim}(\proj 1)"] & & \sm{x:A_\infty} B_\infty(x) \arrow[dl,"\proj 1"] \\
\phantom{\sm{x:A_\infty} B_\infty(x)} & A_\infty & \phantom{\mathsf{colim}\Big(\msm{\mathcal{A}}{\mathcal{B}}\Big)}
\end{tikzcd}
\end{equation*}
in which the top map is an equivalence.
\end{cor}

\begin{proof}
Since $\square\mathcal{B}$ and $\mathsf{EqF}\mathcal{B}$ both satisfy the universal property of the equifibrant replacement of $\mathcal{B}$, they are the same sequences over $\mathcal{A}$. It follows that the square
\begin{equation*}
\begin{tikzcd}
\msm{\mathcal{A}}{\square\mathcal{B}} \arrow[r] \arrow[d,swap,"\proj 1"] & \Delta(\mathsf{colim}(\msm{\mathcal{A}}{\mathcal{B}})) \arrow[d,"\Delta(\mathsf{colim}(\mathcal{A}))"] \\
\mathcal{A} \arrow[r,swap,"\mathsf{seq\usc{}in}"] & \Delta(\mathsf{colim}(\mathcal{A}))
\end{tikzcd}
\end{equation*}
is a pullback square of type sequences. We conclude that the top arrow is a colimiting cocone, so the result follows.
\end{proof}

\begin{cor}\label{thm:colim_id}
Consider a type sequence
\begin{equation*}
\begin{tikzcd}
A_0 \arrow[r] & A_1 \arrow[r] & A_2 \arrow[r] & \cdots.
\end{tikzcd}
\end{equation*}
Then the canonical map
\begin{equation*}
\mathsf{colim}(f^n(x)= f^n(y)) \to (\mathsf{seq\usc{}in}_0(x)=\mathsf{seq\usc{}in}_0(y))
\end{equation*}
is an equivalence for every $x,y:A_0$.
\end{cor}

\begin{thm}\label{thm:seq_colim_pb}
Consider a sequence
\begin{equation*}
\begin{tikzcd}[column sep=small,row sep=small]
& C_0 \arrow[dl] \arrow[dd] \arrow[rr] & & C_1 \arrow[dl] \arrow[dd] \arrow[rr] & & C_2 \arrow[dl] \arrow[dd] \arrow[rr] & & \cdots \\
A_0 \arrow[dd] \arrow[rr,crossing over] & & A_1 \arrow[rr,crossing over] & & A_2 \arrow[rr,crossing over] & & \cdots \\
& B_0 \arrow[dl] \arrow[rr] & & B_1 \arrow[dl] \arrow[rr] & & B_2 \arrow[dl] \arrow[rr] & & \cdots \\
X_0 \arrow[rr] & \phantom{B_2} & X_1 \arrow[rr] \arrow[from=uu,crossing over] & \phantom{B_2} & X_2 \arrow[rr] \arrow[from=uu,crossing over] & \phantom{B_2} & \cdots
\end{tikzcd}
\end{equation*}
of pullback squares. Then the sequential colimit
\begin{equation*}
\begin{tikzcd}
C_\infty \arrow[r] \arrow[d] & B_\infty \arrow[d] \\
A_\infty \arrow[r] & X_\infty
\end{tikzcd}
\end{equation*}
is again a pullback square.
\end{thm}

\begin{proof}
Since sequential colimits commute with $\Sigma$ and identity types, they commute with pullbacks.
\end{proof}

\begin{cor}\label{thm:colim_fiberseq}
Consider a sequence of fiber sequences
\begin{equation*}
\begin{tikzcd}
F_0 \arrow[d] \arrow[r] & F_1 \arrow[d] \arrow[r] & F_2 \arrow[d] \arrow[r] & \cdots \\
E_0 \arrow[d] \arrow[r] & E_1 \arrow[d] \arrow[r] & E_2 \arrow[d] \arrow[r] & \cdots \\
B_0 \arrow[r] & B_1 \arrow[r] & B_2 \arrow[r] & \cdots
\end{tikzcd}
\end{equation*}
in which all maps and homotopies are assumed to be pointed. Then the colimit
\begin{equation*}
\begin{tikzcd}
F_\infty \hookrightarrow E_\infty \twoheadrightarrow B_\infty
\end{tikzcd}
\end{equation*}
is again a fiber sequence.
\end{cor}

\begin{cor}
Consider a type sequence
\begin{equation*}
\begin{tikzcd}
A_0 \arrow[r] & A_1 \arrow[r] & A_2 \arrow[r] & \cdots
\end{tikzcd}
\end{equation*}
of pointed types (and pointed maps between them). Then the canonical map
\begin{equation*}
\mathsf{colim}(\loopspace{A_{n}}) \to \loopspace{A_{\infty}}
\end{equation*}
is an equivalence.
\end{cor}

\begin{prp}\label{thm:colim_hlevel}
Consider a type sequence
\begin{equation*}
\begin{tikzcd}
A_0 \arrow[r] & A_1 \arrow[r] & A_2 \arrow[r] & \cdots
\end{tikzcd}
\end{equation*}
If each $A_n$ is $k$-truncated, then so is the sequential colimit $A_\infty$.
\end{prp}

\begin{proof}
We prove the claim by induction on $k\geq -2$. The base case is trivial, and the inductive step follows since sequential colimits commute with identity types.
\end{proof}

\begin{thm}
Consider a type sequence
\begin{equation*}
\begin{tikzcd}
A_0 \arrow[r] & A_1 \arrow[r] & A_2 \arrow[r] & \cdots
\end{tikzcd}
\end{equation*}
Then the canonical map
\begin{equation*}
\mathsf{colim}_n \trunc{k}{A_n} \to \trunc{k}{A_\infty}
\end{equation*}
is an equivalence. 
\end{thm}

\begin{thm}\label{thm:colim_hgroup}
Consider a type sequence
\begin{equation*}
\begin{tikzcd}
A_0 \arrow[r] & A_1 \arrow[r] & A_2 \arrow[r] & \cdots
\end{tikzcd}
\end{equation*}
Then the canonical map
\begin{equation*}
\mathsf{colim}_n(\pi_k(A_n))\to \pi_k(A_\infty)
\end{equation*}
is a group isomorphism, for any $k\geq 1$.
\end{thm}

\chapter{Compact types}\label{chap:compact}

In this chapter we introduce the notion of compact type, and we show that the compact types are closed under finite coproducts, $\Sigma$-types, finite products, pushouts, reflexive coequalizers, and retracts. In particular it follows that the spheres are compact.

In \cref{sec:localization_compact} we show that if $P$ and $Q$ are families of compact types, then the subuniverse of $F$-local types is reflective. Of course, one might impose stronger assumptions on the type theory so that the subuniverse of $F$-local types is reflective for any family $F$ of maps (see for instance \cite{RijkeShulmanSpitters}). The setting of this dissertation is that of univalent type theory in which the universes are closed under homotopy pushouts, so we have to find a different way to define localizations.

The idea is to approximate the $F$-localization of a type $X$ by iterated `quasi $F$-local extensions' of $X$, by which we mean a map $l:X\to Y$ equipped with a diagonal filler for the square
\begin{equation*}
\begin{tikzcd}
X^{Q_i} \arrow[r,"l\circ\blank"] \arrow[d,swap,"\blank\circ F_i"] & Y^{Q_i} \arrow[d,"\blank\circ F_i"] \\
X^{P_i} \arrow[r,swap,"l\circ \blank"] & Y^{P_i},
\end{tikzcd}
\end{equation*}
for each $i:I$. We show in \cref{prp:colim_local} that, provided that $F$ is a family of maps between compact types, the sequential colimit of type sequence
\begin{equation*}
\begin{tikzcd}
X_0 \arrow[r,"l_0"] & X_1 \arrow[r,"l_1"] & X_2 \arrow[r,"l_2"] & \cdots
\end{tikzcd}
\end{equation*}
is $F$-local if each $l_n$ has the structure of a quasi $F$-local extension. We show that there is an initial quasi $F$-local extension $l:X\to QL_F X$, for every type $X$, and in \cref{thm:localization} we show that the sequential colimit $X\to QL_F^\infty$ is an $F$-localization. We establish in \cref{thm:localization} that if $F$ is a family of maps between compact types, then the subuniverse of $F$-local types is a reflective subuniverse. Furthermore, if $A$ is a family of compact types, then the subuniverse of $A$-null types is a modality.

Dan Christensen pointed out that this construction of $F$-localization already appears in section 1.B of \cite{Farjoun}, where arbitrary localizations are constructed by iterating the initial quasi $F$-local extension transfinitely many times. We stick to the case $\omega$ since larger ordinals are not yet as well understood in homotopy type theory. A minor note is that in Theorem B.5 Dror Farjoun only establishes that the transfinite colimit $QL_F^\kappa X$ localizes the homotopy groups. Our proofs are therefore different than Dror Farjoun's.

\section{Compact types}

\begin{defn}
Consider a type sequence $\mathcal{A}\jdeq (A,f)$, and let $X$ be a type. We define the type sequence $\mathcal{A}^X$ to consist of
\begin{equation*}
\begin{tikzcd}
A_0^X \arrow[r,"f_0\circ \blank"] & A_1^X \arrow[r,"f_1\circ\blank"] & A_2^X \arrow[r,"f_2\circ\blank"] & \cdots. 
\end{tikzcd}
\end{equation*}
Furthermore, given a natural transformation $\tau:\mathsf{Seq}(\mathcal{A},\mathcal{B})$ of type sequences, we define the natural transformation $\tau^X$ to consist of
\begin{equation*}
\begin{tikzcd}[column sep=large]
A_0^X \arrow[r,"f_0\circ \blank"] \arrow[d,swap,"\tau_0\circ\blank"] & A_1^X \arrow[r,"f_1\circ\blank"] \arrow[d,swap,"\tau_1\circ\blank"] & A_2^X \arrow[r,"f_2\circ\blank"] \arrow[d,swap,"\tau_2\circ\blank"] & \cdots \\
B_0^X \arrow[r,swap,"g_0\circ \blank"] & B_1^X \arrow[r,swap,"g_1\circ\blank"] & B_2^X \arrow[r,swap,"g_2\circ\blank"] & \cdots
\end{tikzcd}
\end{equation*}
where the naturality squares commute by whiskering the naturality squares of $\tau$. 
\end{defn}

In particular, if $\tau:\mathsf{Seq}(\mathcal{A},\Delta Y)$ is a cocone with vertex $Y$, then $\tau^X :\mathsf{Seq}(\mathcal{A}^X,(\Delta Y)^X)$ is a cocone with vertex $Y^X$, since $(\Delta Y)^X \jdeq \Delta (Y^X)$.  

\begin{defn}
We call a type $X$ \define{(sequentially) compact} if for any type sequence
\begin{equation*}
\begin{tikzcd}
A_0 \arrow[r] & A_1 \arrow[r] & A_2 \arrow[r] & \cdots
\end{tikzcd}
\end{equation*}
the map
\begin{equation*}
\mathsf{seq\usc{}ind}(\mathsf{constr}(\mathcal{A})^X)
  : \tfcolim(X\to A_n)\to(X\to A_\infty)
\end{equation*}
is an equivalence.
\end{defn}

\begin{eg}
The empty type is compact, simply because $X^\emptyt$ is contractible for any type $X$, and a sequential colimit of contractible types is contractible.
\end{eg}

\begin{eg}
The unit type is compact because $\lam{x}\mathsf{const}_x:X\to X^\unit$ is an equivalence for any type $X$. Thus we obtain a natural equivalence
\begin{equation*}
\begin{tikzcd}
A_0 \arrow[d] \arrow[r] & A_1 \arrow[d] \arrow[r] & A_2 \arrow[d] \arrow[r] & \cdots \\
A_0^\unit \arrow[r] & A_1^\unit \arrow[r] & A_2^\unit \arrow[r] & \cdots.
\end{tikzcd}
\end{equation*}
Thus we obtain a commuting triangle
\begin{equation*}
\begin{tikzcd}[column sep=0]
 & A_\infty \arrow[dl] \arrow[dr] & \phantom{\mathsf{colim}(A_n^\unit)} \\
\mathsf{colim}(A_n^\unit) \arrow[rr] & & A_\infty^\unit.
\end{tikzcd}
\end{equation*}
where two out of three maps are equivalences. Therefore the unit type is compact by the 3-for-2 property.
\end{eg}

We will show in the remainder of this section that the compact types are closed under finite coproducts, $\Sigma$-types (and therefore also finite products), pushouts, reflexive coequalizers, and retracts. In the following example we note that compact types are not closed under pullbacks, and are not closed under exponentiation either.

\begin{rmk}
The compact types are not closed under identity types, because $\Z$ is not compact
and $\eqv{\Z}{\Omega(\Sn^1)}$. To see that $\Z$ is not compact, consider
the sequence $A_n\defeq\nat_{-n}$, where $\nat_{-n}$ is a version of the natural
numbers of starting at $-n$ and $A_n\to A_{n+1}$ is the inclusion map. The colimit
of $A_n$ is $\Z$ and the image of the map $\tfcolim(\Z\to A_n)\to (\Z\to\Z)$ is
the set of maps that are bounded from below. Thus, it is not surjective.

Also, the compact types are not closed under exponentiation. We have
the equivalence $\eqv{(\Sn^1\to\Sn^1)}{\Sn^1\times\Z}$, which is not
compact.
\end{rmk}

\begin{lem}
Finite coproducts of compact types are compact.
\end{lem}

\begin{proof}
We have already seen that the empty type is compact. 
Thus it remains to show that compact types are closed under disjoint
sums. 

Suppose that $X$ and $Y$ are compact. Then we have a commuting diagram of the form
\begin{equation*}
\begin{tikzcd}
\mathsf{colim}(\mathcal{A}^{X+Y}) \arrow[r] \arrow[d] & \mathsf{colim}(\mathcal{A}^X\times \mathcal{A}^Y) \arrow[r] & \mathsf{colim}(\mathcal{A}^X)\times\mathsf{colim}(\mathcal{A}^Y) \arrow[d] \\
A_{\infty}^{X+Y} \arrow[rr] & & A_\infty^X \times A_\infty^Y
\end{tikzcd}
\end{equation*}
where all but one of the maps are equivalences. Thus the remaining map is an equivalence too.
\end{proof}

\begin{defn}
Let $\mathcal{A}:X\to \mathsf{Seq}$ be an indexed type sequence. Then we define the type sequence $\prd{x:X}\mathcal{A}(x)$ to consist of
\begin{equation*}
\begin{tikzcd}[column sep=7em]
\prd{x:X}A_0(x) \arrow[r,"{\lam{h}{x}f_0(x,h(x))}"] & \prd{x:X}A_1(x) \arrow[r,"{\lam{h}{x}f_1(x,h(x))}"] & \cdots
\end{tikzcd}
\end{equation*}
\end{defn}

\begin{prp}
Let $\mathcal{A}:X\to\mathsf{Seq}$ be an indexed type sequence, indexed by a compact type $X$.
Then the canonical map
\begin{equation*}
\tfcolim\Big(\prd{x:X}A_n(x)\Big) \to \Big(\prd{x:X} A_\infty(x)\Big)
\end{equation*}
is an equivalence.
\end{prp}

\begin{proof}
Consider the commuting diagram
\begin{equation*}
\begin{tikzcd}
\mathsf{colim}\Big(\prd{x:X}\mathcal{A}(x)\Big) \arrow[d] \arrow[r] & \prd{x:X}A_\infty(x) \arrow[d] \arrow[r] & \unit \arrow[d,"\mathsf{const}_{\idfunc[X]}"] \\
\mathsf{colim}\Big(X\to\sm{x:X}A_n(x)\Big) \arrow[r] & \Big(X\to \sm{x:X}A_\infty(x)\Big) \arrow[r,swap,"\proj 1\circ\blank"] & (X\to X).
\end{tikzcd}
\end{equation*}
The square on the right and the outer rectangle are pullback squares. Hence the square on the left is a pullback square. Now observe that the type $\sm{x:X}A_\infty(X)$ is the sequential colimit of the type sequence $\sm{x:X}A_n(x)$, by \cref{thm:colim_sigma}. Therefore it follows by the assumption that $X$ is compact, that the bottom map in the left square is an equivalence. We conclude that the top map in the left square is an equivalence, proving the claim.
\end{proof}

\begin{cor}
Suppose $X$ is compact, and $Y(x)$ is compact for each $x:X$. Then the total space $\sm{x:X}Y(x)$ is also compact.
\end{cor}

\begin{proof}
Let $\mathcal{A}$ be a type sequence. Then we have the following commuting pentagon
\begin{equation*}
\begin{tikzcd}[column sep=-5em]
& \mathsf{colim}\Big(\mathcal{A}^{\sm{x:X}Y(x)}\Big) \arrow[rr] \arrow[dl] & & A_\infty^{\sm{x:X}Y(x)} \arrow[dr] \\
\mathsf{colim}\Big(\prd{x:X}\mathcal{A}^{Y(x)}\Big) \arrow[drr] & & & & \prd{x:X} A_\infty^{Y(x)} \\
& & \prd{x:X}\mathsf{colim}\Big(\mathcal{A}^{Y(x)}\Big) \arrow[urr] & \phantom{\mathsf{colim}\Big(\mathcal{A}^{\sm{x:X}Y(x)}\Big)} & \phantom{\mathsf{colim}\Big(\prd{x:X}\mathcal{A}^{Y(x)}\Big)}
\end{tikzcd}
\end{equation*}
in which all but the top map are known to be equivalences. Hence it follows that the top map is an equivalence, which shows that $\sm{x:X}Y(x)$ is compact.
\end{proof}

\begin{cor}
Compact types are closed under finite products.
\end{cor}

\begin{prp}\label{prp:pushout_compact}
Compact types are closed under pushouts.
\end{prp}

\begin{proof}
Consider a span $\mathcal{S}\jdeq (S,f,g)$ from $X$ to $Y$, where $X$, $Y$, and $S$ are assumed to be compact, and let $\mathcal{A}$ be a type sequence.
Then we have the commuting pentagon
\begin{equation*}
\begin{tikzcd}[column sep=-5em]
& \mathsf{colim}\Big(\mathcal{A}^{X\sqcup^{S} Y}\Big) \arrow[rr] \arrow[dl] &[-2em] &[-2em] A_\infty^{X\sqcup^S Y} \arrow[dr] \\
\mathsf{colim}\Big(\mathcal{A}^X\times_{\mathcal{A}^S}\mathcal{A}^Y\Big) \arrow[drr] & & & & A_\infty^X \times_{A_\infty^S} A_\infty^Y \\
& & \mathsf{colim}(\mathcal{A}^X)\times_{\mathsf{colim}(\mathcal{A}^S)}\mathsf{colim}(\mathcal{A}^Y) \arrow[urr] & \phantom{\mathsf{colim}\Big(\mathcal{A}^{X\sqcup^{S} Y}\Big)} & \phantom{\mathsf{colim}\Big(\mathcal{A}^X\times_{\mathcal{A}^S}\mathcal{A}^Y\Big)}
\end{tikzcd}
\end{equation*}
In this diagram, the downwards maps from the top left and from the top right are equivalences by the universal property of pushouts. The downwards map to the left side to the bottom is an equivalence since sequential colimits preserve pullbacks by \cref{thm:seq_colim_pb}. The map from the bottom to the right side is an equivalence by the assumption that $X$, $Y$, and $S$ are all compact types. Therefore we conclude that the top map is an equivalence, which proves the claim.
\end{proof}

\begin{cor}
Compact types are closed under joins, suspensions, and wedges. Furthermore, the cofiber of a map between compact types is again compact. In particular, all the spheres and the finite dimensional real and complex projective spaces are compact, and the smash product of compact pointed types is again compact.
\end{cor}

\begin{defn}
A reflexive graph $\mathcal{A}$ is said to be compact its types of vertices and (the total space of) edges are compact.
\end{defn}

\begin{cor}
The reflexive coequalizer of a compact reflexive graph is again compact.
\end{cor}

\begin{proof}
This follows immediately from \cref{prp:pushout_compact,thm:rcoeq_is_pushout}.
\end{proof}

\begin{prp}
Compact types are closed under retracts.
\end{prp}

\begin{proof}
Suppose that $Y$ is a retract of a compact type $X$, i.e.~we have $i:Y\to X$ and $r:X\to Y$ such that $r\circ i=\mathrm{id}_Y$. 
Note that for any type $Z$, the type $Z^Y$ is a retract of $Z^X$, since we have the section-retraction pair
\begin{equation*}
\begin{tikzcd}
Z^Y \arrow[r,"\blank\circ r"] & Z^X \arrow[r,"\blank\circ i"] & Z^Y.
\end{tikzcd}
\end{equation*}
It follows that for any type sequence $\mathcal{A}$ we have a section-retraction pair
\begin{equation*}
\begin{tikzcd}
\tfcolim(\mathcal{A}^Y) \arrow[r] \arrow[d] & \tfcolim(\mathcal{A}^X) \arrow[r] \arrow[d] & \tfcolim(\mathcal{A}^Y) \arrow[d]\\
A_\infty^Y \arrow[r] & A_\infty^X \arrow[r] & A_\infty^Y
\end{tikzcd}
\end{equation*}
of morphisms. Note that the downward morphism in the middle is an equivalence by the assumption that $X$ is compact. Recall from Theorem 4.7.4 of \cite{hottbook} that a retract of an equivalence is again an equivalence, so the claim follows.
\end{proof}

\section{Localizing at maps between compact types}\label{sec:localization_compact}
In this section, we assume to have $P,Q:I\to\UU$ and a family of maps $F:\prd{i:I} P(i)\to Q(i)$.

Recall type of diagonal fillers for a commuting square
\begin{equation*}
\begin{tikzcd}
A \arrow[d,swap,"f"] \arrow[r,"i"] & B \arrow[d,"g"] \\
X \arrow[r,swap,"j"] & Y
\end{tikzcd}
\end{equation*}
with $H:j\circ f\htpy g\circ i$ is defined to be the fiber of the gap map
\begin{equation*}
\mathsf{gap}:B^X \to B^A\times_{Y^A} Y^X
\end{equation*}
at the point $(i,j,\mathsf{eq\usc{}htpy}(H))$.

\begin{defn}
A \define{quasi $F$-local extension} of $X$ consists of a map $l:X\to Y$ equipped with a diagonal filler for the commuting square
\begin{equation*}
\begin{tikzcd}
X^{Q_i} \arrow[r,"l\circ\blank"] \arrow[d,swap,"\blank\circ F_i"] & Y^{Q_i} \arrow[d,"\blank\circ F_i"] \\
X^{P_i} \arrow[r,swap,"l\circ \blank"] & Y^{P_i},
\end{tikzcd}
\end{equation*}
for each $i:I$.
\end{defn}

\begin{lem}\label{lem:qflocal_local}
For any map $l:X\to Y$ into an $F$-local type $Y$, the type of diagonal fillers for the commuting square
\begin{equation*}
\begin{tikzcd}
X^{Q_i} \arrow[r,"l\circ\blank"] \arrow[d,swap,"\blank\circ F_i"] & Y^{Q_i} \arrow[d,"\blank\circ F_i"] \\
X^{P_i} \arrow[r,swap,"l\circ \blank"] & Y^{P_i},
\end{tikzcd}
\end{equation*}
is contractible. In particular, any such map is a quasi $F$-local extension.
\end{lem}

\begin{proof}
If $Y$ is $F$-local, then the map $F_i^\ast : Y^{Q_i}\to Y^{P_i}$ is an equivalence, for any $i:I$. Since equivalences are right orthogonal to any map, the claim follows.
\end{proof}

\begin{lem}\label{lem:equiv_colim_diagfiller}
Suppose $h:\mathsf{Seq}(\mathcal{A},\mathcal{B})$ is a natural transformation of type sequences equipped with the structure of a diagonal filler for each naturality square
\begin{equation*}
\begin{tikzcd}
A_n \arrow[r,"f_n"] \arrow[d,swap,"h_n"] & A_{n+1} \arrow[d,"h_{n+1}"] \\
B_n \arrow[r,swap,"g_n"] \arrow[ur,densely dotted] & B_{n+1}.
\end{tikzcd}
\end{equation*}
Then the map $h_\infty : A_\infty\to B_\infty$ is an equivalence. 
\end{lem}

\begin{proof}
Let us write $j_n:B_n\to A_{n+1}$ for the diagonal fillers, which come equipped with homotopies
\begin{align*}
J_n & : f_n\htpy j_n\circ h_n \\
K_n & : h_{n+1}\circ j_n \htpy g_n \\
L_n & : \ct{(h_{n+1}\cdot J_n)}{(K_n\cdot h_n)} \htpy H_n
\end{align*}
Note that $j$ is a natural transformation from $\mathcal{B}$ to the shifted sequence $\mathsf{shift}(\mathcal{A})$, with the homotopies $\ct{(f_{n+1}\cdot K_n)}{(J_{n+1}\cdot g_n)}: f_{n+1}\circ j_n \htpy j_{n+1}\circ g_n$ filling the naturality squares.

We claim that the composite $j\circ h$ is just the natural transformation $f:\mathcal{A}\to \mathsf{shift}(\mathcal{A})$. Indeed, we have the homotopy $J_n : j_n\circ h_n\htpy f_n$. Moreover, the type of $j_{n+1}\cdot L_n$
\begin{equation*}
\ct{(j_{n+1}\cdot h_{n+1}\cdot J_n)}{(j_{n+1}\cdot K_n\cdot h_n)} \htpy j_{n+1}\cdot H_n
\end{equation*}
is equivalent to the type of homotopies
\begin{align*}
& \ct{(J_{n+1}\cdot f_n)}{(j_{n+1}\cdot h_{n+1} \cdot J_n)}{(j_{n+1}\cdot K_n\cdot h_n)}{(j_{n+1}\cdot H_n)} \\
& \qquad\qquad \htpy \ct{\mathsf{htpy\usc{}refl}_{(f_{n+1}\circ f_n)}}{(J_{n+1}\cdot f_n)},
\end{align*}
filling the diagram
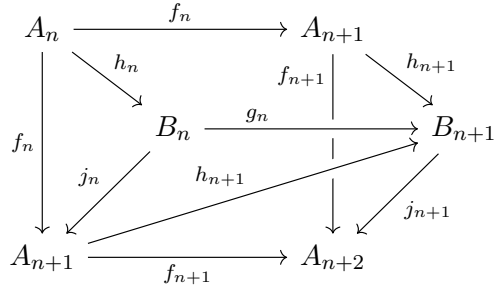
\begin{equation*}
\begin{tikzcd}
A_n \arrow[rr,"f_n"] \arrow[dr,"h_n"] \arrow[dd,swap,"f_n"] &[-1em] & A_{n+1} \arrow[dd,swap,"f_{n+1}" very near start] \arrow[dr,"h_{n+1}"] &[-1em] \\
& B_n \arrow[dl,swap,"j_n"] \arrow[rr,crossing over,"g_n" near start] & & B_{n+1} \arrow[dl,"j_{n+1}"] \\[1em]
A_{n+1} \arrow[urrr,crossing over,"h_{n+1}"] \arrow[rr,swap,"f_{n+1}"] & \phantom{A_{n+2}} & A_{n+2} & \phantom{A_{n+2}}
\end{tikzcd}
\end{equation*}
Therefore we obtain the coherence of the naturality squares needed to conclude that $f=j\circ h$ as natural transformations.

Likewise, the composite $\mathsf{shift}(h) \circ j$ is just the natural transformation $g:\mathcal{B}\to \mathsf{shift}(\mathcal{B})$. Therefore we obtain a commuting diagram
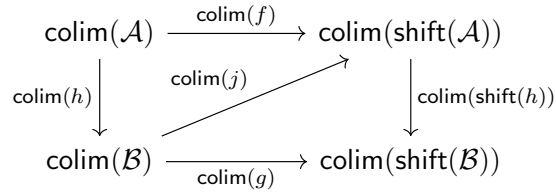
\begin{equation*}
\begin{tikzcd}[column sep=huge,row sep=large]
\mathsf{colim}(\mathcal{A}) \arrow[d,swap,"\mathsf{colim}(h)"] \arrow[r,"\mathsf{colim}(f)"] & \mathsf{colim}(\mathsf{shift}(\mathcal{A})) \arrow[d,"\mathsf{colim}(\mathsf{shift}(h))"] \\
\mathsf{colim}(\mathcal{B}) \arrow[r,swap,"\mathsf{colim}(g)"] \arrow[ur,"\mathsf{colim}(j)"] & \mathsf{colim}(\mathsf{shift}(\mathcal{B}))
\end{tikzcd}
\end{equation*}
Since both $\tfcolim(f)$ and $\tfcolim(g)$ are equivalences, it follows by the 6-for-2 property of equivalences (see Exercise 4.5 of \cite{hottbook}), that the remaining maps are equivalences. In particular $\tfcolim(h)$ is an equivalence. 
\end{proof}

\begin{prp}\label{prp:colim_local}
Suppose $F$ is a family of maps between compact types, and let
\begin{equation*}
\begin{tikzcd}
X_0 \arrow[r,"l"] & X_1 \arrow[r,"l"] & X_2 \arrow[r,"l"] & \cdots
\end{tikzcd}
\end{equation*}
be a sequence of types in which each $X_{n+1}$ is a quasi $F$-local extension of $X_n$. Then the sequential colimit $X_\infty$ is $F$-local.
\end{prp}

\begin{proof}
We have a natural transformation $F_i^\ast : \mathsf{Seq}(\mathcal{X}^{Q_i},\mathcal{X}^{P_i})$ for which each naturality square comes equipped with a diagonal filler
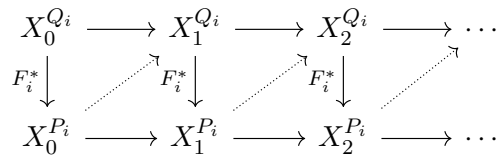
\begin{equation*}
\begin{tikzcd}
X_0^{Q_i} \arrow[d,swap,"F_i^\ast"] \arrow[r] & X_1^{Q_i} \arrow[d,swap,"F_i^\ast"] \arrow[r] & X_2^{Q_i} \arrow[d,swap,"F_i^\ast"] \arrow[r] & \cdots \\
X_0^{P_i} \arrow[r] \arrow[ur,densely dotted] & X_1^{P_i} \arrow[r] \arrow[ur,densely dotted] & X_2^{P_i} \arrow[r] \arrow[ur,densely dotted] & \cdots
\end{tikzcd}
\end{equation*}
by the assumption that each $l_n:X_n\to X_{n+1}$ is a quasi $F$-local extension. Therefore $\mathsf{colim}(F_i^\ast)$ is an equivalence. Note that we have the commuting square
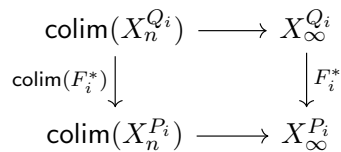
\begin{equation*}
\begin{tikzcd}
\mathsf{colim}(X_n^{Q_i}) \arrow[r] \arrow[d,swap,"\mathsf{colim}(F_i^\ast)"] & X_{\infty}^{Q_i} \arrow[d,"F_i^\ast"] \\
\mathsf{colim}(X_n^{P_i}) \arrow[r] & X_\infty^{P_i}
\end{tikzcd}
\end{equation*}
in which the top and bottom maps are equivalences by the assumption that $P_i$ and $Q_i$ are compact. Therefore we conclude that $F_i^\ast : X_\infty^{Q_i}\to X_\infty^{P_i}$ is an equivalence. It follows that $X_\infty$ is $F$-local.
\end{proof}

Our goal is now to construct the initial quasi $F$-local extension of a type $X$, and use it to show that the subuniverse of $F$-local types is reflective. For any $i:I$, we have the maps
\begin{align*}
F_i : P_i \to Q_i \\
F_i^\ast : X^{Q_i} \to X^{P_i}.
\end{align*}
In the following characterization of quasi $F$-local extensions we use the pushout-product $F_i^\ast\mathbin{\square} F_i$, which is defined as follows
\begin{equation*}
\begin{tikzcd}
X^{Q_i}\times P_i \arrow[r] \arrow[d] & X^{Q_i}\times Q_i \arrow[ddr,bend left=15] \arrow[d] \\
X^{P_i}\times P_i \arrow[drr,bend right=15] \arrow[r] & \big(X^{P_i} \times P_i\big)\sqcup^{\big(X^{Q_i}\times P_i\big)}\big(X^{Q_i}\times Q_i\big) \arrow[dr,densely dotted,swap,"F_i^\ast\mathbin{\square} F_i"] \\
& & X^{P_i}\times Q_i
\end{tikzcd}
\end{equation*}
by the universal property of pushouts. Also note that the outer square in the diagram
\begin{equation*}
\begin{tikzcd}
X^{Q_i}\times P_i \arrow[r] \arrow[d] & X^{Q_i}\times Q_i \arrow[ddr,bend left=15,"\mathsf{ev}"] \arrow[d] &[2em] \\
X^{P_i}\times P_i \arrow[drr,bend right=15,swap,"\mathsf{ev}"] \arrow[r] & \big(X^{P_i} \times P_i\big)\sqcup^{\big(X^{Q_i}\times P_i\big)}\big(X^{Q_i}\times Q_i\big) \arrow[dr,densely dotted,swap,"\varepsilon_i"] \\
& & X
\end{tikzcd}
\end{equation*}
commutes, so we get a map $\varepsilon_i$ as indicated.

\begin{lem}\label{lem:qflocal_characterize}
For any type $l:X\to Y$, the type of $I$-indexed families of diagonal fillers for the squares
\begin{equation*}
\begin{tikzcd}
X^{Q_i} \arrow[r,"l\circ\blank"] \arrow[d,swap,"\blank\circ F_i"] & Y^{Q_i} \arrow[d,"\blank\circ F_i"] \\
X^{P_i} \arrow[r,swap,"l\circ \blank"] & Y^{P_i}
\end{tikzcd}
\end{equation*}
is equivalent to the type of morphisms $j: \Big(\sm{i:I}X^{P_i}\times Q_i\Big)\to Y$ equipped with a homotopy witnessing that the square
\begin{equation*}
\begin{tikzcd}[column sep=6em]
\sm{i:I}\Big(\big(X^{P_i} \times P_i\big)\sqcup^{\big(X^{Q_i}\times P_i\big)}\big(X^{Q_i}\times Q_i\big)\Big) \arrow[r,"\total{\lam{i}F_i^\ast\mathbin{\square} F_i}"] \arrow[d,swap,"\lam{(i,t)}\varepsilon_i(t)"] & \sm{i:I} X^{P_i}\times Q_i \arrow[d] \\
X \arrow[r] & Y
\end{tikzcd}
\end{equation*}
commutes.
\end{lem}

\begin{proof}
It suffices to show that for each $i:I$, the type of diagonal fillers in the first description is equivalent to the type of maps $j:X^{P_i}\times Q_i\to Y$ equipped with a homotopy witnessing that the square
\begin{equation*}
\begin{tikzcd}[column sep=huge]
\Big(\big(X^{P_i} \times P_i\big)\sqcup^{\big(X^{Q_i}\times P_i\big)}\big(X^{Q_i}\times Q_i\big)\Big) \arrow[r,"F_i^\ast\mathbin{\square} F_i"] \arrow[d,swap,"\varepsilon_i"] & X^{P_i}\times Q_i \arrow[d] \\
X \arrow[r] & Y
\end{tikzcd}
\end{equation*}
commutes. The type $X^{P_i}\to Y^{Q_i}$ is equivalent to the type $X^{P_i}\times Q_i\to Y$ by adjointness. 

By the dependent universal property of pushouts established in \cref{thm:pushout_up}, the type $l\circ \varepsilon_i \htpy j\circ (F_i^\ast\mathbin{\square} F_i)$ is equivalent to the type of triples $(K,L,M)$ consisting of
\begin{align*}
K & : \prd{f:P_i\to X}{p:P_i} l(f(p))=j(f,F_i(p)) \\ 
L & : \prd{g:Q_i\to X}{q:Q_i} l(g(q))=j(g\circ F_i,q) \\
M & : \prd{g:Q_i\to X}{p:P_i} K(g\circ F_i,p)=L(g,F_i(p)).
\end{align*}
The type of such quadruples $(j,K,L,M)$ is also equivalent to the fiber of the gap map
\begin{equation*}
(X^{P_i}\to Y^{Q_i})\to \Big(X^{Q_i}\to Y^{Q_i}\Big)\times_{\big(X^{Q_i}\to Y^{P_i}\big)}\Big(X^{P_i}\to Y^{P_i}\Big)
\end{equation*}
at $(l^{Q_i},l^{P_i},\mathsf{htpy\usc{}refl})$. 
\end{proof}

\begin{defn}
For any type $X$, we define the initial quasi $F$-local extension $l:X\to QL_FX$ of $X$ to be the pushout
\begin{equation*}
\begin{tikzcd}[column sep=large]
\sm{i:I}\Big(\big(X^{P_i} \times P_i\big)\sqcup^{\big(X^{Q_i}\times P_i\big)}\big(X^{Q_i}\times Q_i\big)\Big) \arrow[r,"F_i^\ast\mathbin{\square} F_i"] \arrow[d,swap,"\lam{(i,t)}\varepsilon_i(t)"] & \sm{i:I} X^{P_i}\times Q_i \arrow[d] \\
X \arrow[r,swap,"l"] & QL_FX
\end{tikzcd}
\end{equation*}
\end{defn}

\begin{rmk}
The initial quasi $F$-local extension of a type $X$ is initial by the universal property of the pushout and the characterization of \cref{lem:qflocal_characterize}.
\end{rmk}

\begin{eg}
Consider the family of maps consisting just of the terminal projection $\bool\to\unit$. We sketch an argument that the quasi $\bool$-local extension of $X$ is equivalent to $\join{X}{X}$, i.e.~that we have a pushout square
\begin{equation*}
\begin{tikzcd}
(X^2\times \bool)\sqcup^{X\times \bool} X \arrow[r] \arrow[d] & X^2 \arrow[d] \\
X \arrow[r] & \join{X}{X}
\end{tikzcd}
\end{equation*}
To see this, we note first that the pushouts
\begin{equation*}
\begin{tikzcd}
S \arrow[d] \arrow[r] & B \arrow[d] & S+S \arrow[d] \arrow[r] & S \arrow[d] \\
A \arrow[r] & A \sqcup^S B & A+B \arrow[r] & (A+B)\sqcup^{(S+S)} S
\end{tikzcd}
\end{equation*}
are equivalent, for any span $A \leftarrow S \rightarrow B$. Therefore it follows that the pushouts
\begin{equation*}
\begin{tikzcd}
X \arrow[d] \arrow[r] & X^2 \arrow[d] & X\times \bool \arrow[d] \arrow[r] & X \arrow[d] \\
X^2 \arrow[r] & X^2 \sqcup^X X^2 & X^2\times \bool \arrow[r] & (X^2\times\bool)\sqcup^{(X\times\bool)} X
\end{tikzcd}
\end{equation*}
are equivalent. Therefore it suffices to compute the pushout
\begin{equation*}
\begin{tikzcd}
X^2 \sqcup^X X^2 \arrow[r,"\nabla_{\delta_X}"] \arrow[d,swap,"{[\pi_1,\pi_2,\mathsf{htpy\usc{}refl}_{\idfunc[X]}]}"] & X^2 \arrow[d] \\
X \arrow[r] & X\sqcup^{(X^2 \sqcup^X X^2)}X^2.
\end{tikzcd}
\end{equation*}
Now we observe that for any cube
\begin{equation*}
\begin{tikzcd}
& A_{111} \arrow[dl] \arrow[d] \arrow[dr] \\
A_{110} \arrow[d] & A_{101} \arrow[dl] \arrow[dr] & A_{011} \arrow[d] \\
A_{100} \arrow[dr] & A_{010} \arrow[d] \arrow[from=ul,crossing over] \arrow[from=ur,crossing over] & A_{001} \arrow[dl] \\
& A_{000}
\end{tikzcd}
\end{equation*}
the following are equivalent:
\begin{enumerate}
\item The cube is cocartesian.
\item The the square
\begin{equation*}
\begin{tikzcd}
A_{011}\sqcup^{A_{111}} A_{101} \arrow[r] \arrow[d] & A_{001} \arrow[d] \\
A_{010}\sqcup^{A_{110}} A_{100} \arrow[r] & A_{000}
\end{tikzcd}
\end{equation*}
is cocartesian.
\end{enumerate}
We cite \cite{Munson} for this observation, and the analogous result in homotopy type theory is work in progress. Using the above equivalence, it follows that the pushout $X^2\sqcup ^X X^2$ is the colimit of the diagram
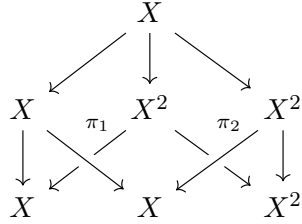
\begin{equation*}
\begin{tikzcd}
& X \arrow[dl] \arrow[d] \arrow[dr] \\
X \arrow[d] & X^2 \arrow[dl,swap,"\pi_1" very near start] \arrow[dr] & X^2 \arrow[dl,crossing over,swap,"\pi_2" very near start] \arrow[d] \\
X & X \arrow[from=ul,crossing over] & X^2
\end{tikzcd}
\end{equation*}
Furthermore, by the same equivalence we see that another way of computing this colimit is as the following pushout
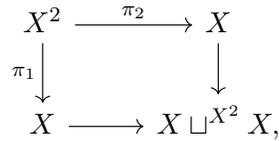
\begin{equation*}
\begin{tikzcd}
X^2 \arrow[r,"\pi_2"] \arrow[d,swap,"\pi_1"] & X \arrow[d] \\
X \arrow[r] & X\sqcup^{X^2} X,
\end{tikzcd}
\end{equation*}
which is the join $\join{X}{X}$.
\end{eg}

\begin{prp}
If $X$ is $F$-local, then $l:X\to QL_F X$ is an equivalence.
\end{prp}

\begin{proof}
If $X$ is $F$-local, then each $F_i^\ast$ is an equivalence. It follows that the pushout-product $F_i^\ast\mathbin{\square} F_i$ is an equivalence. Since pushouts of equivalences are equivalences, the claim follows.
\end{proof}

\begin{thm}\label{thm:localization}
For any type $X$ and any family $F:\prd{a:A}P(a)\to Q(a)$ of maps between
compact types, the subuniverse of $F$-local types is reflective.
\end{thm}

\begin{proof}
We define the localization $\eta:X\to LX$ as the sequential colimit
\begin{equation*}
\begin{tikzcd}
X \arrow[r] & QL_F X \arrow[r] & QL_F^2 X \arrow[r] & \cdots
\end{tikzcd}
\end{equation*}
Since each $l:QL_F^n X\to QL_F^{n+1}X$ has the structure of a quasi $F$-local extension, it follows by \cref{prp:colim_local} that $LX$ is $F$-local. It remains to show that for any $F$-local type $Y$, the precomposition map
\begin{equation*}
\eta^\ast : (LX\to Y)\to (X\to Y)
\end{equation*}
is an equivalence.

We first show that the map
\begin{equation*}
l^\ast : (QL_F X\to Y)\to (X\to Y)
\end{equation*}
is an equivalence. To see this, note that we have a commuting triangle
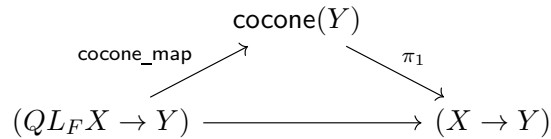
\begin{equation*}
\begin{tikzcd}[column sep=tiny]
& \mathsf{cocone}(Y) \arrow[dr,"\pi_1"] & \phantom{(QL_F X\to Y)} \\
(QL_F X\to Y) \arrow[rr] \arrow[ur,"\mathsf{cocone\usc{}map}"] & & (X\to Y)
\end{tikzcd}
\end{equation*}
We note that by \cref{lem:qflocal_characterize}, the fiber of $\pi_1:\mathsf{cocone}(Y)\to (X\to Y)$ is equivalent to the type of diagonal fillers of the square
\begin{equation*}
\begin{tikzcd}
X^{Q_i} \arrow[r,"l\circ\blank"] \arrow[d,swap,"\blank\circ F_i"] & Y^{Q_i} \arrow[d,"\blank\circ F_i"] \\
X^{P_i} \arrow[r,swap,"l\circ \blank"] & Y^{P_i},
\end{tikzcd}
\end{equation*}
This type of diagonal fillers is contractible by \cref{lem:qflocal_local}, since $Y$ is assumed to be $F$-local. In other words, the map $\pi_1:\mathsf{cocone}(Y)\to (X\to Y)$ is an equivalence. Since $\mathsf{cocone\usc{}map}:(QL_F X\to Y)\to \mathsf{cocone}(Y)$ is an equivalence by the universal property of $QL_F X$, we conclude by the 3-for-2 property that the map
\begin{equation*}
l^\ast : (QL_F X\to Y)\to (X\to Y)
\end{equation*}
is an equivalence.

We conclude that any map $g:X\to Y$ extends uniquely to a cocone on the type sequence
\begin{equation*}
\begin{tikzcd}
X \arrow[r] & QL_F X \arrow[r] & QL_F^2 X \arrow[r] & \cdots,
\end{tikzcd}
\end{equation*}
and therefore it follows that $\eta^\ast:(LX\to Y)\to (X\to Y)$ is an equivalence.
\end{proof}

\begin{cor}
For any family $N:I\to \UU$ of compact types, the subuniverse of $N$-null types is a modality.
\end{cor}

\backmatter

\printbibliography

\end{document}